\documentclass[11pt]{article}
\usepackage[letterpaper, hmargin=1in, top=1.1in, bottom=1.3in, footskip=0.7in]{geometry} % 11pt loose
\usepackage{titlesec}
\titleformat{\subsection}[runin]{\normalfont\bfseries}{\thesubsection.}{.5em}{}[.]\titlespacing{\subsection}{0pt}{2ex plus .1ex minus .2ex}{.8em}
\titleformat{\subsubsection}[runin]{\normalfont\itshape}{\thesubsubsection.}{.3em}{}[.]\titlespacing{\subsubsection}{0pt}{1ex plus .1ex minus .2ex}{.5em}
\titleformat{\paragraph}[runin]{\normalfont\itshape}{\theparagraph.}{.3em}{}[.]\titlespacing{\paragraph}{0pt}{1ex plus .1ex minus .2ex}{.5em}

\usepackage[labelfont=sc,font=small,labelsep=period]{caption}
\usepackage{amsmath}
\usepackage{amssymb}
\usepackage{amsfonts}
\usepackage{latexsym}
\usepackage{amsthm}
\usepackage{amsxtra}
\usepackage{amscd}
\usepackage{bbm}
\usepackage{mathrsfs}
\usepackage{bm}
\usepackage{graphicx}

\usepackage{graphicx, color}

\definecolor{darkred}{rgb}{0.9,0,0.3}
\definecolor{darkblue}{rgb}{0,0.3,0.9}
\definecolor{darkgreen}{rgb}{0,0.8,0.2}

\definecolor{vdarkred}{rgb}{0.6,0,0.2}
\definecolor{vdarkblue}{rgb}{0,0.2,0.6}
\usepackage[pdftex, colorlinks, linkcolor=vdarkblue,citecolor=vdarkred]{hyperref}

\usepackage[nottoc,notlof,notlot]{tocbibind}
\usepackage{cite} %Enabling `[1-4]` - style citing
\numberwithin{equation}{section}
\numberwithin{figure}{section}
\theoremstyle{plain} %plain, definition, remark
\newtheorem{theorem}{Theorem}[section]
\newtheorem*{theorem*}{Theorem}
\newtheorem{lemma}[theorem]{Lemma}
\newtheorem*{lemma*}{Lemma}
\newtheorem{corollary}[theorem]{Corollary}
\newtheorem*{corollary*}{Corollary}
\newtheorem{proposition}[theorem]{Proposition}
\newtheorem*{proposition*}{Proposition}

\newtheorem*{conjecture*}{Conjecture}

\theoremstyle{definition} %plain, definition, remark
\newtheorem{definition}[theorem]{Definition}
\newtheorem*{definition*}{Definition}

\newtheorem*{example*}{Example}
\newtheorem{remark}[theorem]{Remark}
\newtheorem*{remark*}{Remark}

\newtheorem*{assumption*}{Assumption}

\renewcommand{\cal}{\mathcal}

\newcommand{\ee}{\mathrm{e}}
\newcommand{\ii}{\mathrm{i}}
\newcommand{\dd}{\mathrm{d}}
\newcommand{\col}{\mathrel{\vcenter{\baselineskip0.75ex \lineskiplimit0pt \hbox{.}\hbox{.}}}}
\newcommand*{\deq}{\mathrel{\vcenter{\baselineskip0.65ex \lineskiplimit0pt \hbox{.}\hbox{.}}}=}

\renewcommand{\leq}{\leqslant}
\renewcommand{\geq}{\geqslant}
\renewcommand{\epsilon}{\varepsilon}

\DeclareMathOperator{\tr}{Tr}

\DeclareMathOperator{\re}{Re}

\DeclareMathOperator{\conn}{conn}

\date{}
\begin{document}
\title{A microscopic derivation of Gibbs measures for nonlinear Schr\"{o}dinger equations with unbounded interaction potentials}
\maketitle
\begin{center}
\author{
Vedran Sohinger \footnote{University of Warwick, Mathematics Institute, Zeeman Building, Coventry CV4 7AL, United Kingdom.
\\
Email: {\tt V.Sohinger@warwick.ac.uk}.}}
\end{center}
\begin{abstract}
We study the derivation of the Gibbs measure for the nonlinear Schr\"{o}dinger equation (NLS) from many-body quantum thermal states in the high-temperature limit. In this paper, we consider the nonlocal NLS with defocusing and unbounded $L^p$ interaction potentials on $\mathbb{T}^d$ for $d=1,2,3$. This extends the author's earlier joint work with Fr\"{o}hlich, Knowles, and Schlein \cite{FrKnScSo1}, where the regime of defocusing and bounded interaction potentials was considered. When $d=1$, we give an alternative proof of a result previously obtained by Lewin, Nam, and Rougerie \cite{Lewin_Nam_Rougerie}. 

Our proof is based on a perturbative expansion in the interaction. When $d=1$, the thermal state is the grand canonical ensemble. As in \cite{FrKnScSo1}, when $d=2,3$, the thermal state is a modified grand canonical ensemble, which allows us to estimate the remainder term in the expansion. The terms in the expansion are analysed using a graphical representation and are resummed by using Borel summation. By this method, we are able to prove the result for the optimal range of $p$ and obtain the full range of defocusing interaction potentials which were studied in the classical setting when $d=2,3$ in the work of Bourgain \cite{Bourgain_1997}.
\end{abstract}

\tableofcontents

\section{Introduction}

\subsection{Setup of problem}

We consider the domain $\Lambda=\mathbb{T}^d=\mathbb{R}^d/\mathbb{Z}^d$, where $d=1,2,3$ with the standard operations of addition and subtraction.
The \emph{one-body Hamiltonian} is given by
\begin{equation}
\label{one_body_Hamiltonian}
h\;=\;-\Delta+\kappa\,,
\end{equation}
for a fixed chemical potential $\kappa>0$. This is a densely-defined positive operator on $\mathfrak{H} \deq L^2(\Lambda)$.
The eigenvalues of $h$ are 
\begin{equation}
\label{lambda_k}
\lambda_k\;=\;4\pi^2 |k|^2+\kappa\,,\quad k \in \mathbb{Z}^d\,,
\end{equation}
with corresponding $L^2$-normalised eigenvectors 
\begin{equation}
\label{e_k}
e_k(x)\;=\;\ee^{2 \pi \ii k \cdot x}\,.
\end{equation}

We study the \emph{nonlinear Schr\"{o}dinger equation (NLS)}
\begin{equation}
\label{NLS_introduction}
\ii \partial_t u + (\Delta-\kappa) u = (w*|u|^2)\,u\,,
\end{equation}
where $w \in L^p(\Lambda)$ for some $1 \leq p \leq \infty$ is either \emph{pointwise nonnegative} or of \emph{positive type}, i.e.\ $\hat{w} \geq 0$ pointwise. The equation \eqref{NLS_introduction} is sometimes referred to as the \emph{nonlocal NLS} or the \emph{Hartree equation}. 
Furthermore $w$ is referred to as the \emph{interaction potential}.
The NLS \eqref{NLS_introduction} corresponds to the Hamiltonian equations of motion associated with 
the Hamiltonian
\begin{equation}
\label{classical_Hamiltonian_introduction}
H(u)\;=\;\int_{\Lambda}\,\dd x\,\big(|\nabla u|^2+\kappa |u|^2\big)+\frac{1}{2}\int_{\Lambda}\,\dd x\,\int_{\Lambda}\,\dd y\, |u(x)|^2\,w(x-y)\,|u(y)|^2\,
\end{equation}
acting on the space of fields $u: \Lambda \rightarrow \mathbb{C}$, where the Poisson bracket is given by
\begin{equation*}
\{u(x),\bar{u}(y)\}\;=\; \ii \delta(x-y)\,, \quad \{u(x),u(y)\}\;=\;\{\bar{u}(x), \bar{u}(y)\}\,.
\end{equation*}

The \emph{Gibbs measure} associated with the Hamiltonian \eqref{classical_Hamiltonian_introduction}
is the probability measure $\mathbf{P}$ on the space of fields $u:\Lambda \rightarrow \mathbb{C}$ formally given by
\begin{equation}
\label{Gibbs_measure_introduction}
\dd \mathbf{P} (u)\;\deq\; \frac{1}{Z}\ee^{-H(u)} \,\dd u\,,
\end{equation}
for $Z$ a (positive) normalisation constant and $\dd u$ the formally defined Lebesgue measure on the space of fields.
The problem of the rigorous construction of probability measures as in \eqref{Gibbs_measure_introduction} was first addressed in the constructive field theory literature in the 1970s, see \cite{Glimm_Jaffe,Nelson1,Simon74} and the references therein. We also refer the reader to the subsequent references \cite{LRS,McKean_Vaninsky1,McKean_Vaninsky2}.
The invariance of \eqref{Gibbs_measure_introduction} under the flow of \eqref{NLS_introduction} was rigorously established in the work of Bourgain \cite{B,Bourgain_ZS,B1,B3} and Zhidkov \cite{Zhidkov}. Subsequently, this led to the study of global solutions of NLS-type equations with random initial data of low regularity, see
\cite{BourgainBulut,BourgainBulut2,BourgainBulut4,BrydgesSlade,BurqThomannTzvetkov,Cacciafesta_deSuzzoni1,DengTzvetkovVisciglia,GLV1,GLV2,NORBS,NRBSS,OQ,Tz1,Tz}. In this context, the invariance of \eqref{Gibbs_measure_introduction} serves as a substitute of a conservation law at low regularity.

The NLS \eqref{NLS_introduction} can be viewed as \emph{classical limit} of many-body quantum dynamics.
More precisely, given $n \in \mathbb{N}$, we consider the $n$-body Hamiltonian 
\begin{equation} 
\label{n-body_Hamiltonian_introduction}
H^{(n)}\;\deq\;\sum_{i = 1}^n\big(-\Delta_{x_i}+\kappa\big)+\lambda \sum_{1 \leq i < j\leq n}w(x_{i}-x_{j})\,,
\end{equation}
which acts on the bosonic Hilbert space $L^2_{\mathrm{sym}}(\Lambda^n)$. This is defined to be the subspace of elements of $L^2(\Lambda^n)$ which are invariant under permutation of the arguments $x_1,\ldots,x_n$. The interaction strength $\lambda>0$ is 
taken to be of order $1/n$, thus implying that both terms in \eqref{n-body_Hamiltonian_introduction} are of comparable size. 
Given \eqref{n-body_Hamiltonian_introduction}, the \emph{$n$-body Schr\"{o}dinger equation} is 
\begin{equation}
\label{n-body_Schrodinger_equation_introduction}
\ii \partial_t \Psi_{n,t}\;=\;H^{(n)}\,\Psi_{n,t}\,.
\end{equation}
One is interested in studying the limit as $n \rightarrow \infty$ and comparing the limiting dynamics in \eqref{n-body_Schrodinger_equation_introduction} with suitably chosen initial data to that in \eqref{NLS_introduction}. The first rigorous result of this type was proved by Hepp \cite{H} and it was extended by Ginibre and Velo to more singular interactions \cite{GV}. Subsequently, this problem was studied in various different contexts. For further results, we refer the reader to \cite{AFP,AN,CLS,CP,CH,CH2,CH_2018,ESY1,ESY2,ESY3,ESY4,ESY5,EESY,ES,EY,FKP,FKS,FKS,HS,KP,RS,S2,S,2S} and the references therein. The problem of quantum fluctuations around the classical dynamics has been studied in \cite{BCS,BKS,BNNS,BSS,XC,GM,GMM1,GMM2,LNS,NN}.

In this paper, we study Gibbs states associated with \eqref{n-body_Hamiltonian_introduction}. These are equilibrium states of \eqref{n-body_Hamiltonian_introduction} at a given temperature $\tau>0$. More precisely, the Gibbs state at $\tau>0$ is the operator on $L^2_{\mathrm{sym}}(\Lambda^n)$ given by
\begin{equation}
\label{Gibbs_state_introduction}
\frac{1}{Z_{\tau}^{(n)}}\,\ee^{-H^{(n)}/\tau}\,,
\end{equation}
where $Z_{\tau}^{(n)}\;\deq\;\tr\,\ee^{-H^{(n)}/\tau}$. Note that then the operator \eqref{Gibbs_state_introduction} has trace equal to $1$.

Our goal is to relate the Gibbs states \eqref{Gibbs_state_introduction} to Gibbs measures \eqref{Gibbs_measure_introduction} when the temperature $\tau$ tends to infinity. 
The first result in this direction was obtained by Lewin, Nam, and Rougerie \cite{Lewin_Nam_Rougerie}. The precise notion of convergence is that of the corresponding ($r$-particle) correlation functions, which we henceforth refer to as the \emph{microscopic derivation of the Gibbs measure}. For precise statements, see Section \ref{Statement of the main results} below. 
In \cite{Lewin_Nam_Rougerie}, the authors treat the $d=1$ problem as well as the problem in higher dimensions with non-local and non-translation invariant interactions. The approach in \cite{Lewin_Nam_Rougerie} is based on a variational method and the de Finetti theorem. In the author's joint work with Fr\"{o}hlich, Knowles, and Schlein \cite{FrKnScSo1}, the result is obtained for bounded, translation-invariant interactions when $d=2,3$. Here, the approach is based on a perturbative expansion in the interaction and a suitable resummation of the obtained terms. For technical reasons, in \cite{FrKnScSo1}, it is necessary to modify the grand canonical ensemble defined in \eqref{GCE} below and work with its modification \eqref{modified_GCE}. For $d=2$ and for suitably regular interaction potentials, the result was obtained for the unmodified grand canonical ensemble in \cite{Lewin_Nam_Rougerie3}.
When $d=1$, the regime of subharmonic trapping (including the harmonic oscillator) was studied in \cite{Lewin_Nam_Rougerie2}.
The time-dependent problem when $d=1$ was studied in \cite{FrKnScSo2}. The analogous problem was previously analysed on the lattice in \cite[Chapter 3]{Knowles_Thesis}.
In all of the aforementioned works on the continuum, one assumes a suitable positivity on the interaction, i.e.\ one works in the \emph{defocusing} regime.
An expository account of \cite{Lewin_Nam_Rougerie}, is given in \cite{Lewin_Nam_Rougerie_Survey1}. Furthermore, an expository account of \cite{Lewin_Nam_Rougerie3} is given in \cite{Lewin_Nam_Rougerie_Survey2,Lewin_Nam_Rougerie_Survey3}.

It is possible to study related problems in different regimes. When working with zero temperature, the system is in the ground state of \eqref{n-body_Hamiltonian_introduction}. In this case, one is interested in proving convergence of the ground state energy of \eqref{n-body_Hamiltonian_introduction} towards the ground state of \eqref{classical_Hamiltonian_introduction}. This has been studied in \cite{Bach,BL,BBCS1,Brennecke_Schlein,GS,K,LNR0,LNSS,LY,LS,FSV,RW,LSY,S}.
The regime of fixed temperature was studied in \cite{LNR0,LNSS}. For a more detailed discussion on the classical limit and equilibrium states, we refer the reader to the introduction of \cite{FrKnScSo1} and to the expository texts \cite{BenPorSch_Review,Golse_Review,Sch_Review}.

Throughout this paper, we consider interaction potentials $w$ which are either \emph{$d$-admissible} in the sense of Definition \ref{interaction_potential_w} or \emph{endpoint-admissible} in the sense of Definition \ref{endpoint_admissible_w} below. 
Note that the nonnegativity properties (ii)-(iii) of Definition \ref{interaction_potential_w} and property (ii) of Definition \ref{endpoint_admissible_w} correspond to the assumption that the nonlinearity is \emph{defocusing} \footnote{The condition (iii) in Definition \ref{endpoint_admissible_w} is needed for technical reasons and should not be thought of as part of the defocusing assumption, see Remark \ref{absolute_value_unbounded_negative_Sobolev}. It is possible that the condition can be relaxed, but we do not address this issue here.}. 
The case $p=\infty$ has already been considered in \cite{Lewin_Nam_Rougerie,FrKnScSo1}. The main contribution of this work is to obtain a microscopic derivation of the Gibbs measure for the NLS when $w$ does not  belong to $L^\infty(\Lambda)$. 
Hence in Definitions \ref{interaction_potential_w} and \ref{endpoint_admissible_w}, we always assume that $w \notin L^{\infty}(\Lambda)$. Our goal is to obtain the result for the \emph{optimal} range of integrability on $w$, as in \cite{Bourgain_1997}.

The results of this paper can be viewed as a step in the direction of studying more singular interaction potentials. In order to motivate this, we note that in the classical setting, the Gibbs measure is well-defined for a wide range of interactions, including very singular ones. The known results on the microscopic derivation of the Gibbs measure typically apply for sufficiently regular interaction potentials. In the long run, one would be interested in eliminating this discrepancy in the choice of interaction potentials.

%We would like to apply the perturbative expansion from \cite{FrKnScSo1}.
%A challenge in doing so is the fact that many of the methods used in \cite{FrKnScSo1} rely crucially on the boundedness of $w$.

\subsubsection*{Notation and conventions}
We use the convention that $\mathbb{N}=\{0,1,2,3,\ldots\}$. Throughout the paper, $C>0$ denotes a finite positive constant, which can vary from line to line. Given quantities $a_1,a_2,\ldots$, we write $C(a_1,a_2,\ldots)$ for a finite positive constant that depends only on these quantities.
We sometimes also write $X \lesssim_{a_1,a_2,\ldots} Y$ if $X \leq C(a_1,a_2,\ldots) Y$. Likewise, we write $X \gtrsim_{a_1,a_2,\ldots} Y$ if $Y \lesssim_{a_1,a_2,\ldots} X$. We write $X \lesssim Y$ and $X \gtrsim Y$ if we do not need to keep track of the parameters. If $X \lesssim Y$ and $Y \lesssim X$, we write $X \sim Y$.
We use the convention that positive constants with indices $C_0,C_1,C_2>0$ depend on the dimension $d$ and the chemical potential $\kappa$ in \eqref{one_body_Hamiltonian}. In this case, we will suppress the dependence on these quantities in the notation. 

In the sequel, we omit the integration domain $\Lambda$ if it is clear from context that this is the set over which we are integrating. In other words, we use $\int \dd x \equiv \int_{\Lambda} \dd x$.
We use the convention that inner products are linear in the second variable.

For a statement $A$ we denote by
\begin{equation*}
\mathbf{1}_A \;\deq\;
\begin{cases}
1\,&\mbox{if $A$ is true}
\\
0\,&\mbox{if $A$ is false}
\end{cases}
\end{equation*}
the corresponding indicator function.

\subsection{The classical system and Gibbs measures}
Let us consider the probability space $(\mathbb{C}^{\mathbb{Z}^d},\mathcal{G},\mu)$, where $\mathcal{G}$ denotes the product sigma-algebra, and $\mu \deq \bigotimes_{k \in \mathbb{Z}^d} \mu_k$, where for all $k \in \mathbb{Z}^d$, we have $\mu_k=\frac{1}{\pi}\ee^{-|z|^2}\dd z$ ($\dd z$ denotes Lebesgue measure on $\mathbb{C}$). In other words, the $\mu_k$ are independent standard complex Gaussians. The points of the probability space are denoted by $\omega = (\omega_k)_{k \in \mathbb{Z}^d} \in \mathbb{C}^{\mathbb{Z}^d}$. The \emph{classical free field} is defined by
\begin{equation}
\label{classical_free_field}
\phi\equiv\;\phi^{\omega}\;\deq\;\sum_{k \in \mathbb{Z}^d} \frac{\omega_k}{\sqrt{\lambda_k}}\ee^{2 \pi \ii k \cdot x}\,.
\end{equation}
One obtains that
\begin{equation}
\label{classical_free_field_regularity}
\phi \in L^2(\mu;H^s(\Lambda))\,\,\,\mbox{for}\,\, s<1-\frac{d}{2}\,.
\end{equation}
Here, $H^s(\Lambda)$ denotes the $L^2$-based inhomogeneous Sobolev space on $\Lambda$ of order $s$. In order to deduce \eqref{classical_free_field_regularity}, one uses \eqref{lambda_k} to see that $\tr h^{s-1}<\infty$.
We refer the reader to \cite[Section 1.2]{FrKnScSo1}
for further details on the construction of the classical free field if one takes more general one-body Hamiltonians $h$ in \eqref{one_body_Hamiltonian}. 

We now state the precise assumptions on the interaction potentials $w$ that we consider throughout the paper. There are two possibilities. The first type of interaction potential is defined for all $d=1,2,3$.
\begin{definition}($d$-admissible interaction potentials)
\label{interaction_potential_w}
\\
We say that $w : \Lambda \rightarrow \mathbb{C}$ is \textbf{$d$-admissible} if $w$ is an even function that satisfies the following properties.
\begin{itemize}
\item[(i)]  $w \in L^p(\Lambda)$ for $p \in \mathcal{P}_d$, where
\begin{equation}
\label{P_d}
\mathcal{P}_d \;\deq\;
\begin{cases}
[1,\infty)\,\,&\mbox{if}\,\,\,d=1
\\
(1,\infty) \,\,&\mbox{if}\,\,\,d=2
\\
(3,\infty) \,\,&\mbox{if}\,\,\,d=3.
\end{cases}
\end{equation}
\item[(ii)] If $d=1$, $w \geq 0$ \emph{pointwise}.
\item[(iii)] If $d=2,3$, then $w$ is of \emph{positive type}.
\item[(iv)] $w \notin L^\infty(\Lambda)$.
\end{itemize}
\end{definition}

When $d=2$, we also consider the case when $w \in L^1(\Lambda)$, provided that we add further assumptions. 
Throughout the sequel, $\langle x \rangle \deq (1+|x|^2)^{1/2}$ denotes the Japanese bracket. 

\begin{definition}(Endpoint-admissible interaction potentials)
\label{endpoint_admissible_w}
\\
Let $d=2$. We say that $w: \Lambda \rightarrow \mathbb{C}$ is \textbf{endpoint-admissible} if $w$ is an even function that satisfies the following properties.
\begin{itemize}
\item[(i)] $w \in L^1(\Lambda)$.
\item[(ii)] $w$ is of \emph{positive type}. 
\item[(iii)] $w \geq 0$ \emph{pointwise}.
\item[(iv)] There exist $\epsilon>0$ and $L>0$ such that for all $k \in \mathbb{Z}^2$ we have $\hat{w}(k) \leq L \langle k \rangle^{-\epsilon}$.
\item[(v)] $w \notin L^{\infty}(\Lambda)$. 
\end{itemize}
\end{definition}
We note the the classes of interaction potentials given in Definitions \ref{interaction_potential_w} and \ref{endpoint_admissible_w} are indeed non-empty.
\begin{lemma}
\label{admissible_w_nonempty}
\begin{itemize}
Let $d=1,2,3$ be given.
\item[(i)] The class of $d$-admissible interaction potentials given by Definition \ref{interaction_potential_w} is non-empty.
\item[(ii)] For $d=2$ and $\epsilon>0$ small, the class of endpoint admissible interaction potentials given by Definition \ref{endpoint_admissible_w} is non-empty.
\end{itemize}
\end{lemma}
The proof of Lemma \ref{admissible_w_nonempty} is given in Appendix \ref{The_interaction_potential}.

In the sequel, we assume that the interaction potential $w$ is either $d$-admissible or endpoint-admissible as given by Definitions \ref{interaction_potential_w} and \ref{endpoint_admissible_w}. When $d=1$, the \emph{classical interaction} is defined as
\begin{equation}
\label{W_1D_unbounded}
W\;\deq\;\frac{1}{2} \int \dd x\,\dd y\, |\phi(x)|^2\,w(x-y)\,|\phi(y)|^2\,.
\end{equation}
Note that $W \geq 0$ almost surely by Definition \ref{interaction_potential_w} (ii). Furthermore, $W<\infty$ almost surely. Namely, by \eqref{classical_free_field_regularity} and Sobolev embedding, we have that $\phi \in L^4(\Lambda)$ almost surely. Since $w \in L^1(\Lambda)$ by Definition \ref{interaction_potential_w} (i), we use Young's and H\"{o}lder's inequality to conclude that 
\begin{equation}
\label{W_1D_unbounded2}
W\;\leq\;\|\phi\|_{L^4(\Lambda)}^4 \,\|w\|_{L^1(\Lambda)} \;<\;\infty\,
\end{equation}
almost surely.

When $d=2,3$, we need to perform a renormalisation of the interaction by using \emph{Wick ordering}.
More precisely, given $K \in \mathbb{N}$, we define the \emph{truncated classical free field}
\begin{equation}
\label{phi_K}
\phi_{[K]}\equiv\;\phi_{[K]}^{\omega}\;\deq\;\sum_{|k| \leq K} \frac{\omega_k}{\sqrt{\lambda_k}}\ee^{2 \pi \ii k \cdot x}\,.
\end{equation} 
and associated density\footnote{Note that since we are working on the torus with one-body Hamiltonian \eqref{one_body_Hamiltonian}, the quantity $\int \dd \mu \,|\phi_{[K]}(x)|^2$ is constant by translation invariance. Hence, we can take $\varrho_{[K]}$ to be a constant. In the setting of more general $\Lambda$ and $h$, this is a function of $x \in \Lambda$, see \cite[Section 1.6]{FrKnScSo1} for a detailed explanation.}
\begin{equation}
\label{varrho_K}
\varrho_{[K]}\;\deq\;\int \dd \mu \,|\phi_{[K]}(0)|^2\,.
\end{equation}
The \emph{truncated Wick-ordered classical interaction} is given by
\begin{equation}
\label{W_K_unbounded}
W_{[K]} \;\deq\; \frac{1}{2} \int \dd x\,\dd y\,\big(|\phi_{[K]}(x)|^2-\varrho_{[K]}\big)\,w(x-y)\,\big(|\phi_{[K]}(y)|^2-\varrho_{[K]}\big)\,.
\end{equation}
The Wick-ordered interaction is obtained as an appropriate limit of the $W_{[K]}$.
\begin{lemma}(Definition of $W$ for $d=2,3)$
\label{W_Wick-ordered}
\begin{itemize}
\item[(i)]
Let $d=2,3$ and let $w$ be $d$-admissible as in Definition \ref{interaction_potential_w}.
Then $(W_{[K]})_{K \in \mathbb{N}}$ is a Cauchy sequence in $\bigcap_{m \geq 1}L^m(\mu)$.
\item[(ii)]
Let $d=2$ and let $w$ be endpoint-admissible as in Definition \ref{endpoint_admissible_w}. Then the conclusion in part (i) also holds.
\end{itemize}
In both cases, we denote the corresponding limit by $W$. 
\end{lemma}
The proof of Lemma \ref{W_Wick-ordered} (i) is given in Section \ref{Expansion_classical} and the proof of Lemma \ref{W_Wick-ordered} (ii) is given in Section \ref{Endpoint-admissible interaction potentials_3}.

The \emph{classical state} $\rho(\cdot)$ associated with the one-body Hamiltonian $h$ and a $d$-admissible or endpoint-admissible interaction potential $w$ is defined as 
\begin{equation}
\label{rho_unbounded}
\rho(X)\;\deq\;\frac{\int X\,\ee^{-W}\,\dd \mu}{\int \ee^{-W}\,\dd \mu}\,.
\end{equation}
In \eqref{rho_unbounded}, $X$ is a random variable. Given $r \in \mathbb{N}$, the \emph{classical $r$-particle correlation} $\gamma_r$ is defined to be the operator on $\mathfrak{H}^{(r)}$ with kernel given by
\begin{equation}
\label{gamma_r_unbounded}
\gamma_r(x_1,\ldots,x_r;y_1,\ldots,y_r)\;\deq\;\rho\big(\bar{\phi}(y_1)\,\cdots\,\bar{\phi}(y_r)\,\phi(x_1)\,\cdots\,\phi(x_r)\big)\,.
\end{equation}
It can be shown that the family $(\gamma_r)_{r \in \mathbb{N}}$ determines the moments of the classical Gibbs state $\rho(\cdot)$, see \cite[Remark 1.2]{FrKnScSo1}.

\subsection{The quantum system and Gibbs states}
Given $n \in \mathbb{N}$, we denote by $\mathfrak{H}^{(n)}$ the \emph{$n$-particle space}, i.e.\ the bosonic subspace $L^2_{\mathrm{sym}}(\Lambda^n)$ of $L^2(\Lambda^n)$. The \emph{bosonic Fock space} is then defined as
\begin{equation*}
\mathcal{F}\;\equiv\;\mathcal{F}(\mathcal{\mathfrak{H}})\;\deq\;\bigoplus_{n \in \mathbb{N}}\mathfrak{H}^{(n)}\,.
\end{equation*}
On the $n$-particle space $\mathfrak{H}^{(n)}$, we work with $n$-body Hamiltonians of the form
\begin{equation}
\label{H^{(n}_unbounded}
H^{(n)}\;\deq\;\sum_{i=1}^{n}h_i+\lambda \sum_{1 \leq i <j \leq n} \tilde{w}(x_i-x_j)\,.
\end{equation}
In \eqref{H^{(n}_unbounded}, $h_i$ denotes the one-body Hamiltonian $h$ acting in the $x_i$-variable, $\lambda>0$ is the interaction strength and $\tilde{w} \in L^{\infty}(\Lambda)$ is an interaction potential. In particular, $\tilde{w}$ is not $d$-admissible\footnote{Here, one can consider unbounded interaction potentials. The methods in the sequel require bounded interactions. Hence we add this restriction.}. Note that $H^{(n)}$ is a densely-defined self-adjoint operator on $\mathfrak{H}^{(n)}$. 

As in \cite{FrKnScSo1,FrKnScSo2}, we introduce the large parameter $\tau>0$, which has the interpretation of the temperature. Our goal is to show that the classical state \eqref{rho_unbounded} is obtained as the $\tau \rightarrow \infty$ limit of thermal states of suitably chosen $\tau$-dependent $n$-body Hamiltonians of the form \eqref{H^{(n}_unbounded}. We now make this precise. Since we are interested in the limit $\tau \rightarrow \infty$, we always consider $\tau \geq 1$.

Let us first consider the case $d=1$ for a fixed temperature $\tau$.
We choose the coupling constant $\lambda=\frac{1}{\tau}$ in \eqref{H^{(n}_unbounded} (for the justification of this choice, see \cite[Section 1.4]{FrKnScSo1}). We first approximate a $1$-admissible interaction potential by bounded potentials. 

\begin{lemma}(Approximation of the interaction potential: $d=1$)
\label{w_1D_approximation}
\\
Let $d=1$. Let $1 \leq p <\infty$ and let $w \in L^p(\Lambda)$ be pointwise nonnegative. Given $\beta>0$ and $\tau \geq 1$, there exists $w_\tau \in L^p(\Lambda)$ with the following properties.
\begin{itemize}
\item[(i)] $w_\tau \geq 0$ pointwise.
\item[(ii)] $w_\tau \in L^{\infty}(\Lambda)$ and $\|w_\tau\|_{L^\infty(\Lambda)} \leq \tau^{\beta}$.
\item[(iii)] $\|w_{\tau}\|_{L^p(\Lambda)} \leq \|w\|_{L^p}$. 
\item[(iv)] $w_\tau \rightarrow w$ in $L^p(\Lambda)$ as $\tau \rightarrow \infty$.
\end{itemize}
\end{lemma}

\begin{proof}[Proof of Lemma \ref{w_1D_approximation}]
Let us take $w_\tau \;\deq\; w \, \mathbf{1}_{\,0 \leq w \leq \tau^{\beta}}$. Then $w_\tau$ satisfies the wanted properties.
\end{proof}
We hence take $\tilde{w}=w_\tau$ by applying Lemma \ref{w_1D_approximation} for $w$ a $1$-admissible interaction potential. In other words, in \eqref{H^{(n}_unbounded} we consider
\begin{equation}
\label{H^{(n}_unbounded_tau}
H^{(n)} \;\equiv\;H^{(n)}_{[\tau]}\;\deq\;\sum_{i=1}^{n}h_i+\frac{1}{\tau} \sum_{1 \leq i <j \leq n} w_{\tau}(x_i-x_j)\,,
\end{equation}
which we rescale by $\frac{1}{\tau}$ and extend to all of Fock space $\mathcal{F}$ to obtain the \emph{quantum Hamiltonian}
\begin{equation}
\label{H_tau1}
H_{\tau}\;\deq\;\frac{1}{\tau} \bigoplus_{n \in \mathbb{N}} H^{(n)}_{[\tau]}\,.
\end{equation}

We rewrite $H_\tau$ using second-quantisation. In order to do this, we first recall the relevant definitions.
Given $f \in \mathfrak{H}$, we define the \emph{bosonic annihilation and creation operators} $b(f)$ and $b^*(f)$ acting on $\mathcal{F}$. They act on $\Psi=(\Psi^{(n)})_{n \in \mathbb{N}} \in \mathcal{F}$ as
\begin{align}
\label{b(f)}
\big(b(f)\Psi\big)^{(n)}(x_1,\ldots,x_n)&\;\deq\; \sqrt{n+1} \int_{\Lambda} \,\dd x\, \bar{f}(x)\,\Psi^{(n+1)}(x,x_1,\ldots,x_n)\,,
\\
\label{b^*(f)}
\big(b^*(f)\Psi\big)^{(n)}(x_1,\ldots,x_n)&\;\deq\;\frac{1}{\sqrt{n}} \sum_{i=1}^{n} f(x_i)\,\Psi^{(n-1)}(x_1,\ldots,x_{i-1},x_{i+1},\ldots,x_n)\,.
\end{align}
The operators $b(f)$ and $b^*(f)$ defined in \eqref{b(f)}--\eqref{b^*(f)} are unbounded closed operators on $\mathcal{F}$, which are each other's adjoints. They satisfy the \emph{canonical commutation relations}, i.e.\ for all $f,g \in \mathfrak{H}$ we have
\begin{equation}
\label{CCR}
[b(f),b^{*}(g)]\;=\;\langle f, g \rangle\,,\quad [b(f),b(g)]\;=\;[b^{*}(f),b^{*}(g)]\;=\;0\,.
\end{equation}
In \eqref{CCR} $[\cdot,\cdot]$ denotes the commutator, namely $[A,B]=AB-BA$.

As in \cite{FrKnScSo1,FrKnScSo2}, we work with \emph{rescaled creation and annihilation operators}. Given $f \in \mathfrak{H}$, we define
\begin{equation}
\label{quantum_fields}
\phi_{\tau}(f)\;\deq\;\frac{1}{\sqrt{\tau}}\,b(f)\,,\quad \phi^{*}_{\tau}(f)\;\deq\;\frac{1}{\sqrt{\tau}}\,b^*(f)\,.
\end{equation}
By construction, $\phi_\tau$ and $\phi^*_\tau$ are operator-valued distributions and we can write
\begin{equation}
\label{distribution_kernels}
\phi_{\tau}(f)\;=\; \int \dd x\,\bar{f}(x)\,\phi_{\tau}(x)\,,\quad \phi^{*}_{\tau}(f)\;=\;\int \dd x\, f(x)\,\phi^{*}_{\tau}(x)\,.
\end{equation}
By \eqref{b(f)}--\eqref{quantum_fields}, the distribution kernels $\phi_{\tau}(x), \phi^{*}_{\tau}(x)$ given by \eqref{distribution_kernels} satisfy the commutation relations
\begin{equation}
\label{CCR_tau}
[\phi_{\tau}(x),\phi^{*}_{\tau}(y)]\;=\;\frac{1}{\tau}\delta(x-y)\,,\quad [\phi_{\tau}(x),\phi_{\tau}(y)]\;=\;[\phi^{*}_{\tau}(x),\phi^{*}_{\tau}(y)]\;=\;0\,.
\end{equation}
With the above notation, we can rewrite \eqref{H_tau1} as
\begin{equation} 
\label{H_tau2}
H_\tau \;=\;\int \dd x\,\dd y\,\phi^{*}_{\tau}(x)\,h(x;y)\,\phi_{\tau}(y)+\frac{1}{2} \int \dd x\,\dd y\, \phi^{*}_{\tau}(x)\,\phi^{*}_{\tau}(y)\,w_\tau (x-y)\,\phi_{\tau}(x)\,\phi_{\tau}(y)\,,
\end{equation}
where $h(x;y)=\sum_{k \in \mathbb{Z}^d}\lambda_k\,\ee^{2\pi \ii k \cdot(x-y)}$ is the operator kernel of \eqref{one_body_Hamiltonian}.

Let us now consider the case when $d=2,3$. We first note an approximation result that allows us to approximate the $d$-admissible interaction interaction potential by suitably chosen bounded interaction potentials.

\begin{lemma}(Approximation of a $d$-admissible interaction potential: $d=2,3$)
\label{w_positive_approximation}
\\
Let $d=2,3$. Let $1 \leq p<\infty$ and let $w \in L^p(\Lambda)$ be of positive type. Given $\beta>0$ and $\tau \geq 1$, there exists $w_\tau \in L^p(\Lambda)$ with the following properties.
\begin{itemize}
\item[(i)] $w_\tau$ is of positive type.
\item[(ii)] $w_\tau \in L^{\infty}(\Lambda)$ and $\|w_\tau\|_{L^\infty(\Lambda)} \leq \tau^{\beta}$.
\item[(iii)] $\|w_\tau\|_{L^p(\Lambda)} \leq C \|w\|_{L^p(\Lambda)}$, for some $C \equiv C(d)>0$.
\item[(iv)] $w_\tau \rightarrow w$ in $L^p(\Lambda)$ as $\tau \rightarrow \infty$.
\end{itemize}
\end{lemma}
When $d=2$ and the interaction potential is endpoint-admissible, we use a different approximation argument.
\begin{lemma}(Approximation of an endpoint-admissible interaction potential: $d=2$)
\\
\label{w_endpoint_admissible_approximation}
Let $d=2$ and let $w \in L^1(\Lambda)$ be endpoint-admissible as in Definition \ref{endpoint_admissible_w} above. Let $\delta \deq \epsilon/2$ for $\epsilon$ as in Definition \ref{endpoint_admissible_w} (iv). Let $\beta>0$ and $\tau \geq 1$ be given. Then, there exists $w_\tau \in H^{-1+\delta}(\Lambda) \cap L^1(\Lambda)$ with the following properties. 
\begin{itemize}
\item[(i)] $\|w_\tau\|_{H^{-1+\delta}(\Lambda)} \leq \|w\|_{H^{-1+\delta}(\Lambda)} \leq C L$.
\item[(ii)] $w_\tau$ is of positive type.
\item[(iii)] $\hat{w}_\tau \leq L$ for $L$ as in Definition \ref{endpoint_admissible_w} (iv).
\item[(iv)] $w_\tau \geq 0$ pointwise.
\item[(v)] $w_\tau \in L^{\infty}(\Lambda)$ and $\|w_\tau\|_{L^\infty(\Lambda)} \leq \tau^{\beta}$.
\item[(vi)] $\|w_\tau\|_{L^1(\Lambda)} \leq C \|w\|_{L^1(\Lambda)}$.
\item[(vii)] $w_\tau \rightarrow w$ in $H^{-1+\delta}(\Lambda)$ as $\tau \rightarrow \infty$.
\end{itemize}
\end{lemma}
We prove Lemma \ref{w_positive_approximation} and Lemma \ref{w_endpoint_admissible_approximation} in Appendix \ref{The_interaction_potential}.

For $d=2,3$, we need to Wick-order the quantum Hamiltonian in the same spirit as in the classical system. In order to do this, we first consider the \emph{free quantum Hamiltonian} 
\begin{equation} 
\label{H_{tau,0}_unbounded}
H_{\tau,0}\;\deq\; \int  \dd x \, \dd y\,\phi^{*}_{\tau}(x)\,h(x;y)\,\phi_{\tau}(y)\,.
\end{equation}
The \emph{free quantum state} $\rho_{\tau,0}(\cdot)$ associated with \eqref{H_{tau,0}_unbounded} is defined as
\begin{equation}
\label{rho_{tau,0}_unbounded}
\rho_{\tau,0}(\mathcal{A})\;\deq\;\frac{\tr(\mathcal{A}\,\ee^{- H_{\tau,0}})}{\tr(\ee^{-H_{\tau, 0}})}\,,
\end{equation}
for $\mathcal{A}$ a closed operator on $\mathcal{F}$. The \emph{quantum density} $\varrho_\tau$ is given by
\begin{equation} 
\label{def_varrho_tau_unbounded}
\varrho_\tau\;\deq\;\rho_{\tau,0}\big(\phi^*_\tau(0) \phi_\tau(0)\big)\,.
\end{equation}
When $d=2,3$, we work with the \emph{renormalised quantum Hamiltonian}, which is given by
\begin{equation}
\label{H_tau_renormalised_unbounded} 
H_{\tau}\;\deq\;H_{\tau,0}+W_{\tau}\,,
\end{equation}
where $W_\tau$ denotes the \emph{renormalised quantum interaction}
\begin{equation} 
\label{W_renormalised_unbounded}
W_{\tau}\;\deq\;\frac{1}{2} \int \dd x \, \dd y\,\big(\phi_{\tau}^{*} (x)\phi_\tau (x) - \varrho_{\tau} \big)\,w_{\tau}(x-y)\,\big(\phi^{*}_{\tau} (y) \phi_{\tau} (y)-\varrho_{\tau}\big) \,.
\end{equation}
Here, $w_\tau$ is obtained by applying Lemma \ref{w_positive_approximation} for $w$ a $d$-admissible interaction potential and Lemma \ref{w_endpoint_admissible_approximation} for $w$ an endpoint-admissible interaction potential.

Note that the decomposition \eqref{H_tau_renormalised_unbounded} is still valid for the $1D$ quantum Hamiltonian \eqref{H_tau2}, if we again take the free quantum Hamiltonian $H_{\tau,0}$ to be given by \eqref{H_{tau,0}_unbounded}, and if the quantum interaction $W_\tau$ is defined to be the second term on the right-hand side of \eqref{H_tau2}. We use these definitions in the sequel.

\begin{remark}
\label{quantum_density_remark}
One can show (see for instance \cite[(1.31) and (1.33)]{FrKnScSo1}) that the quantum density \eqref{def_varrho_tau_unbounded} diverges as $\tau \rightarrow \infty$. More precisely, it diverges like $\log \tau$ when $d=2$ and like $\sqrt{\tau}$ when $d=3$. In order to deduce this, we are again using the fact that we are working on the torus with one-body Hamiltonian \eqref{one_body_Hamiltonian}. This allows us to deduce that the quantity $\rho_{\tau,0}\big(\phi^*_\tau(x) \phi_\tau(x)\big)$ is independent of $x \in \Lambda$.
\end{remark}

Having defined the many-body quantum Hamiltonian $H_\tau=H_{\tau,0}+W_\tau$, we define the \emph{grand canonical ensemble} 
\begin{equation}
\label{GCE}
P_\tau \;\deq\;\ee^{-H_\tau}\,.
\end{equation} 
This is an (unnormalised) density operator on $\mathcal{F}$. As in \cite{FrKnScSo1}, we introduce a modification of the grand canonical ensemble when $d=2,3$. We fix a parameter $\eta \in \big[0,\frac{1}{4}\big]$ and let
\begin{equation}
\label{modified_GCE}
P_\tau^{\eta} \;\deq\; \ee^{-\eta H_{\tau,0}} \,\ee^{-(1-2\eta)H_{\tau,0}-W_\tau}\,\ee^{-\eta H_{\tau,0}} 
\end{equation} 
be a \emph{modified grand canonical ensemble}.
Throughout the paper, we let $\eta=0$ when $d=1$ and for $d=2,3$ we work with $\eta>0$. In particular, when $d=1$, the definitions \eqref{GCE} and \eqref{modified_GCE} coincide. For the motivation to study \eqref{modified_GCE} as a modification of the grand canonical ensemble when $d=2,3$, we refer the reader to \cite[Sections 1.6 and 2.7]{FrKnScSo1}.

The quantum state $\rho_{\tau}^{\eta}(\cdot)$ associated with \eqref{modified_GCE} is defined as
\begin{equation}
\label{quantum_state}
\rho_{\tau}^{\eta}(\mathcal{A})\;\deq\; \frac{\tr(\mathcal{A} \,P_{\tau}^{\eta})}{\tr(P_{\tau}^{\eta})}\,,
\end{equation}
for $\mathcal{A}$ a closed operator on $\mathcal{F}$. Our goal is to compare \eqref{quantum_state} and \eqref{rho_unbounded}. Analogously to \eqref{gamma_r_unbounded}, given $r \in \mathbb{N}$, we define the \emph{quantum $r$-particle correlation function} $\gamma_{\tau,r}^{\eta}$ as the operator on $\mathfrak{H}^{(r)}$ with kernel given by
\begin{equation}
\label{gamma_{tau,r}_unbounded}
\gamma_{\tau,r}^{\eta}(x_1,\ldots,x_r;y_1,\ldots,y_r)\;\deq\;\rho_{\tau}^{\eta}\big(\phi_{\tau}^{*}(y_1)\cdots \phi_{\tau}^{*}(y_r)\,\phi_{\tau}(x_1)\cdots \phi_{\tau}(x_r)\big)\,.
\end{equation}
As in the classical setting, the family $(\gamma_{\tau,r}^{\eta})_{r \in \mathbb{N}}$ determines the quantum Gibbs state $\rho_{\tau}^{\eta}(\cdot)$, see \cite[Remark 1.4]{FrKnScSo1}\footnote{Here, the claim is shown only in the case when $\eta=0$, but the proof carries over to the case when $\eta \neq 0$ since the rescaled particle operator commutes with $H_{\tau,0}$ and $W_{\tau}$.}. When $\eta=0$ (which we only consider when $d=1$), we write
\begin{align}
\notag
\gamma_{\tau,r}(x_1,\ldots,x_r;y_1,\ldots,y_r)\;\deq\;\gamma_{\tau,r}^{0}(x_1,\ldots,x_r;y_1,\ldots,y_r)
\\
\label{gamma_{tau,r}_unbounded2}
\;=\;\frac{\tr\big(\phi_{\tau}^{*}(y_1)\cdots \phi_{\tau}^{*}(y_r)\,\phi_{\tau}(x_1)\cdots \phi_{\tau}(x_r) \,P_{\tau}\big)}{\tr(P_{\tau})}\,.
\end{align}

We emphasise that, since we are working on the torus with one-body Hamiltonian \eqref{one_body_Hamiltonian}, the \emph{counterterm problem} studied in \cite{FrKnScSo1,Lewin_Nam_Rougerie3} does not occur. More precisely, it is trivial in the sense that it is obtained by a shift of the chemical potential, see \cite[Case (i) in Section 1.6]{FrKnScSo1} for a detailed explanation. We do not address this issue further in the remainder of the paper.
For a more detailed discussion of the quantum system, we refer the reader to \cite[Sections 1.4-1.6]{FrKnScSo1}.

\subsection{Statement of the main results}
\label{Statement of the main results}

Let us now state our main results. We first consider the case when $w$ is a $d$-admissible interaction potential.

\begin{theorem}[Convergence for $d$-admissible interaction potentials]
\label{main_result} 
Consider $\Lambda=\mathbb{T}^d$ for $d=1,2,3$.
Let $h$ be defined as in \eqref{one_body_Hamiltonian} for fixed $\kappa>0$. Let the interaction potential $w$ be $d$-admissible as in Definition \ref{interaction_potential_w}. 
Let the classical interaction $W$ be defined as in \eqref{W_1D_unbounded}
when $d=1$, and as in Lemma \ref{W_Wick-ordered} (i) when $d=2,3$. Let the classical $r$-particle correlation function $\gamma_r$ be defined as in \eqref{gamma_r_unbounded} and \eqref{rho_unbounded}. 
\\
For $\tau \geq 1$, let $w_\tau \in L^{\infty}(\Lambda)$ be obtained by Lemma \ref{w_1D_approximation} when $d=1$ and by Lemma \ref{w_positive_approximation} when $d=2,3$. In each case, assume that 
\begin{equation}
\label{B_d}
\beta \in \mathcal{B}_d\;\deq\;
\begin{cases}
(0,1)\,\,&\mbox{if}\,\,\,d=1,2
\\
(0,\frac{1}{2})\,\,&\mbox{if}\,\,\,d=3\,.
\end{cases}
\end{equation}
When $d=1$, let the quantum $r$-particle correlation function $\gamma_{\tau,r}$ be defined as in \eqref{gamma_{tau,r}_unbounded2}, \eqref{GCE}, and \eqref{H_tau2}. When $d=2,3$, let the quantum $r$-particle correlation function $\gamma_{\tau,r}^{\eta}$ be defined as in \eqref{gamma_{tau,r}_unbounded}, \eqref{quantum_state}, \eqref{modified_GCE}, \eqref{W_renormalised_unbounded}, and \eqref{H_{tau,0}_unbounded}.
\\ 
We then have the following convergence results for all $r \in \mathbb{N}$.
\begin{itemize}
\item[(i)] When $d=1$,
\begin{equation}
\label{Theorem_1D}
\lim_{\tau \rightarrow \infty}\|\gamma_{\tau,r}-\gamma_r\|_{\mathfrak{S}^1(\mathfrak{H}^{(r)})}\;=\;0\,.
\end{equation}
\item[(ii)] When $d=2,3$, 
\begin{equation}
\label{Theorem_2D_and_3D}
\lim_{\tau \rightarrow \infty}\|\gamma_{\tau,r}^{\eta}-\gamma_{r}\|_{\mathfrak{S}^2(\mathfrak{H}^{(r)})}\;=\;0\,,
\end{equation}
whenever $\eta \in (0,\frac{1}{4}]$.
\end{itemize}
\end{theorem}

We also consider the case when $d=2$ and when $w$ is an endpoint-admissible interaction potential.

\begin{theorem}[Convergence for endpoint-admissible interaction potentials]
\label{endpoint_admissible_result} 
Consider $\Lambda=\mathbb{T}^2$ and $h$ as in \eqref{one_body_Hamiltonian} for fixed $\kappa>0$. Let the interaction potential $w$ be endpoint-admissible as in Definition \ref{endpoint_admissible_w}. Let the classical interaction $W$ be defined as in Lemma \ref{W_Wick-ordered} (ii). Let $\gamma_r$ be given by \eqref{gamma_r_unbounded}. Given $\tau \geq 1$, let $w_\tau \in L^{\infty}(\Lambda)$ be obtained by Lemma \ref{w_endpoint_admissible_approximation}  for $\beta \in \mathcal{B}_2 = (0,1)$ as in \eqref{B_d}. Let $\gamma_{\tau,r}$ be given as in \eqref{gamma_{tau,r}_unbounded} with $\eta \in (0,\frac{1}{4}]$. Then for all $r \in \mathbb{N}$, we have
\begin{equation}
\label{Theorem_endpoint_admissible}
\lim_{\tau \rightarrow \infty}\|\gamma_{\tau,r}^{\eta}-\gamma_{r}\|_{\mathfrak{S}^2(\mathfrak{H}^{(r)})}\;=\;0\,.
\end{equation}
\end{theorem}

Before proceeding with the ideas of the proofs, let us make some comments on Theorem \ref{main_result} and Theorem \ref{endpoint_admissible_result}. We first note that the topologies in which one has convergence in \eqref{Theorem_1D}, \eqref{Theorem_2D_and_3D}, and \eqref{Theorem_endpoint_admissible} heuristically come from the fact that, for $h$ given by \eqref{one_body_Hamiltonian}, we have $\tr h^{-1}<\infty$ when $d=1$ and $\tr h^{-1}=\infty\,,\,\tr h^{-2}<\infty$ when $d=2,3$. For a more detailed explanation of this point, we refer the reader to \cite[Section 1.6]{FrKnScSo1}.

A variant of the $d=1$ result in Theorem \ref{main_result} (i) can be deduced from the work of Lewin, Nam, and Rougerie \cite[Theorem 5.3]{Lewin_Nam_Rougerie}.
Note that, in the latter approach, one can take $w_\tau=w$ in \eqref{H_tau2}.
More precisely, this follows since the condition \cite[(5.1)]{Lewin_Nam_Rougerie} is satisfied if $w \in L^1(\Lambda)$. This is verified by arguing as in \eqref{W_1D_unbounded} and using the fact that in 1D, by Wick's theorem, 
\begin{equation*}
\mathbb{E}_{\mu}\big(\|\phi\|_{L^4(\mathbb{T})}^4\big)=2\Big[\mathbb{E}_{\mu}\big(\|\phi\|_{L^2(\mathbb{T})}^2\big)\Big]^2<\infty\,. 
\end{equation*}
 We omit the details.
In Theorem \ref{main_result} (i), we give an alternative proof of this type of result. As we will see, the proofs of Theorem \ref{main_result} (i) and Theorem \ref{main_result} (ii) can both be done within a unified framework. We hence present the result in the $1D$ setting for completeness, and emphasise that the main contribution of Theorem \ref{main_result} is the result in $2D$ and $3D$.

A version of \eqref{Theorem_2D_and_3D} when $d=2$ was recently shown with $\eta=0$ by Lewin, Nam, and Rougerie \cite{Lewin_Nam_Rougerie3}. In this case, the authors consider an interaction potential $w \in L^1(\mathbb{T}^2)$ of positive type satisfying the assumption that 
\begin{equation*}
\sum_{k \in \mathbb{Z}^2} \hat{w}(k) (1+|k|^{\alpha})\;<\;\infty\,,
\end{equation*} 
for some $\alpha>0$ (see \cite[(3.3)]{Lewin_Nam_Rougerie3}). 

%This is a stronger assumption than the one that we are making in Theorem \ref{main_result} (ii).

The only previously known version of the convergence \eqref{Theorem_2D_and_3D} when $d=3$ is the result from the author's earlier joint work with Fr\"{o}hlich, Knowles, and Schlein \cite[Theorem 1.6]{FrKnScSo1}, which was done for $w \in L^{\infty}(\Lambda)$. This assumption was crucially used in the proof. Theorem \ref{main_result} (ii) for $d=3$ is an extension of this result to unbounded $w$. 

Note that the integrability assumptions on $w$ in Theorem \ref{main_result} when $d=3$ correspond to those for the interaction potentials considered in the classical setting by Bourgain \cite{Bourgain_1997}. 
More precisely, in this work, it is assumed that $w \in L^1(\mathbb{T}^3)$ satisfies 
\begin{equation}
\label{Bourgain_1997_w_assumption}
|\hat{w}(k)|\;\leq\;C \langle k \rangle^{-2-\epsilon}\,,
\end{equation}
for some $\epsilon>0$ and for all $k \in \mathbb{Z}^3$, (see \cite[(17)]{Bourgain_1997}). Observe that \eqref{Bourgain_1997_w_assumption} implies 
\begin{equation}
\label{Bourgain_1997_w_assumption_2}
w \in L^{3+}(\mathbb{T}^3)\,,
\end{equation}
(see also \cite[(29)]{Bourgain_1997}). Namely, by \eqref{Bourgain_1997_w_assumption}, we have $\hat{w} \in \ell^{3/2+}(\mathbb{Z}^3)$ and we deduce \eqref{Bourgain_1997_w_assumption_2} by the Hausdorff-Young inequality.

We note that the full analysis of \cite{Bourgain_1997} applies without any sign assumption on the interaction, whereas in this paper we are always considering the \emph{defocusing} regime. 
In particular, in the case when $w$ is of positive type, the construction in \cite{Bourgain_1997} follows without the truncation of the Wick-ordered mass. We do not present the details here. For the setup of the truncation of Wick-ordered mass in the focusing regime, we refer the reader to \cite[(12) and Proposition 1]{Bourgain_1997}.
Let us note that the assumption \eqref{Bourgain_1997_w_assumption} would formally (up to a factor of $\langle k \rangle^{-\epsilon}$) correspond to a Coulomb potential. It is expected to be optimal in terms of control of the Fourier coefficients (see the discussion following \cite[Proposition 3]{Bourgain_1997}).

%In light of this, the assumption $p>3$ in Theorem \ref{main_result} (ii) can be interpreted as being the optimal one. 

One can also see the optimality of the assumption that $p>3$ when $d=3$ as follows. 
We expand the exponential  $\int \ee^{-z W} \,\dd \mu$ and obtain the classical perturbative expansion corresponding to $\xi= \emptyset$ in \eqref{A^xi} defined below. Applying the classical Wick theorem in the formula \eqref{a^xi_m} for the terms of the expansion, one can show that the first term $a_1^{\emptyset}$ satisfies 
\begin{align}
\notag
&a_1^{\emptyset} \sim \int \dd x\, \int \dd y \, w(x-y)\, G^2(x;y) 
\\
\label{optimality}
&\;=\; \int \dd x\, \int \dd y \, w(x-y)\, G^2(x-y;0) \;=\; \int \dd x \,w(x)\, G^2(x;0)\,,
\end{align}
where $G \;\deq\; h^{-1}$ is the classical Green function. One has that $G \in L^{q}(\mathbb{T}^3 \times \mathbb{T}^3)$ for $q \in [1,3)$. Hence, by duality, finiteness of \eqref{optimality} in general requires that we consider $w \in L^{3+}(\mathbb{T}^3)$.

Furthermore, we note that the integrability and regularity assumptions on $w$ in Theorem \ref{endpoint_admissible_result} correspond to those considered when $d=2$ in \cite{Bourgain_1997}. In particular, Definition \ref{endpoint_admissible_w} (iv) corresponds to \cite[(16)]{Bourgain_1997}, see also \cite[(35)]{Bourgain_1997}. As in the $d=3$ setting, we do not need to truncate the Wick-ordered mass since we are working in the defocusing regime.
When $d=2$, we have that the classical Green function satisfies $G \in L^{q}(\mathbb{T}^2 \times \mathbb{T}^2)$ for $q \in [1,\infty)$. Hence, using \eqref{optimality} we deduce that the assumption that $p>1$ is optimal in terms of integrability and one cannot work with general $w \in L^1(\mathbb{T}^2)$ unless one makes additional assumptions (such as Definition \ref{endpoint_admissible_w} (iv) above). For a related discussion on \eqref{optimality} and the optimality of $p$, we refer the reader to \cite[Remark 5.4]{Lewin_Nam_Rougerie3}.

We conjecture that there exist interaction potentials $w \in L^1(\mathbb{T}^2)$ which are endpoint-admissible, but which are not $2$-admissible. In particular, this would imply that one cannot deduce Theorem \ref{endpoint_admissible_result} from Theorem \ref{main_result}. Our main motivation for showing Theorem \ref{endpoint_admissible_result} was to obtain the $2D$ defocusing variant of the result from \cite{Bourgain_1997} under the same integrability and regularity assumptions.

We note that the earlier results hold on more general subsets of $\mathbb{R}^d$ \cite{Lewin_Nam_Rougerie} or when $\Lambda=\mathbb{R}^d$ \cite{FrKnScSo1,FrKnScSo2,Lewin_Nam_Rougerie2}. In all of these cases, it is necessary to make appropriate assumptions on the spectral properties of the one-body Hamiltonian. In this work, we consider exclusively the case when the domain is the torus and when the one-body Hamiltonian is given by \eqref{one_body_Hamiltonian}. This allows us to use Fourier analysis and construct the interaction potential $w_\tau$ when $d=2,3$ as in Lemmas \ref{w_positive_approximation}--\ref{w_endpoint_admissible_approximation} above. Furthermore, using Fourier analysis lets us analyse more closely the quantum and classical Green functions (e.g.\ we use their translation invariance in Section \ref{Endpoint-admissible interaction potentials}). In this paper, we do not address the problem for more general domains or one-body Hamiltonians. 
\subsection{Strategy of proof}

Our strategy is to apply a perturbative expansion in the interaction, as in \cite{FrKnScSo1}. This expansion is applied both in the quantum and in the classical setting. We organise the terms in the expansion by means of a graphical representation. We then resum this expansion by means of Borel summation. This is possible to do provided that we have appropriate bounds on the explicit terms and on the remainder term. 

Let us now state the precise form of Borel summation that we apply. The goal is to deduce the convergence of a family of analytic functions from the convergence of their coefficients in a series expansion around zero. In our context, the functions considered also depend on a parameter $\xi$. Given $R \geq 1$, we let 
\begin{equation} 
\label{cal_C_R_unbounded}
\mathcal{C}_{R}\;\deq\;\{z \in \mathbb{C}\,,\re z^{-1}>R^{-1}\}\,.
\end{equation}

\begin{proposition}
\label{Borel_summation_unbounded}
Suppose that $(A^{\xi})_{\xi}$ and $(A_{\tau}^{\xi})_{\xi,\tau}$ are families of analytic functions in $\mathcal{C}_R$.
Both families are indexed by a parameter $\xi$ from an arbitrary index set. The second family is furthermore indexed by an additional parameter $\tau \geq 1$. Suppose that for $M \in \mathbb{N}$, we have the asymptotic expansions in $\mathcal{C}_{R}$
\begin{equation}
\label{Borel_summation_unbounded1}
A^\xi(z)\;=\;\sum_{m=0}^{M-1}a^\xi_{m} z^m + R^\xi_M(z)\,,\,\,\, A^\xi_\tau(z)\;=\;\sum_{m=0}^{M-1} a^\xi_{\tau,m} z^m + R^\xi_{\tau,M}(z)\,,
\end{equation}
where the explicit terms in \eqref{Borel_summation_unbounded1} satisfy
\begin{equation}
\label{Borel_summation_unbounded2}
\sup_{\xi,\tau} |a^\xi_{\tau,m}|+ \sup_\xi |a^\xi_m|
 \;\leq\; \nu \sigma^{m}m!\,,
\end{equation}
and the remainder terms satisfy
\begin{equation}
\label{Borel_summation_unbounded3}
\sup_{\xi,\tau} |R^\xi_{\tau,M}(z)|+ \sup_{\xi} |R^\xi_{M}(z)|
 \;\leq\; \nu \sigma^{M}  M!  |z|^M\,,
\end{equation}
for all $z \in \mathcal{C}_{R}$ and for some $\nu>0$ and $\sigma \geq 1$, both which are independent of $m$ and $M$.
\\
Furthermore, suppose that we have
\begin{equation}
\label{Borel_summation_unbounded4}
\lim_{\tau \rightarrow \infty}\sup_{\xi} |a^\xi_{\tau,m}-a^\xi_{m}|\;=\;0\,.
\end{equation}
Then, we have for all $z \in \mathcal{C}_{R}$ that 
\begin{equation} 
\label{Borel_summation_unbounded5}
\lim_{\tau \rightarrow \infty} \sup_{\xi} |A^\xi_{\tau}(z) - A^{\xi}(z)| \;=\;0\,.
\end{equation}
\end{proposition}
Proposition \ref{Borel_summation_unbounded} corresponds to \cite[Theorem A.1]{FrKnScSo1} and is proved in \cite[Appendix A]{FrKnScSo1}. Its proof is based on the formulation of Borel summation given by Sokal in \cite{Sokal} (we refer the reader to \cite[Theorem 136]{Hardy},\cite{Nevanlinna}, and \cite{Watson} for earlier versions of Borel summation). 

In order to prove \eqref{Borel_summation_unbounded2}, we would like to further develop the ideas based on the graphical analysis from \cite[Section 2.4]{FrKnScSo1}. Note that this method in its original form relied crucially on the boundedness of $w$. 
In particular, this made it possible to reduce the case $w=1$ at the price multiplying by appropriate powers of the finite constant $\|w\|_{L^{\infty}(\Lambda)}$. The latter was then used in order to cancel the time evolutions applied to the quantum Green function and close the estimate. For details, we refer the reader to the proof of \cite[Lemma 2.18]{FrKnScSo1}.

Instead of directly applying the methods from \cite[Section 2.4]{FrKnScSo1}, we will apply a \emph{splitting} of the time-evolved quantum Green function, see Section \ref{The splitting of the time-evolved quantum Green functions} below. The idea is to appropriately split the time-evolved quantum Green function which avoids taking Hilbert-Schmidt norms of objects which are uniformly bounded only in the operator norm. A variant of this idea was used when $d=1$ and $w=\delta$ in \cite[Section 4.4]{FrKnScSo1} (thus giving an alternative proof of the corresponding result for the $1D$ local NLS in \cite{Lewin_Nam_Rougerie}). The analysis in the case $d=1$ is easier in the sense that the quantum Green function is bounded. When $d=2,3$, this is no longer true and we have to keep careful track of the integrability parameters.

In the quantum setting, we work with \emph{bounded, $\tau$-dependent} interaction potentials $w_{\tau}$ given by Lemmas \ref{w_1D_approximation}--\ref{w_endpoint_admissible_approximation} above. These interaction potentials satisfy the appropriate positivity assumptions, have a controlled growth in $\tau$ and converge to the original interaction potential $w$ as $\tau \rightarrow \infty$. 
In order to construct such $w_\tau$ when $d=2,3$, we work on the Fourier domain and apply the \emph{transference principle}  \cite[Theorem VII.3.8]{Stein_Weiss}.
Since $\|w_\tau\|_{L^{\infty}(\Lambda)}$ is not uniformly bounded in $\tau$, we have to be careful about how we distribute these factors among the connected components in the graphical representation. 
This issue was already present in \cite[Section 4.4]{FrKnScSo1} and its solution carries over to our setting, with appropriate modifications.
For details, see \eqref{cal W_unbounded}--\eqref{product w tau_unbounded} below.

Throughout the work, we repeatedly apply variants of Sobolev embedding on the torus. We henceforth apply Sobolev embedding without further comment that we are working on the periodic setting.
For a self-contained proof of this fact, we refer the reader to \cite[Corollary 1.2]{Benyi_Oh}.

\subsection{Structure of the paper}
In Section \ref{Analysis of the quantum system}, we analyse the quantum problem when $w$ is a $d$-admissible interaction potential. In particular, in Sections \ref{General framework and setup of the perturbative expansion}--\ref{The graphical representation}, we set up the perturbative expansion and the graphical representation. This is similar to \cite[Sections 2.1--2.4]{FrKnScSo1}, with appropriate modifications. For the convenience of the reader and for completeness, we give an overview of the construction. For of the proofs, we refer the reader to \cite{FrKnScSo1}. In Section \ref{The splitting of the time-evolved quantum Green functions}, we give the details of the splitting of the time-evolved quantum Green functions. In Section \ref{Bounds on the explicit terms}, we prove the bounds on the explicit terms and in Section \ref{Convergence of the explicit terms}, we analyse their convergence. In Section \ref{Bounds on the remainder term}, we prove the bounds on the remainder term.
Section \ref{Analysis of the classical system} is devoted to the analysis of the classical system when $w$ is a $d$-admissible interaction potential. In Section \ref{Analysis of the classical system: General framework}, we recall the main definitions and rigorously justify the setup the classical problem. In Section \ref{Expansion_classical}, we set up the perturbative expansion. In Section \ref{Proof of Theorem 1}, we give the proof of Theorem \ref{main_result}. Section \ref{Endpoint-admissible interaction potentials} is devoted to the analysis of the problem when $d=2$ and when $w$ is an endpoint-admissible interaction potential.
The general framework is set up in Section \ref{Endpoint-admissible interaction potentials_1}. The quantum system is analysed in
Section \ref{Endpoint-admissible interaction potentials_2}. The classical system is analysed in Section \ref{Endpoint-admissible interaction potentials_3}. In Section \ref{Endpoint-admissible interaction potentials_4}, we give the proof of Theorem \ref{endpoint_admissible_result}.
Throughout Section \ref{Endpoint-admissible interaction potentials}, one explicitly uses the translation invariance of the quantum and classical Green functions. In Appendix \ref{The_interaction_potential}, we prove Lemma \ref{admissible_w_nonempty}, Lemma \ref{w_positive_approximation}, and Lemma \ref{w_endpoint_admissible_approximation} stated above.

\subsubsection*{Acknowledgements}
The author would like to thank Zied Ammari, J\"{u}rg Fr\"{o}hlich, Sebastian Herr, Antti Knowles, Mathieu Lewin, Benjamin Schlein, and Daniel Ueltschi for helpful discussions and comments.

\section{Analysis of the quantum system}
\label{Analysis of the quantum system}

\subsection{General framework and setup of the perturbative expansion}
\label{General framework and setup of the perturbative expansion}
Throughout this section, we assume that $w$ is a $d$-admissible interaction potential as in Definition \ref{interaction_potential_w}.
Before setting up the perturbative expansion, we define some more notation.
Given $r \in \mathbb{N}$, we denote by 
\begin{equation*}
\mathbf{B}_r \;\deq\; \{\xi \in \mathfrak{S}^2(\mathfrak{H}^{(r)})\,, \|\xi\|_{\mathfrak{S}^2(\mathfrak{H}^{(r)})} \leq 1 \,, \xi \,\,\mbox{is self-adjoint}\}\,.
\end{equation*}
In the sequel, we work with operators $\xi \in \mathbf{C}_r$, where
\begin{equation}
\label{C_r}
\mathbf{C}_r\;\deq\;
\begin{cases}
\mathbf{B}_r \cup \{\mathrm{Id}_r\}\,\,&\mbox{if}\,\,d=1
\\
\mathbf{B}_r\,\,&\mbox{if}\,\,d=2,3\,.
\end{cases}
\end{equation}

Given $\xi \in \mathbf{C}_r$, we lift it to an operator on $\mathcal{F}$ by
\begin{equation}
\label{theta_tau_unbounded}
\Theta_\tau(\xi)\;\deq\;\int \dd x_1 \cdots \dd x_r\,\dd y_1 \cdots \dd y_r\,\xi(x_1, \dots, x_r; y_1, \dots, y_r)\,\phi_{\tau}^*(x_1) \cdots \phi_{\tau}^*(x_r)\, \phi_{\tau}(y_1) \cdots \phi_{\tau}(y_r)\,,
\end{equation}
where we recall \eqref{distribution_kernels} for the definition of $\phi_{\tau}(x)$ and $\phi_{\tau}^*(y)$.
We would like to compute the quantity
$\rho_\tau^{\eta}(\Theta_\tau(\xi))$ for $\xi \in \mathbf{C}_r$ and for $\rho_\tau$ given as in \eqref{quantum_state}. Note that, when $d=1$, we are always setting $\eta=0$.
In order to set up the Borel summation argument, we introduce a complex coupling constant $z$ in front of the interaction. 
More precisely, we write
\begin{equation}
\label{rho_tau_fraction_unbounded}
\rho_{\tau}^{\eta}(\Theta_{\tau}(\xi))\;=\;\frac{\tr (\Theta_{\tau}(\xi)\,P^{\eta}_{\tau})}{\tr (P^{\eta}_{\tau})}\;=\;\frac{\hat{\rho}_{\tau,1}^{\eta} (\Theta_{\tau}(\xi))}{\hat{\rho}_{\tau,1}^{\eta}(\mathrm{Id})}\,,
\end{equation}
where for $z \in \mathbb{C}$ with $\re z \geq 0$ and $\mathcal{A}$ a closed operator on $\mathcal{F}$ we define
\begin{equation} 
\label{rho_tau_hat_unbounded}
\hat{\rho}_{\tau,z}^{\eta}(\mathcal{A})\;\deq\;\frac{\tr\big(\mathcal{A} \, \ee^{-\eta H_{\tau,0}} \, \ee^{-(1 - 2\eta) H_{\tau,0}  - z W_{\tau}} \,\ee^{-\eta H_{\tau,0}}\big)
}{\tr (\ee^{-H_{\tau,0}})}\,.
\end{equation}
Note that in \eqref{rho_tau_fraction_unbounded}, $\mathrm{Id}=\Theta_{\tau}(\emptyset)$.
We set $\mathcal{A}=\Theta_{\tau}(\xi)$ in \eqref{rho_tau_hat_unbounded} and we expand in the parameter $z$ to a fixed finite order.
\begin{lemma}
\label{Duhamel_expansion_unbounded}
Let $r \in \mathbb{N}$ and $\xi \in \mathbf{C}_r$ be given. Then, for $z \in \mathbb{C}$ with $\re z \geq 0$ and all $M \in \mathbb{N}$ we have
\begin{equation}
\label{Duhamel_expansion_unbounded_A}
A_{\tau}^{\xi}(z)\;\deq\; \hat{\rho}_{\tau,z} (\Theta_{\tau}(\xi))\;=\;\sum_{m=0}^{M-1}a_{\tau,m}^{\xi}\,z^m+R_{\tau,M}^{\xi}(z)\,,
\end{equation}
where the explicit terms are given by
\begin{align}
\notag
&a_{\tau,m}^{\xi} \;\deq\; 
\tr \bigg((-1)^m \frac{1}{(1-2\eta)^m} \int_\eta^{1-\eta} \dd t_1 \,\int_\eta^{t_1} \dd t_2 \cdots \int_\eta^{t_{m-1}} \,\dd t_m\\
\label{Explicit_term_a_unbounded}
&\times \Theta_{\tau}(\xi)\,\ee^{-(1-t_1)H_{\tau,0}}\,W_{\tau}\,\ee^{-(t_1-t_2)H_{\tau,0}}\,W_\tau \cdots \ee^{-(t_{m-1}-t_m)H_{\tau,0}} \, W_\tau\, \ee^{-t_m H_{\tau,0}} \bigg) \bigg/ \tr \big(\ee^{-H_{\tau,0}}\big)\,, 
\end{align}
for $m=0,1,\ldots, m-1$, and the remainder term is given by
\begin{align}
\notag
&R_{\tau,M}^{\xi}(z)\;\deq\; 
\tr \Bigg((-1)^M\frac{z^M}{(1-2\eta)^{M}} \int_0^{1-2\eta} \,\dd t_1\, \int_0^{t_1} \,\dd t_2  \cdots \int_0^{t_{M-1}} \,\dd t_M \,
\Theta_\tau(\xi)\,\ee^{-(1-\eta-t_1) H_{\tau,0}}\,W_\tau 
\\
\label{Remainder_term_R_unbounded}
&\times 
\ee^{-(t_1-t_2)H_{\tau,0}}\,W_\tau \ee^{-(t_2 - t_3) H_{\tau,0}}\cdots 
\, W_\tau  \, \ee^{-t_M (H_{\tau,0}+\frac{z}{1-2\eta}W_\tau)}\,\ee^{-\eta H_{\tau,0}}
\Bigg) \bigg/ \tr \big(\ee^{-H_{\tau,0}}\big)\,. 
\end{align}
\end{lemma}
\begin{proof}
This was proved in \cite[Lemma 2.1]{FrKnScSo1} by using a Duhamel expansion.
\end{proof}

\subsection{The graphical representation}
\label{The graphical representation}
We now set up the graphical representation which allows us to rewrite and systematically analyse the explicit terms given in \eqref{Explicit_term_a_unbounded} above.
An analogous framework was already set up in detail in \cite[Sections 2.3-2.4]{FrKnScSo1} when $d=2,3$, and in \cite[Sections 4.1-4.2]{FrKnScSo1} when $d=1$.
We just recall the main definitions and results and we refer the reader to \cite{FrKnScSo1} for their proofs. We also refer the reader to the aforementioned sections in \cite{FrKnScSo1} for the motivation of the definitions and for further examples.

Let us first consider the case when $d=2,3$. Later, we explain the necessary modifications in the case when $d=1$.
Before giving the precise definitions, we briefly recall the motivation for the construction of the graphs.

Given operators $\mathcal{A}$ and $\mathcal{B}$, both of which are linear in $\phi_{\tau}$ and $\phi_{\tau}^{*}$, we denote their renormalised product as
\begin{equation} \label{Wick ordering quadratic}
: \mathcal{A}\,\mathcal{B}:\;=\;\mathcal{A}\,\mathcal{B}-\rho_{\tau,0}(\mathcal{A}\,\mathcal{B})\,,
\end{equation}
for $\rho_{\tau,0}$ given by \eqref{rho_{tau,0}_unbounded} above.
With this notation, arguing as in \cite[(2.18)]{FrKnScSo1}, we get that for all $m \in \mathbb{N}$ the quantity $a_{\tau,m}^{\xi}$ given by \eqref{Explicit_term_a_unbounded} can be rewritten as
\begin{equation} 
\label{def_a_tau_unbounded}
a_{\tau,m}^{\xi}\;=\;(-1)^m\,\frac{1}{(1-2\eta)^m\,2^m}\,\int_{\eta}^{1-\eta}\dd t_1\int_{\eta}^{t_1}\dd t_2\cdots\int_{\eta}^{t_{m-1}} \,\dd t_m\,f_{\tau,m}^{\xi}(t_1,\ldots,t_m)\,,
\end{equation}
for
\begin{align}
\notag 
&f^{\xi}_{\tau, m}(t_1,\ldots,t_m)\;\deq\;\int\dd x_1\,\ldots\,\dd x_{m+r}\,\dd y_1\,\ldots\, \dd y_{m+r}
\\
\notag 
&\Bigg[\prod_{i=1}^{m} \,w_{\tau}(x_i-y_i)\Bigg]\,\cdot\,\xi(x_{m+1}, \dots, x_{m+r}; y_{m+1}, \dots, y_{m+r})\\ 
\notag
&\quad \times \rho_{\tau,0} \Bigg(\,\prod_{i=1}^{m} \bigg\{\Big[:\big(\ee^{t_i h/\tau} \phi^{*}_{\tau}\big)(x_i) \, \big(\ee^{-t_ih/\tau}\phi_\tau\big)(x_i): \Big]\,\Big[:\big(\ee^{t_ih/\tau} \phi^{*}_{\tau}\big)(y_i) \, \big(\ee^{-t_ih/\tau}\phi_{\tau}\big)(y_i):\Big]\,\bigg\}
\\
\label{Definition of f_{tau,m}^{xi}_unbounded}
&\quad \times \prod_{i=1}^{r} \phi_{\tau}^{*}(x_{m+i})\,\prod_{i=1}^{r} \phi_\tau(y_{m+i}) \Bigg)\,. 
\end{align}
Here $\ee^{th/\tau}\phi_{\tau}$ and $\ee^{th/\tau}\phi_{\tau}^{*}$ are operator-valued distributions defined by Fourier series as
\begin{equation*} 
%\label{time_evolved_operators}
\big(\ee^{th/\tau}\phi_{\tau}\big)(x)\;\deq\;\sum_{k \in \mathbb{Z}^d}\,\ee^{t \lambda_k/\tau}\,\phi_{\tau}(e_k)\,e_k(x)\,,\quad
\big(\ee^{th/\tau}\phi_{\tau}^{*}\big)(x)\;\deq\;\sum_{k \in \mathbb{Z}^d} \ee^{t \lambda_k/\tau}\,\phi_{\tau}^{*}(e_k)\,\bar{e}_k(x)\,,
%\bar{u}_{\tau,k}(x) \phi_{\tau,k}^*\,.
\end{equation*}
where we recall \eqref{lambda_k}--\eqref{e_k}. For a further discussion on time-evolved operators, see \cite[Section 2.3]{FrKnScSo1}.
In \eqref{Definition of f_{tau,m}^{xi}_unbounded} and in the sequel, $\prod$ refers to the product of the operators taken in fixed order of increasing indices from left to right.
Note that, in contrast to \cite[(2.18)]{FrKnScSo1}, in \eqref{Definition of f_{tau,m}^{xi}_unbounded} the interaction is $w_{\tau}$ and the one body Hamiltonian $h$ is independent of $\tau$.

The expression in \eqref{Definition of f_{tau,m}^{xi}_unbounded} can be written in terms of pairings given by the quantum Wick theorem (see \cite[Appendix B] {FrKnScSo1} for a self-contained summary). This motivates the introduction of a graph structure through the series of definitions given below. We first consider an abstract set of vertices $\mathcal{X}\equiv \mathcal{X}(m,r)$ consisting of $4m+2r$ elements, by which we encode the occurrences of the operators $\phi_{\tau}$ and $\phi_{\tau}^{*}$ in \eqref{Definition of f_{tau,m}^{xi}_unbounded}.

\begin{definition}(The vertex set $\mathcal{X}$; $d=2,3$)
\\
\label{cal X_unbounded}
Let $m,r\in\mathbb{N}$ be given. 
\begin{enumerate}
\item
The set $\mathcal{X}\equiv \mathcal{X}(m,r)$ consists of all triples $(i,\vartheta,\delta)$, where $i \in \{1,\ldots,m+1\}$. Furthermore, for $i \in \{1,\ldots,m\}$, we have $\vartheta \in \{1,2\}$ and for $i=m+1$, we have $\vartheta \in \{1,\ldots,r\}$. Finally, $\delta \in \{-1,1\}$.
\\
We also denote elements $(i,\vartheta,\delta)$ of $\mathcal{X}$ by $\alpha$. For $\alpha=(i,\vartheta,\delta) \in \mathcal{X}$,  we write its components as 
$i_{\alpha} \deq i,\vartheta_{\alpha} \deq \vartheta,\delta_{\alpha} \deq \delta$ respectively.
\item We give a linear order $\leq$ on $\mathcal{X}$ by ordering its elements in increasing order as
\begin{align*}
&(1,1,+1)\,,(1,1,-1),\,(1,2,+1),\,(1,2,-1),\,\ldots,
(m,1,+1),\,(m,1,-1),\,(m,2,+1),\,(m,2,-1)\,,
\\
&(m+1,1,+1)\,,\dots\,,(m+1, p,+1)\,,(m+1,1,-1)\,,\ldots\,,(m+1, p, -1)\,.
\end{align*}
For $\alpha,\beta \in \mathcal{X}$, we say that $\alpha<\beta$ when $\alpha \leq \beta$ and $\alpha \neq \beta$.
\end{enumerate}
\end{definition}

We want to rewrite \eqref{def_a_tau_unbounded}--\eqref{Definition of f_{tau,m}^{xi}_unbounded} in terms of variables labelled by the vertex set $\mathcal{X}$ defined above. To this end, we assign to each vertex $\alpha=(i,\vartheta,\delta) \in \mathcal{X}$ a spatial integration label $x_\alpha \equiv x_{i,\vartheta,\delta}$. Furthermore, to each $i \in \{1,\ldots,m\}$, we assign a time integration label $t_i$, and we set $t_{m+1}\deq 0$. In the sequel, we also use the convention $t_{\alpha}\equiv t_i$ if $\alpha=(i,\vartheta,\delta)$.
With the above notations, we let
\begin{equation} 
\label{x_t_unbounded}
\mathbf{x}\;\deq\;(x_{\alpha})_{\alpha \in \mathcal{X}} \in \Lambda^{\mathcal{X}}\,,\quad\mathbf{t}\;\deq\;(t_{\alpha})_{\alpha \in \mathcal {X}} \in \mathbb{R}^{\mathcal{X}}\,.
\end{equation}

In the sequel, we always work with $\mathbf{t} \in \mathfrak{A} \equiv \mathfrak{A}(m)$, for the simplex
\begin{equation} 
\label{def_fraA_unbounded}
\mathfrak{A}(m)\;\deq\;\Big\{\mathbf{t} \in \mathbb{R}^{\mathcal{X}} : t_{i,\vartheta,\delta} = t_i \mbox{ with } 0 = t_{m+1} \leq \eta < t_m < t_{m - 1} < \cdots < t_2 < t_1 < 1 - \eta\Big\}\,.
\end{equation}
In this case, we work with $(t_1,\ldots,t_m)$ in the support of the integral in \eqref{def_a_tau_unbounded} (up to measure zero). By Definition \ref{cal X_unbounded} and \eqref{def_fraA_unbounded}, it follows that that for all
$\alpha,\beta \in \mathcal{X}$, we have 
\begin{equation} 
\label{t_ordering_unbounded}
\alpha<\beta\,\,\implies\,\,0\leq t_\alpha-t_\beta < 1\,.
\end{equation}

\begin{definition} (The set of pairings $\mathfrak{R}$)
\label{def_pairing_unbounded}
\\
Let $m,r \in \mathbb{N}$ be given.
Let $\Pi$ be a pairing of $\mathcal{X} \equiv \mathcal{X}(m,r)$. In other words, $\Pi$ is a one-regular graph on $\mathcal{X}$. 
The edges of $\Pi$ are then ordered pairs $(\alpha,\beta) \in \mathcal{X}^2$ with $\alpha<\beta$.
We denote by $\mathfrak{R} \equiv \mathfrak{R}(m,r)$ the set of pairings $\Pi$ of $\mathcal{X}$ satisfying the following properties.
\begin{itemize}
\item[(i)]
For each $(\alpha, \beta) \in \Pi$ we have $\delta_{\alpha} \, \delta_{\beta} = -1$.
\item[(ii)]
For each $i =\{1,\ldots,m\}$ and $\vartheta \in \{1,2\}$, we have $((i,\vartheta,+1), (i,\vartheta,-1))\notin\Pi$.
\end{itemize}
\end{definition}
For an example of the pairing $\Pi \in \mathfrak{R}$ on the set of vertices $\mathcal{X}$, we refer the reader to \cite[Fig. 1]{FrKnScSo1}.
Note that condition (ii) in Definition \ref{def_pairing_unbounded} corresponds to Wick ordering.
With notation as in \eqref{x_t_unbounded}, we define a family of operator-valued distributions 
$(\mathcal{K}_{\alpha}(\mathbf{x},\mathbf{t}))_{\alpha \in \mathcal{X}}$ by
\begin{align*}
\mathcal{K}_{\alpha}(\mathbf{x},\mathbf{t})\;\deq\;
\begin{cases}
\big(\ee^{t_\alpha h/\tau} \phi_{\tau}^{*}\big)(x_\alpha) &\mbox{if }\delta_\alpha=1\\
\big(\ee^{-t_\alpha h/\tau} \phi_{\tau}\big)(x_\alpha) &\mbox{if }\delta_\alpha=-1\,.
\end{cases}
\end{align*}
\begin{definition} (The value of $\Pi \in \mathfrak{R}$)
\label{def:I_tau_PI_unbounded}
\\
Let $m,r \in \mathbb{N}$, $\Pi \in \mathfrak{R}(m,r)$ be given. We then define the \emph{value of $\Pi$} at $\mathbf{t} \in \mathfrak{A}(m)$ as
\begin{align} 
\notag
&\mathcal{I}_{\tau,\Pi}^{\xi} (\mathbf{t})\;\deq\;\int_{\Lambda^{\mathcal{X}(m,r)}} \dd \mathbf{x}\,\prod_{i=1}^{m} \bigg[ w_{\tau}(x_{i,1,1}-x_{i,2,1}) \prod_{\vartheta=1}^{2} \delta(x_{i,\vartheta,1}-x_{i,\vartheta,-1})\bigg]
\\
\label{def_cal_I_unbounded}
&\qquad \times \xi(x_{m+1,1,1},\ldots,x_{m+1,r,1}; x_{m+1,1,-1}, \dots, x_{m+1,r,-1})\prod_{(\alpha,\beta) \in \Pi}\rho_{\tau,0} \big(\mathcal{K}_{\alpha}(\mathbf{x},\mathbf{t})\,\mathcal{K}_{\beta}(\mathbf{x},\mathbf{t})
\big)\,.
\end{align}
\end{definition}
We can now rewrite \eqref{Definition of f_{tau,m}^{xi}_unbounded}.

\begin{lemma}
\label{Wick_application lemma_unbounded}
For each $m,r \in \mathbb{N}$ and $\mathbf{t} \in \mathfrak{A}(m)$, we have $f_{\tau,m}^\xi(\mathbf{t}) = \sum_{\Pi \in \mathfrak{R}(m,r)} \mathcal{I}_{\tau,\Pi}^{\xi}(\mathbf{t})$.
Furthermore, we have 
\begin{equation}
\label{Wick_application_identity_unbounded_2D_3D}
a^{\xi}_{\tau,m}\;=\;\frac{(-1)^m}{(1-2\eta)^m \,2^m}\int_{\mathfrak{A}(m)} \dd \mathbf{t}\,\sum_{\Pi \in \mathfrak{R}(m,r)} \mathcal{I}^{\xi}_{\tau,\Pi}(\mathbf{t})\,.
\end{equation}
\end{lemma}
\begin{proof}
The claim follows by applying \eqref{def_a_tau_unbounded}, the quantum Wick theorem, and arguing as in the proof of \cite[Lemma 2.8]{FrKnScSo1}. 
\end{proof}

We encode the pairing $\Pi \in \mathfrak{R}(m,r)$ of the vertex set $\mathcal{X}(m,r)$ given in Definition \ref{def_pairing_unbounded} as a multigraph (i.e.\ a graph that can have multiple edges) on a collapsed set of vertices.
\begin{definition} (The edge-coloured multigraph associated with $\Pi \in \mathfrak{R}$)
\\
\label{def_collapsed_graph_unbounded}
Let $m,r \in \mathbb{N}$, $\Pi \in \mathfrak{R} \equiv \mathfrak{R}(m,r)$ be given. We define an edge-coloured undirected multigraph $(\mathcal{V},\mathcal{E},\sigma) \equiv (\mathcal{V}_\Pi,\mathcal{E}_\Pi,\sigma_\Pi)$, with a \emph{colouring} $\sigma:\mathcal{E} \rightarrow \{-1,1\}$ according to the following.
\begin{enumerate}
\item
On $\mathcal{X} \equiv \mathcal{X}(m,r)$, we give the equivalence relation $\sim$, where $\alpha \sim \beta$ if and only if $i_{\alpha} = i_{\beta} \leq m$ and $\vartheta_{\alpha} = \vartheta_{\beta}$. The collapsed vertex set $\mathcal{V} \equiv \mathcal{V}(m,r) \deq \{[\alpha] \,, \alpha \in \mathcal{X}\}$ is defined to be the set of equivalence classes of $\mathcal X$ under $\sim$. Furthermore, we write $\mathcal{V} = \mathcal{V}_2 \sqcup \mathcal{V}_1$,
for
\begin{align*}
&\mathcal{V}_2 \;\equiv\; \mathcal{V}_2(m,r) \;\deq\; \{(i,\vartheta)\,,1 \leq i \leq m\,, 1 \leq \vartheta \leq 2\}\,, 
\\
&\mathcal{V}_1 \;\equiv\; \mathcal{V}_2(m,r)\;\deq\; \{(m+1,\vartheta,\pm 1)\col 1 \leq \vartheta \leq r\}\,.
\end{align*}
\item
$\mathcal{V}$ carries a total order $\leq$ which is inherited from $\mathcal{X}$, i.e.\ we say that $[\alpha] \leq [\beta]$ whenever $\alpha \leq \beta$. 
This is independent of the choice of representatives $\alpha,\beta \in \mathcal{X}$.
\item
Given a pairing $\Pi \in \mathfrak{R}$, from each edge $(\alpha,\beta) \in \Pi$, we obtain an edge $e =\{[\alpha], [\beta]\} \in \mathcal{E}$. Furthermore, we define the associated colouring by $\sigma(e)\deq\delta_{\beta}$.
\item
By $\conn(\mathcal{E})$, we denote the set of connected components of $\mathcal{E}$. In other words, we have $\mathcal{E}= \bigsqcup_{\mathfrak{P} \in \conn(\mathcal{E})} \mathfrak{P}$. We refer to the connected components $\mathfrak{P}$ of $\mathcal{E}$ as its \emph{paths}.
\end{enumerate}
\end{definition}

In the sequel, we refer to the multigraph $(\mathcal{V},\mathcal{E})$ just as $\mathcal{E}$, since the vertex set $\mathcal{V}$ is uniquely determined for given $m,r \in \mathbb{N}$. For an example of such a multigraph, we refer the reader to \cite[Fig. 2]{FrKnScSo1}. We note the following properties of the multigraph $\mathcal{E}$ associated with any $\Pi \in \mathfrak{R}$ \cite[Lemma 2.12]{FrKnScSo1}.
\begin{itemize}
\item[(1)] Every vertex in $\mathcal{V}_2$ has degree 2.
\item[(2)] Every vertex in $\mathcal{V}_1$ has degree 1.
\item[(3)] $\mathcal{E}$ has no loops, i.e.\,cycles of length 1.
\end{itemize}
In particular, we can view each $\mathfrak{P} \in \conn(\mathcal{E})$ as a path of $\mathcal{E}$ in the standard graph-theoretic sense.

We adapt the space and time labels given by \eqref{x_t_unbounded} to the setting of the collapsed vertex set $\mathcal{V}$.
Given $\mathbf{x}=(x_\alpha)_{\alpha \in \mathcal{X}} \in \Lambda^{\mathcal{X}}, \mathbf{t}=(t_\alpha)_{\alpha \in \mathcal{X}} \in \mathbb{R}^{\mathcal{X}}$, we define $\mathbf{y}=(y_a)_{a \in \mathcal{V}} \in \Lambda^{\mathcal{V}}, \mathbf{s} \in (s_a)_{a \in \mathcal{V}} \in \mathbb{R}^{\mathcal{V}}$ according to 
\begin{equation} 
\label{def_ys_unbounded}
y_{[\alpha]}\;\deq\;x_{\alpha}\,,\quad
s_{[\alpha]}\;\deq\;t_{\alpha}\,,
\end{equation}
for $\alpha \in \mathcal{X}$, with $[\cdot]$ given as in Definition \ref{def_collapsed_graph_unbounded} (i) above. Note that \eqref{def_ys_unbounded} is independent of the choice of representative $\alpha \in \mathcal{X}$. Furthermore, by \eqref{t_ordering_unbounded}, we have 
\begin{equation} 
\label{s_ordered_unbounded}
0 \leq s_a - s_b < 1 \,\,\, \mbox{whenever} \,\, a<b\,.
\end{equation}
Also, 
\begin{equation}
\label{s_ordered equality_unbounded}
s_a=s_b \quad \iff \quad i_a = i_b\,,
\end{equation}
where we let $i_{[\alpha]} \deq i_{\alpha}$, 
which is independent of the choice of representative.
Given $a \in \mathcal{V}_2$, we write it as $a=(i_a,\vartheta_a)$.
In the sequel, given $\mathbf{t} \in \mathbb{R}^{\mathcal{X}}$, we sometimes write $\mathbf{s} \equiv \mathbf{s} (\mathbf{t})$ for the quantity defined in \eqref{def_ys_unbounded} above without further comment. 

We distinguish two possible types of connected components of $\mathcal{E}$. 
\begin{definition}(Open and closed paths)
\\
\label{Open_and_Closed_paths_unbounded}
Let $\mathfrak{P} \in \conn(\mathcal{E})$ be given. We say that $\mathfrak{P}$ is a \emph{closed path} if all of its vertices belong to $\mathcal{V}_2$. Otherwise, we say that it is an \emph{open path}.
\end{definition}

Given $\mathfrak{P} \in \conn(\mathcal{E})$, let 
$\mathcal{V}(\mathfrak{P}) \deq \bigcup_{e \in \mathfrak{P}} e$ denote the set of its vertices. Furthermore, for $i=1,2$, we write 
$\mathcal{V}_i(\mathfrak{P}) \deq \mathcal{V}(\mathfrak{P}) \cap \mathcal{V}_i$ denote the set of vertices of $\mathfrak{P}$ that belong to $\mathcal{V}_i$. In particular, we have a decomposition $\mathcal{V}(\mathfrak{P})= \mathcal{V}_2(\mathfrak{P}) \sqcup \mathcal{V}_1(\mathfrak{P})$. By construction, one deduces that any open path $\mathfrak{P} \in \conn(\mathcal{E})$ has two distinct endpoints in $\mathcal{V}_1$ and that its remaining $|\mathcal{V}(\mathfrak{P})|-2$ vertices belong to $\mathcal{V}_2$. 

We rewrite \eqref{def_cal_I_unbounded} in terms of \emph{time-evolved quantum Green functions}. We first recall several definitions (for a detailed discussion see \cite[Section 2.3]{FrKnScSo1}).
The \emph{quantum Green function} is given by
\begin{equation}
\label{G_tau}
G_{\tau} \;\deq\; \frac{1}{\tau(\ee^{h/\tau}-1)}\,.
\end{equation}
More generally, we consider \emph{time-evolved quantum Green function}
\begin{equation}
\label{G}
G_{\tau,t} \;\deq\; \frac{\ee^{-t h/\tau}}{\tau(\ee^{h/\tau}-1)}\,\,\,\mbox{for}\,\,t >-1\,.
\end{equation}
In particular, we have $G_{\tau,0}=G_\tau$.
\\
Furthermore, we work with the \emph{time-evolved delta function}
\begin{equation}
\label{S}
S_{\tau,t} \;\deq\; \ee^{-t h/\tau}\,\,\,\mbox{for}\,\,t \geq 0\,.
\end{equation}
In the sequel, the operator kernels of $G_{\tau,t}$ and $S_{\tau,t}$ will be denoted as $G_{\tau,t}(x;y)$ and $S_{\tau,t}(x;y)$ respectively. For the admissible values of $t$, these are both measures on $\Lambda^2$. Furthermore, these kernels are pointwise nonnegative and symmetric (for a proof of this, see \cite[Lemma 2.9]{FrKnScSo1}).
We state the result of \cite[Lemma 2.10]{FrKnScSo1}, which allows us to make the connection with the factors in \eqref{def_cal_I_unbounded}.
\begin{lemma}
\label{Correlation functions}
Let $\alpha,\beta \in \mathcal{X}$ with $\alpha<\beta$ be given. The following statements hold.
\begin{itemize}
\item[(i)]
If $\delta_\alpha= +1, \delta_\beta= -1$, and $t_{\alpha}-t_{\beta}<1$, then
\begin{equation*}
\rho_{\tau,0} \big(\mathcal{K}_{\alpha} (\mathbf{x},\mathbf{t})\,\mathcal{K}_{\beta} (\mathbf{x},\mathbf{t})\big)\;=\; 
G_{\tau,-(t_\alpha - t_\beta)}(x_\alpha; x_\beta)\;\geq\;0\,.
\end{equation*}
\item[(ii)]
If $\delta_\alpha=-1, \delta_{\beta}=+1$, and $t_\alpha-t_\beta \geq 0$, then
\begin{equation*}
\rho_{\tau,0} \big(\mathcal{K}_{\alpha}(\mathbf{x},\mathbf{t})\,\mathcal{K}_{\beta}(\mathbf{x},\mathbf{t})\big)\;=\; 
G_{\tau, t_\alpha - t_\beta}(x_\alpha;x_\beta)+\frac{1}{\tau}\,S_{\tau, t_\alpha-t_\beta}(x_\alpha; x_\beta)\;\geq\;0\,.
\end{equation*}
\end{itemize}
\end{lemma}
%For the proof, we refer the reader to the proof of \cite[Lemma 2.10]{FrKnScSo1}.

\begin{definition}
\label{def_J_e_unbounded}
Let $\mathbf{y}= (y_a)_{a \in \mathcal{V}} \in \Lambda^{\mathcal{V}}$ and $\mathbf{s}=(s_a)_{a \in \mathcal{V}} \in \mathbb{R}^{\mathcal{V}}$ which satisfies conditions \eqref{s_ordered_unbounded}--\eqref{s_ordered equality_unbounded} be given. Furthermore, let $e =\{a,b\} \in \mathcal{E}$ with $a,b \in \mathcal{V}, a<b$ be given. We then define the labels $\mathbf{y}_e \deq (y_a, y_b) \in \Lambda^{e}$ as well as the integral kernel
\begin{equation}
\label{J_e_unbounded}
\mathcal{J}_{\tau,e}(\mathbf{y}_e,\mathbf{s})\;\deq\;G_{\tau,\sigma(e)(s_a-s_b)}(y_a; y_b) + \frac{\mathbf{1}_{\sigma(e) = +1}}{\tau} \, S_{\tau,s_a-s_b}(y_a;y_b)\,,
\end{equation}
\end{definition}
Note that the integral kernel $\mathcal{J}_{\tau,e}$ given by \eqref{J_e_unbounded} is Hilbert-Schmidt unless
$\sigma(e)=+1$ and $i_a=i_b$.
Furthermore, we note that
\begin{equation}
\label{J_nonnegative}
\mathcal J_{\tau, e}(\mathbf{y_e}, \mathbf{s}) \;\geq\;0\,.
\end{equation}
In the sequel, we use the splitting of the variable $\mathbf{y}$, which is given by
\begin{equation}
\label{unbounded_y1_y2}
\mathbf{y}\;=\; (\mathbf{y_1},\mathbf{y_2})\,,
\end{equation}
where $\mathbf{y_i} \deq (y_a)_{a \in \mathcal{V}_i}, i = 1,2$. We also write
\begin{equation} 
\label{xi_y_1_unbounded}
\xi(y_{m+1,1,+1},\ldots,y_{m+1,r,+1}; y_{m+1,1,-1}, \dots, y_{m+1,r,-1}) \;=\; \xi(\mathbf{y_1})\,.
\end{equation}

\begin{lemma} 
\label{lem:I_rep_unbounded}
With $\mathbf{s}$ as defined in \eqref{def_ys_unbounded} above, we have
\begin{equation}
\label{I_representation_unbounded}
\mathcal{I}_{\tau,\Pi}^{\xi}(\mathbf{t})\;=\;\int_{\Lambda^{\mathcal{V}}} \dd \mathbf{y}\,\Bigg[\prod_{i=1}^{m} w_{\tau}(y_{i,1}-y_{i,2})\Bigg] \xi(\mathbf{y_1})\,\prod_{e \in \mathcal{E}} \mathcal{J}_{\tau, e}(\mathbf{y}_e,\mathbf{s})\,.
\end{equation}
\end{lemma}
\begin{proof}
We rewrite \eqref{def_cal_I_unbounded} by using Lemma \ref{Correlation functions}, Definition \ref{def_J_e_unbounded}, \eqref{xi_y_1_unbounded}, and arguing analogously as in the proof of \cite[Lemma 2.15]{FrKnScSo1}.
\end{proof}

We now explain the necessary modifications of the graph structure when $d=1$. 
Recall that the quantum Hamiltonian is now given by \eqref{H_tau2} above. In particular, the quantum interaction is given by
\begin{equation}
\label{W_tau_1D}
W_{\tau}\;\deq\;\frac{1}{2} \int \dd x\,\dd y\, \phi^{*}_{\tau}(x)\,\phi^{*}_{\tau}(y)\,w_\tau (x-y)\,\phi_{\tau}(x)\,\phi_{\tau}(y)\,,
\end{equation}
which is normal-ordered, i.e.\ the factors of $\phi^{*}_{\tau}$ come before the factors of $\phi_{\tau}$.
Furthermore, in \eqref{W_tau_1D}, we do not renormalise the interaction, which allows for the presence of loops in the graph structure. We again perform a Taylor expansion to order $M$ of the quantity $A^{\xi}_{\tau}(z) \deq \hat{\rho}_{\tau,z}(\Theta_{\tau}(\xi))$ in the parameter $z \in \mathbb{C}, \,\re z \geq 0$. Here $\hat{\rho}_{\tau,z}(\cdot)$ is defined as in \eqref{rho_tau_hat_unbounded}, except that now we take $\eta=0$, and the quantum interaction $W_{\tau}$ as in \eqref{W_tau_1D}. After the necessary modifications, the explicit terms $a_{\tau,m}^{\xi}$ and the remainder term $R_{\tau,M}^{\xi}(z)$ in this expansion are given as in \eqref{Explicit_term_a_unbounded}--\eqref{Remainder_term_R_unbounded} above.

Recalling \eqref{C_r}, we need to consider the cases when $\xi \in \mathbf{B}_r$ and $\xi=\mathrm{Id}_r$ separately. Let us consider first the case when $\xi \in \mathbf{B}_r$.
In light of the normal-ordering of the interaction \eqref{W_tau_1D}, one needs to modify the order in Definition \ref{cal X_unbounded} (ii). In particular, we order the vertices $(i,\vartheta,\delta) \in \mathcal{X}(m,r)$ according to the lexicographical order of the corresponding string $i \delta \vartheta$, where for the $\delta$ component we define that $+1<-1$.
The new order is also denoted as $\leq$. Due to the lack of renormalisation in \eqref{W_tau_1D}, we also need to consider a larger class of pairings than the class $\mathfrak{R}$ given in Definition \ref{def_pairing_unbounded} above. 
\begin{definition}(The set of pairings $\mathfrak{Q}$)
\\
\label{def_pairing_1D_unbounded}
Let $m,r \in \mathbb{N}$ be given. We denote by $\mathfrak{Q} \equiv \mathfrak{Q}(m,r)$ the set of pairings $\Pi$ of $\mathcal{X}(m,r)$ such that for all
$(\alpha,\beta) \in \Pi$ we have $\delta_\alpha \delta_\beta=-1$.
\end{definition}

We accordingly modify Definition \ref{def_collapsed_graph_unbounded}, by which we assign to every pairing $\Pi \in \mathfrak{Q}$ a corresponding edge-coloured multigraph $(\mathcal{V}_{\Pi},\mathcal{E}_{\Pi},\sigma_{\Pi})
\equiv(\mathcal{V},\mathcal{E},\sigma)$. By Definition \ref{def_pairing_1D_unbounded}, the new graphs can have loops, which are of the form $e=\{a,a\}$ for $a \in \mathcal{V}_2$. Here, $\mathcal{V}_2$ is defined exactly as in Definition \ref{def_collapsed_graph_unbounded} (i). Given $\Pi \in \mathfrak{Q}(m,r)$, we define $\mathcal{I}_{\tau,\Pi}^{\xi} (\mathbf{t})$ as in Definition \ref{def:I_tau_PI_unbounded}. 
One can then show that (see \cite[(4.3)]{FrKnScSo1})
\begin{equation}
\label{Wick_application_identity_unbounded}
a^{\xi}_{\tau,m}\;=\;\frac{(-1)^m}{2^m}\int_{\mathfrak{A}(m)} \dd \mathbf{t}\,\sum_{\Pi \in \mathfrak{Q}(m,r)} \mathcal{I}^{\xi}_{\tau,\Pi}(\mathbf{t})\,.
\end{equation}
The quantity $\mathcal{J}_{\tau,e}(\mathbf{y}_e,\mathbf{s})$ is defined as in Definition \ref{def_J_e_unbounded}, but with the $d=1$ variant of the order $\leq$ (obtained from the lexicographic order as explained above). With these appropriately modified definitions, we obtain that Lemma \ref{lem:I_rep_unbounded} holds.

Let us now explain the modifications needed when $\xi=\mathrm{Id}_r$. The definition of the (uncollapsed) vertex set $\mathcal{X}(m,r)$ remains the same as $d=1$ setting with $\xi \in \mathbf{B}_r$. We need to modify the definition of the edge-coloured multigraph. 
\begin{definition}(The edge-coloured multigraph associated with $\Pi \in \mathfrak{Q}$ when $\xi=\mathrm{Id}_r$)
\\
\label{collapsed_graph_1D_delta_unbounded}
Let $m,r \in \mathbb{N}$, $\Pi \in \mathfrak{Q}(m,r) \equiv \mathfrak{Q}$ be given. We define an edge-coloured undirected multigraph
$(\tilde{\mathcal{V}}_{\Pi},\tilde{\mathcal{E}}_{\Pi},\tilde{\sigma}_{\Pi}) \equiv (\tilde{\mathcal{V}},\tilde{\mathcal{E}},\tilde{\sigma})$, with a colouring $\tilde{\sigma}: \mathcal{E} \rightarrow \{-1,1\}$ according to the following rules.

\begin{itemize}
\item[(i)] We introduce an equivalence relation $\sim'$ on $\mathcal{X} \equiv \mathcal{X}(m,r)$ under which  $\alpha \sim' \beta$ if and only if $i_{\alpha}=i_{\beta}$ and $\vartheta_{\alpha}=\vartheta_{\beta}$. The collapsed vertex set $\tilde{\mathcal{V}} \equiv \tilde{\mathcal{V}}(m,r) \deq \{[\alpha] \,, \alpha \in \mathcal{X}\}$ is defined to be the set of equivalence classes of $\mathcal X$ under $\sim'$. 
\\
Furthermore, we write
$\tilde{\mathcal{V}}=\tilde{\mathcal{V}}_{2} \sqcup \tilde{\mathcal{V}}_{1}$,
where we define
\begin{align*}
&\tilde{\mathcal{V}}_{2} \;\deq\; \{(i,\vartheta)\,, 1 \leq i \leq m\,, 1 \leq \vartheta \leq 2\}
\\
&\tilde{\mathcal{V}}_{1} \;\deq\; \{(m+1,\vartheta)\,, 1 \leq \vartheta \leq r\}\,.
\end{align*}

\item[(ii)] $\tilde{\mathcal{V}}$ carries a total order that is inherited from $\mathcal{X}$. Note that we are using the $d=1$ variant of the order $\leq$ on $\mathcal{X}$ here.
\item[(iii)] Given a pairing $\Pi \in \mathfrak{Q}$, from each edge $(\alpha, \beta) \in \Pi$, we obtain an edge $e =\{[\alpha], [\beta]\} \in \tilde{\mathcal{E}}$. Furthermore, we define the associated colouring by $\tilde{\sigma}(e) \deq \delta_{\beta}$. 
\item[(iv)] By $\conn(\tilde{\mathcal{E}})$, we denote the set of connected components of $\tilde{\mathcal{E}}$. In other words, we have $\tilde{\mathcal{E}}= \bigsqcup_{\mathfrak{P} \in \conn(\tilde{\mathcal{E}})} \mathfrak{P}$. We refer to the connected components $\mathfrak{P}$ of $\tilde{\mathcal{E}}$ as its \emph{paths}.
\end{itemize}
\end{definition}
By construction we have that all the paths $\mathfrak{P} \in \conn(\tilde{\mathcal{E}})$ as defined above are closed.
Similarly as in \eqref{def_ys_unbounded}, given $\mathbf{x}=(x_{\alpha})_{\alpha \in \mathcal{X}} \in \Lambda^{\mathcal{X}}, \mathbf{t}=(t_{\alpha})_{\alpha \in \mathcal{X}} \in \mathbb{R}^{\mathcal{X}}$, we define
$\mathbf{y}=(y_a)_{a \in \tilde{\mathcal{V}}} \in \Lambda^{\tilde{\mathcal{V}}}$, $\mathbf{s}=(s_a)_{a \in \tilde{\mathcal{V}}}\in \mathbb{R}^{\tilde{\mathcal{V}}}$. The only difference is that we are now considering the collapsed vertex set $\tilde{\mathcal{V}}$ and the equivalence relation $\sim'$ from Definition \ref{collapsed_graph_1D_delta_unbounded}.
We adapt \eqref{unbounded_y1_y2} by setting $\mathbf{y_i} \deq(y_a)_{a \in \tilde{\mathcal{V}}_i}$, $i=1,2$.
For $\mathfrak{P} \in \conn (\mathcal{E})$, $\tilde{\mathcal{V}}(\mathfrak{P})$ denotes its set of vertices. For $i=1,2$, we let 
$\tilde{\mathcal{V}}_i(\mathfrak{P}) \deq \tilde{\mathcal{V}}(\mathfrak{P}) \cap \tilde{\mathcal{V}}_i$. Furthermore, for an edge $e \in \tilde{\mathcal{E}}$, the quantity $\mathcal{J}_{\tau,e}$ is given as in Definition \ref{def_J_e_unbounded}, when we replace each of $\mathcal{V},\mathcal{E},\sigma$ by $\tilde{\mathcal{V}},\tilde{\mathcal{E}},\tilde{\sigma}$ respectively.

In this context, we also need to modify the definition of the value of $\Pi \in \mathfrak{Q}$ (see  Definition \ref{def:I_tau_PI_unbounded} and Lemma \ref{lem:I_rep_unbounded} above).
Let $m,r \in \mathbb{N}$ and a pairing $\Pi \in \mathfrak{Q}(m,r)$ be given. Let $(\tilde{\mathcal{V}},\tilde{\mathcal{E}},\tilde{\sigma})$ the associated graph from Definition \ref{collapsed_graph_1D_delta_unbounded}. For $\mathbf{t} \in \mathfrak{A} \equiv \mathfrak{A}(m)$, we let
\begin{equation}
\label{value_of_Pi_1D_delta}
\mathcal{I}^{\xi}_{\tau,\Pi}(\mathbf{t})\;\deq\;\int_{\Lambda^{\tilde{\mathcal{V}}}}\dd \mathbf{y}\,\Bigg[ \prod_{i=1}^{m} w_{\tau}(y_{i,1}-y_{i,2})\Bigg] \prod_{e \in \tilde{\mathcal{E}}} \mathcal{J}_{\tau, e}(\mathbf{y}_e,\mathbf{s})\,.
\end{equation}
Note that \eqref{Wick_application_identity_unbounded} then still holds.

Throughout the sequel, we streamline the notation for pairings in such a a way that we combine 
the notation for all dimensions and values of $\xi$ when $d=1$.

Let $m,r \in \mathbb{N}$ be given. Recalling Definitions \ref{def_pairing_unbounded} and \ref{def_pairing_1D_unbounded}, we define
\begin{equation}
\label{R(d,m,r)}
\mathfrak{R}(d,m,r)\;\deq\;
\begin{cases}
\mathfrak{R}(m,r)\,\,\,&\mbox{for}\,\, d=2,3
\\
\mathfrak{Q}(m,r)\,\,\,&\mbox{for}\,\, d=1\,.
\end{cases}
\end{equation}
In particular, for $\xi \in \mathbf{C}_r$, we can rewrite \eqref{Wick_application_identity_unbounded_2D_3D} and \eqref{Wick_application_identity_unbounded} as
\begin{equation}
\label{Wick_application_identity_unbounded_all_D}
a^{\xi}_{\tau,m}\;=\;\frac{(-1)^m}{(1-2\eta)^m \,2^m}\int_{\mathfrak{A}(m)} \dd \mathbf{t}\,\sum_{\Pi \in \mathfrak{R}(d,m,r)} \mathcal{I}^{\xi}_{\tau,\Pi}(\mathbf{t})\,.
\end{equation}
We also write 
\begin{equation}
\label{function_f_all_D}
f^{\xi}_{\tau,m}(\mathbf{t}) \;=\; \sum_{\Pi \in \mathfrak{R}(d,m,r)} \mathcal{I}^{\xi}_{\tau,\Pi}(\mathbf{t})\,.
\end{equation}
Note that, for $d=2,3$, this corresponds to \eqref{Definition of f_{tau,m}^{xi}_unbounded} above. When $d=1$, the latter expression is appropriately modified to take into account the different graph structure.
As a convention, when $\xi \in \mathbf{B}_r$, we henceforth always adopt the necessary modifications of the above definitions when $d=1$ without explicit mention. For instance, we use expressions such as \eqref{I_representation_unbounded} for all values of $d=1,2,3$.

\subsection{The splitting of the time-evolved quantum Green functions}
\label{The splitting of the time-evolved quantum Green functions}

For $t \in (-1,1)$ we consider the quantity
\begin{equation}
\label{Q_{tau,t}}
Q_{\tau,t} \;\deq\;
\begin{cases}
G_{\tau,t}+\frac{1}{\tau}S_{\tau,t}\,\,\,&\mbox{for}\,\, t \in (0,1)
\\
G_{\tau,t}\,\,\,&\mbox{for}\,\, t \in (-1,0]\,.
\end{cases}
\end{equation}
Note that this is well-defined by \eqref{S} and \eqref{G}. However, note that the Hilbert-Schmidt norm of $G_{\tau,t}$ is unbounded as $t \rightarrow -1$. Likewise, the Hilbert-Schmidt norm of $S_{\tau,t}$ is unbounded as $t \rightarrow 0+$.
The only norm in which $G_{\tau,t}$ and $S_{\tau,t}$ are bounded uniformly in the appropriate set of $t$ is the operator norm. This norm is too weak to apply our analysis.
%In particular, we have 
%\begin{equation}
%\label{Q_{tau,0}}
%Q_{\tau,0}=G_\tau\,.
%\end{equation}

As in \cite[Section 4.4]{FrKnScSo1}, we introduce a splitting of the operator $Q_{\tau,t}$ defined in \eqref{Q_{tau,t}} as
\begin{equation}
\label{Q_splitting}
Q_{\tau,t}\;=\; Q_{\tau,t}^{(1)}+\frac{1}{\tau} Q_{\tau,t}^{(2)}\,,
\end{equation}
where for $t \in (-1,1)$, we define
\begin{equation}
\label{Q^12}
Q^{(1)}_{\tau,t} \;\deq\; \frac{\ee^{-\{t\}h/\tau}}{\tau(\ee^{h/\tau}-1)}\,,\qquad Q^{(2)}_{\tau,t} \;\deq\; 
\begin{cases}
\ee^{-\{t\}h/\tau}\,\,\,&\mbox{for}\,\, t \in (-1,0) \cup (0,1)\\
0\,\,\,&\mbox{for}\,\, t=0\,.
\end{cases}
\end{equation}
In \eqref{Q^12} and in the sequel, $\{t\}$ denotes the fractional part of $t$, i.e.\ $\{t\} \deq t-\lfloor t \rfloor$, where $\lfloor \cdot \rfloor$ is the floor function.
By \cite[Lemma 2.9]{FrKnScSo1}, we have that for all $x,y \in \Lambda$ and $t \in (-1,1)$ 
\begin{equation}
\label{positivity}
Q^{(1)}_{\tau,t}(x;y) \;=\; Q^{(1)}_{\tau,t}(y;x) \;\geq\;0\,,\quad Q^{(2)}_{\tau,t}(x;y) \;=\; Q^{(2)}_{\tau,t}(y;x) \;\geq\;0\,.
\end{equation}
In particular, substituting this into \eqref{Q_splitting}, we have
\begin{equation}
\label{Q_positivity}
Q_{\tau,t}(x;y)\;=\;Q_{\tau,t}(y;x)\;\geq\;0\,.
\end{equation}
In what follows, we analyse the boundedness properties of the operators $Q^{(j)}_{\tau,t}$ defined in \eqref{Q^12}.

\begin{proposition}
\label{Q^12_bounds}
There exist constants $C_1,C_2>0$ such that for all $\tau \geq 1$ and $t \in (-1,1)$ the following properties hold.
\begin{itemize}
\item[(i)] $\|Q^{(1)}_{\tau,t}\|_{L^\infty(\Lambda^2)} \;\leq\; 
\begin{cases}
C_1 \,\,&\mbox{if}\,\,\,d=1
\\
C_1 \log \tau \,\,&\mbox{if}\,\,\,d=2
\\
C_1 \sqrt{\tau} \,\,&\mbox{if}\,\,\,d=3.
\end{cases}
$
\item[(ii)] For all $x \in \Lambda$ we have 
\begin{equation}
\label{Q12_bounds_ii}
\int \dd y \, Q^{(2)}_{\tau,t}(x;y) \;=\; \int \dd y \, Q^{(2)}_{\tau,t}(y;x) \;\leq\; 1\,.
\end{equation}
\item[(iii)] If $\{t\} \geq \frac{1}{2}$, then we have
\begin{equation}
\label{Q12_bounds_iii}
\|Q^{(2)}_{\tau,t}\|_{L^\infty(\Lambda^2)} \;\leq\; C_1 \tau^{d/2}\,.
\end{equation}
\item[(iv)] For all $x,y \in \Lambda$ we have
\begin{equation}
\label{Q12_bounds_iv}
Q_{\tau,t}(x;y)\;\geq\; Q^{(1)}_{\tau,t}(x;y) \;\geq\; C_2\,.
\end{equation}
\end{itemize}
\end{proposition}

\begin{remark}
We note that a version of Proposition \ref{Q^12_bounds} was proved when $d=1$ in \cite[Lemma 4.12]{FrKnScSo1}. In the analogue of \eqref{Q12_bounds_iii}, i.e. in \cite[Lemma 4.12 (iii)]{FrKnScSo1} the weaker bound of $C_1 \tau$ was given. This was sufficient for the proof of \cite[Theorem 1.9]{FrKnScSo1}, and would also be sufficient for the proof of the $d=1$ version of our result here.
\end{remark}

\begin{proof}[Proof of Proposition \ref{Q^12_bounds}]

\begin{itemize}
\item[(i)] By the triangle inequality we have
\begin{multline}
\label{Q1_bound1}
\|Q^{(1)}_{\tau,t}\|_{L^\infty(\Lambda^2)} \;=\; \Bigg\| \sum_{k \in \mathbb{Z}^d} \frac{\ee^{-\{t\}\lambda_k/\tau}}{\tau(\ee^{\lambda_k/\tau}-1)} \,\ee^{2\pi \ii k\cdot (x-y)} \Bigg\|_{L^{\infty}(\Lambda^2)}\;\leq\; \sum_{k \in \mathbb{Z}^d} \frac{1}{\tau(\ee^{\lambda_k/\tau}-1)}\,,
\\
\;\leq\; \mathop{\sum_{k \in \mathbb{Z}^d}}_{|k| \leq \sqrt{\tau}} \frac{1}{\lambda_k} +  \mathop{\sum_{k \in \mathbb{Z}^d}}_{|k| > \sqrt{\tau}} \frac{C}{\tau}\,\ee^{-|k|^2/\tau} \;\lesssim\;  \mathop{\sum_{k \in \mathbb{Z}^d}}_{|k| \leq \sqrt{\tau}} \frac{1}{|k|^2 + 1} + \frac{1}{\tau} \int_{|x| \geq \sqrt{\tau}} \ee^{-|x|^2/\tau} \,\dd x\,.
\end{multline}
We know that 
\begin{equation}
\label{Q1bound_2}
\mathop{\sum_{k \in \mathbb{Z}^d}}_{|k| \leq \sqrt{\tau}} \frac{1}{|k|^2 + 1} \;\leq\; 
\begin{cases}
C  \,\,&\mbox{if}\,\,\,d=1
\\
C\log \tau \,\,&\mbox{if}\,\,\,d=2
\\
C\sqrt{\tau} \,\,&\mbox{if}\,\,\,d=3.
\end{cases}
\end{equation}
Furthermore, by performing the substitution $y=x/\sqrt{\tau}$, it follows that 
\begin{equation}
\label{Q1bound_3}
\frac{1}{\tau} \int_{|x| \geq \sqrt{\tau}} \ee^{-|x|^2/\tau} \,\dd x \;\leq\; C\,\tau^{d/2-1}\,.\end{equation}
Substituting \eqref{Q1bound_2} and \eqref{Q1bound_3} into \eqref{Q1_bound1}, we deduce (i).
\item[(ii)] The claim \eqref{Q12_bounds_ii} was proved for $d=1$ in \cite[Lemma 4.12 (ii)]{FrKnScSo1}. The proof carries over to $d=2,3$.
\item[(iii)] Suppose that $\{t\} \geq \frac{1}{2}$. By the triangle inequality we have
\begin{equation*}
\|Q^{(2)}_{\tau,t}\|_{L^\infty(\Lambda^2)}=
\|\ee^{-\{t\}h/\tau}\|_{L^\infty (\Lambda^2)} \leq \sum_{k \in \mathbb{Z}^d} \ee^{-\frac{1}{2}\lambda_k/\tau}
\;\lesssim\; \int_{\mathbb{R}^d} \ee^{-C|x|^2/\tau} \,\dd x\,.
\end{equation*}
By performing the substitution $y=x/\sqrt{\tau}$, we deduce (iii).
\item[(iv)] %The claim \eqref{Q12_bounds_iv} was proved for $d=1$ in \cite[Lemma 4.12 (iv)]{FrKnScSo1}. The proof carries over to $d=2,3$.
We recall from \cite[Lemma 2.9]{FrKnScSo1} that $S_{\tau,t}(x;y) \geq 0$.
Hence, by \eqref{Q_{tau,t}}--\eqref{Q^12}, it suffices to show that 
\begin{equation}
\label{Q12_bounds_iv_1a}
G_{\tau,t}(x;y) \;\geq\; C_2\,.
\end{equation}
We present the details of the proof of \eqref{Q12_bounds_iv_1a} based on the Poisson summation formula (alternatively, one can use the Feynman-Kac formula, but we do not take this approach here).
We write, for $z \deq x-y$ and $\lambda_k$ as in \eqref{lambda_k}
\begin{align}
\notag
G_{\tau,t}(x;y) &\;=\; \sum_{k \in \mathbb{Z}^d} \frac{\ee^{-t \lambda_k/\tau}}{\tau(\ee^{\lambda_k/\tau}-1)} \,\ee^{2\pi \ii k \cdot z}\;=\;\frac{1}{\tau} \sum_{k \in \mathbb{Z}^d} \sum_{\ell=1}^{\infty} \ee^{-\ell \lambda_k/\tau} \,\ee^{-t \lambda_k/\tau}\,\ee^{2\pi \ii k \cdot z}
\\
\label{Q12_bounds_iv_1}
&\;=\;\frac{1}{\tau}  \sum_{\ell=1}^{\infty} \sum_{k \in \mathbb{Z}^d} \ee^{-\ell(4 \pi^2 |k|^2+\kappa)/\tau}\,\ee^{-t(4\pi^2 |k|^2+\kappa)/\tau} \,\ee^{2\pi \ii k\cdot z}\,.
\end{align}
Interchanging the orders of summation in the above calculation is justified by the exponential decay of the factors.
We rewrite the expression in \eqref{Q12_bounds_iv_1} as 
\begin{equation}
\label{Q12_bounds_iv_2}
\frac{\ee^{-\frac{\kappa t}{\tau}}}{\tau} \sum_{\ell=1}^{\infty} \ee^{-\frac{\ell \kappa}{\tau}} \Bigg\{\sum_{k \in \mathbb{Z}^d} \ee^{- \pi \big(\sqrt{\frac{4\pi(\ell+t)}{\tau}}k\big)^2}\,\ee^{2\pi \ii k \cdot z}\Bigg\}\,.
\end{equation}
Recalling that for $\xi \in \mathbb{R}^d$ and $\lambda>0$ we have
\begin{equation*}
\big(\lambda^{d/2} \ee^{-\pi[\lambda (x+z)]^2}\big){\,\widehat{}\,}(\xi)\;=\;\ee^{-\pi (\xi/\lambda)^2}\,\ee^{2 \pi \ii \xi \cdot z}\,,
\end{equation*}
and by applying the Poisson summation formula, we deduce that \eqref{Q12_bounds_iv_2} equals
\begin{align}
\notag
&\frac{\ee^{-\frac{\kappa t}{\tau}}}{\tau} \sum_{\ell=1}^{\infty} \ee^{-\frac{\ell \kappa}{\tau}} \Bigg\{\sum_{n \in \mathbb{Z}^d} \bigg(\frac{\tau}{4 \pi (\ell+t)}\bigg)^{d/2}\,\ee^{-\frac{\tau}{4 \pi (\ell+t)} (n+z)^2}\Bigg\}
\\
\label{Q12_bounds_iv_3}
&\;\geq\; \frac{C}{\tau} \, \sum_{\ell=1}^{\infty} \ee^{-\frac{\ell \kappa}{\tau}} \Bigg\{\sum_{n \in \mathbb{Z}^d} \bigg(\frac{\tau}{4 \pi (\ell+t)}\bigg)^{d/2}\,\ee^{-\frac{\tau}{4 \pi (\ell+t)} (n+z)^2}\Bigg\}
\,.
\end{align}
In the above inequality, we used that $\frac{\kappa t}{\tau}=\mathcal{O}(1)$.

For fixed $\ell \geq 1$, we note that 
\begin{equation*}
\mathcal{R}_{\ell} \;\deq\; \sum_{n \in \mathbb{Z}^d} \bigg(\frac{\tau}{4 \pi (\ell+t)}\bigg)^{d/2}\,\ee^{-\frac{\tau}{4 \pi (\ell+t)} (n+z)^2}
\end{equation*}
is a Riemann sum corresponding to the integral 
\begin{equation*}
\int_{\mathbb{R}^d} \ee^{-\pi |u|^2} \,\dd u\;=\;1
\end{equation*}
with mesh size $\sqrt{\frac{\tau}{4 \pi (\ell+t)}}$. In particular, we note that $\mathcal{R}_{\ell} \geq \frac{1}{2}$ provided that $\sqrt{\frac{\tau}{4 \pi (\ell+t)}}$ is sufficiently small. This holds for all $\ell \geq \ell_0$, where $\ell_0 \sim \tau$ is chosen sufficiently large. Therefore, for such $\ell_0$, we have by \eqref{Q12_bounds_iv_3} that 
\begin{equation*}
G_{\tau,t}(x;y) \;\geq\; \frac{C}{\tau}\sum_{\ell=\ell_0}^{\infty} \ee^{-\frac{\ell \kappa}{\tau}} \;=\; \frac{C}{\tau} \frac{\ee^{-\frac{\ell_0 \kappa}{\tau}}}{1-\ee^{-\frac{\kappa}{\tau}}} \sim \frac{1}{\tau (1-\ee^{-\frac{\kappa}{\tau}})} \gtrsim 1\,. 
\end{equation*}
Above, we used the fact that $\ell_0 \sim \tau$. We hence deduce \eqref{Q12_bounds_iv_1a}.
\end{itemize}
\end{proof}

By \eqref{Q^12} and Proposition \ref{Q^12_bounds} (i), we know that the quantum Green function $G_{\tau}$ defined in \eqref{G_tau} is bounded in $L^\infty(\Lambda \times \Lambda)$ uniformly in $\tau$ when $d=1$. This fact was used repeatedly in the analysis in \cite[Section 4.4]{FrKnScSo1}.  In higher dimensions, this is no longer true. We now give the appropriate bounds on $G_{\tau,0}$.

\begin{proposition}
\label{G_{tau,0}_bound}
Let us fix $q \in \mathcal{Q}_d$, where
\begin{equation}
\label{Q_d}
\mathcal{Q}_d \;\deq \;
\begin{cases}
[1,\infty]\,\,&\mbox{if}\,\,\,d=1
\\
[1,\infty)\,\,&\mbox{if}\,\,\,d=2
\\
[1,3) \,\,&\mbox{if}\,\,\,d=3.
\end{cases}
\end{equation}
\begin{itemize}
\item[(i)]
We have
\begin{equation}
\label{G_{tau,0}_bound_claim}
\|G_\tau(x;\cdot)\|_{L^q(\Lambda)} \;=\;\|G_\tau(\cdot;x)\|_{L^q(\Lambda)} \;\leq\; C(q)\,,
\end{equation}
uniformly in $x \in \Lambda$, $\tau \geq 1$. In particular, we have
\begin{equation}
\label{G_{tau,0}_bound_claim2}
\|G_\tau\|_{L^q(\Lambda^2)}\;\leq\;C(q)\,.
\end{equation}
\item[(ii)] More generally, for $t \in (-1,1)$, we have
\begin{equation*}
\|Q^{(1)}_{\tau,t}(x;\cdot)\|_{L^q(\Lambda)}\;=\;\|Q^{(1)}_{\tau,t}(\cdot\,;x)\|_{L^q(\Lambda)}\;\leq\;C(q)\,,
\end{equation*}
uniformly in $t,x,\tau$.
\end{itemize}
\end{proposition}
\begin{proof}
\begin{itemize}
\item[(i)]
We note that \eqref{G_{tau,0}_bound_claim2} follows immediately from \eqref{G_{tau,0}_bound_claim}. We now prove \eqref{G_{tau,0}_bound_claim}.
Let $x \in \Lambda$ and $\tau \geq 1$ be given.
Recall that by \eqref{positivity} with $t=0$, we have $G_\tau(x;\cdot)=G_\tau(\cdot;x)$, so it suffices to prove the bound on $\|G_\tau(x,\cdot)\|_{L^q(\Lambda)}$. The claim when $d=1$ follows from Proposition \ref{Q^12_bounds} (i) and H\"{o}lder's inequality, so we consider the cases when $d=2,3$.
By \eqref{G_tau}, we have that 
\begin{equation*}
G_\tau(x;y)\;=\;\sum_{k \in \mathbb{Z}^d} \frac{1}{\tau(\ee^{\lambda_k/\tau}-1)}\,\ee^{2\pi \ii k \cdot (x-y)}\,.
\end{equation*}
Therefore, for $k \in \mathbb{Z}^d$ we have
\begin{equation*}
\big|\big(G_\tau(x;\cdot)\big)\,\,\widehat{}\,\,(k)\big|\;\leq\; \frac{1}{\tau(\ee^{\lambda_k/\tau}-1)}\;\leq\;\frac{1}{\lambda_k} \;\leq\; \frac{C}{|k|^2+1}\,.
\end{equation*}
We hence deduce that 
\begin{equation}
\label{G_{tau,0}_bound_proof_s1}
\|G_\tau(x;\cdot)\|_{H^s(\Lambda)} \;\leq\;C(s)\,,\,\,\, \mbox {for} \quad 0 \;\leq\; s \;<\; 2-\frac{d}{2}\,.
\end{equation}
By using \eqref{G_{tau,0}_bound_proof_s1} and Sobolev embedding on the torus, \eqref{G_{tau,0}_bound_claim} follows. 
\item[(ii)] Using \eqref{Q^12}--\eqref{positivity}  and arguing as in (i), we deduce that
\begin{equation}
\label{Q_{tau,t}_bound_proof_s1}
\big|\big(Q_{\tau,t}^{(1)}(x;\cdot)\big)\,\,\widehat{}\,\,(k)\big|\;=\;\big|\big(Q_{\tau,t}^{(1)}(\cdot;x)\big)\,\,\widehat{}\,\,(k)\big|\;\leq\;\frac{\ee^{-\{t\}\lambda_k/\tau}}{\tau(\ee^{\lambda_k/\tau}-1)}\;\leq\;\frac{1}{\tau(\ee^{\lambda_k/\tau}-1)}\;\leq\;\frac{1}{\lambda_k}\,.
\end{equation}
We now proceed as in the proof of \eqref{G_{tau,0}_bound_claim2} to deduce the claim.
\end{itemize}
\end{proof}

We can relate the set $\mathcal{P}_d$ defined in \eqref{P_d} with the set $\mathcal{Q}_d$ defined in \eqref{Q_d}.
\begin{lemma}
\label{QP_embedding}
For $p \in \mathcal{P}_d$ we have $2p' \in \mathcal{Q}_d$.
\end{lemma}
\begin{proof}
This follows from \eqref{P_d} and \eqref{Q_d}.
\end{proof}

\subsection{Bounds on the explicit terms}
\label{Bounds on the explicit terms}
Our goal in this section is to first prove an upper bound as in \eqref{Borel_summation_unbounded2} on the explicit terms $a_{\tau,m}$. Throughout this section, we fix $p \in \mathcal{P}_d$ and a $d$-admissible interaction potential $w \in L^p(\Lambda)$. Furthermore, we let $\beta \in \cal B_d$, as defined in \eqref{B_d} above.
With the parameters $p,\beta$ chosen in this way, let $w_\tau$ be as in Lemma \ref{w_1D_approximation} for $d=1$ and as in Lemma \ref{w_positive_approximation} for $d=2,3$. 
We recall the definition of $\mathfrak{R}(d,m,r)$ from \eqref{R(d,m,r)}, as well as that $\mathcal{I}_{\tau,\Pi}^{\xi}(\mathbf{t})$ is given by Definition \ref{def:I_tau_PI_unbounded} when $\xi \in \mathbb{B}_r$ and by \eqref{value_of_Pi_1D_delta} when $d=1$ and $\xi=\mathrm{Id}_r$. Furthermore, we recall \eqref{function_f_all_D}. We now state the upper bounds that we prove.

\begin{proposition}
\label{Product of subgraphs}
Fix $m,r \in \mathbb{N}$. Given $\Pi \in \mathfrak{R}(d,m,r)$, $\mathbf{t} \in \mathfrak{A}(m)$, the following estimates hold uniformly in  $\xi \in \mathbf{C}_r$ and $\tau \geq 1$.
\begin{itemize}
\item[(i)] We have
\begin{equation*}
\big| \mathcal{I}_{\tau,\Pi}^{\xi} (\mathbf{t}) \big| \leq C_0^{m+r}(1+\|w\|_{L^{p}(\Lambda)})^{m}\,,
\end{equation*}
for some $C_0>0$.
\item[(ii)] $|a^{\xi}_{\tau,m}| \leq (Cr)^{r} \,C^m  \,(1+\|w\|_{L^p(\Lambda)})^{m}\,m!\,$.
\item[(iii)] $|f^{\xi}_{\tau,m}(\mathbf{t})| \leq C^{m+r}\,(1+\|w\|_{L^p(\Lambda)})^m\,(2m+r)!\,$.  
\end{itemize}
\end{proposition}

\begin{proof}
Let us note that (ii) is obtained from (i) by applying \eqref{Wick_application_identity_unbounded_all_D}, as well as the fact that $|\mathfrak{R}(d,m,r)| \leq (2m+r)!\,$. Similarly, (i) implies (iii) by \eqref{function_f_all_D}.
We now prove (i).
Let us consider first the case when $\xi \in \mathbf{B}_r$. Then, by Lemma \ref{lem:I_rep_unbounded}, $\mathcal I_{\tau,\Pi}^{\xi}(\mathbf{t})$ is given by \eqref{I_representation_unbounded}.
The multigraph $\mathcal{E}$ corresponding to $\Pi$ decomposes into paths $\mathfrak{P}_1, \ldots, \mathfrak{P}_k$, ordered in an arbitrary way.
Recalling \eqref{J_nonnegative}, we deduce from \eqref{I_representation_unbounded} and the path decomposition of $\mathfrak{P}$ that
\begin{equation}
\label{Product of subgraphs unbounded 1}
\big|\mathcal I_{\tau,\Pi}^{\xi}(\mathbf{t})\big| \;\leq\; \int_{\Lambda^{\mathcal V}} \dd \mathbf{y} \, \Bigg[\prod_{i=1}^{m} \big|w_\tau(y_{i,1}-y_{i,2})\big|\Bigg] \,\big|\xi(\mathbf{y_1})\big|\, \Bigg[\prod_{j=1}^{k}\, \prod_{e\, \in\, \mathfrak{P}_j} \mathcal J_{\tau, e}(\mathbf{y_e}, \mathbf{s})\Bigg]\,.
\end{equation}
In the sequel, we repeatedly use the nonnegativity of $\mathcal{J}_{\tau,e}(\mathbf{y_e}, \mathbf{s})$, i.e.\ \eqref{J_nonnegative} without further comment.
Since $\|w_\tau\|_{L^\infty}$ is not bounded uniformly in $\tau$, we cannot estimate \eqref{Product of subgraphs unbounded 1} by arguing as in \cite[Section 2.4]{FrKnScSo1}. Instead, we need to systematically distribute the factors of $w_\tau$ among the paths $\mathfrak{P}_1,\ldots,\mathfrak{P}_k$, similarly as in the proof of \cite[Proposition 4.11]{FrKnScSo1}.

As in \cite[Section 4.4]{FrKnScSo1}, given a vertex $a=(i_a,\vartheta_a) \in \mathcal{V}_2$, we define the vertex $a^* \in \mathcal{V}_2$ by 
\begin{equation}
\label{a_star_unbounded}
a^* \deq (i_a,3-\vartheta_a)\,.
\end{equation}
Furthermore, the factors of $w_\tau$ are distributed among the paths $\mathfrak{P}_1,\ldots,\mathfrak{P}_k$ according to the rule that for $a \in \mathcal{V}_2$ the function $\mathcal{W}_\tau^{a}$ is given by
\begin{align} 
\label{cal W_unbounded}
\mathcal{W}_\tau^{a} \;\deq\;
\begin{cases}
w_\tau &\mbox{if }a \in \mathfrak{P}_j,\,a^{*} \in \mathfrak{P}_{\ell} \,\, \mbox{for } 1 \leq j<\ell \leq k \\
1 &\mbox{if }a \in  \mathfrak{P}_j,\,a^{*} \in  \mathfrak{P}_{\ell} \,\, \mbox{for } 1 \leq \ell<j \leq k \\
w_\tau &\mbox{if }a,a^* \in  \mathfrak{P}_j \,\, \mbox{for } 1 \leq j \leq k \,\, \mbox{and } \vartheta_a=1 \\
1 &\mbox{if }a,a^* \in  \mathfrak{P}_j \,\, \mbox{for } 1 \leq j \leq k \,\, \mbox{and } \vartheta_a=2\,.
\end{cases}
\end{align}
In other words, we put the factor of $w_\tau(y_a-y_{a^*})$ with the vertex in $\{a,a^*\}$ belonging to the cycle of lowest index. In case these two indices are the same, we put it with the vertex whose second component $\vartheta$ is equal to $1$. By \eqref{cal W_unbounded}, we have that 
\begin{equation}
\label{product w tau_unbounded}
\prod_{i=1}^{m} w_\tau(y_{i,1}-y_{i,2}) \;=\; \prod_{a \in \mathcal{V}_2} \mathcal{W}_{\tau}^{a}(y_{a}-y_{a^*})\,.
\end{equation}
Moreover, since $p \in \mathcal{P}_d$, we also have by Lemma \ref{w_1D_approximation} (ii)--(iii) (when $d=1$) and by Lemma \ref{w_positive_approximation} (ii)--(iii) (when $d=2,3$) that $\mathcal{W}_\tau^a$ satisfies
\begin{align}
\label{cal W bound_unbounded2}
\|\mathcal{W}_{\tau}^a\|_{L^{\infty}(\Lambda)} &\;\leq\; \tau^{\beta}
\\
\label{cal W bound_unbounded}
\|\mathcal{W}_{\tau}^a\|_{L^p(\Lambda)} &\;\leq\; 1+\|w_{\tau}\|_{L^p(\Lambda)}\;\leq\;1+ C\|w\|_{L^p(\Lambda)}\,.
\end{align}
We note that the term $1$ on the right-hand side of \eqref{cal W bound_unbounded} corresponds to the case when $\mathcal{W}_{\tau}^a=1$ and when $\|w\|_{L^p(\Lambda)}$ is small.

We now adapt the method from the proof of \cite[Proposition 4.11]{FrKnScSo1}. The goal is to integrate out the $\mathbf{y_2}$-variables by successively integrating them out in the paths $\mathfrak{P}_1, \dots, \mathfrak{P}_k$.  %For completeness, we present the details of the setup and refer the reader to \cite{FrKnScSo1} for the details of the proofs, which carry over to our setting. 
Let us first recall some useful notation, which was used in \cite{FrKnScSo1}.
Given $1 \leq j \leq k+1$, we define 
\begin{equation}
\label{V_{2,j}}
\mathcal V_{2,j} \deq \mathcal V_2 \,\Big\backslash \,\Big(\bigcup_{\ell=1}^{j-1} \mathcal{V}_2(\mathfrak{P}_{\ell})\Big)\,.
\end{equation}
This is the set of variables that remain after we have integrated out the vertices which occur in the first $j-1$ paths $\mathfrak{P}_1, \ldots, \mathfrak{P}_{j-1}$. In particular, we have $\mathcal{V}_{2,1}=\mathcal{V}_2$ and $\mathcal{V}_{2,k+1}=\emptyset$. Given a subset $A \subseteq \mathcal{V}$, we introduce the variable $\mathbf{y}_A \deq (y_a)_{a \in A}$. Similarly, given $\alpha \in [1,\infty]$, we denote by $L^\alpha_A$ the space $L^\alpha_{\mathbf{y}_A}$ with the corresponding norm. 

By \eqref{product w tau_unbounded}, it follows that the right-hand side of \eqref{Product of subgraphs unbounded 1} equals

\begin{equation*}
\int_{\Lambda^{\mathcal{V}_1}} \dd \mathbf{y_1} \, \big|\xi(\mathbf{y_1})\big| \int_{\Lambda^{\mathcal{V}_2}} \dd \mathbf{y_2} \,\prod_{j=1}^{k} \Bigg[ \prod_{e \in \mathfrak{P}_j} \mathcal{J}_{\tau, e}(\mathbf{y}_e, \mathbf{s}) \, \prod_{a \in \mathcal {V}_2(\mathfrak{P}_j)} \big|\mathcal{W}_\tau^{a}(y_{a}-y_{a^*})\big|\Bigg]\,,
\end{equation*}
which by applying the Cauchy-Schwarz inequality in $\mathbf{y_1}$ and recalling that $\|\xi\|_{L^2_{\mathbf{y_1}}} \leq 1$ is
\begin{equation}
\label{Product of subgraphs unbounded 2}
\;\leq\;
\Bigg\|\int_{\Lambda^{\mathcal{V}_2}} \dd \mathbf{y_2} \, \prod_{j=1}^{k} \Bigg[\prod_{e \in \mathfrak{P}_j} \mathcal{J}_{\tau, e}(\mathbf{y}_e, \mathbf{s})\, \prod_{a \in \mathcal{V}_2(\mathfrak{P}_j)} \big| \mathcal{W}_\tau^{a}(y_{a}-y_{a^*})\big|\Bigg] \Bigg\|_{L^{2}_{\mathbf{y_1}}}\,.
\end{equation}
The claim of the proposition follows from \eqref{Product of subgraphs unbounded 2} if we show that for all $1 \leq \ell \leq k$ we have the recursive inequality
\begin{align}
\notag
&\Bigg\|\int_{\Lambda^{\mathcal{V}_{2,\ell}}} \dd \mathbf{y}_{\mathcal{V}_{2,\ell}} \, \prod_{j=\ell}^{k} \Bigg[\prod_{e \in \mathfrak{P}_j} \mathcal{J}_{\tau, e}(\mathbf{y}_e, \mathbf{s})\, \prod_{a \in \mathcal{V}_2(\mathfrak{P}_j)} \big|\mathcal{W}_\tau^{a}(y_{a}-y_{a^*})\big|\Bigg]\Bigg\|_{L^{\infty}_{\mathcal{V}_2 \setminus \mathcal{V}_{2,\ell}}L^{2}_{\mathbf{y_1}}}
\\
\notag
&\;\leq\; C_0^{\,|\mathcal{V}(\mathfrak{P}_{\ell})|}\,\big(1+\|w\|_{L^p(\Lambda)}\big)^{\,|\mathcal{V}_2(\mathfrak{P}_{\ell})|} \,
\\
\label{Inductive inequality 1 unbounded}
&\times
\Bigg\|\int_{\Lambda^{\mathcal{V}_{2,\ell+1}}} \dd \mathbf y_{\mathcal{V}_{2,\ell+1}} \, \prod_{j=\ell+1}^{k} \Bigg[\prod_{e \in \mathfrak{P}_j} \mathcal{J}_{\tau, e}(\mathbf y_e, \mathbf s)\, \prod_{a \in \mathcal{V}_2(\mathfrak{P}_j)} \big|\mathcal{W}_\tau^{a}(y_{a}-y_{a^*})\big|\Bigg]\Bigg\|_{L^{\infty}_{\mathcal{V}_2 \setminus \mathcal{V}_{2,\ell+1}} L^{2}_{\mathbf{y_1}}}\,,
\end{align}
for some $C_0>0$. We first show that \eqref{Inductive inequality 1 unbounded} follows if we prove that
\begin{align}
\notag
&\Bigg\|\prod_{e \in \mathfrak{P}_{\ell}} \mathcal{J}_{\tau, e}(\mathbf{y}_e, \mathbf{s})\, \prod_{a \in \mathcal{V}_2(\mathfrak{P}_{\ell})} \big|\mathcal{W}_\tau^{a}(y_{a}-y_{a^*})\big|\Bigg\|_{L^{2}_{\mathbf{y_1}} L^{\infty}_{\mathcal{V}^{*}(\mathfrak{P}_{\ell})}  L^1_{\mathcal{V}_2(\mathfrak{P}_{\ell})}}
\\ 
\label{Inductive inequality 2 unbounded}
&\;\leq\;C_0^{\,|\mathcal{V}(\mathfrak{P}_{\ell})|}\,\big(1+\|w\|_{L^p(\Lambda)}\big)^{\,|\mathcal{V}_2(\mathfrak{P}_{\ell})|}\,,
\end{align}
where for $\mathfrak{P} \in \conn (\mathcal{E})$, $\mathcal{V}^*(\mathfrak{P}) \deq \{a^{*}:\,a \in \mathcal{V}_2 (\mathfrak{P}) \} \setminus \mathcal{V}_2(\mathfrak{P})$. In other words, $\mathcal{V}^*(\mathfrak{P})$ denotes the vertices in $\mathcal{V}_2$ that are connected to some vertex in $\mathfrak{P}$ by an interaction $w_\tau$, but which do not belong to the set $\mathcal V_2(\mathfrak{P})$.

In order to see that \eqref{Inductive inequality 2 unbounded} implies \eqref{Inductive inequality 1 unbounded}, we fix $1 \leq \ell \leq k$ and observe that we can write the set $\mathcal{V}_{2,\ell}$ as the disjoint union
\begin{equation*}
\label{V2l}
\mathcal{V}_{2,\ell}\;=\;\mathcal{V}_2(\mathfrak{P}_\ell) \,\sqcup\, \Big(\mathcal{V}_{2,\ell+1} \,\cap\,\mathcal{V}^{*}(\mathfrak{P}_\ell)\Big) \,\sqcup\, \Big(\mathcal{V}_{2,\ell+1} \setminus \mathcal{V}^{*}(\mathfrak{P}_\ell)\Big)\,.
\end{equation*}
We can hence rewrite the expression on the left-hand side of the inequality in \eqref{Inductive inequality 1 unbounded} as
\begin{multline}
\label{L 1 L infty A_unbounded}
\Bigg\|\int_{\Lambda^{\mathcal{V}_{2,\ell+1} \cap \mathcal{V}^*(\mathfrak{P}_\ell)}} \dd \mathbf{y}_{\mathcal{V}_{2,\ell+1} \cap \mathcal{V}^*(\mathfrak{P}_\ell)} \,
\int_{\Lambda^{\mathcal{V}_2(\mathfrak{P}_\ell)}} \dd \mathbf{y}_{\mathcal{V}_2(\mathfrak{P}_\ell)} \,\prod_{e \in \mathfrak{P}_\ell} \mathcal{J}_{\tau, e}(\mathbf{y}_e, \mathbf{s})\, \prod_{a \in \mathcal{V}_2(\mathfrak{P}_\ell)} \big|\mathcal{W}^a_{\tau}(y_{a}-y_{a^{*}})\big|
\\
\Bigg\{\int_{\Lambda^{\mathcal{V}_{2,\ell+1} \setminus \mathcal{V}^*(\mathfrak{P}_\ell)}} \dd \mathbf{y}_{\mathcal{V}_{2,\ell+1} \setminus \mathcal{V}^*(\mathfrak{P}_\ell)} \,
\prod_{j=\ell+1}^{k} \Bigg[\prod_{e \in \mathfrak{P}_j} \mathcal{J}_{\tau,e}(\mathbf{y}_e, \mathbf{s})\, \prod_{a \in \mathcal {V}_2(\mathfrak{P}_j)}\big|\mathcal{W}^a_{\tau}(y_a-y_{a^{*}})\big|\Bigg]  \Bigg\}\Bigg\|_{L^{\infty}_{\mathcal{V}_2 \setminus \mathcal{V}_{2,\ell}}L^{2}_{\mathbf{y_1}}}\,,
\\
\;=\; \Bigg\|\Bigg\|\int_{\Lambda^{\mathcal{V}_{2,\ell+1} \cap \mathcal{V}^*(\mathfrak{P}_\ell)}} \dd \mathbf{y}_{\mathcal{V}_{2,\ell+1} \cap \mathcal{V}^*(\mathfrak{P}_\ell)} \,
\int_{\Lambda^{\mathcal{V}_2(\mathfrak{P}_\ell)}} \dd \mathbf{y}_{\mathcal{V}_2(\mathfrak{P}_\ell)} \,\prod_{e \in \mathfrak{P}_\ell} \mathcal{J}_{\tau, e}(\mathbf{y}_e, \mathbf{s})\, \prod_{a \in \mathcal{V}_2(\mathfrak{P}_\ell)}\big|\mathcal{W}^a_{\tau}(y_a-y_{a^{*}})\big|
\\
\Bigg\{\int_{\Lambda^{\mathcal{V}_{2,\ell+1} \setminus \mathcal{V}^*(\mathfrak{P}_\ell)}} \dd \mathbf{y}_{\mathcal{V}_{2,\ell+1} \setminus \mathcal{V}^*(\mathfrak{P}_\ell)} \,
\prod_{j=\ell+1}^{k} \Bigg[\prod_{e \in \mathfrak{P}_j} \mathcal{J}_{\tau,e}(\mathbf{y}_e, \mathbf{s})\, \prod_{a \in \mathcal {V}_2(\mathfrak{P}_j)}\big|\mathcal{W}^a_{\tau}(y_a-y_{a^{*}})\big|\Bigg]  \Bigg\} \Bigg\|_{L^{2}_{\mathbf{y_1}}} \Bigg\|_{L^{\infty}_{\mathcal{V}_2 \setminus \mathcal{V}_{2,\ell}}}\,.
\end{multline}
In \eqref{L 1 L infty A_unbounded}, we first integrate in $\mathbf{y}_{\mathcal{V}_2(\mathfrak{P}_\ell)}$, keeping in mind that the 
$\mathbf{y}_{\mathcal{V}_{2,\ell+1} \setminus \mathcal{V}^{*}(\mathfrak{P}_\ell)}$ integral does not depend on $\mathbf{y}_{\mathcal{V}_2(\mathfrak{P}_\ell)}$. We then apply an $L^{\infty}-L^1$ H\"{o}lder inequality in the 
$\mathbf{y}_{\mathcal{V}_{2,\ell+1} \cap \mathcal{V}^{*}(\mathfrak{P}_{\ell})}$ variable, putting in $L^{\infty}_{\mathbf{y}_{\mathcal{V}_{2,\ell+1} \cap \mathcal{V}^{*}(\mathfrak{P}_{\ell})}}$ the $\mathbf{y}_{\mathcal{V}_2(\mathfrak{P}_\ell)}$ integral, i.e.\ the quantity
\begin{equation*}
I\;\deq\;\int_{\Lambda^{\mathcal{V}_2(\mathfrak{P}_\ell)}} \dd \mathbf{y}_{\mathcal{V}_2(\mathfrak{P}_\ell)} \,\prod_{e \in \mathfrak{P}_\ell} \mathcal{J}_{\tau, e}(\mathbf{y}_e, \mathbf{s})\, \prod_{a \in \mathcal{V}_2(\mathfrak{P}_\ell)}\big|\mathcal{W}^a_{\tau}(y_a-y_{a^{*}})\big|\,,
\end{equation*}
which is a function of $(\mathbf{y_1},\mathbf{y}_{\mathcal{V}^*(\mathfrak{P}_\ell)})$.
For fixed $\mathbf{y_1}$, we estimate from above the factor 
\begin{equation*}
\|I\|_{L^{\infty}_{\mathbf{y}_{\mathcal{V}_{2,\ell+1} \cap \mathcal{V}^{*}(\mathfrak{P}_{\ell})}}}
\end{equation*}
by taking suprema in the remaining variables in  $\mathbf{y}_{\mathcal{V}^*(\mathfrak{P}_\ell)}$, i.e.\ in $\mathbf{y}_{\mathcal{V}^*(\mathfrak{P}_\ell) \setminus \mathcal{V}_{2,\ell+1}}$.
Next we take the $L^2_{\mathbf{y_1}}$ norm, keeping in mind that each variable $y_a, a \in \mathcal{V}_1$ appears in at most one of the composite variables $\mathbf{y}_{e}, e \in \mathfrak{P}_j, \ell \leq j \leq k$.
Taking suprema in all of the remaining variables, we conclude that \eqref{L 1 L infty A_unbounded} is
\begin{align}
\notag
&\;\leq\;
\Bigg\|\prod_{e \in \mathfrak{P}_\ell} \mathcal{J}_{\tau,e}(\mathbf{y}_e,\mathbf{s})\,\prod_{a \in \mathcal{V}_2(\mathfrak{P}_\ell)}\big|\mathcal{W}^a_{\tau}(y_a-y_{a^{*}})\big|\Bigg\|_{L^{2}_{\mathbf{y_1}} L^{\infty}_{\mathcal{V}^*(\mathfrak{P}_\ell)} L^1_{\mathcal{V}_2(\mathfrak{P}_\ell)}}
\\
\label{L1 Linfty_unbounded}
&\times 
\Bigg\|\int_{\Lambda^{\mathcal{V}_{2,\ell+1}}} \dd \mathbf{y}_{\mathcal{V}_{2,\ell+1}}\,\prod_{j=\ell+1}^{k} \Bigg[\prod_{e \in \mathfrak{P}_j} \mathcal{J}_{\tau, e}(\mathbf{y}_e, \mathbf{s})\, \prod_{a \in \mathcal{V}_2(\mathfrak{P}_j)} \big|\mathcal{W}^a_{\tau}(y_a-y_{a^{*}})\big|\Bigg] \Bigg\|_{L^{\infty}_{\mathcal{V}_2 \setminus \mathcal{V}_{2,\ell+1}}L^{2}_{\mathbf{y_1}}}\,.
\end{align}
From \eqref{L1 Linfty_unbounded}, we obtain that the claim of the proposition indeed follows from  \eqref{Inductive inequality 2 unbounded}.

Let us reformulate \eqref{Inductive inequality 2 unbounded}. For $\mathfrak{P} \in \conn (\mathcal{E})$, we define
\begin{equation}
\label{I(P)}
\mathfrak{I}(\mathfrak {P}) \;\deq\; \int \dd \mathbf{y}_{\mathcal{V}_2(\mathfrak{P})} \, \prod_{e \in \mathfrak{P}} \mathcal{J}_{\tau, e}(\mathbf{y}_e, \mathbf{s})\, \prod_{a \in \mathcal{V}_2(\mathfrak{P})} \big|\mathcal{W}_\tau^{a}(y_{a}-y_{a^*})\big|\,.
\end{equation}
Note that $\mathfrak{I}(\mathfrak{P})$ is a function of $(\mathbf{y_1}, \mathbf{y}_{\mathcal{V}^{*}(\mathfrak{P})})$. Furthermore, \eqref{Inductive inequality 2 unbounded} is equivalent to showing that
\begin{equation}
\label{Inductive inequality 2 unbounded_B}
\|\mathfrak{I}(\mathfrak{P})\|_{L^{2}_{\mathbf{y_1}} L^{\infty}_{\mathcal{V}^{*}(\mathfrak{P})}}\;\leq\;C_0^{\,|\mathcal {V}(\mathfrak{P})|}
\,\,\big(1+\|w\|_{L^p(\Lambda)}\big)^{\,|\mathcal{V}_2(\mathfrak{P})|}\,,
\end{equation}
when $\mathfrak{P}=\mathfrak{P}_{\ell}$. We now prove \eqref{Inductive inequality 2 unbounded_B} for $\mathfrak{P} \in \conn (\mathcal{E})$ by induction on $n \deq |\mathcal{V}(\mathfrak{P})|$. In doing so, we need to keep track of the time carried by each edge of $\mathfrak{P}$ in the sense that we precisely define now. 

Let us first assume that $\mathfrak{P}$ is a closed path in the sense of Definition \ref{Open_and_Closed_paths_unbounded} above. We list its edges as $e_1,e_2,\ldots,e_{n}$, where the edges $e_{j}$ and $e_{j+1}$ are incident at vertex $a_j \in \mathcal{V}(\mathfrak{P})$. 
For an example, see Figure \ref{Closed_Cycle_Example_Graph}.
\begin{figure}[!ht]
\begin{center}
\includegraphics[scale=0.5]{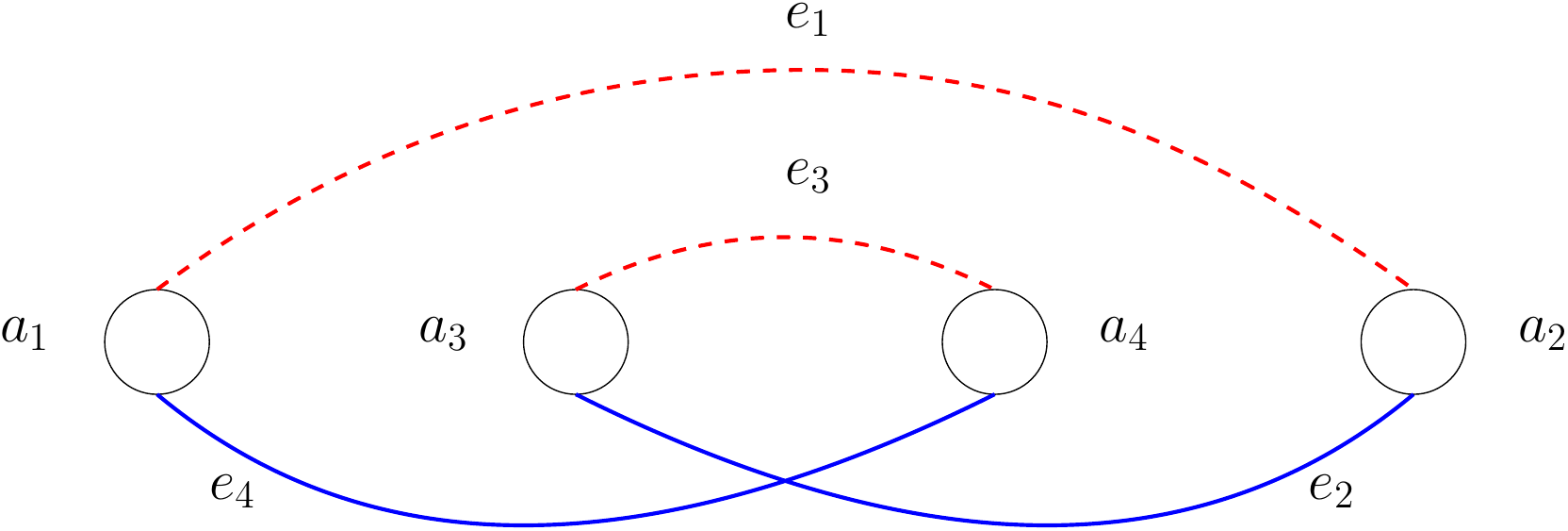}
\end{center}
\caption{This is an example of a closed path $\mathfrak{P}$ of size $4$. We order the vertices in $\mathcal{V}_2$ from left to right. The edges with colour $\sigma=+1$ are depicted by the full blue lines, whereas edges with colour $\sigma=-1$ are depicted by red dashed lines.
\label{Closed_Cycle_Example_Graph}
}
\end{figure}

Here, and in the sequel, we consider all indices modulo $n$. Given $j \in \{1,\ldots,n\}$, the time carried by the edge $e_j$ is given by
\begin{align}
\zeta_j \;\deq\;
\label{zeta_unbounded_j}
\begin{cases}
\sigma(e_1)\,(s_{a_j}-s_{a_{j+1}}) &\mbox{if }a_1<a_2\\
-\sigma(e_1)\,(s_{a_j}-s_{a_{j+1}}) &\mbox{otherwise,}
\end{cases}
\end{align}
where $\sigma$ is given by Definition \ref{def_collapsed_graph_unbounded} (iii). Note that $\zeta_j$ is an intrinsically defined quantity that one can read off from \eqref{J_e_unbounded} and \eqref{I_representation_unbounded}. By \eqref{zeta_unbounded_j}, we obtain
\begin{equation}
\label{path_sum_1_unbounded}
\sum_{j=1}^{n} \zeta_j \;=\; \pm \sum_{j=1}^{n} \sigma(e_1)\,(s_{a_j}-s_{a_{j+1}})\;=\;0\,,
\end{equation}
and for all $i \in \{1,\ldots,n\}, k \in \{1,\ldots,n-1\}$
\begin{equation}
\label{path_sum_2_unbounded}
\sum_{j=i}^{i+k} \zeta_j \;=\; \pm \sum_{j=i}^{i+k}\sigma(e_1)\,(s_{a_j}-s_{a_{j+1}}) \;=\; \pm \sigma(e_1)\,(s_{a_i}-s_{a_{i+k+1}}) \in (-1,1)\,.
\end{equation}
In the above calculations, we used the assumption that $\mathbf{t} \in \mathfrak{A}(m)$.
Analogous results to \eqref{path_sum_1_unbounded}-\eqref{path_sum_2_unbounded} hold for open paths $\mathfrak{P}$ after appropriate modifications.

\subsubsection*{Base of induction: $n=1$ (for $d=1$) and $n=2$ (for $d=2,3$)}

If $n=1$, then $\mathfrak{P}$ is a loop based at some $a_1 \in \mathcal{V}_2$.
Due to the Wick ordering, we note that this situation occurs only when $d=1$. In this case, $\zeta_1=0$, i.e.\ there is no time-evolution applied to the quantum Green function. See Figure \ref{Induction_Graph_1} below.

\begin{figure}[!ht]
\begin{center}
\includegraphics[scale=0.5]{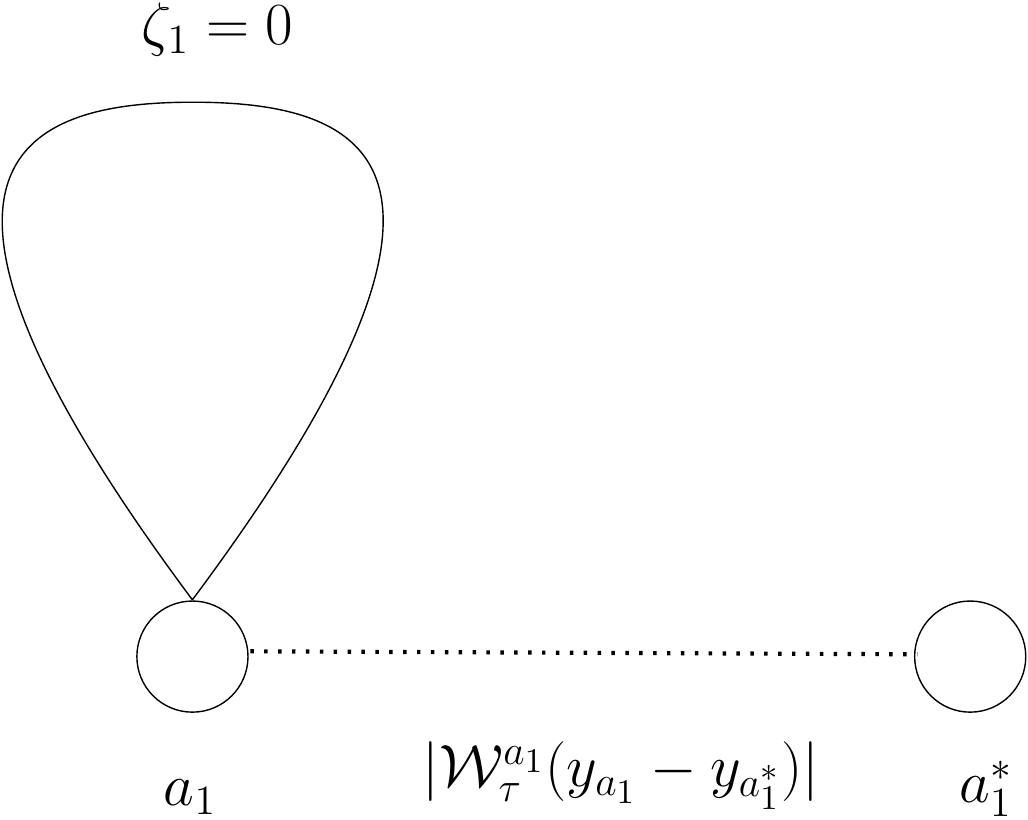}
\end{center}
\caption{In the sequel, we no longer draw the vertices in order corresponding to $\leq$ . The full lines depict the time-evolved Green functions. On these edges, we write the associated time $\zeta$. The dotted lines depict the interactions through the factors of $|\mathcal{W}_{\tau}^{a}|$.
\label{Induction_Graph_1}}
\end{figure}

By using the fact that $a_1^* \neq a_1$ and H\"{o}lder's inequality in $y_{a_1}$, we hence have
\begin{align}
\notag
&\mathfrak{I}(\mathfrak{P}) \;=\; \int \dd y_{a_1}\,\big|\mathcal{W}_\tau^{a_1} (y_{a_1}-y_{a_1^*})\big| \, G_{\tau}(y_{a_1};y_{a_1}) \;\leq\; \|\mathcal{W}_\tau^{a_1}\|_{L^1(\Lambda)}\,
\|G_\tau\|_{L^\infty(\Lambda^2)}
\\
\label{Base_n=1}
&\;\leq\;\|\mathcal{W}_\tau^{a_1}\|_{L^p(\Lambda)}\,
\|G_\tau\|_{L^\infty(\Lambda^2)}\;\leq\;C+C\|w_{\tau}\|_{L^p(\Lambda)} \;\leq\; C\big(1+ \|w\|_{L^p(\Lambda)}\big)\,.
\end{align}
Above, we used \eqref{cal W bound_unbounded} and Proposition \ref{G_{tau,0}_bound} (i) when $d=1$. We also used that $L^p(\Lambda) \hookrightarrow L^1(\Lambda)$ by compactness of $\Lambda$. We note that \eqref{Base_n=1} is an acceptable upper bound.

%In the sequel, we use such inclusions without further comment.

If $n=2$, we consider three cases.
\begin{itemize}
\item[(1)] $\mathfrak{P}$ is a closed path connecting $a_1,a_2 \in \mathcal{V}_2$ satisfying $a_2=a_1^{*}$. See Figure \ref{Induction_Graph_2}.

\begin{figure}[!ht]
\begin{center}
\includegraphics[scale=0.5]{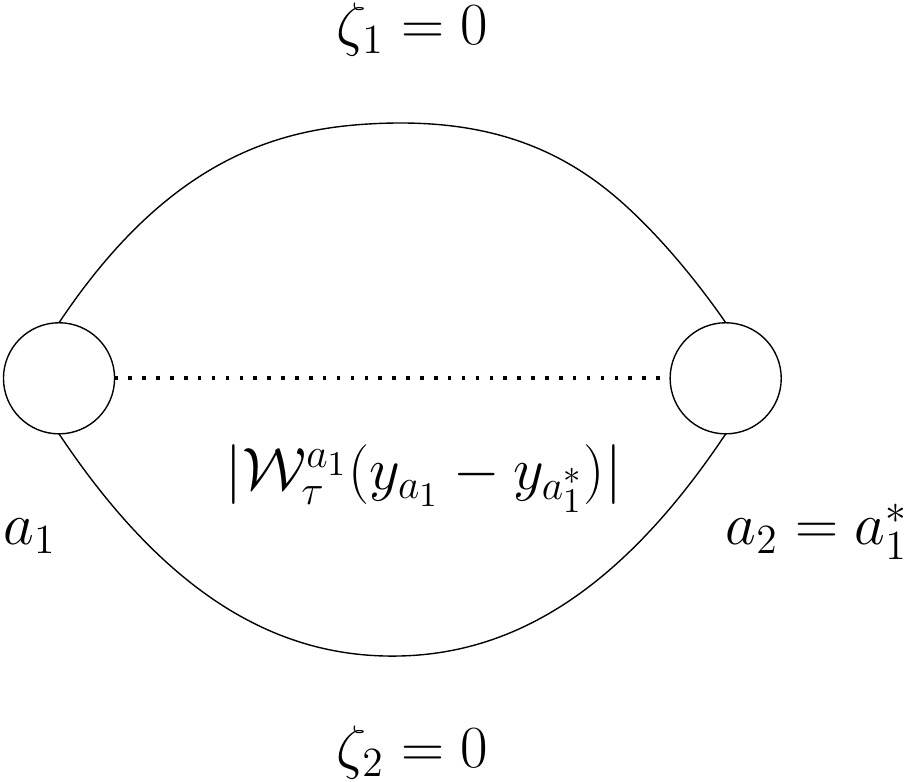}
\end{center}
\caption{
\label{Induction_Graph_2}}
\end{figure}

In this case, we have $\zeta_1=\zeta_2=0$. We can assume, without loss of generality, that $\vartheta_{a_1}=1$. Hence, \eqref{cal W_unbounded} yields that $\mathcal{W}_{\tau}^{a_1}=w_{\tau}$ and $\mathcal{W}_{\tau}^{a_2}=1$. So, we obtain 
\begin{multline}
\label{Base_n=2_(1)}
\mathfrak{I}(\mathfrak{P})\;=\;\int \dd y_{a_1}\,\dd y_{a_2}\,|w_{\tau}(y_{a_1}-y_{a_2})|\,G_{\tau}(y_{a_1};y_{a_2})\,G_{\tau}(y_{a_2};y_{a_1})
\\
+\frac{1}{\tau} \,\int \dd y_{a_1}\,|w_{\tau}(0)|\,G_{\tau}(y_{a_1};y_{a_1})
\;\leq\;\|w_{\tau}\|_{L^p(\Lambda)}\,\|G_\tau\|_{L^{2p'}(\Lambda^2)}^2+\frac{1}{\tau} \|w_{\tau}\|_{L^{\infty}(\Lambda)} \,\|G_{\tau}\|_{L^{\infty}(\Lambda^2)}
\\
\;\leq\;C\|w\|_{L^p(\Lambda)}+\tau^{-1+\beta}\,\|G_{\tau}\|_{L^{\infty}(\Lambda^2)}
\,.
\end{multline}
In order to get the first term on the last line of \eqref{Base_n=2_(1)}, we used \eqref{Q_positivity} with $t=0$, H\"{o}lder's inequality in $y_{a_1},y_{a_2}$, \eqref{cal W bound_unbounded} (with $\mathcal{W}_{\tau}^a=w_\tau$), and Proposition \ref{G_{tau,0}_bound} (i). The application of the latter is justified by Lemma \ref{QP_embedding}. For the second term on the last line of \eqref{Base_n=2_(1)}, we used \eqref{cal W bound_unbounded2} and Proposition \ref{Q^12_bounds} (i) with $t=0$ to deduce that this quantity is
\begin{equation}
\label{Base_n=2_(1)_2}
\;\leq\;
\begin{cases}
C \tau^{-1+\beta} \,\,&\mbox{if}\,\,\,d=1
\\
C \tau^{-1+\beta}\,\log \tau
\,\,&\mbox{if}\,\,\,d=2
\\
C \tau^{-1/2+\beta}
\,\,&\mbox{if}\,\,\,d=3\,.
\end{cases}
\end{equation}
Recalling \eqref{B_d}, we conclude that \eqref{Base_n=2_(1)_2} is an acceptable upper bound. 
%\begin{equation}
%\frac{1}{\tau} \|w_{\tau}\|_{L^{\infty}(\Lambda)} \,\|G_{\tau}\|_{L^{\infty}(\Lambda^2)}\;\leq\;C \tau^{-\epsilon_0}
%\end{equation}
%for some $\epsilon_0 \equiv \epsilon_0(\beta,d)>0$. 
%The bounds on both terms in \eqref{Base_n=2_(1)} are therefore acceptable.

\item[(2)] $\mathfrak{P}$ is a closed path connecting $a_1,a_2 \in \mathcal{V}_2$ satisfying $a_2 \neq a_1^{*}$. See Figure \ref{Induction_Graph_3} below.

We note that in this case we have $\zeta_1=\pm \sigma(e_1)\,(s_{a_1}-s_{a_2}) \in (-1,1) \setminus \{0\}$. Therefore,
\begin{equation}
\label{Induction_Base_Case_2}
\mathfrak{I} (\mathfrak{P})\;=\;\int \dd y_{a_1}\,\dd y_{a_2}\,\big|\mathcal{W}_\tau^{a_1}(y_{a_1}-y_{a_1^*})\big|\,
\big|\mathcal{W}_\tau^{a_2}(y_{a_2}-y_{a_2^*})\big|\,Q_{\tau,\zeta_1}(y_{a_1};y_{a_2})\,Q_{\tau,-\zeta_1}(y_{a_2};y_{a_1})\,.
\end{equation}
We now split the factors $Q_{\tau,\zeta_1}(y_{a_1};y_{a_2})$ and $Q_{\tau,-\zeta_1}(y_{a_2};y_{a_2})$ in \eqref{Induction_Base_Case_2} according to \eqref{Q_splitting} and rewrite $\mathfrak{I} (\mathfrak{P})$ as
\begin{align}
\notag
&\int \dd y_{a_1}\,\dd y_{a_2}\,\big|\mathcal{W}_\tau^{a_1}(y_{a_1}-y_{a_1^*})\big|\,\big|\mathcal{W}_\tau^{a_2}(y_{a_2}-y_{a_2^*})\big|\,Q_{\tau,\zeta_1}^{(1)}(y_{a_1};y_{a_2})\,Q_{\tau,-\zeta_1}^{(1)}(y_{a_2};y_{a_1})
\\
\notag
+\frac{1}{\tau} &\int \dd y_{a_1}\,\dd y_{a_2}\,\big|\mathcal{W}_\tau^{a_1}(y_{a_1}-y_{a_1^*})\big|\,\big|\mathcal{W}_\tau^{a_2}(y_{a_2}-y_{a_2^*})\big|\,Q_{\tau,\zeta_1}^{(2)}(y_{a_1};y_{a_2})\,Q_{\tau,-\zeta_1}^{(1)}(y_{a_2};y_{a_1})
\\
\notag
+\frac{1}{\tau} &\int \dd y_{a_1}\,\dd y_{a_2}\,\big|\mathcal{W}_\tau^{a_1}(y_{a_1}-y_{a_1^*})\big|\,\big|\mathcal{W}_\tau^{a_2}(y_{a_2}-y_{a_2^*})\big|\,Q_{\tau,\zeta_1}^{(1)}(y_{a_1};y_{a_2})\,Q_{\tau,-\zeta_1}^{(2)}(y_{a_2};y_{a_1})
\\
\label{Induction_Base_Case_2B}
+\frac{1}{\tau^2} &\int \dd y_{a_1}\,\dd y_{a_2}\,\big|\mathcal{W}_\tau^{a_1}(y_{a_1}-y_{a_1^*})\big|\,\big|\mathcal{W}_\tau^{a_2}(y_{a_2}-y_{a_2^*})\big|\,Q_{\tau,\zeta_1}^{(2)}(y_{a_1};y_{a_2})\,Q_{\tau,-\zeta_1}^{(2)}(y_{a_2};y_{a_1})\,.
\end{align}
Note that all the integrands in \eqref{Induction_Base_Case_2B} are nonnegative by \eqref{positivity}.

Using the inequality
\begin{equation*}
Q_{\tau,\zeta_1}^{(1)}(y_{a_1};y_{a_2})\,Q_{\tau,-\zeta_1}^{(1)}(y_{a_2};y_{a_1}) \;\leq\; \frac{1}{2} \Big(\big[Q_{\tau,\zeta_1}^{(1)}(y_{a_1};y_{a_2})\big]^2+\big[Q_{\tau,-\zeta_1}^{(1)}(y_{a_2};y_{a_1})\big]^2\Big)
\end{equation*}
in the first integrand, and applying H\"{o}lder's inequality in $y_{a_1}$ and $y_{a_2}$, we deduce that the first term in \eqref{Induction_Base_Case_2B} is
\begin{align}
\label{Induction_Base_Case_2B_2}
&\;\leq\;\frac{1}{2}\|\mathcal{W}_\tau^{a_1}\|_{L^p(\Lambda)}\,\|\mathcal{W}_\tau^{a_2}\|_{L^p(\Lambda)}\, \Big(\|Q_{\tau,\zeta_1}^{(1)}\|_{L^{2p'}(\Lambda^2)}^2+\|Q_{\tau,-\zeta_1}^{(1)}\|_{L^{2p'}(\Lambda^2)}^2\Big)
\\
\notag
&\;\leq\;C \big(1+\|w\|_{L^p(\Lambda)}\big)^2\,,
\end{align}
which is an acceptable upper bound.
In the above inequality, we used \eqref{cal W bound_unbounded} and Proposition \ref{G_{tau,0}_bound} (ii), which is justified by Lemma \ref{QP_embedding}.

By applying \eqref{positivity} and Proposition \ref{Q^12_bounds} (ii), as well as H\"{o}lder's inequality, the second term in \eqref{Induction_Base_Case_2B} is 

\begin{multline*}
\;\leq\;\frac{1}{\tau} \,\|\mathcal{W}^{a_2}_{\tau}\|_{L^\infty(\Lambda)}\,\|Q_{\tau,-\zeta_1}^{(1)}\|_{L^\infty(\Lambda^2)}\, \int \dd y_{a_1}\,\big|\mathcal{W}^{a_1}_{\tau}(y_{a_1}-y_{a_1^*})\big|\,\int \dd y_{a_2}\,Q_{\tau,\zeta_1}^{(2)}(y_{a_1};y_{a_2})
\\
\;\leq\; \frac{1}{\tau} \,\|\mathcal{W}^{a_2}_{\tau}\|_{L^\infty(\Lambda)}\,\|Q_{\tau,-\zeta_1}^{(1)}\|_{L^\infty(\Lambda^2)}\,\|\mathcal{W}_\tau^{a_1}\|_{L^p(\Lambda)}\,.
\end{multline*}
By \eqref{cal W bound_unbounded2}, Proposition \ref{Q^12_bounds} (i), and \eqref{cal W bound_unbounded}, this expression is
\begin{equation}
\label{Induction_Base_Case_2B_l.o.t.1}
\;\leq\;
\begin{cases}
C \tau^{-1+\beta}\,\big(1+\|w\|_{L^p(\Lambda)}\big)\,\,&\mbox{if}\,\,\,d=1
\\
C \tau^{-1+\beta}\,\log \tau\,\big(1+\|w\|_{L^p(\Lambda)}\big) \,\,&\mbox{if}\,\,\,d=2
\\
C \tau^{-1/2+\beta}\,\big(1+\|w\|_{L^p(\Lambda)}\big)\,&\mbox{if}\,\,\,d=3\,.
\end{cases}
\end{equation}
In particular, by \eqref{B_d}, this is an acceptable bound. %(recall that we are considering $\tau \geq 1$).
We estimate the third term in \eqref{Induction_Base_Case_2B} by an analogous argument. %as the one used to estimate the second term.

For the fourth term in \eqref{Induction_Base_Case_2B}, we consider two possible cases for the values of the time label $\zeta_1 \in (-1,0) \cup (0,1)$.

\begin{itemize}
\item[(a)] $\{\zeta_1\} \geq \frac{1}{2}$, i.e.\ $\zeta_1 \in [-\frac{1}{2},0) \cup [\frac{1}{2},1)$.

In this case, the fourth term in \eqref{Induction_Base_Case_2B} is
\begin{align}
\notag
&\;\leq\;\frac{1}{\tau^2}\,\|\mathcal{W}^{a_2}_{\tau}\|_{L^\infty(\Lambda)}\,\|Q_{\tau,\zeta_1}^{(2)}\|_{L^\infty(\Lambda^2)} \,\int \dd y_{a_1}\, \big|\mathcal{W}^{a_1}_{\tau}(y_{a_1}-y_{a_1^*})\big|\,\int \dd y_{a_2}\,Q_{\tau,-\zeta_1}^{(2)}(y_{a_2};y_{a_1})
\\
\label{Induction_Base_Case_2B_Term4}
&\;\leq\;\frac{1}{\tau^2}\,\|\mathcal{W}^{a_1}_{\tau}\|_{L^p(\Lambda)}\,\|\mathcal{W}^{a_2}_{\tau}\|_{L^\infty(\Lambda)}\,\|Q_{\tau,\zeta_1}^{(2)}\|_{L^\infty(\Lambda^2)}\,.
\end{align}
In the last line, we used Proposition \ref{Q^12_bounds} (ii).
By H\"{o}lder's inequality, \eqref{cal W bound_unbounded2}, \eqref{cal W bound_unbounded}, Proposition \ref{Q^12_bounds} (iii), and \eqref{B_d},  the expression in \eqref{Induction_Base_Case_2B_Term4} is 
\begin{equation}
\label{Induction_Base_Case_2B_Term4_l.o.t.}
\;\leq\; C\tau^{-2+\beta+d/2}\,\,\big(1+\|w\|_{L^p(\Lambda)}\big)\;\leq\; C\,\big(1+\|w\|_{L^p(\Lambda)}\big)\,.
\end{equation}
\item[(b)] $\{\zeta_1\} < \frac{1}{2}$, i.e.\ $\zeta_1 \in (-1,-\frac{1}{2}) \cup (0,\frac{1}{2})$.

In this case, we have $\{-\zeta_1\}>\frac{1}{2}$. We estimate the fourth term in \eqref{Induction_Base_Case_2B} as
\begin{align}
\notag
&\;\leq\; \frac{1}{\tau^2}\,\|\mathcal{W}_{\tau}^{a_2}\|_{L^{\infty}(\Lambda)}\,\|Q_{\tau,-\zeta_1}^{(2)}\|_{L^{\infty}(\Lambda^2)}\,
\\
\label{Induction_Base_Case_2B_Term4_l.o.t.2}
& \times \int \dd y_{a_1}\, 
\big|\mathcal{W}_\tau^{a_1}(y_{a_1}-y_{a_1^{*}})\big| 
\int \dd y_{a_2}\,Q_{\tau,\zeta_1}^{(2)}(y_{a_1};y_{a_2})\,,
\end{align}
\end{itemize}
and deduce the bound as in (i) by replacing $\zeta_1$ by $-\zeta_1$.

\begin{figure}[!ht]
\begin{center}
\includegraphics[scale=0.5]{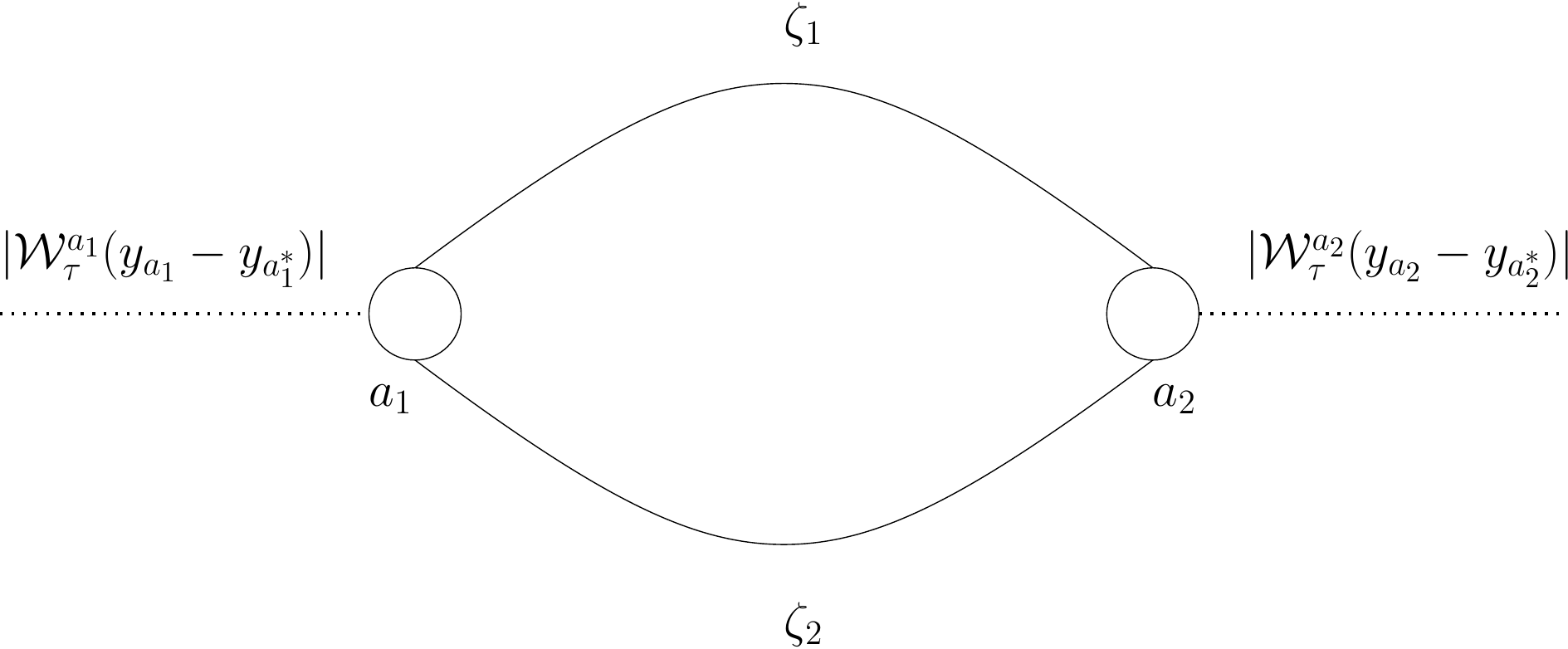}
\end{center}
\caption{
\label{Induction_Graph_3}}
\end{figure}

\item[(3)] $\mathfrak{P}$ is an open path with vertices $b_1,b_2 \in \mathcal{V}_1$.
%Here, we are using the convention for vertex names established in \cite{FrKnScSo1}.
In this case, we have
$\mathfrak{I} (\mathfrak{P})=G_\tau$ and we note that
\begin{equation}
\label{Case_3_Induction_Base}
\|\mathfrak{I} (\mathfrak{P})\|_{L^2_{\mathbf{y_1}}L^{\infty}_{\mathcal{V}^{*}(\mathfrak{P})}}\;=\;\|G_\tau\|_{L^2(\Lambda^2)}\;\leq\;C\,,
\end{equation}
uniformly in $\tau$ by
Proposition \ref{G_{tau,0}_bound} (i).
\end{itemize}
This completes the base step. Hence, \eqref{Inductive inequality 2 unbounded_B} holds when $n \leq 2$.

\subsubsection*{Inductive step}
Let $n \in \mathbb{N}, n \geq 3$ be given. Suppose that \eqref{Inductive inequality 2 unbounded_B} holds for all $\mathfrak{P} \in \conn(\mathcal{E})$ with $|\mathcal{V}(\mathfrak{P})| \leq n-1$. We show that this implies that \eqref{Inductive inequality 2 unbounded_B} holds when $|\mathcal{V}(\mathfrak{P})|=n$. The inductive step proceeds by integrating out an appropriately chosen vertex in $\mathcal{V}_2(\mathfrak{P})$.
The vertex $a \in \mathcal{V}_2(\mathfrak{P})$ which we integrate out is chosen such that it satisfies the property that 
\begin{equation}
\label{integration_vertex_inductive_step}
a^{*} \in \mathcal{V}_2(\mathfrak{P}) \implies \vartheta_a=1\,.
\end{equation}
Note that such a vertex always exists. Namely, if $c \in \mathcal{V}_2(\mathfrak{P})$ does not satisfy this property, then $c^{*} \in \mathcal{V}_2(\mathfrak{P})$ and it satisfies $\vartheta_{c^{*}}=1$. Therefore, $c^{*} \in \mathcal{V}_2$ satisfies the wanted property.

Having found such an $a \in \mathcal{V}_2(\mathfrak{P})$, we write $c_2=a$ and note that there exist $c_1,c_3 \in \mathcal{V}(\mathfrak{P})$ such that $\{c_1,c_2\}, \{c_2,c_3\} \in \mathfrak{P}$.
In what follows, we consider the case when $c_1=a_1,c_2=a_2,c_3=a_3$. The general case follows by analogous arguments when we appropriately shift the indices\footnote{In the case when $\mathfrak{P}$ is a closed path, by cyclically relabelling its vertices, we can assume without loss of generality that $a=a_2$.}.

By choice of $a \equiv a_2$ and \eqref{cal W_unbounded}, we have that the only dependence on the variable $y_a \equiv y_{a_2}$ in the integrand in \eqref{I(P)} is through an expression which is
\begin{align}
\notag
\;\leq\; 
\big|\mathcal{W}_{\tau}^{a_2}(y_{a_2}-y_{a_2^*})\big|\,&\bigg[Q_{\tau,\zeta_1}(y_{a_1};y_{a_2})+\frac{\mathbf{1}_{\zeta_1=0}}{\tau} \delta(y_{a_1}-y_{a_2})\bigg]\,
\\
\label{y_a_dependence_unbounded}
\times &\bigg[
Q_{\tau,\zeta_2}(y_{a_2};y_{a_3})+\frac{\mathbf{1}_{\zeta_2=0}}{\tau}\delta(y_{a_2}-y_{a_3})\bigg]
\,.
\end{align}
Since $\mathbf{t} \in \mathfrak{A}(m)$, it follows that it is not possible to have $\zeta_1=0$ and $\zeta_2=0$ at the same time. In other words, at most one of the delta functions in \eqref{y_a_dependence_unbounded} can appear.

Before proceeding, let us first estimate the contribution to the $y_{a_2}$ integral of each of the terms coming from a single delta function in \eqref{y_a_dependence_unbounded}. We have
\begin{align}
\notag
&\frac{\mathbf{1}_{\zeta_1=0}}{\tau}\,\int\dd y_{a_2}\,\big|\mathcal{W}_{\tau}^{a_2}(y_{a_2}-y_{a_2^*})\big|\,\delta(y_{a_1}-y_{a_2})\,Q_{\tau,\zeta_2}(y_{a_2};y_{a_3}) 
\\
\label{Induction_Step_delta_1_unbounded}
&\;=\;\frac{\mathbf{1}_{\zeta_1=0}}{\tau}\,\big|\mathcal{W}_\tau^{a_2}(y_{a_1}-y_{a_2^*})\big|\, Q_{\tau,\zeta_2}(y_{a_1};y_{a_3}) 
\;\leq\; C\,\tau^{\beta-1}\,\mathbf{1}_{\zeta_1=0}\,Q_{\tau,\zeta_1+\zeta_2} (y_{a_1};y_{a_3})\,.
\end{align}
In the last line we used \eqref{cal W bound_unbounded2}.
By analogous arguments, we have
\begin{equation}
\label{Induction_Step_delta_2_unbounded}
\frac{\mathbf{1}_{\zeta_2=0}}{\tau}\,\int \dd y_{a_2}\,\big|\mathcal{W}_\tau^{a_2}(y_{a_2}-y_{a_2^{*}})\big|\,Q_{\tau,\zeta_1}(y_{a_1};y_{a_2}) \, \delta(y_{a_2}-y_{a_3})
\;\leq\; C\,\tau^{\beta-1}\,\mathbf{1}_{\zeta_2=0}\,Q_{\tau,\zeta_1+\zeta_2} (y_{a_1};y_{a_3})\,.
\end{equation}
In what follows, we prove that
\begin{multline}
\label{Induction_Step_unbounded}
\int \dd y_{a_2}\,\big|\mathcal{W}_\tau^{a_2}(y_{a_2}-y_{a_2^{*}})\big| \,Q_{\tau,\zeta_1}(y_{a_1};y_{a_2}) \, Q_{\tau,\zeta_2}(y_{a_2};y_{a_3}) 
\\
\;\leq\;C\,\big(1+\|w\|_{L^p(\Lambda)}\big)\, \Big[1+ \,\mathbf{1}_{\zeta_1+\zeta_2 \neq 0}\,Q_{\tau,\zeta_1+\zeta_2} (y_{a_1};y_{a_3})\Big]\,.
\end{multline}
Assuming \eqref{Induction_Step_unbounded} for the moment and recalling \eqref{I(P)}, we deduce by \eqref{Induction_Step_delta_1_unbounded}--\eqref{Induction_Step_unbounded}, \eqref{B_d}, and Proposition \ref{Q^12_bounds} (iv) that 
\begin{equation}
\label{I(P)_induction_unbounded}
\mathfrak{I}(\mathfrak{P})\;\leq\;C_0\,\big(1+\|w\|_{L^p(\Lambda)}\big)\, \mathfrak{I}(\mathfrak{\hat{P}})\,,
\end{equation}
where $\hat{\mathfrak{P}}$ denotes the (open or closed) path that we obtain from $\mathfrak{P}$ after deleting the vertex $a_2$ and by replacing its edges $\{a_1,a_2\}$ (carrying time $\zeta_1$) and $\{a_2,a_3\}$ (carrying time $\zeta_2$) with the edge $\{a_1,a_3\}$ which now carries time $\zeta_1+\zeta_2$. See Figure \ref{Graph_1} below.
The times carried by the edges of the new path $\mathfrak{\hat{P}}$ still satisfy \eqref{path_sum_1_unbounded} and \eqref{path_sum_2_unbounded}.  We note that \eqref{I(P)_induction_unbounded} together with the induction base allows us to deduce \eqref{Inductive inequality 2 unbounded_B}.

\begin{figure}[!ht]
\begin{center}
\includegraphics[scale=0.5]{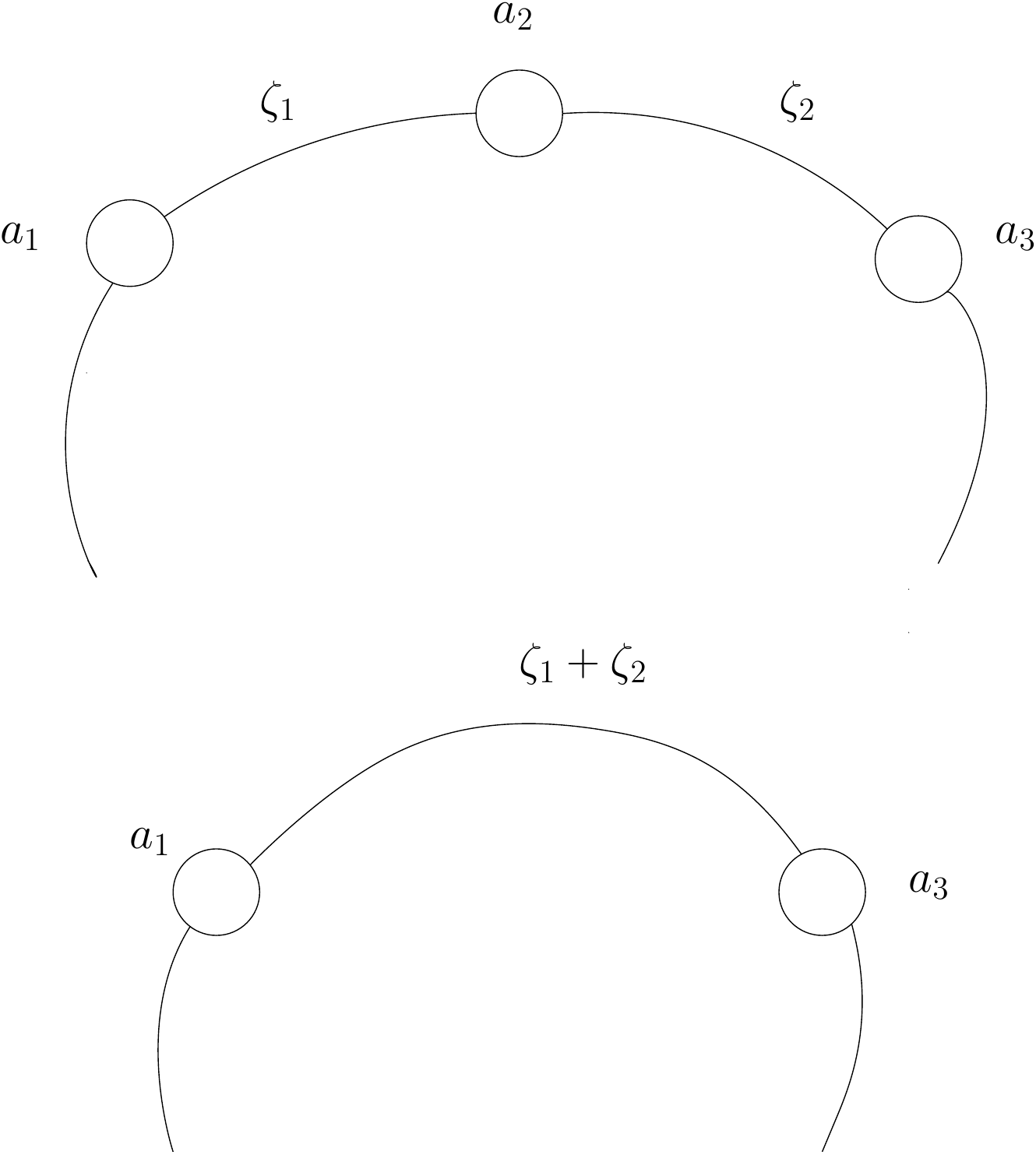}
\end{center}
\caption{The top depicts path $\mathfrak{P}$ and the bottom depicts path $\hat{\mathfrak{P}}$ obtained after having integrated out the vertex $a_2$. Here, we omitted the dotted lines corresponding to the interactions for clarity.
\label{Graph_1}}
\end{figure}

Let us now prove \eqref{Induction_Step_unbounded}. We apply \eqref{Q_splitting} and rewrite the left-hand side of \eqref{Induction_Step_unbounded} as
\begin{align}
\notag
&\int \dd y_{a_2}\,\big|\mathcal{W}^{a_2}_{\tau}(y_{a_2}-y_{a_2^*})\big|\,Q_{\tau,\zeta_1}^{(1)}(y_{a_1};y_{a_2})\,Q_{\tau,\zeta_2}^{(1)}(y_{a_2};y_{a_3})
\\
\notag
+\frac{1}{\tau}&\,\int \dd y_{a_2}\,\big|\mathcal{W}^{a_2}_{\tau}(y_{a_2}-y_{a_2^*})\big|\,Q_{\tau,\zeta_1}^{(2)}(y_{a_1};y_{a_2})\,Q_{\tau,\zeta_2}^{(1)}(y_{a_2};y_{a_3})
\\
\notag
+\frac{1}{\tau}&\,\int \dd y_{a_2}\,\big|\mathcal{W}^{a_2}_{\tau}(y_{a_2}-y_{a_2^*})\big|\,Q_{\tau,\zeta_1}^{(1)}(y_{a_1};y_{a_2})\,Q_{\tau,\zeta_2}^{(2)}(y_{a_2};y_{a_3})
\\
\label{Induction_Step_unbounded_2}
+\frac{1}{\tau^2}&\,\int \dd y_{a_2}\,\big|\mathcal{W}^{a_2}_{\tau}(y_{a_2}-y_{a_2^*})\big|\,Q_{\tau,\zeta_1}^{(2)}(y_{a_1};y_{a_2})\,Q_{\tau,\zeta_2}^{(2)}(y_{a_2};y_{a_3})\,.
\end{align}
We now bound each of the four terms in \eqref{Induction_Step_unbounded_2}. In estimating each term, we first apply H\"{o}lder's inequality in $y_{a_2}$.

The first term in \eqref{Induction_Step_unbounded_2} is 
\begin{equation}
\label{Term_1}
\;\leq\; \|\mathcal{W}_{\tau}^{a_2}\|_{L^p(\Lambda)}\,\|Q_{\tau,\zeta_1}^{(1)}(y_{a_1};\cdot)\|_{L^{2p'}(\Lambda)} \,
\|Q_{\tau,\zeta_2}^{(1)}(\cdot\,;y_{a_3})\|_{L^{2p'}(\Lambda)} \;\leq\;C \big(1+\|w\|_{L^p(\Lambda)}\big)\,,
\end{equation}
by using \eqref{cal W bound_unbounded} and Proposition \ref{G_{tau,0}_bound} (ii). The application of the latter is justified by Lemma \ref{QP_embedding}.

The second term in \eqref{Induction_Step_unbounded_2} is
\begin{equation*}
\;\leq\;\frac{1}{\tau}\,\|\mathcal{W}_{\tau}^{a_2}\|_{L^\infty(\Lambda)}\,\|Q^{(1)}_{\tau,\zeta_2}\|_{L^{\infty}(\Lambda^2)}\,\int \dd y_{a_2} \,Q_{\tau,\zeta_1}^{(2)}(y_{a_2};y_{a_3})\,,
\end{equation*}
which by \eqref{cal W bound_unbounded2} and Proposition \ref{Q^12_bounds} (i)--(ii) is
\begin{equation}
\label{Step_l.o.t._1}
\;\leq\;
\begin{cases}
C \tau^{-1+\beta}\,&\mbox{if}\,\,\,d=1
\\
C \tau^{-1+\beta} \,\log \tau \,\,&\mbox{if}\,\,\,d=2
\\
C \tau^{-1/2+\beta} \,\,&\mbox{if}\,\,\,d=3\,,
\end{cases}
\end{equation}
which in turn is bounded uniformly in $\tau \geq 1$ by using \eqref{B_d}. 
The third term in \eqref{Induction_Step_unbounded_2} is bounded analogously as the second term.

The fourth term in \eqref{Induction_Step_unbounded_2} is
\begin{align}
\notag
&\;\leq\; \frac{1}{\tau^2}\,\|\mathcal{W}_{\tau}^{a_2}\|_{L^\infty(\Lambda)}\,\int \dd y_{a_2}\,
Q^{(2)}_{\tau,\zeta_1}(y_{a_1};y_{a_2})\,Q^{(2)}_{\tau,\zeta_2}(y_{a_2};y_{a_3})
\\
\label{Induction_Step_unbounded_2_4th_term}
&\;\leq\; 
\frac{1}{\tau^{2-\beta}}\,\int \dd y_{a_2}\,
Q^{(2)}_{\tau,\zeta_1}(y_{a_1};y_{a_2})\,Q^{(2)}_{\tau,\zeta_2}(y_{a_2};y_{a_3})\,.
\end{align}
Here, we used \eqref{cal W bound_unbounded2}. Furthermore, by using \eqref{Q^12}, we note that the expression in 
\eqref{Induction_Step_unbounded_2_4th_term} is zero when $\zeta_1=0$ or $\zeta_2=0$.
Therefore, it suffices to consider the case when 
\begin{equation}
\label{zeta_nonzero}
\zeta_1,\zeta_2 \in (-1,0) \cup (0,1)\,.
\end{equation}
We need to consider separately the cases when $\zeta_1+\zeta_2 \neq 0$ and when $\zeta_1+\zeta_2=0$.
\begin{itemize}
\item[(1)] $\zeta_1+\zeta_2 \neq 0$.

By \eqref{Q^12}, \eqref{zeta_nonzero}, and the semigroup property, we can rewrite the expression in \eqref{Induction_Step_unbounded_2_4th_term} as
\begin{equation}
\label{Induction_Step_unbounded_2_4th_term_B}
\frac{C}{\tau^{2-\beta}}\,\int \dd y_{a_2}\,\ee^{-\{\zeta_1\}h/\tau}(y_{a_1};y_{a_2})\,\ee^{-\{\zeta_2\}h/\tau}(y_{a_2};y_{a_3})
\;=\;\frac{C}{\tau^{2-\beta}}\,\ee^{-(\{\zeta_1\}+\{\zeta_2\})h/ \tau}(y_{a_1};y_{a_3})\,.
\end{equation}
We now consider two possibilities.
\begin{itemize}
\item[(a)] $\{\zeta_1\}+\{\zeta_2\} \geq 1$.

In this case, by the proof of Proposition \ref{Q^12_bounds} (iii), the expression in \eqref{Induction_Step_unbounded_2_4th_term_B} is 
\begin{equation}
\label{Step_l.o.t._2}
\;\leq\;C\,\tau^{-2+\beta+d/2}\,,
\end{equation}
which is bounded by using \eqref{B_d}.

\item[(b)] $\{\zeta_1\}+\{\zeta_2\} < 1$.

In this case, we know that $\{\zeta_1\}+\{\zeta_2\}=\{\zeta_1+\zeta_2\}$. Hence, the expression in \eqref{Induction_Step_unbounded_2_4th_term_B} is
\begin{equation*}
\;=\;\frac{C}{\tau^{2-\beta}}\,\ee^{-\{\zeta_1+\zeta_2\}h/ \tau}(y_{a_1};y_{a_3})\,,
\end{equation*}
which by \eqref{Q^12}, followed by \eqref{positivity}, \eqref{B_d}, and \eqref{Q_splitting} is 
\begin{equation}
\label{Step_l.o.t._3}
\;=\;\frac{C}{\tau^{2-\beta}}\,Q_{\tau,\zeta_1+\zeta_2}^{(2)}(y_{a_1};y_{a_3}) \;\leq\; 
C\, \tau^{-1+\beta}\,Q_{\tau,\zeta_1+\zeta_2}(y_{a_1};y_{a_3})\,.
\end{equation}
This is an acceptable bound due to \eqref{B_d}.
Note that the application of \eqref{Q^12} above is justified since $\zeta_1+\zeta_2 \neq 0$.
\end{itemize}
\item[(2)] $\zeta_1+\zeta_2=0$.

By \eqref{zeta_nonzero} we have that $\zeta_2=-\zeta_1 \neq 0$. We now need to consider two possibilities.
\begin{itemize}
\item[(a)]  $\{\zeta_1\} \geq \frac{1}{2}$, i.e.\ $\zeta_1 \in [-\frac{1}{2},0) \cup [\frac{1}{2},1)$.

We bound the fourth term in \eqref{Induction_Step_unbounded_2} as
\begin{equation}
\label{Step_l.o.t._4}
\leq\; \frac{1}{\tau^2} \, \|\mathcal{W}_\tau^{a_2}\|_{L^{\infty}(\Lambda)}\,\|Q_{\tau,\zeta_1}^{(2)}\|_{L^{\infty}(\Lambda^2)}\,\int \dd y_{a_2}\,Q_{\tau,-\zeta_1}^{(2)}(y_{a_2};y_{a_3})\;\leq\; C\,\tau^{-2+\beta+d/2}\,,
\end{equation}
by \eqref{cal W bound_unbounded2} and Proposition \ref{Q^12_bounds} (ii)--(iii). The above quantity is bounded by using \eqref{B_d}.
\item[(b)]   $\{\zeta_1\} < \frac{1}{2}$, i.e.\ $\zeta_1 \in (-1,-\frac{1}{2}) \cup (0,\frac{1}{2})$.

In this case, we have $\{-\zeta_1\}>1/2$, and we bound the fourth term in \eqref{Induction_Step_unbounded_2} as
\begin{equation}
\label{Step_l.o.t._5}
\leq\; \frac{1}{\tau^2} \, \|\mathcal{W}_\tau^{a_2}\|_{L^{\infty}(\Lambda)}\,\|Q_{\tau,-\zeta_1}^{(2)}\|_{L^{\infty}(\Lambda^2)}\,\int \dd y_{a_2}\,Q_{\tau,\zeta_1}^{(2)}(y_{a_1};y_{a_2})\;\leq\; C\,\tau^{-2+\beta+d/2}\,,
\end{equation}
again by \eqref{cal W bound_unbounded2} and Proposition \ref{Q^12_bounds} (ii)--(iii). This quantity is bounded by using \eqref{B_d}.

\end{itemize}
\end{itemize}

Putting the estimates \eqref{Term_1}--\eqref{Step_l.o.t._1}, \eqref{Step_l.o.t._2}--\eqref{Step_l.o.t._5} on all of the terms in \eqref{Induction_Step_unbounded_2} together, we deduce \eqref{Induction_Step_unbounded}.
The inductive step now follows. This proves part (i) when $\xi \in \mathbf{B}_r$.

The claim in the case when $d=1$ and $\xi=\mathrm{Id}_p$ is shown by analogous arguments. 
We recall Definition \ref{collapsed_graph_1D_delta_unbounded}. All of the connected components of $\tilde{\mathcal{E}}$, which we denote by $\tilde{\mathfrak{P}}_1, \ldots, \tilde{\mathfrak{P}}_k$, are now closed paths (which can be loops). 
In this case, $\mathcal I_{\tau,\Pi}^{\xi}(\mathbf{t})$ is given by \eqref{value_of_Pi_1D_delta} and \eqref{Product of subgraphs unbounded 1} is replaced by 
\begin{equation}
\label{Product of subgraphs unbounded 1_delta}
\big|\mathcal I_{\tau,\Pi}^{\xi}(\mathbf{t})\big| \;\leq\; \int_{\Lambda^{\tilde{\mathcal V}}} \dd \mathbf{y} \, \Bigg[\prod_{i=1}^{m} w_\tau(y_{i,1}-y_{i,2})\Bigg] \,\Bigg[\prod_{j=1}^{k}\, \prod_{e\, \in\, \tilde{\mathfrak{P}}_j} \mathcal J_{\tau, e}(\mathbf{y_e}, \mathbf{s})\Bigg]\,.
\end{equation}
Recall that $w_\tau$ is pointwise nonnegative since $d=1$ and we are constructing it using Lemma \ref{w_1D_approximation}.
For $a \in \mathcal{V}_2$, $\mathcal{W}_\tau^{a}$ is still defined as in \eqref{cal W_unbounded}.
In particular, \eqref{product w tau_unbounded} holds and the right-hand side of \eqref{Product of subgraphs unbounded 1_delta} is
\begin{equation}
\label{Product of subgraphs unbounded 2_delta}
\;\leq\;
\Bigg\|\int_{\Lambda^{\tilde{\mathcal{V}}_2}} \dd \mathbf{y_2} \, \prod_{j=1}^{k} \Bigg[\prod_{e \in \tilde{\mathfrak{P}}_j} \mathcal{J}_{\tau, e}(\mathbf{y}_e, \mathbf{s})\, \prod_{a \in \mathcal{V}_2(\tilde{\mathfrak{P}}_j)} \mathcal{W}_\tau^{a}(y_{a}-y_{a^*})\Bigg] \Bigg\|_{L^{\infty}_{\mathbf{y_1}}}\,,
\end{equation}
where we note that in \eqref{Product of subgraphs unbounded 2_delta}, we are considering $\mathbf{y_1} \equiv \mathbf{y}_{\tilde{\mathcal{V}}_1}$. 
We modify \eqref{V_{2,j}} by defining, for $1 \leq j \leq k+1$
\begin{equation*}
\tilde{\mathcal V}_{2,j} \deq \tilde{\mathcal V}_2 \,\Big\backslash \,\Big(\bigcup_{\ell=1}^{j-1} \tilde{\mathcal{V}}_2(\tilde{\mathfrak{P}}_{\ell})\Big)\,.
\end{equation*}
As in \eqref{a_star_unbounded}, for $a=(i_a,\vartheta_a) \in \tilde{\mathcal{V}}_2$, we define $a^{*}=(i_a,3-\vartheta_a) \in \tilde{\mathcal{V}}_2$.
Analogously to \eqref{Inductive inequality 1 unbounded}, we reduce the claim to showing that for all $1 \leq \ell \leq k$ we have the recursive inequality
\begin{align}
\notag
&\Bigg\|\int_{\Lambda^{\tilde{\mathcal{V}}_{2,\ell}}} \dd \mathbf{y}_{\tilde{\mathcal{V}}_{2,\ell}} \, \prod_{j=\ell}^{k} \Bigg[\prod_{e \in \tilde{\mathfrak{P}}_j} \mathcal{J}_{\tau, e}(\mathbf y_e, \mathbf s)\, \prod_{a \in \tilde{\mathcal{V}}_2(\tilde{\mathfrak{P}}_j)} \mathcal{W}_\tau^{a}(y_{a}-y_{a^*})\Bigg]\Bigg\|_{L^{\infty}_{\tilde{\mathcal{V}}_2 \setminus \tilde{\mathcal{V}}_{2,\ell}}L^{\infty}_{\mathbf{y_1}} }
\\
\notag
&\;\leq\;C_0^{\,|\tilde{\mathcal{V}}(\tilde{\mathfrak{P}}_{\ell})|}\,\big(1+\|w\|_{L^p(\Lambda)}\big)^{\,|\tilde{\mathcal{V}}_2(\mathfrak{P}_{\ell})|} \,
\\
\label{Inductive inequality 1 unbounded_delta}
&\times
\Bigg\|\int_{\Lambda^{\tilde{\mathcal{V}}_{2,\ell+1}}} \dd \mathbf{y}_{\tilde{\mathcal{V}}_{2,\ell+1}} \, \prod_{j=\ell+1}^{k} \Bigg[\prod_{e \in \tilde{\mathfrak{P}}_j} \mathcal{J}_{\tau, e}(\mathbf{y}_e, \mathbf{s})\, \prod_{a \in \tilde{\mathcal{V}}_2(\mathfrak{P}_j)} \mathcal{W}_\tau^{a}(y_{a}-y_{a^*})\Bigg]\Bigg\|_{L^{\infty}_{\tilde{\mathcal{V}}_2 \setminus \tilde{\mathcal{V}}_{2,\ell+1}} L^{\infty}_{\mathbf{y_1}}}\,,
\end{align}
for some $C_0>0$.

Similarly as in \eqref{I(P)}, for $\mathfrak{P} \in \conn (\tilde{\mathcal{E}})$, we let
\begin{equation*}
%\label{I(P)_tilde}
\tilde{\mathfrak{I}}(\mathfrak {P}) \;\deq\; \int \dd \mathbf{y}_{\tilde{\mathcal{V}}_2(\mathfrak{P})} \, \prod_{e \in \mathfrak{P}} \mathcal{J}_{\tau, e}(\mathbf{y}_e, \mathbf{s})\, \prod_{a \in \tilde{\mathcal{V}}_2(\mathfrak{P})} \mathcal{W}_\tau^{a}(y_{a}-y_{a^*})\,.
\end{equation*}
Arguing as in the proof when $\xi \in \mathbf{B}_p$, we obtain that \eqref{Inductive inequality 1 unbounded_delta} follows if we prove that
\begin{equation}
\label{Inductive inequality 2 unbounded_B_tilde}
\|\tilde{\mathfrak{I}}(\mathfrak{P})\|_{L^{\infty}_{\mathbf{y_1}} L^{\infty}_{\tilde{\mathcal{V}}^{*}(\mathfrak{P})}}\;\leq\;C_0^{\,|\tilde{\mathcal{V}}(\mathfrak{P})|}\,\big(1+\|w\|_{L^p(\Lambda)}\big)^{\,|\tilde{\mathcal{V}}_2(\mathfrak{P})|}\,,
\end{equation}
for all $\mathfrak{P} \in \conn (\tilde{\mathcal{E}})$. The inductive proof of \eqref{Inductive inequality 2 unbounded_B} given earlier lets us obtain \eqref{Inductive inequality 2 unbounded_B_tilde} since we are now working in dimension $d=1$. More precisely, we only need to modify case (3) of the induction base (see \eqref{Case_3_Induction_Base} above). In particular, when $d=1$, we can replace the $L^2_{\mathbf{y_1}}$ norm by the $L^{\infty}_{\mathbf{y_1}}$ norm in \eqref{Case_3_Induction_Base} since $\|G_\tau\|_{L^\infty(\Lambda^2)} \leq C$, uniformly in $\tau$ by Proposition \ref{G_{tau,0}_bound} (i) and the claim follows. This finishes the proof of (i).
\end{proof}

\begin{remark}
\label{Product_of_subgraphs_Remark}
Let us summarize the inductive proof given in Proposition \ref{Product of subgraphs} above. We consider the case when $\xi \in \mathbf{B}_r$. The case when $d=1$ and $\xi=\mathrm{Id}_r$ is treated analogously with appropriate modifications in the notation. With notation as in the proof of Proposition \ref{Product of subgraphs}, let $\mathfrak{P} \in \conn (\mathcal{E})$ and recall that  $\mathfrak{I}(\mathfrak{P})$ is given by \eqref{I(P)}. We again write $n \deq |\mathcal{V}(\mathfrak{P})|$. 

When $n=1$ and when $\mathfrak{P}$ is a loop based at $a_1 \in \mathcal{V}_2$, we have by \eqref{Base_n=1} that 
\begin{equation}
\label{Product_of_subgraphs_Remark_n=1}
\mathfrak{I}(\mathfrak{P})\;\leq\;\|\mathcal{W}_\tau^{a_1}\|_{L^p(\Lambda)}\,\|G_\tau\|_{L^\infty(\Lambda^2)}\,.
\end{equation}
When $n=2$, let $c_1,c_2$ denote the vertices in $\mathfrak{P}$. If $c_1,c_2 \in \mathcal{V}_2$ and $c_2=c_1^*$, we have by \eqref{Base_n=2_(1)}--\eqref{Base_n=2_(1)_2} that
\begin{equation}
\label{Product_of_subgraphs_Remark_n=2a}
\mathfrak{I}(\mathfrak{P}) \;\leq\;
\|w_{\tau}\|_{L^p(\Lambda)}\,\|G_\tau\|_{L^{2p'}(\Lambda^2)}^2 + \mathcal{O}(\tau^{-\epsilon_0})\,
\end{equation}
for some $\epsilon_0>0$.
\\
If  $c_1,c_2 \in \mathcal{V}_2$ and $c_2 \neq c_1^*$, we have by \eqref{Induction_Base_Case_2}--\eqref{Induction_Base_Case_2B_Term4_l.o.t.2} that 
\begin{equation}
\label{Product_of_subgraphs_Remark_n=2b}
\mathfrak{I}(\mathfrak{P}) \;\leq\; \frac{1}{2}\|\mathcal{W}_\tau^{a_1}\|_{L^p(\Lambda)}\,\|\mathcal{W}_\tau^{a_2}\|_{L^p(\Lambda)}\, \Big(\|Q_{\tau,\zeta_1}^{(1)}\|_{L^{2p'}(\Lambda^2)}^2+\|Q_{\tau,-\zeta_1}^{(1)}\|_{L^{2p'}(\Lambda^2)}^2\Big)+ \mathcal{O}(\tau^{-\epsilon_0})\,
\end{equation}
for some $\epsilon_0>0$. For the positivity of $\epsilon_0$ in \eqref{Product_of_subgraphs_Remark_n=2a}--\eqref{Product_of_subgraphs_Remark_n=2b}, we used \eqref{B_d}.
\\
If $c_1,c_2 \in \mathcal{V}_1$, i.e.\ if $\mathfrak{P}$ is an open path, we have by \eqref{Case_3_Induction_Base} that 
\begin{equation}
\label{Product_of_subgraphs_Remark_n=2c}
\mathfrak{I}(\mathfrak{P})\;\leq\;\|G_{\tau}\|_{L^2(\Lambda^2)}\,.
\end{equation}
When $n \geq 3$, we note that by \eqref{Induction_Step_unbounded_2}--\eqref{Step_l.o.t._5}, the estimate in  \eqref{I(P)_induction_unbounded} can be rewritten as 
\begin{equation}
\label{Product_of_subgraphs_Remark_n>=3}
\mathfrak{I}(\mathfrak{P})\;\leq\;C\,\Big(\|\mathcal{W}_{\tau}^{a_2}\|_{L^p(\Lambda)}\,\|Q_{\tau,\zeta_1}^{(1)}(y_{a_1};\cdot)\|_{L^{2p'}(\Lambda)} \,
\|Q_{\tau,\zeta_2}^{(1)}(\cdot\,;y_{a_3})\|_{L^{2p'}(\Lambda)}  + \mathcal{O}(\tau^{-\epsilon_0})\Big)\, \mathfrak{I}(\mathfrak{\hat{P}})\,,
\end{equation}
for some $\epsilon_0>0$. Here, we again used \eqref{B_d} and recalled the definition of $\hat{\mathfrak{P}}$ given in \eqref{I(P)_induction_unbounded}.
\\
We note that, in \eqref{Product_of_subgraphs_Remark_n=2a}--\eqref{Product_of_subgraphs_Remark_n=2b} and \eqref{Product_of_subgraphs_Remark_n>=3}, the $\mathcal{O}(\tau^{-\epsilon_0})$ contribution comes from all of the $Q^{(2)}$ and delta function factors. All the factors involving only $Q^{(1)}$ give the leading order terms in \eqref{Product_of_subgraphs_Remark_n=1}--\eqref{Product_of_subgraphs_Remark_n>=3}.
\end{remark}

\begin{remark}
\label{Product_of_subgraphs_Remark2_A}
We also note that, in proving \eqref{Product_of_subgraphs_Remark_n=1}--\eqref{Product_of_subgraphs_Remark_n>=3}, the bounds giving us the leading order terms (i.e.\ not the terms involving $\mathcal{O}(\tau^{-\epsilon_0})$) were obtained by applying only \eqref{cal W bound_unbounded} and never by applying \eqref{cal W bound_unbounded2}. The leading order terms in the upper bounds \eqref{Product_of_subgraphs_Remark_n=1}--\eqref{Product_of_subgraphs_Remark_n>=3} are the ones that we obtain by estimating expressions involving only factors of $Q^{(1)}$ and no $Q^{(2)}$ or delta function factors. The precise details of these steps are given in \eqref{Base_n=1}, \eqref{Base_n=2_(1)}, \eqref{Induction_Base_Case_2B_2}, and \eqref{Term_1} above.
\end{remark}

\subsection{Convergence of the explicit terms}
\label{Convergence of the explicit terms}

We now study the convergence of the explicit terms as $\tau \rightarrow \infty$. Throughout this section, we use the convention that when we are working with an $r$-particle operator $\xi \in \mathbf{C}_r$, we write
\begin{equation}
\label{hat_V,E,sigma}
(\hat{\mathcal{V}},\hat{\mathcal{E}},\hat{\sigma}) \;\deq\;
\begin{cases}
(\mathcal{V},\mathcal{E},\sigma) &\mbox{if }\,\,\xi \in \mathbf{B}_r
\\
(\tilde{\mathcal{V}},\tilde{\mathcal{E}},\tilde{\sigma}) &\mbox{if } \,\,d=1\,\, \mbox{and } \xi=\mathrm{Id}_r\,.
\end{cases}
\end{equation}
Here we recall Definitions \ref{def_collapsed_graph_unbounded} and \ref{collapsed_graph_1D_delta_unbounded} above.

%Let $\hat{\mathcal{E}}$ denote either the set $\mathcal{E}$ from Definition \ref{def_collapsed_graph_unbounded} when $\xi \in \mathbf{B}_r$ or the set $\tilde{\mathcal{E}}$ from Definition \ref{collapsed_graph_1D_delta_unbounded} when $d=1$ and $\xi=\mathrm{Id}_r$. We accordingly let $\hat{\sigma}$ denote either $\sigma$ or $\tilde{\sigma}$.

By Definition \ref{def_J_e_unbounded} and \eqref{Q_{tau,t}}, it follows that for all $e = \{a,b\} \in \hat{\mathcal{E}}$ with $a <b$, we have
\begin{equation}
\label{J_tau_e_Q}
\mathcal{J}_{\tau, e}(\mathbf{y}_e,\mathbf{s}) \;=\; Q_{\tau, \hat{\sigma}(e) (s_a - s_b)}(y_a; y_b) + \frac{\mathbf{1}_{\hat{\sigma}(e) = +1}\,\mathbf{1}_{i_a=i_b}}{\tau} \,\delta(y_a-y_b)\,.
\end{equation}
In light of \eqref{J_tau_e_Q}, we define 
\begin{equation}
\label{J_tau_e^1}
\mathcal{J}_{\tau, e}^{(1)}(\mathbf{y}_e,\mathbf{s}) \;\deq\; Q^{(1)}_{\tau, \hat{\sigma}(e) (s_a - s_b)}(y_a; y_b)\,.
\end{equation}
and 
\begin{align}
\notag
&\mathcal{J}_{\tau, e}^{(2)}(\mathbf{y}_e,\mathbf{s}) \;\deq\; \mathcal{J}_{\tau, e}(\mathbf{y}_e,\mathbf{s})-\mathcal{J}_{\tau, e}^{(1)}(\mathbf{y}_e,\mathbf{s}) 
\\
\label{J_tau_e^2}
&\;=\; \frac{1}{\tau}\,Q^{(2)}_{\tau, \hat{\sigma}(e) (s_a - s_b)}(y_a; y_b)+\frac{\mathbf{1}_{\hat{\sigma}(e) = +1}\,\mathbf{1}_{i_a=i_b}}{\tau} \,\delta(y_a-y_b)\,.
\end{align}
In the last line, we used \eqref{Q_splitting}.
By \eqref{positivity} and \eqref{J_tau_e_Q}--\eqref{J_tau_e^2}, we deduce that 
\begin{equation}
\label{J_tau_ordering}
0\;\leq\; \mathcal{J}_{\tau,e}^{(i)}(\mathbf{y}_e,\mathbf{s})\;\leq\; \mathcal{J}_{\tau,e}(\mathbf{y}_e,\mathbf{s})\,,\quad i=1,2\,.
\end{equation}
In the sequel, we use the nonnegativity of $\mathcal{J}_{\tau,e}^{(i)}(\mathbf{y}_e,\mathbf{s})$ for $i=1,2$ without further comment.

We now define new quantities analogous to those considered in \eqref{I_representation_unbounded} and \eqref{value_of_Pi_1D_delta} above.
\begin{definition}
\label{I1_definition}
Let $m,r \in \mathbb{N}, \mathbf{t} \in \mathfrak{A}(m), \Pi \in \mathfrak{R}(d,m,r)$, and $\xi \in \mathbf{C}_r$ be given. We define the following quantities.
\begin{itemize}
\item[(i)] If $\xi \in \mathbf{B}_r$, we let
\begin{equation}
\label{I1_representation_unbounded}
\mathcal{I}_{\tau,\Pi}^{1,\xi}(\mathbf{t})\;\deq\;\int_{\Lambda^{\mathcal{V}}} \dd \mathbf{y}\,\Bigg[\prod_{i=1}^{m} w_{\tau}(y_{i,1}-y_{i,2})\Bigg] \xi(\mathbf{y_1})\,\prod_{e \in \mathcal{E}} \mathcal{J}^{(1)}_{\tau, e}(\mathbf{y}_e,\mathbf{s})\,.
\end{equation}
\item[(ii)] If $d=1$ and $\xi=\mathrm{Id}_r$, we let
\begin{equation}
\label{I1_representation_unbounded_1D_delta}
\mathcal{I}^{1,\xi}_{\tau,\Pi}(\mathbf{t})\;\deq\;\int_{\Lambda^{\tilde{\mathcal{V}}}}\dd \mathbf{y}\,\Bigg[ \prod_{i=1}^{m} w_{\tau}(y_{i,1}-y_{i,2})\Bigg] \prod_{e \in \tilde{\mathcal{E}}} \mathcal{J}^{(1)}_{\tau, e}(\mathbf{y}_e,\mathbf{s})\,.
\end{equation}
\end{itemize}
\end{definition}
We now state an approximation result that makes precise the heuristic that the terms coming from $Q^{(2)}$ and delta function factors are lower order and hence vanish in the limit as $\tau \rightarrow \infty$.
\begin{lemma}
\label{Approximation_result_Q1}
Fix $m,r \in \mathbb{N}$. Given $\Pi \in \mathfrak{R}(d,m,r)$, $\mathbf{t} \in \mathfrak{A}(m)$, we have that 
\begin{equation}
\label{Approximation_result_Q1_claim}
\big|\mathcal{I}^{\xi}_{\tau,\Pi}(\mathbf{t}) - \mathcal{I}^{1,\xi}_{\tau,\Pi} (\mathbf{t})\big| \rightarrow 0 \quad \mbox{as}\quad \tau \rightarrow \infty\quad \mbox{uniformly in} \quad \xi \in \mathbf{C}_r\,.
\end{equation}
\end{lemma}
\begin{proof}
We first prove \eqref{Approximation_result_Q1_claim} when $\xi \in \mathbf{B}_r$. We do this by a telescoping argument. 
First, we introduce an arbitrary strict total order $\prec$ on the edges in $\mathcal{E}$. Furthermore, we write $e_1 \succ e_2$ if $e_2 \prec e_1$.
By \eqref{J_tau_e^1}--\eqref{J_tau_e^2}, we have
\begin{align}
\notag
&\prod_{e \in \mathcal{E}} \mathcal{J}_{\tau, e}(\mathbf{y}_e,\mathbf{s})-\prod_{e \in \mathcal{E}} \mathcal{J}^{(1)}_{\tau, e}(\mathbf{y}_e,\mathbf{s})
\\
\label{Telescoping_1}
&\;=\;\sum_{e_0 \in \mathcal{E}}  \Bigg[\Bigg(\prod_{e \prec e_0} \mathcal{J}_{\tau, e}(\mathbf{y}_e,\mathbf{s})\Bigg) \, \mathcal{J}^{(2)}_{\tau,e_0}(\mathbf{y}_{e_0},\mathbf{s})\,\Bigg(\prod_{e \succ e_0} \mathcal{J}^{(1)}_{\tau, e}(\mathbf{y}_e,\mathbf{s})\Bigg)\Bigg]\,.
\end{align}
By applying \eqref{Telescoping_1}, we  estimate
\begin{align}
\notag
\big|\mathcal{I}^{\xi}_{\tau,\Pi}(\mathbf{t}) - \mathcal{I}^{1,\xi}_{\tau,\Pi} (\mathbf{t})\big| &\;\leq\; \sum_{e_0 \in \mathcal{E}} \int_{\Lambda^{\mathcal{V}}} \dd \mathbf{y}\,\Bigg[\prod_{i=1}^{m} \big|w_{\tau}(y_{i,1}-y_{i,2})\big|\Bigg] \,|\xi(\mathbf{y_1})|\,
\\
\label{Telescoping_1_application}
&\times \Bigg[\Bigg(\prod_{e \prec e_0} \mathcal{J}_{\tau, e}(\mathbf{y}_e,\mathbf{s})\Bigg) \, \mathcal{J}^{(2)}_{\tau,e_0}(\mathbf{y}_{e_0},\mathbf{s})\,\Bigg(\prod_{e \succ e_0} \mathcal{J}^{(1)}_{\tau, e}(\mathbf{y}_e,\mathbf{s})\Bigg)\Bigg]\,,
\end{align}
which by \eqref{J_tau_ordering} is 
\begin{equation}
\label{e_0_telescoping_sum}
\;\leq\; \sum_{e_0 \in \mathcal{E}} \int_{\Lambda^{\mathcal{V}}} \dd \mathbf{y}\,\Bigg[\prod_{i=1}^{m} \big|w_{\tau}(y_{i,1}-y_{i,2})\big|\Bigg] \,|\xi(\mathbf{y_1})|\,
\Bigg[\Bigg(\prod_{e \neq e_0} \mathcal{J}_{\tau, e}(\mathbf{y}_e,\mathbf{s})\Bigg) \, \mathcal{J}^{(2)}_{\tau,e_0}(\mathbf{y}_{e_0},\mathbf{s})\Bigg]\,.
\end{equation}
Given $e_0 \in \mathcal{E}$, let $\mathcal{I}^{\xi}_{\tau,\Pi,e_0}(\mathbf{t})$
denote the corresponding summand in \eqref{e_0_telescoping_sum}.
The quantity $\mathcal{I}^{\xi}_{\tau,\Pi,e_0}(\mathbf{t})$ differs from the expression on the right-hand side of \eqref{Product of subgraphs unbounded 1} only by replacing the factor $\mathcal{J}_{\tau,e_0}(\mathbf{y}_{e_0},\mathbf{s})$ by $\mathcal{J}^{(2)}_{\tau,e_0}(\mathbf{y}_{e_0},\mathbf{s})$.

Note that the proof of Proposition \ref{Product of subgraphs} gives us that $\mathcal{I}^{\xi}_{\tau,\Pi,e_0}(\mathbf{t}) \rightarrow 0$ as $\tau \rightarrow \infty$ uniformly in $\xi \in \mathbf{B}_r$. Let us now explain this in more detail. 
For $\mathfrak{P} \in \conn (\mathcal{E})$, we define
\begin{align}
\notag
&\mathfrak{I}_{e_0}(\mathfrak {P}) \;\deq\; \int \dd \mathbf{y}_{\mathcal{V}_2(\mathfrak{P})} \, \Bigg(\prod_{e \in \mathfrak{P} \setminus \{e_0\}} \mathcal{J}_{\tau, e}(\mathbf{y}_e, \mathbf{s})\Bigg)\, 
\\
\label{I_{e_0}_definition}
&\times \Bigg(\prod_{e \in \mathfrak{P} \cap \{e_0\}} \mathcal{J}_{\tau, e}^{(2)}(\mathbf{y}_e, \mathbf{s})\Bigg) \,\prod_{a \in \mathcal{V}_2(\mathfrak{P})} \big|\mathcal{W}_\tau^{a}(y_{a}-y_{a^*})\big|\,.
\end{align}
In particular, for $\mathfrak{I}(\mathfrak{P})$ as defined in \eqref{I(P)} above, we have $\mathfrak{I}_{e_0}(\mathfrak{P})=\mathfrak{I}(\mathfrak{P})$ if and only if $e_0 \notin \mathfrak{P}$.
We find the path $\mathfrak{P} \in \conn (\mathcal{E})$ which contains the edge $e_0$ and we follow the arguments given in Remark \ref{Product_of_subgraphs_Remark}. 
\\
Namely, if $n \deq |\mathcal{V}(\mathfrak{P})|=1$, then by arguing as in \eqref{Base_n=1}, we deduce that $\mathfrak{I}_{e_0}(\mathfrak {P}) =0$.
\\
If $n=2$, and if $\mathfrak{P}$ is a closed path, we argue as in \eqref{Product_of_subgraphs_Remark_n=2a}--\eqref{Product_of_subgraphs_Remark_n=2b} and deduce that $\mathfrak{I}_{e_0}(\mathfrak {P}) =\mathcal{O}(\tau^{-\epsilon_0})$ for some $\epsilon_0>0$. If $\mathfrak{P}$ is an open path of length $2$, then $\mathfrak{I}_{e_0}(\mathfrak {P}) =0$.
\\
If $n \geq 3$, then we arrange that the vertex $a_2 \in \mathcal{V}_2$ from the proof of Proposition \ref{Product of subgraphs} is adjacent to the edge $e_0$. By arguing as in \eqref{Product_of_subgraphs_Remark_n>=3}, it follows that 
\begin{equation}
\label{I_{e_0}_bound}
\mathfrak{I}_{e_0}(\mathfrak {P})  \;\leq\; C \tau^{-\epsilon_0} \mathfrak{I}(\hat{\mathfrak{P}})\,,
\end{equation}
for some $\epsilon_0>0$ and for $\hat{\mathfrak{P}}$ as in \eqref{I(P)_induction_unbounded}.
The claim for $\xi \in \mathbf{B}_r$ now follows. The claim for $d=1$ and $\xi=\mathrm{Id}_r$ follows by analogous arguments.
\end{proof}

Recalling the definition \eqref{one_body_Hamiltonian} of $h$, the \emph{classical Green function} is given by
\begin{equation}
\label{classical_Green_function}
G \;\deq\; h^{-1}\,.
\end{equation}
One has that for all $x,y \in \Lambda$ 
\begin{equation}
\label{G_positive}
G(x;y) \;=G(y;x) \;\geq\; 0\,.
\end{equation}
For a proof of \eqref{G_positive}, see \cite[Lemma 2.23]{FrKnScSo1}.

\begin{definition}
\label{J_e_classical_unbounded}
Let $\mathbf{y} = (y_a)_{a \in \hat{\mathcal{V}}} \in \Lambda^{\hat{\mathcal{V}}}$ be given. With every edge $e=\{a,b\} \in \hat{\mathcal{E}}$ such that $a<b$, we associate the integral kernel
$\mathcal{J}_{e}(\mathbf{y}_e) \deq G(y_a;y_b)$.
\end{definition}
By \eqref{G_positive}, it follows that $\mathcal{J}_{e}(\mathbf{y}_e) \geq 0$. We henceforth use this nonnegativity property without further comment. Note that $\mathcal{J}_{e}(\mathbf{y}_e)$ is independent of time.

We now define the quantity $\mathcal{I}_{\Pi}^{\xi}$, which is a formal limit as $\tau \rightarrow \infty$ of $\mathcal{I}_{\tau,\Pi}^{\xi}(\mathbf{t})$.
\begin{definition}
\label{I_infinity_unbounded}
Let $m,r \in \mathbb{N}, \mathbf{t} \in \mathfrak{A}(m), \Pi \in \mathfrak{R}(d,m,r)$, and $\xi \in \mathbf{C}_r$ be given. We define the following quantities.
\begin{itemize}
\item[(i)] If $\xi \in \mathbf{B}_r$, we let
\begin{equation*}
\mathcal{I}_{\Pi}^{\xi}\;\deq\;\int_{\Lambda^{\mathcal{V}}} \dd \mathbf{y}\,\Bigg[\prod_{i=1}^{m} w(y_{i,1}-y_{i,2})\Bigg] \xi(\mathbf{y_1})\,\prod_{e \in \mathcal{E}} \mathcal{J}_{e}(\mathbf{y}_e)\,.
\end{equation*}
\item[(ii)] If $d=1$ and $\xi=\mathrm{Id}_r$, we let
\begin{equation*}
\mathcal{I}^{\xi}_{\Pi}\;\deq\;\int_{\Lambda^{\tilde{\mathcal{V}}}} \dd \mathbf{y}\,\Bigg[ \prod_{i=1}^{m} w (y_{i,1}-y_{i,2})\Bigg] \prod_{e \in \tilde{\mathcal{E}}} \mathcal{J}_{e}(\mathbf{y}_e)\,.
\end{equation*}
\end{itemize}
\end{definition}

\begin{proposition}
\label{Approximation_result_2_unbounded}
Let $m,r \in \mathbb{N}$, $\Pi \in \mathfrak{R}(d,m,r)$, and  $\mathbf{t} \in \mathfrak{A}(m)$ be given. We have that 
\begin{equation}
\label{Approximation_result_2_unbounded_convergence}
\mathcal{I}^{\xi}_{\tau,\Pi}(\mathbf{t}) \rightarrow \mathcal{I}^{\xi}_{\Pi} \quad \mbox{as}\quad \tau \rightarrow \infty\quad \mbox{uniformly in} \,\, \xi \in \mathbf{C}_r\,.
\end{equation}
\end{proposition}
We emphasise that the convergence in \eqref{Approximation_result_2_unbounded_convergence} is not uniform in $\mathbf{t} \in \mathfrak{A}(m)$. Before proceeding with the proof of Proposition \ref{Approximation_result_2_unbounded}, we note a useful convergence result.
\begin{lemma}
\label{Q1_convergence_lemma}
Let $t \in (-1,1)$ and $q \in \mathcal{Q}_d$, for $\mathcal{Q}_d$ as defined in \eqref{Q_d} be given. Then, we have
\begin{equation}
\label{Q1_convergence_lemma1}
\|Q^{(1)}_{\tau,t}(x;\cdot)-G(x;\cdot)\|_{L^q(\Lambda)}\;=\;\|Q^{(1)}_{\tau,t}(\cdot;x)-G(\cdot;x)\|_{L^q(\Lambda)} \rightarrow 0 \quad \mbox{as}\quad \tau \rightarrow \infty\,, 
\end{equation}
uniformly in $x \in \Lambda$. In particular, we have 
\begin{equation}
\label{Q1_convergence_lemma2}
\|Q^{(1)}_{\tau,t}-G\|_{L^q(\Lambda^2)} \rightarrow 0 \quad \mbox{as}\quad \tau \rightarrow \infty\,.
\end{equation}
\end{lemma}
\begin{proof}[Proof of Lemma \ref{Q1_convergence_lemma}]
We argue similarly as in the proof of Proposition \ref{G_{tau,0}_bound}. Note that \eqref{Q1_convergence_lemma2} follows from \eqref{Q1_convergence_lemma1}, so it suffices to prove \eqref{Q1_convergence_lemma1}. By \eqref{positivity} and  \eqref{G_positive} we get the equality of the first two expressions in \eqref{Q1_convergence_lemma2}.
For $x,y \in \Lambda$, we compute by \eqref{Q^12} and \eqref{classical_Green_function} 
\begin{equation}
\label{Q1_convergence_lemma3}
Q^{(1)}_{\tau,t}(x;y)-G(x;y)\;=\; Q^{(1)}_{\tau,t}(y;x)-G(y;x)\;=\;\sum_{k \in \mathbb{Z}^d}\bigg(\frac{\ee^{-\{t\}\lambda_k/\tau}}{\tau(\ee^{\lambda_k/\tau}-1)}-\frac{1}{\lambda_k}\bigg)\,\ee^{2\pi \ii k \cdot (x-y)}\,.
\end{equation}
We note that for fixed $k \in \mathbb{Z}^d$ we have 
\begin{equation}
\label{Q1_convergence_lemma4}
\frac{\ee^{-\{t\}\lambda_k/\tau}}{\tau(\ee^{\lambda_k/\tau}-1)}-\frac{1}{\lambda_k} \rightarrow 0  \quad \mbox{as}\quad \tau \rightarrow \infty\,,
\end{equation}
and 
\begin{equation}
\label{Q1_convergence_lemma5}
\bigg|\frac{\ee^{-\{t\}\lambda_k/\tau}}{\tau(\ee^{\lambda_k/\tau}-1)}-\frac{1}{\lambda_k}\bigg| \;\leq\; \frac{C}{|k|^2+1}\,.
\end{equation}
The claim now follows using \eqref{Q1_convergence_lemma3}--\eqref{Q1_convergence_lemma5}, the dominated convergence theorem and arguing as in the proof of Proposition \ref{G_{tau,0}_bound}.
\end{proof}
We now have the necessary tools to prove Proposition \ref{Approximation_result_2_unbounded}.

\begin{proof}[Proof of Proposition \ref{Approximation_result_2_unbounded}]
Recalling \eqref{J_tau_e^1}, we define the following auxiliary quantities.
\begin{itemize}
\item[(i)] If $\xi \in \mathbf{B}_r$, we let
\begin{equation}
\label{I_{tau,Pi}^{2,xi}_definition}
\mathcal{I}_{\tau,\Pi}^{2,\xi}(\mathbf{t})\;\deq\;\int_{\Lambda^{\mathcal{V}}} \dd \mathbf{y}\,\Bigg[\prod_{i=1}^{m} w(y_{i,1}-y_{i,2})\Bigg] \xi(\mathbf{y_1})\,\prod_{e \in \mathcal{E}} \mathcal{J}^{(1)}_{\tau,e}(\mathbf{y}_e,\mathbf{s})\,.
\end{equation}
\item[(ii)] If $d=1$ and $\xi=\mathrm{Id}_r$, we let
\begin{equation*}
\mathcal{I}_{\tau,\Pi}^{2,\xi}(\mathbf{t})\;\deq\;\int_{\Lambda^{\tilde{\mathcal{V}}}} \dd \mathbf{y}\,\Bigg[ \prod_{i=1}^{m} w (y_{i,1}-y_{i,2})\Bigg] \prod_{e \in \tilde{\mathcal{E}}} \mathcal{J}^{(1)}_{\tau,e}(\mathbf{y}_e,\mathbf{s})\,.
\end{equation*}
\end{itemize}
We show that 
\begin{equation}
\label{Approximation_result_2_unbounded_I1-I2}
\big|\mathcal{I}^{1,\xi}_{\tau,\Pi}(\mathbf{t}) - \mathcal{I}^{2,\xi}_{\tau,\Pi} (\mathbf{t})\big| \rightarrow 0 \quad \mbox{as}\quad \tau \rightarrow \infty\quad \mbox{uniformly in} \,\, \xi \in \mathbf{C}_r\,.
\end{equation}
To this end, we use a telescoping argument.
More precisely, we write 
\begin{align}
\notag
&\prod_{i=1}^{m} w_{\tau} (y_{i,1}-y_{i,2})-\prod_{i=1}^{m} w (y_{i,1}-y_{i,2})
\\
\label{Approximation_result_2_unbounded_I1-I2_telescoping}
&\;=\;\sum_{n=1}^{m} \Bigg[\prod_{i=1}^{n-1} w_{\tau}(y_{i,1}-y_{i,2}) \Bigg] \, \Big(w_{\tau}(y_{n,1}-y_{n,2})-w(y_{n,1}-y_{n,2})\Big)\,\Bigg[\prod_{i=n+1}^{m} w(y_{i,1}-y_{i,2})\Bigg]\,.
\end{align}
We fix $n \in \{1,\ldots,m\}$ and consider the contribution to $\mathcal{I}^{1,\xi}_{\tau,\Pi}(\mathbf{t}) - \mathcal{I}^{2,\xi}_{\tau,\Pi} (\mathbf{t})$ coming from the $n$-th term in \eqref{Approximation_result_2_unbounded_I1-I2_telescoping}.
We need to modify \eqref{cal W_unbounded} above. For $a \in \mathcal{V}_2$ and with the same notation as in \eqref{cal W_unbounded}, we define the interaction $\widetilde{\mathcal{W}}_{\tau}^a$ by
\begin{align} 
\label{cal W_tilde}
\widetilde{\mathcal{W}}_\tau^{a} \;\deq\;
\begin{cases}
w_\tau &\mbox{if }a \in \mathfrak{P}_j,\,a^{*} \in \mathfrak{P}_{\ell} \,\, \mbox{for } 1 \leq j<\ell \leq k \,\,\, \mbox{and }i_a<n \\
w_\tau-w &\mbox{if }a \in \mathfrak{P}_j,\,a^{*} \in \mathfrak{P}_{\ell} \,\, \mbox{for } 1 \leq j<\ell \leq k \,\,\, \mbox{and }i_a=n \\
w &\mbox{if }a \in \mathfrak{P}_j,\,a^{*} \in \mathfrak{P}_{\ell} \,\, \mbox{for } 1 \leq j<\ell \leq k \,\,\, \mbox{and }i_a>n \\
1 &\mbox{if }a \in  \mathfrak{P}_j,\,a^{*} \in  \mathfrak{P}_{\ell} \,\, \mbox{for } 1 \leq \ell<j \leq k \\
w_\tau &\mbox{if }a,a^* \in  \mathfrak{P}_j \,\, \mbox{for } 1 \leq j \leq k \,\, \mbox{and } \vartheta_a=1, i_a<n \\
w_\tau-w &\mbox{if }a,a^* \in  \mathfrak{P}_j \,\, \mbox{for } 1 \leq j \leq k \,\, \mbox{and } \vartheta_a=1, i_a=n \\
w &\mbox{if }a,a^* \in  \mathfrak{P}_j \,\, \mbox{for } 1 \leq j \leq k \,\, \mbox{and } \vartheta_a=1, i_a>n \\
1 &\mbox{if }a,a^* \in  \mathfrak{P}_j \,\, \mbox{for } 1 \leq j \leq k \,\, \mbox{and } \vartheta_a=2\,.
\end{cases}
\end{align}
Then the estimate \eqref{cal W bound_unbounded} is still satisfied with $\mathcal{W}_\tau^a$ replaced by $\widetilde{\mathcal{W}}_\tau^{a}$ (note that \eqref{cal W bound_unbounded2} is not). Furthermore, by using Lemma \ref{w_1D_approximation} (iv) when $d=1$ and Lemma \ref{w_positive_approximation} (iv) when $d=2,3$, we have that 
\begin{equation}
\label{2.125}
\|\widetilde{\mathcal{W}}_\tau^{a_0}\|_{L^p(\Lambda)} \rightarrow 0 \quad \mbox{as}\quad \tau \rightarrow \infty\quad \mbox{whenever} \quad \widetilde{\mathcal{W}}_\tau^{a_0}=w_\tau-w\,.
\end{equation}
We note that by \eqref{cal W_tilde}, a unique such $a_0 \in \mathcal{V}_2$ exists.

We now argue as in Remark \ref{Product_of_subgraphs_Remark}. Analogously to \eqref{I(P)}, given $\mathfrak{P} \in \conn (\mathcal{E})$, we define
\begin{equation}
\label{I'(P)}
\mathfrak{I}'(\mathfrak {P}) \;\deq\; \int \dd \mathbf{y}_{\mathcal{V}_2(\mathfrak{P})} \, \prod_{e \in \mathfrak{P}} \mathcal{J}_{\tau, e}^{(1)}(\mathbf{y}_e, \mathbf{s})\, \prod_{a \in \mathcal{V}_2(\mathfrak{P})} \big|\widetilde{\mathcal{W}}_\tau^{a}(y_{a}-y_{a^*})\big|\,.
\end{equation}
We now use \eqref{Product_of_subgraphs_Remark_n=1}--\eqref{Product_of_subgraphs_Remark_n>=3} for $\mathfrak{I}'$ instead of $\mathfrak{I}$. Note that now there are no $\mathcal{O}(\tau^{-\epsilon_0})$ error terms since we are working only with $Q^{(1)}$ factors. Moreover, it is important to use Remark \ref{Product_of_subgraphs_Remark2_A} (properly adapted to this context). In other words, we are only applying \eqref{cal W bound_unbounded} for $\widetilde{\mathcal{W}}_{\tau}^{a}$
and we are never applying \eqref{cal W bound_unbounded2} (which does not hold for $\widetilde{\mathcal{W}}_{\tau}^{a}$).
Furthermore, when applying the induction base \eqref{Product_of_subgraphs_Remark_n=2a} in this context, we replace $w_\tau$ by $w_\tau-w$ when $a_0 \in \{c_1,c_2\}$. Likewise, if $a_0 \in \{a_1,a_2\}$ with notation as in \eqref{Product_of_subgraphs_Remark_n>=3}, then we replace the corresponding factor of $\mathcal{W}_{\tau}^{a_j}$ by $w_\tau-w$ and we estimate it using \eqref{2.125} instead of \eqref{cal W bound_unbounded}. In particular, we deduce that if $\mathfrak{P} \in \conn(\mathcal{E})$ is such that $a_0 \in \mathcal{V}_2(\mathfrak{P})$, then 
\begin{equation}
\label{a_0_vertex_factor1}
\|\mathfrak{I}'(\mathfrak{P})\|_{L^{2}_{\mathbf{y_1}} L^{\infty}_{\mathcal{V}^{*}(\mathfrak{P})}} \rightarrow 0 \quad \mbox{as}\quad \tau \rightarrow \infty\,.
\end{equation}
We also obtain that $\|\mathfrak{I}'(\mathfrak{P})\|_{L^{\infty}_{\mathbf{y_1}} L^{\infty}_{\mathcal{V}^{*}(\mathfrak{P})}} \rightarrow 0$ if $d=1$ and $\xi=\mathrm{Id}_r$ by replacing the base case \eqref{Product_of_subgraphs_Remark_n=2a} with \eqref{Product_of_subgraphs_Remark_n=1} and applying the same arguments. Finally, by arguing as in the proof of Proposition \ref{Product of subgraphs}, for the other $\mathfrak{P} \in \conn(\mathcal{E})$ we have the bounds \eqref{Inductive inequality 2 unbounded_B} when $\xi \in \mathbf{B}_r$ and \eqref{Inductive inequality 2 unbounded_B_tilde} when $d=1$ and $\xi=\mathrm{Id}_r$ with $\mathfrak{I}$ and $\tilde{\mathfrak{I}}$ replaced by $\mathfrak{I}'$ respectively.  
Putting everything together, we obtain \eqref{Approximation_result_2_unbounded_I1-I2}.

We now show that 
\begin{equation}
\label{Approximation_result_2_unbounded_I2-I}
\mathcal{I}^{2,\xi}_{\tau,\Pi}(\mathbf{t}) \rightarrow \mathcal{I}^{\xi}_{\Pi} \quad \mbox{as}\quad \tau \rightarrow \infty\quad \mbox{uniformly in} \,\, \xi \in \mathbf{C}_r\,.
\end{equation}
We use a further telescoping argument.
Let us first consider the case when $\xi \in \mathbf{B}_r$.
By arguing as in \eqref{Telescoping_1_application} and with the same notation, we have 
\begin{align}
\notag
\big|\mathcal{I}^{2,\xi}_{\tau,\Pi}(\mathbf{t}) - \mathcal{I}^{\xi}_{\Pi}\big| &\;\leq\; \sum_{e_0 \in \mathcal{E}} \int_{\Lambda^{\mathcal{V}}} \dd \mathbf{y}\,\Bigg[\prod_{i=1}^{m} \big|w(y_{i,1}-y_{i,2})\big|\Bigg] \,|\xi(\mathbf{y_1})|\,
\\
\label{Telescoping_2_application}
&\times \Bigg[\Bigg(\prod_{e \prec e_0} \mathcal{J}^{(1)}_{\tau, e}(\mathbf{y}_e,\mathbf{s})\Bigg) \, \big|\mathcal{J}^{(1)}_{\tau,e_0}(\mathbf{y}_{e_0},\mathbf{s})-\mathcal{J}_{e_0}(\mathbf{y}_{e_0})\big|\,\Bigg(\prod_{e \succ e_0} \mathcal{J}_{e}(\mathbf{y}_e)\Bigg)\Bigg]\,.
\end{align}
We fix $e_0 \in \mathcal{E}$ and consider the corresponding term on the right-hand side of \eqref{Telescoping_2_application}.
Let us define
\begin{equation}
\label{J_tilde}
\tilde{\mathcal{J}}_{\tau,e}(\mathbf{y}_e,\mathbf{s}) \;\deq\;
\begin{cases}
\mathcal{J}^{(1)}_{\tau, e}(\mathbf{y}_e,\mathbf{s}) &\mbox{if }e \prec e_0\\
\big|\mathcal{J}^{(1)}_{\tau,e_0}(\mathbf{y}_{e_0},\mathbf{s})-\mathcal{J}_{e_0}(\mathbf{y}_{e_0})\big| &\mbox{if }e=e_0\\
\mathcal{J}_{e}(\mathbf{y}_e) &\mbox{if }e \succ e_0\,. 
\end{cases}
\end{equation}
Given $a \in \mathcal{V}_2$, we define $\mathcal{W}^a$ analogously as in \eqref{cal W_unbounded} with $w_\tau$ replaced by $w$. Then $\mathcal{W}^a$ satisfies $\|\mathcal{W}^a\|_{L^p(\Lambda)} \leq 1+ C\|w\|_{L^p(\Lambda)}$ as in  \eqref{cal W bound_unbounded}. Given $\mathfrak{P} \in \conn (\mathcal{E})$, we define
\begin{equation}
\label{I''_definition}
\mathfrak{I}''(\mathfrak {P}) \;\deq\; \int \dd \mathbf{y}_{\mathcal{V}_2(\mathfrak{P})} \, \prod_{e \in \mathfrak{P}} \tilde{\mathcal{J}}_{\tau, e}(\mathbf{y}_e, \mathbf{s})\, \prod_{a \in \mathcal{V}_2(\mathfrak{P})} \big|\mathcal{W}^{a}(y_{a}-y_{a^*})\big|\,.
\end{equation}
We now apply \eqref{Product_of_subgraphs_Remark_n=1}--\eqref{Product_of_subgraphs_Remark_n>=3} and Remark  \ref{Product_of_subgraphs_Remark2_A} with proper modifications to the context of $\mathfrak{I}''$. More precisely, all factors of $w_\tau$ are replaced by $w$. The factor corresponding to the edge $e$, which was previously of the form $Q^{(1)}_{\tau,\zeta}$ gets replaced by $|Q^{(1)}_{\tau,\zeta}-G|$ if $e=e_0$ and by $G$ if $e>e_0$. In order to deduce this, we use \eqref{J_tau_e^1}, Definition \ref{J_e_classical_unbounded}, and \eqref{J_tilde}. Finally, we use Lemma \ref{QP_embedding} and Lemma \ref{Q1_convergence_lemma} and deduce \eqref{Approximation_result_2_unbounded_I2-I} when $\xi \in \mathbf{B}_r$.
The proof of \eqref{Approximation_result_2_unbounded_I2-I} when $d=1$ and $\xi= \mathrm{Id}_r$ proceeds analogously, with minor notational modifications. %, the only difference being that the base of the induction when integrating out the path that contains $a_0$ is given by \label{Product_of_subgraphs_Remark_n=1} and not by \eqref{Product_of_subgraphs_Remark_n=2a}. 
We omit the details.
The claim of the proposition now follows by using Lemma \ref{Approximation_result_Q1}, \eqref{Approximation_result_2_unbounded_I1-I2}, and \eqref{Approximation_result_2_unbounded_I2-I}.
\end{proof}

Given $m,r \in \mathbb{N}$ and $\xi \in \mathbf{C}_r$, we let
\begin{equation}
\label{a_infty}
a^{\xi}_{\infty,m} \;\deq\; \frac{(-1)^m}{m!\,2^m} \sum_{\Pi \in \mathfrak{R}(d,m,r)} \mathcal{I}^{\xi}_\Pi\,,
\end{equation}
for $\mathcal{I}^{\xi}_\Pi$ as given by Definition \ref{I_infinity_unbounded}. We now show that this quantity corresponds to  the limit as $\tau \rightarrow \infty$ of the explicit term $a^{\xi}_{\tau,m}$ given by \eqref{Explicit_term_a_unbounded}.

\begin{proposition}
\label{a_convergence}
Let $m,r \in \mathbb{N}$ be given. We have 
\begin{equation*}
a^{\xi}_{\tau,m} \rightarrow a^{\xi}_{\infty,m} \quad \mbox{as}\quad \tau \rightarrow \infty\quad \mbox{uniformly in} \,\, \xi \in \mathbf{C}_r\,.
\end{equation*}
\end{proposition}
\begin{proof}
The claim follows from  \eqref{Wick_application_identity_unbounded_all_D}, \eqref{a_infty}, Proposition \ref{Approximation_result_2_unbounded}, Proposition \ref{Product of subgraphs} (i), and the dominated convergence theorem.
\end{proof}

\subsection{Bounds on the remainder term}
\label{Bounds on the remainder term}

This section is devoted to estimating the quantum remainder term $R^{\xi}_{\tau,M}(z)$ given by \eqref{Remainder_term_R_unbounded}. Since we have already proved the required upper bound on the explicit terms in Proposition \ref{Product of subgraphs} (ii) above, most of the arguments will follow in a similar way as in the case of bounded interaction potentials \cite[Section 2.7]{FrKnScSo1}. We will emphasise the necessary modifications and refer the reader to \cite[Section 2.7]{FrKnScSo1} for more details and explanations. We recall that for $d=2,3$, we are taking $\eta \in (0,1/4]$ in \eqref{modified_GCE}. Strictly speaking, this was not necessary for the previous sections, but in this section, we use it crucially. For $d=1$, we take $\eta=0$. We now state the main result of the section.
\begin{proposition} 
\label{remainder_term}
Let $r,M \in \mathbb{N}, \xi \in \mathbf{C}_r$, and $z \in \mathbb{C}$ with $\re z \geq 0$ be given. Then $R^{\xi}_{\tau,M}(z)$ given by \eqref{Remainder_term_R_unbounded} satisfies the following bounds.
\begin{itemize}
\item[(i)] If $d=2,3$, we have
\begin{equation*}
|R_{\tau,M}^{\xi}(z)|\;\leq\; \bigg(\frac{Cr}{\eta^2}\bigg)^r \, \Bigg[\frac{C(1+\|w\|_{L^p(\Lambda)})}{\eta^2}\Bigg]^M\,|z|^M\,M!\,.
\end{equation*}
\item[(ii)] If $d=1$, we have 
\begin{equation*}
|R_{\tau,M}^{\xi}(z)|\;\leq\; (Cr)^r \, C^M\,(1+\|w\|_{L^p(\Lambda)})^M\,|z|^M\,M!\,.
\end{equation*}
\end{itemize}
\end{proposition}
\begin{proof}
We first prove (i). Let us recall some notation used in the proof of \cite[Proposition 2.27]{FrKnScSo1}. For $0<t_{M}<\cdots<t_2<t_1<1-2\eta$ (as in the support of the time integral in \eqref{Remainder_term_R_unbounded}, up to measure zero), we let $u_1 \deq 1-2\eta-t_1$ and $u_j \deq t_{j-1}-t_{j}$ for $j=2,\ldots,M$. Furthermore, we let 
$\mathbf{u} \deq (u_1,\ldots,u_M)$ and $|\mathbf{u}|\deq u_1+\cdots +u_M$. We then rewrite \eqref{Remainder_term_R_unbounded} as
\begin{equation}
\label{definition_of_g_1}
R^{\xi}_{\tau,M}(z) \;=\; (-1)^M\,\frac{z^M}{(1-2\eta)^{M}} \, \int_{(0,1-2\eta)^M} \dd \mathbf{u}\, \mathbf{1}_{|\mathbf{u}|<1-2\eta} \, g_{\tau,M}^{\xi}(z,\mathbf{u})\,,
\end{equation}
for 
\begin{align*}
g_{\tau,M}^{\xi}(z,\mathbf{u})\;\deq\;&\tr \Big( \Theta_{\tau}(\xi)\,\ee^{-(\eta+u_1)\,H_{\tau,0}}\,W_\tau\,\ee^{-u_2\,H_{\tau,0}}
\,W_\tau\,\ee^{-u_3\,H_{\tau,0}} \,\cdots
\\
&\cdots \,\ee^{-u_M\,H_{\tau,0}} \,W_{\tau}\,\ee^{-(1-2\eta-|\mathbf{u}|)\,(H_{\tau,0}+\frac{z}{1-2\eta}\,W_\tau)}\,\ee^{-\eta\,H_{\tau,0}}\Big) \Big/ \tr \big(\ee^{-H_{\tau,0}}\big)\,.
\end{align*}
We henceforth work with $\mathbf{u} \in (0,1-2\eta)^M$ such that $|\mathbf{u}|<1-2\eta$. As in \cite[(2.84)]{FrKnScSo1}, we have
\begin{align} 
\notag
|g_{\tau,M}^{\xi}(z,\mathbf{u})|\;\leq\;&\big\|\ee^{-\eta H_{\tau,0}}\,\Theta_\tau(\xi)\,\ee^{-\eta/2\,H_{\tau,0}}\big\|_{\tilde {\mathfrak{S}}^{2/(3\eta)}(\mathcal{F})}\,\big\|\ee^{-(1-2\eta-|\mathbf{u}|)\,(H_{\tau,0}+\frac{z}{1-2\eta}W_\tau)}\big\|_{\tilde{\mathfrak {S}}^{1/(1-2\eta-|\mathbf{u}|)}(\mathcal{F})}
\\
\label{estimate_on_g}
&\times \big\|\ee^{-(\eta/2+u_1)\,H_{\tau,0}}\,W_{\tau}\,
\ee^{-u_2\,H_{\tau,0}}\,W_\tau\,\ee^{-u_3\,H_{\tau,0}}
\cdots \ee^{-u_M\,H_{\tau,0}}\,W_{\tau}\big\|_{\tilde{\mathfrak{S}}^{1/(\eta/2+|\mathbf{u}|)}(\mathcal{F})}\,.
\end{align}
Here, we are working with the rescaled Schatten norms given by
\begin{equation*}
\|\mathcal{A}\|_{\tilde{\mathfrak{S}}^q(\mathcal{F})} \;\deq\;
\begin{cases}
\|\mathcal{A}\|_{\mathfrak{S}^q(\mathcal{F})}\Big/\big(\tr (\ee^{-H_{\tau,0}})\big)^{1/q} &\mbox{if } q \in [1,\infty)
\\
\|\mathcal{A}\|_{\mathfrak{S}^{\infty}(\mathcal{F})} &\mbox{if } q=\infty\,.
\end{cases}
\end{equation*}
The estimate \eqref{estimate_on_g} then follows by applying cyclicity of the trace and H\"{o}lder's inequality in Schatten spaces. For details of the latter, see \cite[Lemma 2.28]{FrKnScSo1}.

We now estimate each of the three factors on the right-hand-side of \eqref{estimate_on_g}.
By \cite[Lemma 2.32]{FrKnScSo1}, the first term on the right-hand side of \eqref{estimate_on_g} is
\begin{equation}
\label{Remainder_term_1_bound}
\;\leq\; (Cr \eta^{-2})^r\,. 
\end{equation}
Note that the proof of \cite[Lemma 2.32]{FrKnScSo1} directly applies since this term does not depend on the interaction.

In order to estimate the second term on the right-hand side of \eqref{estimate_on_g}, we first note that the operator $W_\tau$ is positive. In order to do this, we expand $w_\tau$ into a Fourier series in \eqref{W_renormalised_unbounded} and obtain that 
\begin{equation}
\label{W_tau_positivity}
W_\tau \;=\; \frac{1}{2} \sum_{k \in \mathbb{Z}^d} \hat{w}_{\tau}(k) \, \bigg[\int \dd x\, \ee^{2 \pi \ii k \cdot x} \big(\phi_{\tau}^{*}(x) \phi_{\tau}(x)- \varrho_{\tau}\big)\bigg] \, \bigg[\int \dd y\, \ee^{2 \pi \ii k \cdot y} \big(\phi_{\tau}^{*}(y) \phi_{\tau}(y)- \varrho_{\tau}\big)\bigg]^{*}\,
\end{equation}
which is positive by Lemma \ref{w_positive_approximation} (i). The proof of \cite[Lemma 2.30]{FrKnScSo1} now shows that for all $t \in (0,1)$ we have
\begin{equation}
\label{Remainder_term_2_bound}
\Big\|\ee^{-t\,(H_{\tau,0}+\frac{z}{1-2\eta} \,W_\tau)}\Big\|_{\tilde{\mathfrak {S}}^{1/t}(\mathcal{F})}\,
\;\leq\;1
\end{equation}

By arguing analogously as in the proof of \cite[Lemma 2.29]{FrKnScSo1}, we deduce that the third term on the right-hand side of \eqref{estimate_on_g} is 
\begin{equation}
\label{Remainder_term_3_bound}
\;\leq\; \Big(CM^2 \eta^{-2} (1+\|w\|_{L^p(\Lambda)}\Big)^M\,.
\end{equation}
The only difference in the proof is that we use Proposition \ref{Product of subgraphs} (iii) instead of \cite[Proposition 2.20]{FrKnScSo1} (for the details of this step in the context of bounded interaction potentials, see \cite[(2.88)]{FrKnScSo1}).
Using \eqref{Remainder_term_1_bound}--\eqref{Remainder_term_3_bound} in \eqref{estimate_on_g}, we obtain that
\begin{equation}
\label{Remainder_term_123_bound}
|g_{\tau,M}^{\xi}(z,\mathbf{u})|\;\leq\;\bigg(\frac{Cr}{\eta^2}\bigg)^r\,\Bigg[\frac{C(1+\|w\|_{L^p(\Lambda)})}{\eta^2}\Bigg]^M\,M!\,.
\end{equation}
Substituting \eqref{Remainder_term_123_bound} into \eqref{definition_of_g_1}, we deduce (i).

We note that (ii) follows from the Feynman-Kac formula and the proof of Proposition \ref{Product of subgraphs} with $\xi$ replaced by $\tilde{\xi} \in \mathbf{C}_r$, which is the operator whose kernel is the absolute value of the kernel of $\xi$.
The arguments are analogous to those in \cite[Proposition 4.5, Corollary 4.6]{FrKnScSo1} when $\xi \in \mathbf{B}_r$ with minor modifications when $\xi=\mathrm{Id}_r$ (for the details of the latter see \cite[Section 4.2]{FrKnScSo1}).
\end{proof}
We recall the function $A^{\xi}_{\tau}$ given in \eqref{Duhamel_expansion_unbounded_A}.
\begin{lemma} 
\label{A^xi_{tau,z}_analytic}
Let $r \in \mathbb{N}$ and $\xi \in \mathbf{C}_r$ be given.
The function $A^{\xi}_{\tau}$ is analytic in the right-half plane $\{z \in \mathbb{C}\,, \re z >0\}$.
\end{lemma}
\begin{proof}
We argue analogously as in the proof of \cite[Lemma 2.34]{FrKnScSo1}. The proof carries over since by Lemma \ref{w_positive_approximation} (ii) (when $d=2,3$) and Lemma \ref{w_1D_approximation} (ii) (when $d=1$), the operator $W_{\tau}$ is bounded on the range of $P^{(\leq n)}$, which is defined to be the orthogonal projection onto 
$\bigoplus_{n'=0}^{n} \mathfrak{H}^{(n')} \subseteq \mathcal{F}$. Furthermore, $W_{\tau}$ is positive by \eqref{W_tau_positivity} when $d=2,3$ and by using Lemma \ref{w_1D_approximation} (i) when $d=1$. 
Namely when $d=1$, we use \eqref{CCR_tau} to rewrite \eqref{W_tau_1D} as 
\begin{equation*}
W_{\tau}\;=\;\frac{1}{2} \int \dd x\,\dd y\, w_\tau(x-y) \big[\phi^{*}_{\tau}(x)\,\phi^{*}_{\tau}(y)\big]\,\big[\phi_{\tau}^{*}(x)\,\phi_{\tau}^{*}(y)\big]^{*} \;\geq\;0\,.
\end{equation*}
Once one has these two ingredients, the rest of the proof follows in the same way.
We refer the reader to the proof of \cite[Lemma 2.34]{FrKnScSo1} for the full details.
\end{proof}

\section{Analysis of the classical system}
\label{Analysis of the classical system}

\subsection{General framework}
\label{Analysis of the classical system: General framework}
As in Section \ref{Analysis of the quantum system}, we assume throughout this section that $w$ is a $d$-admissible interaction potential as in Definition \ref{interaction_potential_w}.
In order to study the classical system for $d=2,3$, we first prove Lemma \ref{W_Wick-ordered} (i), which gives us the construction of the classical interaction $W$. Recall that the construction of the classical interaction when $d=1$ does not require renormalisation and is given in \eqref{W_1D_unbounded}--\eqref{W_1D_unbounded2} above. 
Before proceeding to this proof, let us set up some notation and conventions that we will use in all dimensions.

In addition to the classical Green function \eqref{classical_Green_function}, we also work with the \emph{truncated classical Green function}. Given $K \in \mathbb{N}$ and recalling \eqref{lambda_k}--\eqref{e_k}, we define $G_{[K]}: \Lambda \times \Lambda \rightarrow \mathbb{C}$ by
\begin{equation}
\label{G_K}
G_{[K]} (x;y)\;\deq\; \sum_{|k| \leq K} \frac{1}{\lambda_k}\,e_k(x)\,\overline{e_k(y)}\,.
\end{equation}
One can then rewrite \eqref{varrho_K} as $\varrho_{[K]}=G_{[K]}(0;0)=G_{[K]}(x;x)$ for all $x \in \Lambda$. We note a convergence result for the $G_{[K]}$.

\begin{lemma}
\label{GK_convergence_lemma}
Let $q \in \mathcal{Q}_d$, for $\mathcal{Q}_d$ as in \eqref{Q_d}. Then, we have
\begin{equation*}
\|G_{[K]}(x;\cdot)-G(x;\cdot)\|_{L^q(\Lambda)}\;=\;\|G_{[K]}(\cdot;x)-G(\cdot;x)\|_{L^q(\Lambda)} \rightarrow 0 \quad \mbox{as}\quad K \rightarrow \infty\,, 
\end{equation*}
uniformly in $x \in \Lambda$. %In particular, we have 
%\begin{equation*}
%\|G_{[K]}-G\|_{L^q(\Lambda^2)} \rightarrow 0 \quad \mbox{as}\quad K \rightarrow \infty\,.
%\end{equation*}
\end{lemma}

\begin{proof}
For $x,y \in \Lambda$ and $K \in \mathbb{N}$ we write
\begin{equation}
\label{GK_convergence_lemma_proof}
G_{[K]}(x;y)-G(x;y)\;=\;G_{[K]}(y;x)-G(y;x)\;=-\;\sum_{|k| >K}\frac{1}{\lambda_k}\,\ee^{2\pi \ii k \cdot (x-y)}\,.
\end{equation}
The claim now follows by arguing as in the proof of Proposition \ref{G_{tau,0}_bound} (i).
\end{proof}

Throughout this section, we use the classical Wick theorem applied to polynomials of the classical free field \eqref{classical_free_field} or its truncated version \eqref{phi_K}. In order to encode the pairings that one obtains in this way, we can use the graph structure given in Definitions \ref{cal X_unbounded}, \ref{def_pairing_unbounded}, and \ref{def_collapsed_graph_unbounded} when $d=2,3$ and in Definition \ref{def_pairing_1D_unbounded} when $d=1$. The main difference is that now the classical fields commute and so it is no longer necessary to impose the order as in Definition \ref{cal X_unbounded} (ii) when $d=2,3$ (or with appropriate modifications when $d=1$). As in Definition \ref{cal X_unbounded}, each vertex $(i,\vartheta,\delta) \in \mathcal{X}$ corresponds to a factor of $\phi$ or $\bar{\phi}$ ($\phi_{[K]}$ or $\bar{\phi}_{[K]}$ in the truncated setting), where the $\phi,\phi_{[K]}$ are obtained by replacing the $\phi_{\tau}$ and the $\bar{\phi},\bar{\phi}_{[K]}$ are obtained by replacing the $\phi_{\tau}^{*}$. 
In other words, the former case corresponds to $\delta=+1$ and the latter to $\delta=-1$. Furthermore, the $i$ indexes the factors of the interaction $w$ for $i=1,\ldots,m$ and the observable $\xi$ when $i=m+1$.
From context, it will be clear whether we are working in the truncated setting or not. We adapt a lot of conventions and notation from Section \ref{The graphical representation} without further comment. With this setup, we now prove Lemma \ref{W_Wick-ordered} (i). Note that proof of Lemma \ref{W_Wick-ordered} (ii) is given in Section \ref{Endpoint-admissible interaction potentials_3} below. 

\begin{proof}[Proof of Lemma \ref{W_Wick-ordered} (i)]
In order to prove the claim, it suffices to show that $W_{[K]}$ converges in $L^m(\mu)$ for $m$
an even integer. We consider $K_1,\ldots,K_m \in \mathbb{N}$ with $\min\{M_1,\ldots,M_m\}=K$ and we let
$T_{\mathbf{K}}\;\deq\; \int \dd \mu \, W_{[K_1]} \cdots W_{[K_m]}$. We want to show that $T_{\mathbf{K}}$ converges to a limit as $K \rightarrow \infty$ and that this limit does not depend on $\mathbf{K} \deq (K_1,\ldots,K_m)$. 
By \eqref{phi_K} and \eqref{G_K}, we note that for all $x,y \in \Lambda$ and $L, M \in \mathbb{N}$, we have
\begin{equation}
\label{phi_K2}
\int \dd \mu \, \phi_{[L]}(x)\,\bar{\phi}_{[M]}(y)\;=\; G_{[L \wedge M ]}(x;y)\,.
\end{equation}
Here $L \wedge M$ denotes the minimum of $\{L,M\}$.

By using Wick's theorem, Definitions \ref{cal X_unbounded}, \ref{def_pairing_unbounded}, and \ref{def_collapsed_graph_unbounded} and \eqref{phi_K2}, we can write
\begin{equation}
\label{T_K}
T_{\mathbf{K}}\;=\;\frac{1}{2^m} \sum_{\Pi \in \mathfrak{R}(m,0)} \mathcal{I}_{\mathbf{K},\Pi}\,,
\end{equation} 
where with $(\mathcal{V},\mathcal{E}) \equiv (\mathcal{V}_{\Pi},\mathcal{E}_{\Pi})$ as in Definition \ref{def_collapsed_graph_unbounded}, we define
\begin{equation}
\label{I_{K,Pi}}
\mathcal{I}_{\mathbf{K},\Pi}\;\deq\;\int_{\mathcal{V}} \dd \mathbf{y}\,\Bigg[\prod_{i=1}^{m}w(y_{i,1}-y_{i,2})\Bigg]\,
\prod_{\{a,b\} \in \mathcal{E}} G_{[K_{i_a} \wedge K_{i_b}]}(y_a;y_b)\,.
\end{equation}
Furthermore, we let
\begin{equation}
\label{I_Pi}
\mathcal{I}_{\Pi}\;\deq\;\int_{\mathcal{V}} \dd \mathbf{y}\,\Bigg[\prod_{i=1}^{m}w(y_{i,1}-y_{i,2})\Bigg]\,
\prod_{\{a,b\} \in \mathcal{E}} G(y_a;y_b)\,.
\end{equation}
Note that, $\mathcal{I}_{\Pi}=\mathcal{I}^{\emptyset}_{\Pi}$ by taking $r=0$ in Definition \ref{I_infinity_unbounded} (i). In particular, by Proposition \ref{Product of subgraphs} (i) and  Proposition \ref{Approximation_result_2_unbounded}, we have that $\mathcal{I}_{\Pi}$ is well-defined and finite.

By \eqref{T_K}, it suffices to show that for all $\Pi \in \mathfrak{R}(m,0)$ we have
\begin{equation}
\label{I_K_convergence}
\mathcal{I}_{\mathbf{K},\Pi} \rightarrow \mathcal{I}_{\Pi} \quad \mbox{as}\quad K \rightarrow \infty\,.
\end{equation}
The proof of \eqref{I_K_convergence} follows by using an analogous telescoping argument to the one used in the proof of \eqref{Approximation_result_2_unbounded_I2-I} above. The only difference is that, instead of Lemma \ref{Q1_convergence_lemma}, we use Lemma \ref{GK_convergence_lemma}. The proof is now concluded as in \cite[Proof of Lemma 1.5; given in Section 3.1]{FrKnScSo1}.
\end{proof}

\subsection{The perturbative expansion}
\label{Expansion_classical}
Analogously to \eqref{theta_tau_unbounded}, given $\xi \in \mathbf{C}_r$, we define the random variable
\begin{equation}
\label{theta_unbounded}
\Theta(\xi)\;\deq\;\int \dd x_1 \cdots \dd x_r\,\dd y_1 \cdots \dd y_r\,\xi(x_1, \dots, x_r; y_1, \dots, y_r)\,\bar{\phi}(x_1) \cdots \bar{\phi}(x_r)\, \phi(y_1) \cdots \phi(y_r)\,.
\end{equation}
Furthermore, analogously to \eqref{rho_tau_fraction_unbounded}, we write
\begin{equation*}
%\label{rho_fraction_unbounded}
\rho(\Theta(\xi))\;=\;\frac{\int \Theta(\xi)\,\ee^{-zW}\,\dd \mu}{\int \ee^{-zW}\,\dd \mu}\;=\;\frac{\hat{\rho}_{1} (\Theta(\xi))}{\hat{\rho}_{1}(\mathrm{1})}\,,
\end{equation*}
where for $z \in \mathbb{C}$ with $\re z \geq 0$ and $X \equiv X(\omega)$ a random variable we define
\begin{equation} 
\label{rho_hat_unbounded}
\hat{\rho}_{z}(X)\;\deq\;\int X\,\ee^{-zW}\,\dd \mu\,.
\end{equation}
Given $M \in \mathbb{N}$, we expand the exponential in $\ee^{-zW}$ up to order $M$ to obtain that
\begin{equation}
\label{A^xi}
A^{\xi}(z) \;\deq\; \hat{\rho}_z(\Theta(\xi)) \;=\; \sum_{m=0}^{M-1} a^{\xi}_{m}\,z^m + R^{\xi}_{M}(z)\,,
\end{equation}
where
\begin{equation}
\label{a^xi_m}
a^{\xi}_{m} \;\deq\; \frac{(-1)^m}{m!} \int \Theta(\xi)\,W^{m}\,\dd \mu
\end{equation}
and
\begin{equation}
\label{R^xi_m}
R^{\xi}_{m}(z)\;\deq\;\frac{(-1)^{M}\,z^{M}}{M!}\,\int \Theta(\xi)\, W^{M}\,\ee^{-z' \,W}\,\dd \mu\,,
\end{equation}
for some $z' \in [0,z]$.

We recall the definition of $a^{\xi}_{\infty,m}$ given by \eqref{a_infty} above.
\begin{lemma} 
\label{a^xi_m_identity}
Given $m,r \in \mathbb{N}$ and $\xi \in \mathbf{C}_r$, we have $a^{\xi}_m=a^{\xi}_{\infty,m}$.
\end{lemma}

\begin{proof}
When $d=1$, the claim follows from \eqref{W_1D_unbounded}, \eqref{a_infty} and \eqref{a^xi_m} by applying Wick's theorem.
Let us now consider the case when $d=2,3$.
We use Lemma \ref{W_Wick-ordered} (i) and argue as in the proof of \cite[Lemma 3.1]{FrKnScSo1} to deduce that \eqref{a^xi_m} can be rewritten as 
\begin{equation}
\label{a^xi_m2}
a^{\xi}_{m}\;=\; \frac{(-1)^{m}}{m!\,2^{m}}\,\lim_{K \rightarrow \infty} \sum_{\Pi \in \mathfrak{R}(d,m,r)} \mathcal{I}^{\xi}_{[K],\Pi}\,, 
\end{equation}
where for $K \in \mathbb{N}$ we let
\begin{equation}
\label{a^xi_m_identity_proof_1}
\mathcal{I}^{\xi}_{[K],\Pi}\;\deq\;\int_{\Lambda^{\mathcal{V}}} \dd \mathbf{y}\, \Bigg[\prod_{i=1}^{m} w(y_{i,1}-y_{i,2})\Bigg]\,\xi(\mathbf{y_1})\,\prod_{e \in \mathcal{E}} \mathcal{J}_{[K],e}(\mathbf{y}_e)
\end{equation}
and for $e \in \mathcal{E} \equiv \mathcal{E}_{\Pi}$
\begin{equation}
\label{a^xi_m_identity_proof_2}
\mathcal{J}_{[K],e}(\mathbf{y}_e) \;\deq\;
\begin{cases}
G(y_{a};y_{b}) & \mbox{if } e=\{a,b\}\,\,\mbox{for}\,\,a,b \in \mathcal{V}_1
\\
G_{[K]}(y_{a};y_{b}) & \mbox{if } e=\{a,b\}\,\,\mbox{for}\,\,a \in \mathcal{V}_2\,\,\mbox{or}\,\,a \in \mathcal{V}_2\,.
\end{cases}
\end{equation}

By applying a telescoping argument analogously as in the proof of \eqref{Approximation_result_2_unbounded_I2-I}  (see also the proof of \eqref{I_K_convergence} above) and by using Lemma \ref{GK_convergence_lemma}, we deduce that 
\begin{equation}
\label{I^xi_K_convergence}
\mathcal{I}^{\xi}_{[K],\Pi} \rightarrow \mathcal{I}^{\xi}_{\Pi} \quad \mbox{as}\quad K \rightarrow \infty\,.
\end{equation}
Here $\mathcal{I}^{\xi}_{\Pi}$ is given by Definition \ref{I_infinity_unbounded} (i).
The claim now follows from \eqref{a_infty}, \eqref{a^xi_m2}, and \eqref{I^xi_K_convergence}.
\end{proof}
We can now deduce upper bounds on the explicit terms in the expansion \eqref{A^xi}.
\begin{corollary} 
\label{a^xi_m_corollary}
Given $m,r \in \mathbb{N}$ and $\xi \in \mathbf{C}_r$, we have 
\begin{equation*}
|a^{\xi}_{m}|\;\leq\;(Cr)^{r} \,C^m  \,(1+\|w\|_{L^p(\Lambda)})^{m}\,m!\,.
\end{equation*}
\end{corollary}
\begin{proof}
The claim follows from Proposition \ref{Product of subgraphs} (ii), Proposition \ref{a_convergence}, and Lemma \ref{a^xi_m_identity}.
\end{proof}
Furthermore, we can estimate the remainder term \eqref{R^xi_m}.
\begin{proposition} 
\label{R^xi_M_bound_proposition}
Given $M,r \in \mathbb{N}$, $\xi \in \mathbf{C}_r$, and $z \in \mathbb{C}$ with $\re z \geq 0$, we have
\begin{equation*}
|R^{\xi}_{M}(z)|\;\leq\;(Cr)^r \, C^M\,(1+\|w\|_{L^p(\Lambda)})^M\,|z|^M\,M!\,.
\end{equation*}
\end{proposition}
\begin{proof}
We first note that $W \geq 0$. When $d=1$, this follows from \eqref{W_1D_unbounded} and Definition \ref{interaction_potential_w} (ii). 
When $d=2,3$, we argue as in \eqref{W_tau_positivity}. Namely, we use \eqref{W_K_unbounded} and Definition \ref{interaction_potential_w} (iii) to deduce that for $K \in \mathbb{N}$, we have
\begin{equation}
\label{R^xi_M_bound_proposition_1}
W_{[K]} \;=\; \frac{1}{2} \sum_{k \in \mathbb{Z}^d} \hat{w}(k) \, \bigg|\int \dd x\, \ee^{2 \pi \ii k \cdot x} \big(|\phi_{[K]}(x)|^2- \varrho_{[K]} \big)\bigg|^2 \;\geq\;0\,.
\end{equation}
The nonnegativity of $W$ then follows from \eqref{R^xi_M_bound_proposition_1} and Lemma \ref{W_Wick-ordered} (i).

For $\xi \in \mathbf{B}_r$, we deduce the claim by arguing as in the proof of \cite[Lemma 3.3]{FrKnScSo1}.
Namely, we use the nonnegativity of the interaction, take absolute values, and use the Cauchy-Schwarz inequality %and Wick's theorem to deduce that 
\begin{equation}
\label{R^xi_M_bound}
|R^{\xi}_{M}(z)|\;\leq\; \frac{|z|^M}{M!}\,\bigg(\int \Theta(\xi)^{2}\,\dd \mu\bigg)^{1/2}\,\bigg(\int W^{2M}\,\dd \mu\bigg)^{1/2}\,.
%\;\leq\; \frac{|z|^M}{M!}\,(Cr)^{2r}\,\bigg(\int W^{2M}\,\dd \mu\bigg)^{1/2}\,.
\end{equation}
For the second factor in \eqref{R^xi_M_bound} we use Wick's theorem and for the third, we use Corollary \ref{a^xi_m_corollary} with $m=2M$ and $r=0$ (in this case, the observable is $\xi=\emptyset$). We hence deduce that the right-hand side of \eqref{R^xi_M_bound} is
\begin{equation}
\label{R^xi_M_bound_proposition_2}
\;\leq\; \frac{|z|^M}{M!}\,(Cr)^r\,(1+\|w\|_{L^p(\Lambda)})^{M} (CM)^{2M}\;\leq\;(Cr)^r \, C^M\,(1+\|w\|_{L^p(\Lambda)})^M\,|z|^M\,M!\,.
\end{equation}
For $d=1$ and $\xi=\mathrm{Id}_r$, we have $\Theta(\xi)=(\int \dd x\,|\phi(x)|^2)^r \geq 0$.
Hence, again using the nonnegativity of $W$, we obtain that 
\begin{equation*}
|R^{\xi}_{M}(z)|\;\leq\; \frac{|z|^M}{M!}\,\int \Theta(\xi)\, W^{M}\,\dd \mu\;\leq\;(Cr)^r \, C^M\,(1+\|w\|_{L^p(\Lambda)})^M\,|z|^M\,M!\,,
\end{equation*}
by applying Corollary \ref{a^xi_m_corollary} with $m=M$.
\end{proof}
We have a classical analogue of Lemma \ref{A^xi_{tau,z}_analytic}.
\begin{lemma} 
\label{A^xi_analytic}
Let $r \in \mathbb{N}$ and $\xi \in \mathbf{C}_r$ be given.
The function $A^{\xi}$ given by \eqref{A^xi} is analytic in the right-half plane $\{z \in \mathbb{C}\,, \re z >0\}$.
\end{lemma}
\begin{proof}
We can take derivatives in $z$ by differentiating under the integral sign in \eqref{rho_hat_unbounded} with $X=\Theta(\xi)$. This is justified by analogous arguments as in the proof of Proposition \ref{R^xi_M_bound_proposition} (with a higher power of $W$). The analyticity of $A^{\xi}$ in the right-half plane follows.
\end{proof}

\subsection{Proof of Theorem \ref{main_result}}
\label{Proof of Theorem 1}

We now have all the ingredients needed to prove Theorem \ref{main_result}. 
With what we have obtained so far, the method of proof is now analogous to the proof of \cite[Theorem 1.6]{FrKnScSo1} (given in \cite[Section 3.3]{FrKnScSo1}) when $d=2,3$, of \cite[Theorem 1.8]{FrKnScSo1} when $d=1$ and $\xi \in \mathbf{B}_r$ (given in \cite[Section 4.3]{FrKnScSo1}) and of \cite[Theorem 1.9]{FrKnScSo1} when $d=1$ and $\xi=\mathrm{Id}_r$ (given in \cite[Section 4.4]{FrKnScSo1}). We outline the main ideas and refer the reader to \cite{FrKnScSo1} for further details.

\begin{proof}[Proof of Theorem \ref{main_result}]

Let us first show (ii). 
We need to show that
\begin{equation}
\label{main_result_proof_1}
\rho_{\tau}^{\eta}(\Theta_{\tau}(\xi)) \stackrel{\tau \rightarrow \infty}{\longrightarrow} \rho(\Theta(\xi)) \quad \mbox{uniformly in} \,\, \xi \in \mathbf{B}_r\,.
\end{equation}
Namely, assuming \eqref{main_result_proof_1}, by using the duality $\mathfrak{S}^{2}(\mathfrak{H}^{(r)}) \cong \mathfrak{S}^{2}(\mathfrak{H}^{(r)})^{*}$, we indeed have
\begin{equation}
\label{main_result_proof_1A}
\|\gamma^{\eta}_{\tau,r}-\gamma_r\|_{\mathfrak{S}^{2}(\mathfrak{H}^{(r)})}\;=\;\mathop{\mathrm{sup}}_{{\xi \in \mathbf{B}_r}} \big|\tr\,(\gamma^{\eta}_{\tau,r}\xi-\gamma_r \xi)\big|\;=\;\mathop{\mathrm{sup}}_{{\xi \in \mathbf{B}_r}} \big|\rho^{\eta}(\Theta_{\tau}(\xi))-\rho(\Theta(\xi))\big|\stackrel{\tau \rightarrow \infty}{\longrightarrow}0\,.
\end{equation}
Note that \eqref{main_result_proof_1} follows if we prove that 
\begin{equation}
\label{main_result_proof_2}
\hat{\rho}_{\tau,1}^{\eta}(\Theta_{\tau}(\xi)) \stackrel{\tau \rightarrow \infty}{\longrightarrow} \hat{\rho}_{1}(\Theta(\xi)) \quad \mbox{uniformly in} \,\, \xi \in \mathbf{B}_r\,.
\end{equation}
Let us now show \eqref{main_result_proof_2}.
We work with the functions $A^{\xi}_{\tau}$ and $A^{\xi}$ defined in \eqref{Duhamel_expansion_unbounded_A} and \eqref{A^xi} respectively and apply Proposition \ref{Borel_summation_unbounded}. 
By Lemma \ref{A^xi_{tau,z}_analytic} and Lemma \ref{A^xi_analytic}, we know that both of these functions are analytic in the right-half plane. In particular, recalling \eqref{cal_C_R_unbounded}, they are analytic in $\mathcal{C}_R$ for all $R>0$. 
In particular, we consider $R=2$ so that $1 \in \mathcal{C}_R$.

In $\mathcal{C}_2$, we perform the expansion \eqref{Borel_summation_unbounded1} according to \eqref{Duhamel_expansion_unbounded_A} and \eqref{A^xi}.
By Proposition \ref{Product of subgraphs} (ii), Corollary \ref{a^xi_m_corollary}, Proposition \ref{remainder_term} (i), and Proposition \ref{R^xi_M_bound_proposition}, we know that \eqref{Borel_summation_unbounded2}--\eqref{Borel_summation_unbounded3} hold with 
\begin{equation}
\label{nu_sigma_choice}
\nu\;=\;\bigg(\frac{Cr}{\eta^2}\bigg)^r\,,\quad \sigma\;=\;\frac{C(1+\|w\|_{L^p(\Lambda)})}{\eta^2}\,.
\end{equation}
We can now deduce \eqref{main_result_proof_2} by setting $z=1$ in Proposition \ref{Borel_summation_unbounded}.

We now prove (i). Note that \eqref{main_result_proof_1}--\eqref{main_result_proof_1A} still hold when $d=1$, with $\eta=0$ and the appropriate modifications of the graph structure. Furthermore, if we let $\xi=\mathrm{Id}_r$, the proof of \eqref{main_result_proof_2} shows that 
\begin{equation*}
\hat{\rho}_{\tau,1}(\Theta_{\tau}(\mathrm{Id}_r)) \stackrel{\tau \rightarrow \infty}{\longrightarrow} \hat{\rho}_{1}(\Theta(\mathrm{Id}_r))\,, 
\end{equation*}
which implies that 
\begin{equation*}
\rho_{\tau}(\Theta_{\tau}(\mathrm{Id}_r)) \stackrel{\tau \rightarrow \infty}{\longrightarrow} \rho(\Theta(\mathrm{Id}_r))\,.
\end{equation*}
The latter convergence can be rewritten as 
\begin{equation}
\label{main_result_proof_3}
\tr \gamma_{\tau,r} \stackrel{\tau \rightarrow \infty}{\longrightarrow} \tr \gamma_r\,.
\end{equation}
Using \eqref{gamma_r_unbounded} and \eqref{gamma_{tau,r}_unbounded}, one can show that $\gamma_{\tau,r}$ and $\gamma_r$ are positive operators. We can now deduce \eqref{Theorem_1D} from \eqref{main_result_proof_1A} (applied for $d=1$), \eqref{main_result_proof_3} and \cite[Lemma 4.10]{FrKnScSo1} (which, in turn, is based on \cite[Lemma 2.20]{Simon05}). We hence deduce the claim in (i).
\end{proof}

\section{Endpoint-admissible interaction potentials}
\label{Endpoint-admissible interaction potentials}
\subsection{General framework}
\label{Endpoint-admissible interaction potentials_1}
We recall that in Sections \ref{Analysis of the quantum system} and \ref{Analysis of the classical system}, we were always considering interaction potentials $w$, which were $d$-admissible as in Definition \ref{interaction_potential_w}. In this section, we fix $d=2$ and consider an interaction potential $w$ which is endpoint admissible as in Definition \ref{endpoint_admissible_w}. 
Furthermore, for $\tau \geq 1$, we let $w_\tau$ be given by Lemma \ref{w_endpoint_admissible_approximation}. 
In this construction, we are always considering the parameter $\beta \in \mathcal{B}_2 \equiv (0,1)$ as in \eqref{B_d}. We adapt all of the conventions and notation accordingly.
Throughout this section, we fix $\epsilon>0$ small and $L>0$ as in Definition \ref{endpoint_admissible_w} (iv). Furthermore, we fix $\delta=\epsilon/2$ as in Lemma \ref{w_endpoint_admissible_approximation}. In the sequel, we use the observation that $w \in H^{-1+\delta}(\Lambda)$ and that $\|w\|_{H^{-1+\delta}(\Lambda)} \leq CL$ by Lemma \ref{w_endpoint_admissible_approximation} (i).

The graphical setup that we consider in this section is the same as the one given in Section \ref{The graphical representation} when $d=2$. We are always considering Hilbert-Schmidt observables $\xi \in \mathbf{B}_r \equiv \mathbf{C}_r$. All of the other conventions and notation are the same as before.

Let us recall that, by Definition \ref{endpoint_admissible_w}, $w$ is assumed to be pointwise nonnegative (in addition to being of positive type). 
This allows us to work with $w_\tau$ pointwise nonnegative as in Lemma \ref{w_endpoint_admissible_approximation} (iv) (see \eqref{endpoint_admissible_w_tau2}--\eqref{f_tau} below).
These pointwise nonnegativity conditions are used in the proof of the convergence of the explicit terms given in Proposition \ref{Approximation_result_2_unbounded_endpoint}. 
In particular, we can take absolute values and obtain \eqref{a_0_vertex_factor1}--\eqref{I'_endpoint_bound4} below.
This is the only place where we use the pointwise nonnegativity of the interaction, albeit in a crucial way.
At this step, it is important to keep in mind that the absolute value map is not continuous on $H^s(\Lambda)$ for $s<0$, see Remark \ref{absolute_value_unbounded_negative_Sobolev} below.

We recall that the analysis in Sections \ref{Analysis of the quantum system}--\ref{Analysis of the classical system} relied primarily on bounds on the quantum and classical Green functions (see Proposition \ref{Q^12_bounds}, Proposition \ref{G_{tau,0}_bound}, and Lemma \ref{Q1_convergence_lemma} above). Once we had such bounds at our disposal, we never needed to explicitly use the translation invariance of the Green functions.
In this section, we use the translation invariance of the Green functions on the torus in a crucial way. In order to motivate why we need this, we observe that translation invariance allows us to use the decay in Fourier space of $w$ from Definition \ref{endpoint_admissible_w} (iv) to obtain the finiteness of \eqref{optimality} above when $w$ is an endpoint-admissible interaction potential. For further details, see the proof of Proposition \ref{Product of subgraphs_endpoint} below.

Given $x \in \Lambda$ and with notation as in \eqref{G}, we define 
\begin{equation}
\label{G_bold}
\mathbf{G}_{\tau,t}(x)\;\deq\; G_{\tau,t}(x;0)\,\,\,\mbox{for}\,\,t >-1\,.
\end{equation}
Furthermore, we let 
\begin{equation}
\label{G_tau_bold}
\mathbf{G}_{\tau}\;\deq\;\mathbf{G}_{\tau,0}\,.
\end{equation}
Note that \eqref{G_tau_bold} is consistent with \eqref{G_tau} and \eqref{G_bold}.
Recalling \eqref{Q_{tau,t}} and \eqref{Q^12}, we define for $x \in \Lambda$ and $t \in (-1,1)$
\begin{equation}
\label{Q^12_bold}
\mathbf{Q}_{\tau,t}(x) \;\deq\; Q_{\tau,t}(x;0)\,,\quad \mathbf{Q}_{\tau,t}^{(j)}(x)\;\deq\;Q_{\tau,t}^{(j)}(x;0)\,,\,\,j\;=\;1,2\,.
\end{equation}
In this setting, \eqref{Q_splitting} can be rewritten as
\begin{equation}
\label{Q_splitting_bold}
\mathbf{Q}_{\tau,t}\;=\; \mathbf{Q}_{\tau,t}^{(1)}+\frac{1}{\tau} \mathbf{Q}_{\tau,t}^{(2)}\,.
\end{equation}
The appropriate modifications of \eqref{positivity} and \eqref{Q_positivity} also hold. We note several estimates that are analogous to those in Propositions \ref{Q^12_bounds} and \ref{G_{tau,0}_bound}.
\begin{lemma}
\label{Q_bounds_endpoint}
There exist constants $C_1,C_2>0$ such that for all $\tau \geq 1$ and $t \in (-1,1)$, we have the following estimates.
\begin{itemize}
\item[(i)] %($L^{\infty}$ bound.)
$\|\mathbf{Q}_{\tau,t}^{(1)}\|_{L^{\infty}(\Lambda)} \leq C_1 \log \tau$.
\item[(ii)] %(Sobolev bound.)
For all $s<1$ we have
$\|\mathbf{Q}_{\tau,t}^{(1)}\|_{H^s(\Lambda)} \leq C(s)$.
\item[(iii)] $\int \dd y \,\mathbf{Q}_{\tau,t}^{(2)}(y) \leq 1$.
\item[(iv)] If $\{t\} \geq \frac{1}{2}$, then we have $\|\mathbf{Q}_{\tau,t}^{(2)}\|_{L^{\infty}(\Lambda)} \leq C_1 \tau$.
\item[(v)] For all $x \in \Lambda$ we have $\mathbf{Q}_{\tau,t}(x) \geq \mathbf{Q}^{(1)}_{\tau,t}(x) \geq C_2$.
\end{itemize}
\end{lemma}
\begin{proof}
Claims (i), (iii), (iv), and (v) follow immediately from \eqref{Q^12_bold} and Proposition \ref{Q^12_bounds}. In order to prove claim (ii), we recall \eqref{Q_{tau,t}_bound_proof_s1} and use \eqref{Q^12_bold} to deduce that for all $k \in \mathbb{Z}^2$
\begin{equation}
\label{Q_{tau,t}_bound_proof_s1_bold} 
\big|\big(\mathbf{Q}_{\tau,t}^{(1)}\big)\,\,\widehat{}\,\,(k)\big|\;\leq\;\frac{C}{\lambda_k}\,.
\end{equation}
We hence obtain (ii) from \eqref{Q_{tau,t}_bound_proof_s1_bold} by arguing as in the proof of \eqref{G_{tau,0}_bound_proof_s1}.
\end{proof}
In the sequel, we use a general product estimate in Sobolev spaces.

\begin{lemma}
\label{Sobolev_product_estimate}
Let $s \geq 0$ and $\alpha \in (0,1)$ be given. 
Then, for all $f,g \in H^{\max\{s+\alpha,1-\alpha\}}(\Lambda)$, we have
\begin{equation*}
\|f g \|_{H^s(\Lambda)} \lesssim_{s,\alpha} \|f\|_{H^{s+\alpha}(\Lambda)}\,\|g\|_{H^{1-\alpha}(\Lambda)}+\|f\|_{H^{1-\alpha}(\Lambda)}\,\|g\|_{H^{s+\alpha}(\Lambda)}\,.
\end{equation*}
\end{lemma}
\begin{proof}
We have
\begin{multline*}
\|f g\|_{H^s(\Lambda)} \sim_s \Bigg\|\langle k \rangle^{s} \sum_{k' \in \mathbb{Z}^2} \hat{f}(k-k')\,\hat{g}(k')\Bigg\|_{\ell^2_k}
\;\leq\;  \Bigg\|\langle k \rangle^{s} \sum_{k' \in \mathbb{Z}^2} |\hat{f}(k-k')|\,|\hat{g}(k')|\Bigg\|_{\ell^2_k}
\\
\;\lesssim_s\;
\Bigg\|\sum_{k' \in \mathbb{Z}^2} \langle k-k' \rangle^{s} \,|\hat{f}(k-k')|\,|\hat{g}(k')|\Bigg\|_{\ell^2_k}
+ \Bigg\|\sum_{k' \in \mathbb{Z}^2} |\hat{f}(k-k')|\,\langle k' \rangle^{s} \,|\hat{g}(k')|\Bigg\|_{\ell^2_k}\,,
\end{multline*}
which by Plancherel's theorem is 
\begin{equation}
\label{Sobolev_product_estimate_proof_1}
\sim_s \Big\|\Big(\langle D \rangle^{s} \,\mathbf{F}^{-1}\,|\hat{f}|\Big)\,\mathbf{F}^{-1}\,|\hat{g}|\Big\|_{L^2(\Lambda)}+
\Big\|\mathbf{F}^{-1}\,|\hat{f}|\,\Big(\langle D \rangle^{s}\, \mathbf{F}^{-1}\,|\hat{g}|\Big)\Big\|_{L^2(\Lambda)}\,.
\end{equation}
In \eqref{Sobolev_product_estimate_proof_1}, $\langle D \rangle^s$ denotes the Fourier multiplier operator with symbol $\langle k \rangle^s$ and $\mathbf{F}^{-1}$ denotes the inverse Fourier transform. By H\"{o}lder's inequality, the expression in \eqref{Sobolev_product_estimate_proof_1} is
\begin{equation*}
\;\leq\; \big\|\langle D \rangle^{s} \,\mathbf{F}^{-1}\,|\hat{f}|\big\|_{L^{\frac{2}{1-\alpha}}(\Lambda)} \,\big\| \mathbf{F}^{-1}\,|\hat{g}|\big\|_{L^{\frac{2}{\alpha}}(\Lambda)}+\big\| \mathbf{F}^{-1}\,|\hat{f}|\big\|_{L^{\frac{2}{\alpha}}(\Lambda)} \,\big\|\langle D \rangle^{s} \,\mathbf{F}^{-1}\,|\hat{g}|\big\|_{L^{\frac{2}{1-\alpha}}(\Lambda)}\,,
\end{equation*}
which is 
\begin{equation*}
\;\lesssim_{s,\alpha}\; \|f\|_{H^{s+\alpha}(\Lambda)}\,\|g\|_{H^{1-\alpha}(\Lambda)}+\|f\|_{H^{1-\alpha}(\Lambda)}\,\|g\|_{H^{s+\alpha}(\Lambda)}
\end{equation*}
by using the Sobolev embeddings $H^{\alpha} (\Lambda) \hookrightarrow L^{\frac{2}{1-\alpha}}(\Lambda)$ and $H^{1-\alpha}(\Lambda) \hookrightarrow L^{\frac{2}{\alpha}}(\Lambda)$, as well as the fact that $L^2$-based Sobolev norms are invariant under taking absolute values of the Fourier transform.
\end{proof}

\subsection{The quantum system}
\label{Endpoint-admissible interaction potentials_2}
As in Section \ref{Bounds on the explicit terms}, we first need to show the upper bound \eqref{Borel_summation_unbounded2}  on the explicit terms $a_{\tau,m}$ given by \eqref{Explicit_term_a_unbounded} (where now $w_\tau$ is given by Lemma \ref{w_endpoint_admissible_approximation}). More precisely, we prove an analogue of Proposition \ref{Product of subgraphs}.
\begin{proposition}
\label{Product of subgraphs_endpoint}
Fix $m,r \in \mathbb{N}$. Given $\Pi \in \mathfrak{R}(d,m,r)$, $\mathbf{t} \in \mathfrak{A}(m)$, the following estimates hold uniformly in  $\xi \in \mathbf{B}_r$ and $\tau \geq 1$.
\begin{itemize}
\item[(i)] We have
\begin{equation*}
\big| \mathcal{I}_{\tau,\Pi}^{\xi} (\mathbf{t}) \big| \leq C_0^{m+r}(1+\|w\|_{L^1(\Lambda)}+L)^{m}\,,
\end{equation*}
for some $C_0>0$.
\item[(ii)] $|a^{\xi}_{\tau,m}| \leq (Cr)^{r} \,C^m  \,(1+\|w\|_{L^1(\Lambda)}+L)^{m}\,m!\,$.
\item[(iii)] $|f^{\xi}_{\tau,m}(\mathbf{t})| \leq C^{m+r}\,(1+\|w\|_{L^1(\Lambda)}+L)^m\,(2m+r)!\,$.  
\end{itemize}
\end{proposition}
\begin{proof}
As in the proof of Proposition \ref{Product of subgraphs}, claims (ii) and (iii) follow from claim (i), so we need to prove (i).
We argue similarly and use the same notation as in the proof of Proposition \ref{Product of subgraphs} (i). 
Since $w_\tau \geq 0$ by Lemma \ref{w_endpoint_admissible_approximation} (iv), we note that \eqref{Product of subgraphs unbounded 1} now becomes
\begin{equation}
\label{Product of subgraphs unbounded 1_endpoint}
\big|\mathcal I_{\tau,\Pi}^{\xi}(\mathbf{t})\big| \;\leq\; \int_{\Lambda^{\mathcal V}} \dd \mathbf{y} \, \Bigg[\prod_{i=1}^{m} w_\tau(y_{i,1}-y_{i,2})\Bigg] \,\big|\xi(\mathbf{y_1})\big|\, \Bigg[\prod_{j=1}^{k}\, \prod_{e\, \in\, \mathfrak{P}_j} \mathcal J_{\tau, e}(\mathbf{y_e}, \mathbf{s})\Bigg]\,.
\end{equation}
Given $a \in \mathcal{V}_2$, we define $\mathcal{W}_\tau^{a}$ as in \eqref{cal W_unbounded}. By Lemma \ref{w_endpoint_admissible_approximation} (iv), we have that $\mathcal{W}_\tau^{a} \geq 0$ pointwise and, by Lemma \ref{w_endpoint_admissible_approximation} (i), (ii), (iii), (v), (vi), we have that
\begin{align}
\label{W_bound_bold1}
\|\mathcal{W}_{\tau}^a\|_{H^{-1+\delta}(\Lambda)} &\;\leq\; 1 + \|w_\tau\|_{H^{-1+\delta}(\Lambda)} \;\leq 1 + CL\,. 
\\
\label{W_bound_bold2}
0\; \leq\; \widehat{\mathcal{W}}_{\tau}^a &\;\leq\;1+ L\,.
\\
\label{W_bound_bold3}
\|\mathcal{W}_{\tau}^a\|_{L^{\infty}(\Lambda)} &\;\leq\; \tau^{\beta}\,.
\\
\label{W_bound_bold4}
\|\mathcal{W}_{\tau}^a\|_{L^1(\Lambda)} &\;\leq\; 1+\|w_{\tau}\|_{L^1(\Lambda)}\;\leq\;1+ C\|w\|_{L^1(\Lambda)}\,.
\end{align}
By arguing analogously as for \eqref{Inductive inequality 2 unbounded_B}, we reduce the claim to showing that for all $\mathfrak{P} \in \mathrm{conn}(\mathcal{E})$, we have
\begin{equation}
\label{Inductive inequality 2 unbounded_B_endpoint}
\|\mathfrak{I}(\mathfrak{P})\|_{L^{2}_{\mathbf{y_1}} L^{\infty}_{\mathcal{V}^{*}(\mathfrak{P})}}\;\leq\;C_0^{\,|\mathcal {V}(\mathfrak{P})|}
\,\,\big(1+\|w\|_{L^1(\Lambda)}+L\big)^{\,|\mathcal{V}_2(\mathfrak{P})|}\,,
\end{equation}
where
\begin{equation}
\label{I(P)_endpoint}
\mathfrak{I}(\mathfrak {P}) \;\deq\; \int \dd \mathbf{y}_{\mathcal{V}_2(\mathfrak{P})} \, \prod_{e \in \mathfrak{P}} \mathcal{J}_{\tau, e}(\mathbf{y}_e, \mathbf{s})\, \prod_{a \in \mathcal{V}_2(\mathfrak{P})} \mathcal{W}_\tau^{a}(y_{a}-y_{a^*})\,.
\end{equation}
We prove \eqref{Inductive inequality 2 unbounded_B_endpoint} by using induction on $n \deq |\mathcal{V}(\mathfrak{P})|$, as in the proof of \eqref{Inductive inequality 2 unbounded_B}. We define the times associated to the edges of $\mathfrak{P}$ as in \eqref{zeta_unbounded_j} above.

\subsubsection*{Base of induction: $n=2$}

We consider three cases.
\begin{itemize}
\item[(1)] $\mathfrak{P}$ is a closed path connecting $a_1,a_2 \in \mathcal{V}_2$ satisfying $a_2=a_1^{*}$. 

In this case, we have
\begin{align*}
&\mathfrak{I}(\mathfrak{P})\;=\;\int \dd y_{a_1}\,\dd y_{a_2}\,w_{\tau}(y_{a_1}-y_{a_2})\,\mathbf{G}_{\tau}^2(y_{a_1}-y_{a_2})
+\frac{1}{\tau} \,\int \dd y_{a_1}\,w_{\tau}(0)\,\mathbf{G}_{\tau}(0)
\\
&\;=\;\int \dd x \,w_{\tau}(x)\,\mathbf{G}_{\tau}^2(x)
+\frac{1}{\tau} \,w_{\tau}(0)\,\mathbf{G}_{\tau}(0)\,,
\end{align*}
which by duality and Lemma \ref{Sobolev_product_estimate} with $s=1-\delta$ and $\alpha=\delta/2$ is
\begin{align}
\notag
&\;\leq\;\|w_{\tau}\|_{H^{-1+\delta}(\Lambda)}\,\|\mathbf{G}_\tau^2\|_{H^{1-\delta}(\Lambda)}+\frac{1}{\tau} \|w_{\tau}\|_{L^{\infty}(\Lambda)} \,\|\mathbf{G}_{\tau}\|_{L^{\infty}(\Lambda)}
\\
\label{Base_n=2_(1)_endpoint}
&\;\lesssim_{\delta}\; \|w_{\tau}\|_{H^{-1+\delta}(\Lambda)}\,\|\mathbf{G}_\tau\|_{H^{1-\delta/2}(\Lambda)}^2+\frac{1}{\tau} \|w_{\tau}\|_{L^{\infty}(\Lambda)} \,\|\mathbf{G}_{\tau}\|_{L^{\infty}(\Lambda)}\,.
\end{align}
Using Lemma \ref{w_endpoint_admissible_approximation} (i), Lemma \ref{Q_bounds_endpoint} (ii), Lemma \ref{w_endpoint_admissible_approximation} (v), Lemma \ref{Q_bounds_endpoint} (i), and \eqref{B_d}, we deduce that the expression in \eqref{Base_n=2_(1)_endpoint} is 
\begin{equation}
\label{Base_n=2_(1)_endpoint_2}
\;\leq\; C \|w_{\tau}\|_{H^{-1+\delta}(\Lambda)} +C \tau^{-1+\beta}\,\log \tau\;\leq\; C(1+\|w\|_{H^{-1+\delta}(\Lambda)})\;\leq\;C(1+L)\,.
\end{equation}
\item[(2)] $\mathfrak{P}$ is a closed path connecting $a_1,a_2 \in \mathcal{V}_2$ satisfying $a_2 \neq a_1^{*}$.

Using \eqref{Q_splitting_bold} and arguing as in \eqref{Induction_Base_Case_2B}, we deduce that in this case
we can rewrite $\mathfrak{I} (\mathfrak{P})$ as
\begin{align}
\notag
&\int \dd y_{a_1}\,\dd y_{a_2}\, \mathcal{W}_\tau^{a_1}(y_{a_1}-y_{a_1^*})\,\mathcal{W}_\tau^{a_2}(y_{a_2}-y_{a_2^*})\,\mathbf{Q}_{\tau,\zeta_1}^{(1)}(y_{a_1}-y_{a_2})\,\mathbf{Q}_{\tau,-\zeta_1}^{(1)}(y_{a_2}-y_{a_1})
\\
\notag
+\frac{1}{\tau} &\int \dd y_{a_1}\,\dd y_{a_2}\,\mathcal{W}_\tau^{a_1}(y_{a_1}-y_{a_1^*})\,\mathcal{W}_\tau^{a_2}(y_{a_2}-y_{a_2^*})\,\mathbf{Q}_{\tau,\zeta_1}^{(2)}(y_{a_1}-y_{a_2})\,\mathbf{Q}_{\tau,-\zeta_1}^{(1)}(y_{a_2}-y_{a_1})
\\
\notag
+\frac{1}{\tau} &\int \dd y_{a_1}\,\dd y_{a_2}\,\mathcal{W}_\tau^{a_1}(y_{a_1}-y_{a_1^*})\,\mathcal{W}_\tau^{a_2}(y_{a_2}-y_{a_2^*})\,\mathbf{Q}_{\tau,\zeta_1}^{(1)}(y_{a_1}-y_{a_2})\,\mathbf{Q}_{\tau,-\zeta_1}^{(2)}(y_{a_2}-y_{a_1})
\\
\label{Induction_Base_Case_2B_bold}
+\frac{1}{\tau^2} &\int \dd y_{a_1}\,\dd y_{a_2}\,\mathcal{W}_\tau^{a_1}(y_{a_1}-y_{a_1^*})\,\mathcal{W}_\tau^{a_2}(y_{a_2}-y_{a_2^*})\,\mathbf{Q}_{\tau,\zeta_1}^{(2)}(y_{a_1}-y_{a_2})\,\mathbf{Q}_{\tau,-\zeta_1}^{(2)}(y_{a_2}-y_{a_1})\,,
\end{align}
where $\zeta_1=\pm \sigma(e_1)\,(s_{a_1}-s_{a_2}) \in (-1,1) \setminus \{0\}$.
As in the analysis of \eqref{Induction_Base_Case_2B}, we have that all the integrands in \eqref{Induction_Base_Case_2B_bold} are nonnegative.

We can estimate the first term in \eqref{Induction_Base_Case_2B_bold} as
\begin{align}
\notag
&\frac{1}{2} \int \dd y_{a_1}\,\dd y_{a_2}\, \mathcal{W}_\tau^{a_1}(y_{a_1}-y_{a_1^*})\,\mathcal{W}_\tau^{a_2}(y_{a_2}-y_{a_2^*})\,\big[\mathbf{Q}_{\tau,\zeta_1}^{(1)}(y_{a_1}-y_{a_2})\big]^2
\\
\label{Induction_Base_Case_2B_bold_term1}
+&\frac{1}{2} \int \dd y_{a_1}\,\dd y_{a_2}\, \mathcal{W}_\tau^{a_1}(y_{a_1}-y_{a_1^*})\,\mathcal{W}_\tau^{a_2}(y_{a_2}-y_{a_2^*})\,\big[\mathbf{Q}_{\tau,-\zeta_1}^{(1)}(y_{a_2}-y_{a_1})\big]^2\,.
\end{align}
We rewrite the first term in \eqref{Induction_Base_Case_2B_bold_term1} as
\begin{equation*}
\sim \int \dd y_{a_1}\, \mathcal{W}_\tau^{a_1}(y_{a_1}-y_{a_1^*})\, \Big(\big[\mathbf{Q}_{\tau,\zeta_1}^{(1)}\big]^2*\mathcal{W}_{\tau}^{a_2}\Big) (y_{a_1}-y_{a_2^*})\,,
\end{equation*}
which by duality in the $y_{a_1}$ variable is 
\begin{align}
\notag
&\;\leq\; \big\|\mathcal{W}_\tau^{a_1}(\cdot-y_{a_1^*})\big\|_{H^{-1+\delta}(\Lambda)}\,\big\|\big[\mathbf{Q}_{\tau,\zeta_1}^{(1)}\big]^2*\mathcal{W}_{\tau}^{a_2}(\cdot-y_{a_2^*})\big\|_{H^{1-\delta}(\Lambda)}
\\
\label{Induction_Base_Case_2B_bold_term1A}
&\;=\;\big\|\mathcal{W}_\tau^{a_1}\big\|_{H^{-1+\delta}(\Lambda)}\,\big\|\big[\mathbf{Q}_{\tau,\zeta_1}^{(1)}\big]^2*\mathcal{W}_{\tau}^{a_2}\big\|_{H^{1-\delta}(\Lambda)}\,.
\end{align}
Note that by \eqref{W_bound_bold2}, we have 
\begin{equation}
\label{Induction_Base_Case_2B_bold_term1B}
|(f*\mathcal{W}_{\tau}^{a_2}){\,\,\widehat{}\,\,}|\;\leq\; (1+L)|\hat{f}|
\end{equation}
for all $f \in L^1(\Lambda)$.
Using \eqref{Induction_Base_Case_2B_bold_term1B}, we deduce that the expression in \eqref{Induction_Base_Case_2B_bold_term1A} is
\begin{equation}
\label{Induction_Base_Case_2B_bold_term1B_2}
\;\leq\; (1+L)\big\|\mathcal{W}_\tau^{a_1}\big\|_{H^{-1+\delta}(\Lambda)}\,\big\|\big[\mathbf{Q}_{\tau,\zeta_1}^{(1)}\big]^2\big\|_{H^{1-\delta}(\Lambda)}\,.
\end{equation}
Note that, by Lemma \ref{Sobolev_product_estimate} with $s=1-\delta$ and $\alpha=\frac{\delta}{2}$ and Lemma \ref{Q_bounds_endpoint} (ii), we have that
\begin{equation}
\label{Induction_Base_Case_2B_bold_term1C}
\big\|\big[\mathbf{Q}_{\tau,\zeta_1}^{(1)}\big]^2\big\|_{H^{1-\delta}(\Lambda)} \;\leq\; C(\delta)\,\big\|\mathbf{Q}_{\tau,\zeta_1}^{(1)}\big\|_{H^{1-\delta/2}(\Lambda)}^2 \;\leq\; C(\delta)\,.
\end{equation}

By \eqref{W_bound_bold1} and \eqref{Induction_Base_Case_2B_bold_term1C}, it follows that the expression in \eqref{Induction_Base_Case_2B_bold_term1B_2} is 
\begin{equation}
\label{Induction_Base_Case_2B_bold_term1D}
\;\leq\; C(\delta) (1+L)^2\,.
\end{equation}
By analogous arguments, the second term in \eqref{Induction_Base_Case_2B_bold_term1} also satisfies the bound \eqref{Induction_Base_Case_2B_bold_term1D}.
This gives us an acceptable bound for the first term in \eqref{Induction_Base_Case_2B_bold}.

For the second and third term in \eqref{Induction_Base_Case_2B_bold}, we argue analogously as in the proof of \eqref{Induction_Base_Case_2B_l.o.t.1}. We now set $p=1$ and use \eqref{W_bound_bold4} instead of \eqref{cal W bound_unbounded}. Furthermore, instead of Proposition \ref{Q^12_bounds} (ii), we use Lemma \ref{Q_bounds_endpoint} (iii). In the end, we obtain the upper bound of 
\begin{equation}
\label{l.o.t.1_endpoint}
\;\leq\; C \tau^{-1+\beta}\,\log \tau\,(1+\|w\|_{L^1(\Lambda)})
\end{equation}
as in \eqref{Induction_Base_Case_2B_l.o.t.1} when $d=2$ and $p=1$, which is acceptable.

Finally, for the fourth term in \eqref{Induction_Base_Case_2B_bold}, we use Lemma \ref{Q_bounds_endpoint} (iii), (iv), as well as 
\eqref{W_bound_bold3} and \eqref{W_bound_bold4}.
We argue analogously as in \eqref{Induction_Base_Case_2B_Term4}--\eqref{Induction_Base_Case_2B_Term4_l.o.t.2} with $d=2$ and $p=1$ to get the acceptable upper bound
\begin{equation}
\label{l.o.t.2_endpoint}
\;\leq\; C\tau^{-1+\beta}\,(1+\|w\|_{L^1(\Lambda)})\,.
\end{equation}

\item[(3)] $\mathfrak{P}$ is an open path with vertices $b_1,b_2 \in \mathcal{V}_1$. We have
\begin{equation}
\label{n=2_open_path_bound_endpoint}
\|\mathfrak{I} (\mathfrak{P})\|_{L^2_{\mathbf{y_1}}L^{\infty}_{\mathcal{V}^{*}(\mathfrak{P})}}\;=\;\|\mathbf{G}_\tau\|_{L^2(\Lambda)}\;\leq\;C\,,
\end{equation}
uniformly in $\tau$ by Lemma \ref{Q_bounds_endpoint} (ii).
\end{itemize}
This completes the base step and proves \eqref{Inductive inequality 2 unbounded_B_endpoint} when $n=2$.

\subsubsection*{Inductive step}
Let $n \in \mathbb{N}, n \geq 3$ be given. With notation as in the inductive step of the proof of \eqref{Inductive inequality 2 unbounded_B}, we find a vertex $a \in \mathcal{V}_2(\mathfrak{P})$ satisfying \eqref{integration_vertex_inductive_step}. With analogous notation as in the setup of \eqref{y_a_dependence_unbounded}, we get that the dependence on the variable $y_a \equiv y_{a_2}$ in \eqref{I(P)_endpoint} is 
\begin{align}
\notag
\;\leq\; 
\mathcal{W}_{\tau}^{a_2}(y_{a_2}-y_{a_2^*})\,&\bigg[\mathbf{Q}_{\tau,\zeta_1}(y_{a_1}-y_{a_2})+\frac{\mathbf{1}_{\zeta_1=0}}{\tau} \delta(y_{a_1}-y_{a_2})\bigg]\,
\\
\label{y_a_dependence_unbounded_endpoint}
\times &\bigg[
\mathbf{Q}_{\tau,\zeta_2}(y_{a_2}-y_{a_3})+\frac{\mathbf{1}_{\zeta_2=0}}{\tau}\delta(y_{a_2}-y_{a_3})\bigg]
\,.
\end{align}
Furthermore, it is not possible for $\zeta_1=0$ and $\zeta_2=0$ simultaneously. The contributions to the $y_{a_2}$ integral coming from the single delta functions in \eqref{y_a_dependence_unbounded_endpoint} are estimated as in \eqref{Induction_Step_delta_1_unbounded}--\eqref{Induction_Step_delta_2_unbounded} by
\begin{align}
\notag
&\int\dd y_{a_2}\,\mathcal{W}_{\tau}^{a_2}(y_{a_2}-y_{a_2^*})\,\bigg[\frac{\mathbf{1}_{\zeta_1=0}}{\tau}\,\delta(y_{a_1}-y_{a_2})\,\mathbf{Q}_{\tau,\zeta_2}(y_{a_2}-y_{a_3}) + \frac{\mathbf{1}_{\zeta_2=0}}{\tau}\,\mathbf{Q}_{\tau,\zeta_1}(y_{a_1}-y_{a_2}) \, \delta(y_{a_2}-y_{a_3})\bigg]
\\
\label{Induction_Step_delta_unbounded_endpoint}
&\;\leq\; C\,\tau^{-1+\beta}\,\Big(\mathbf{1}_{\zeta_1=0}+\mathbf{1}_{\zeta_2=0}\Big)\,\mathbf{Q}_{\tau,\zeta_1+\zeta_2} (y_{a_1}-y_{a_3})\,.
\end{align}
Analogously to \eqref{Induction_Step_unbounded}, we now prove
\begin{multline}
\label{Induction_Step_unbounded_endpoint}
\int \dd y_{a_2}\,\mathcal{W}_\tau^{a_2}(y_{a_2}-y_{a_2^{*}}) \,\mathbf{Q}_{\tau,\zeta_1}(y_{a_1}-y_{a_2}) \, \mathbf{Q}_{\tau,\zeta_2}(y_{a_2}-y_{a_3}) 
\\
\;\leq\;C\,\big(1+L\big)\, \Big[1+ \,\mathbf{1}_{\zeta_1+\zeta_2 \neq 0}\,\mathbf{Q}_{\tau,\zeta_1+\zeta_2} (y_{a_1}-y_{a_3})\Big]\,.
\end{multline}
Namely, if we know \eqref{Induction_Step_unbounded_endpoint}, by using \eqref{I(P)_endpoint}, \eqref{y_a_dependence_unbounded_endpoint}--\eqref{Induction_Step_delta_unbounded_endpoint}, \eqref{B_d}, and Lemma \ref{Q_bounds_endpoint} (v), we deduce that
\begin{equation}
\label{I(P)_induction_unbounded_endpoint}
\mathfrak{I}(\mathfrak{P})\;\leq\;C_0\,\big(1+L\big)\, \mathfrak{I}(\mathfrak{\hat{P}})\,,
\end{equation}
for $\mathfrak{\hat{P}}$ defined as in \eqref{I(P)_induction_unbounded}. As in the proof of Proposition \ref{Product of subgraphs}, \eqref{I(P)_induction_unbounded_endpoint} implies the inductive step.

Let us now prove \eqref{Induction_Step_unbounded_endpoint}. By \eqref{Q_splitting_bold}, we can rewrite the left-hand side of \eqref{Induction_Step_unbounded_endpoint} as 

\begin{align}
\notag
&\int \dd y_{a_2}\,\mathcal{W}^{a_2}_{\tau}(y_{a_2}-y_{a_2^*})\,\mathbf{Q}_{\tau,\zeta_1}^{(1)}(y_{a_1}-y_{a_2})\,\mathbf{Q}_{\tau,\zeta_2}^{(1)}(y_{a_2}-y_{a_3})
\\
\notag
+\frac{1}{\tau}&\,\int \dd y_{a_2}\,\mathcal{W}^{a_2}_{\tau}(y_{a_2}-y_{a_2^*})\,\mathbf{Q}_{\tau,\zeta_1}^{(2)}(y_{a_1}-y_{a_2})\,\mathbf{Q}_{\tau,\zeta_2}^{(1)}(y_{a_2}-y_{a_3})
\\
\notag
+\frac{1}{\tau}&\,\int \dd y_{a_2}\,\mathcal{W}^{a_2}_{\tau}(y_{a_2}-y_{a_2^*})\,\mathbf{Q}_{\tau,\zeta_1}^{(1)}(y_{a_1}-y_{a_2})\,\mathbf{Q}_{\tau,\zeta_2}^{(2)}(y_{a_2}-y_{a_3})
\\
\label{Induction_Step_unbounded_2_endpoint}
+\frac{1}{\tau^2}&\,\int \dd y_{a_2}\,\mathcal{W}^{a_2}_{\tau}(y_{a_2}-y_{a_2^*})\,\mathbf{Q}_{\tau,\zeta_1}^{(2)}(y_{a_1}-y_{a_2})\,\mathbf{Q}_{\tau,\zeta_2}^{(2)}(y_{a_2}-y_{a_3})\,.
\end{align}
The first term in \eqref{Induction_Step_unbounded_2_endpoint} is
\begin{equation}
\label{Induction_Step_unbounded_2_endpoint_Term1}
\frac{1}{2}\,\int \dd y_{a_2}\,\mathcal{W}^{a_2}_{\tau}(y_{a_2}-y_{a_2^*})\,\big[\mathbf{Q}_{\tau,\zeta_1}^{(1)}(y_{a_1}-y_{a_2})\big]^2+\frac{1}{2}\,\int \dd y_{a_2}\,\mathcal{W}^{a_2}_{\tau}(y_{a_2}-y_{a_2^*})\,\big[\mathbf{Q}_{\tau,\zeta_2}^{(1)}(y_{a_2}-y_{a_3})\big]^2\,.
\end{equation}
By symmetry, we need to consider only the first term in \eqref{Induction_Step_unbounded_2_endpoint_Term1}, which by duality in the $y_{a_2}$ variable is
\begin{equation*}
\leq \frac{1}{2}\, \big\|\mathcal{W}^{a_2}_{\tau}(\cdot-y_{a_2^*})\big\|_{H^{-1+\delta}(\Lambda)}\,\big\|\big[\mathbf{Q}_{\tau,\zeta_1}^{(1)}(y_{a_1}-\cdot)\big]^2\big\|_{H^{1-\delta}(\Lambda)}\;=\; \frac{1}{2}\,\big\|\mathcal{W}^{a_2}_{\tau}\big\|_{H^{-1+\delta}(\Lambda)}\,\big\|\big[\mathbf{Q}_{\tau,\zeta_1}^{(1)}\big]^2\big\|_{H^{1-\delta}(\Lambda)}\,,
\end{equation*}
which, in turn, by \eqref{W_bound_bold1} and \eqref{Induction_Base_Case_2B_bold_term1C} is 
\begin{equation}
\label{Induction_Step_unbounded_2_endpoint_Term1_bound}
\;\leq\; C(1+L)\,.
\end{equation}

The second and third term in \eqref{Induction_Step_unbounded_2_endpoint} are 
\begin{equation}
\label{Induction_Step_unbounded_2_endpoint_Term23_bound}
\;\leq\; C \tau^{-1+\beta} \,\log \tau
\end{equation}
by using the same arguments as in the proof of \eqref{Step_l.o.t._1} with appropriate notational modifications. Here, instead of 
\eqref{cal W bound_unbounded2} and Proposition \ref{Q^12_bounds} (i)--(ii), we use the analogous bound \eqref{W_bound_bold3} and Lemma \ref{Q_bounds_endpoint} (i) and (iii) respectively. With analogous notational modifications as above, the fourth term in \eqref{Induction_Step_unbounded_2_endpoint} is
\begin{equation}
\label{Induction_Step_unbounded_2_endpoint_Term4_bound}
\;\leq\; C \tau^{-1+\beta} \big(1+\,\mathbf{1}_{\zeta_1+\zeta_2 \neq 0}\,\mathbf{Q}_{\tau,\zeta_1+\zeta_2}(y_{a_1}-y_{a_3})\big)
\end{equation}
by using the same arguments as in \eqref{Induction_Step_unbounded_2_4th_term}--\eqref{Step_l.o.t._5}. Here, instead of \eqref{cal W bound_unbounded2} and Proposition \ref{Q^12_bounds} (ii)--(iii), we use \eqref{W_bound_bold3} and Lemma \ref{Q_bounds_endpoint} (iii)--(iv) respectively. Combining \eqref{Induction_Step_unbounded_2_endpoint_Term1_bound}--\eqref{Induction_Step_unbounded_2_endpoint_Term4_bound}, we deduce \eqref{Induction_Step_unbounded_endpoint}.
\end{proof}
Before we proceed, let us make two observations that follow from the proof of Proposition \ref{Product of subgraphs_endpoint}. There are analogues of Remarks \ref{Product_of_subgraphs_Remark} and \ref{Product_of_subgraphs_Remark2_A}.

\begin{remark}
\label{Product_of_subgraphs_Remark_endpoint}
When $n=2$ and $c_2=c_1^*$, in the proof of Proposition \ref{Product of subgraphs_endpoint}, \eqref{Product_of_subgraphs_Remark_n=2a}
gets replaced by 
\begin{equation}
\label{Product_of_subgraphs_Remark_n=2a_endpoint}
\mathfrak{I}(\mathfrak{P}) \;\lesssim\; \|w_{\tau}\|_{H^{-1+\delta}(\Lambda)}\,\|\mathbf{G}_\tau\|_{H^{1-\delta/2}(\Lambda)}^2
+ \mathcal{O}(\tau^{-\epsilon_0})\,
\end{equation}
for some $\epsilon_0>0$. This follows from \eqref{Base_n=2_(1)_endpoint}--\eqref{Base_n=2_(1)_endpoint_2}.
\\
If  $c_1,c_2 \in \mathcal{V}_2$ and $c_2 \neq c_1^*$, \eqref{Product_of_subgraphs_Remark_n=2b} gets replaced by
\begin{align}
\notag
\mathfrak{I}(\mathfrak{P}) &\;\lesssim\; \big\|\mathcal{W}_\tau^{a_1}\big\|_{H^{-1+\delta}(\Lambda)}\,\big\|\widehat{\mathcal{W}}_\tau^{a_2}\big\|_{\ell^{\infty}(\mathbb{Z}^2)} \,\big\|\mathbf{Q}_{\tau,\zeta_1}^{(1)}\big\|_{H^{1-\delta/2}(\Lambda)}^2
\\
\label{Product_of_subgraphs_Remark_n=2b_endpoint}
&+\big\|\widehat{\mathcal{W}}_\tau^{a_1}\big\|_{\ell^{\infty}(\mathbb{Z}^2)} \big\|\mathcal{W}_\tau^{a_2}\big\|_{H^{-1+\delta}(\Lambda)}\,\,\big\|\mathbf{Q}_{\tau,\zeta_2}^{(1)}\big\|_{H^{1-\delta/2}(\Lambda)}^2+\mathcal{O}(\tau^{-\epsilon_0})\,.
\end{align}
for some $\epsilon_0>0$. This follows from the arguments in \eqref{Induction_Base_Case_2B_bold}--\eqref{l.o.t.2_endpoint}.
\\
If $c_1,c_2 \in \mathcal{V}_1$, we again have \eqref{Product_of_subgraphs_Remark_n=2c}.
When $n \geq 3$, \eqref{Product_of_subgraphs_Remark_n>=3} gets replaced by 
\begin{equation}
\label{Product_of_subgraphs_Remark_n>=3_endpoint}
\mathfrak{I}(\mathfrak{P})\;\lesssim\; \Bigg[\big\|\mathcal{W}^{a_2}_{\tau}\big\|_{H^{-1+\delta}(\Lambda)}\,\Big(\big\|\mathbf{Q}_{\tau,\zeta_1}^{(1)}\big\|_{H^{1-\delta/2}(\Lambda)}^2+\big\|\mathbf{Q}_{\tau,\zeta_2}^{(1)}\big\|_{H^{1-\delta/2}(\Lambda)}^2\Big)+ \mathcal{O}(\tau^{-\epsilon_0})\Bigg]\,\mathfrak{I}(\mathfrak{\hat{P}}),
\end{equation}
for some $\epsilon_0>0$.
This follows from the arguments in \eqref{y_a_dependence_unbounded_endpoint}--\eqref{Induction_Step_unbounded_2_endpoint_Term4_bound}.
In \eqref{Product_of_subgraphs_Remark_n=2a_endpoint}--\eqref{Product_of_subgraphs_Remark_n>=3_endpoint}, the $\mathcal{O}(\tau^{-\epsilon_0})$ contributions come from the $\mathbf{Q}^{(2)}$ and delta function factors. All the factors involving only $\mathbf{Q}^{(1)}$ give us the leading order terms in \eqref{Product_of_subgraphs_Remark_n=2a_endpoint}--\eqref{Product_of_subgraphs_Remark_n>=3_endpoint}, and in \eqref{Product_of_subgraphs_Remark_n=2c}, when we are considering an open path with two vertices.
\end{remark}

\begin{remark}
\label{Product_of_subgraphs_Remark2}
In proving \eqref{Product_of_subgraphs_Remark_n=2a_endpoint}--\eqref{Product_of_subgraphs_Remark_n>=3_endpoint}, the bounds giving us the leading order terms were obtained by applying only \eqref{W_bound_bold1}, \eqref{W_bound_bold2}, and \eqref{W_bound_bold4}, and never by applying \eqref{W_bound_bold3}. The leading order terms in the upper bounds \eqref{Product_of_subgraphs_Remark_n=2a_endpoint}--\eqref{Product_of_subgraphs_Remark_n>=3_endpoint} are the ones that we obtain by estimating expressions involving only factors of $\mathbf{Q}^{(1)}$ and no factors of $\mathbf{Q}^{(2)}$ or the delta function. 
For details, see \eqref{Base_n=2_(1)_endpoint}--\eqref{Base_n=2_(1)_endpoint_2}, \eqref{Induction_Base_Case_2B_bold_term1}--\eqref{Induction_Base_Case_2B_bold_term1D}, \eqref{n=2_open_path_bound_endpoint} and \eqref{Induction_Step_unbounded_2_endpoint_Term1}--\eqref{Induction_Step_unbounded_2_endpoint_Term1_bound} above.
\end{remark}

We now study the convergence of the explicit terms, as in Section \ref{Convergence of the explicit terms}.
Let us first appropriately adapt the notation to the current setting. Recalling
\eqref{hat_V,E,sigma}--\eqref{J_tau_e^2} and \eqref{Q^12_bold}, we write 
\begin{align*}
%\label{J_tau_e_Q_endpoint1}
\mathcal{J}_{\tau, e}(\mathbf{y}_e,\mathbf{s}) &\;=\; \mathbf{Q}_{\tau, \sigma(e) (s_a - s_b)}(y_a-y_b) + \frac{\mathbf{1}_{\sigma(e) = +1}\,\mathbf{1}_{i_a=i_b}}{\tau} \,\delta(y_a-y_b)\,.
\\
%\label{J_tau_e_Q_endpoint2}
\mathcal{J}_{\tau, e}^{(1)}(\mathbf{y}_e,\mathbf{s}) &\;=\; \mathbf{Q}^{(1)}_{\tau, \sigma(e) (s_a - s_b)}(y_a-y_b)\,.
\\ 
%\label{J_tau_e_Q_endpoint3}
\mathcal{J}_{\tau, e}^{(2)}(\mathbf{y}_e,\mathbf{s})&\;=\; \frac{1}{\tau}\mathbf{Q}^{(2)}_{\tau, \sigma(e) (s_a - s_b)}(y_a-y_b)+\frac{\mathbf{1}_{\sigma(e) = +1}\,\mathbf{1}_{i_a=i_b}}{\tau} \,\delta(y_a-y_b)\,.
\end{align*}
By properly adapting \eqref{I_representation_unbounded} and \eqref{I1_representation_unbounded} to the setting of endpoint-admissible interaction potentials, we now show that the analogue of the approximation result given in Lemma \ref{Approximation_result_Q1} above holds in this setting.

\begin{lemma}
\label{Approximation_result_Q1_endpoint}
Fix $m,r \in \mathbb{N}$. Given $\Pi \in \mathfrak{R}(d,m,r)$, $\mathbf{t} \in \mathfrak{A}(m)$, we have that 
\begin{equation*}
%\label{Approximation_result_Q1_claim_endpoint}
\big|\mathcal{I}^{\xi}_{\tau,\Pi}(\mathbf{t}) - \mathcal{I}^{1,\xi}_{\tau,\Pi} (\mathbf{t})\big| \rightarrow 0 \quad \mbox{as}\quad \tau \rightarrow \infty\,\, \mbox{uniformly in} \,\, \xi \in \mathbf{B}_r\,.
\end{equation*}
\end{lemma}

\begin{proof}
We argue as in the proof of Lemma \ref{Approximation_result_Q1} and use a telescoping method. More precisely, we argue as in \eqref{Telescoping_1}--\eqref{e_0_telescoping_sum} and reduce to showing that for $e_0 \in \mathcal{E}$, we have
\begin{equation}
\label{Approximation_result_Q1_endpoint_proof1}
\mathcal{I}^{\xi}_{\tau,\Pi,e_0}(\mathbf{t}) \;=\; \int_{\Lambda^{\mathcal{V}}} \dd \mathbf{y}\,\Bigg[\prod_{i=1}^{m} w_{\tau}(y_{i,1}-y_{i,2})\Bigg] \,|\xi(\mathbf{y_1})|\,
\Bigg[\Bigg(\prod_{e \neq e_0} \mathcal{J}_{\tau, e}(\mathbf{y}_e,\mathbf{s})\Bigg) \, \mathcal{J}^{(2)}_{\tau,e_0}(\mathbf{y}_{e_0},\mathbf{s})\Bigg] \rightarrow 0
\end{equation}
as $\tau \rightarrow \infty$ uniformly in $\xi \in \mathbf{B}_r$. %Note that in \eqref{Approximation_result_Q1_endpoint_proof1}, $w_\tau$ appears without the absolute value since it is pointwise nonnegative. 
For $\mathfrak{P} \in \conn (\mathcal{E})$, as in \eqref{I_{e_0}_definition}, we work with
\begin{align}
\notag
&\mathfrak{I}_{e_0}(\mathfrak {P}) \;=\; \int \dd \mathbf{y}_{\mathcal{V}_2(\mathfrak{P})} \, \Bigg(\prod_{e \in \mathfrak{P} \setminus \{e_0\}} \mathcal{J}_{\tau, e}(\mathbf{y}_e, \mathbf{s})\Bigg)\, 
\\
\label{Approximation_result_Q1_endpoint_proof2}
&\times
\Bigg(\prod_{e \in \mathfrak{P} \cap \{e_0\}} \mathcal{J}_{\tau, e}^{(2)}(\mathbf{y}_e, \mathbf{s})\Bigg) \,\prod_{a \in \mathcal{V}_2(\mathfrak{P})} \mathcal{W}_\tau^{a}(y_{a}-y_{a^*})\,.
\end{align}
We now conclude by using Remark \ref{Product_of_subgraphs_Remark_endpoint} and arguing as in the proof of Lemma \ref{Approximation_result_Q1}. More precisely, if $n \deq |\mathcal{V}(\mathfrak{P})|=2$, then we argue as in \eqref{Product_of_subgraphs_Remark_n=2a_endpoint}--\eqref{Product_of_subgraphs_Remark_n=2b_endpoint} and deduce that $\mathfrak{I}_{e_0}(\mathfrak {P}) =\mathcal{O}(\tau^{-\epsilon_0})$ for some $\epsilon_0>0$. If $\mathfrak{P}$ is an open path of length $2$, then $\mathfrak{I}_{e_0}(\mathfrak {P}) =0$. 
Furthermore, if $n \geq 3$, then we obtain the analogue of \eqref{I_{e_0}_bound} by arguing as for \eqref{Product_of_subgraphs_Remark_n>=3_endpoint}.
Throughout the proof, it is important to use the fact that one is not taking absolute values of the interaction in \eqref{Approximation_result_Q1_endpoint_proof1}--\eqref{Approximation_result_Q1_endpoint_proof2}, since the interaction is pointwise nonnegative. This allows us to use the estimate \eqref{W_bound_bold1} (which is in general not true if we add absolute values). See Remark \ref{absolute_value_unbounded_negative_Sobolev} below.
\end{proof}

Recalling the definition \eqref{classical_Green_function} of the classical Green function, we let 
\begin{equation}
\label{G_endpoint}
\mathbf{G}(x) \;\deq\; G(x;0)
\end{equation}
for $x \in \Lambda$. Note that by \eqref{G_positive} we have that $\mathbf{G}$ is pointwise nonnegative.
Recall the quantity $\mathcal{I}_{\Pi}^{\xi}$ given by Definition \ref{I_infinity_unbounded} (i). We now note the following analogue of Proposition \ref{Approximation_result_2_unbounded}.

\begin{proposition}
\label{Approximation_result_2_unbounded_endpoint}
Let $m,r \in \mathbb{N}$, $\Pi \in \mathfrak{R}(d,m,r)$, and $\mathbf{t} \in \mathfrak{A}(m)$ be given. Then \eqref{Approximation_result_2_unbounded_convergence} holds uniformly in $\xi \in \mathbf{B}_r$.
\end{proposition}
In the proof of Proposition \ref{Approximation_result_2_unbounded_endpoint}, we use a modification of Lemma \ref{Q1_convergence_lemma}.
\begin{lemma}
\label{Q1_convergence_lemma_endpoint}
Let $t \in (-1,1)$ and $\alpha \in (0,1)$ be given. Then, we have
\begin{equation*}
\|\mathbf{Q}^{(1)}_{\tau,t}-\mathbf{G}\|_{H^{1-\alpha}(\Lambda)} \rightarrow 0 \quad \mbox{as}\quad \tau \rightarrow \infty\,.
\end{equation*}
\end{lemma}
\begin{proof}[Proof of Lemma \ref{Q1_convergence_lemma_endpoint}]
The claim follows from Lemma \ref{Q1_convergence_lemma} by setting $x=0$ in \eqref{Q1_convergence_lemma1}, recalling \eqref{Q^12_bold}, \eqref{G_endpoint}, and using the Sobolev embedding $H^{1-\alpha}(\Lambda) \hookrightarrow L^{q}(\Lambda)$ for $q=\frac{2}{\alpha} \in [1,\infty) = \mathcal{Q}_2$.
\end{proof}

\begin{proof}[Proof of Proposition \ref{Approximation_result_2_unbounded_endpoint}]
With $\mathcal{I}_{\tau,\Pi}^{2,\xi}(\mathbf{t})$ defined as in \eqref{I_{tau,Pi}^{2,xi}_definition}, we show that 
\begin{equation}
\label{Approximation_result_2_unbounded_I1-I2_endpoint}
\big|\mathcal{I}^{1,\xi}_{\tau,\Pi}(\mathbf{t}) - \mathcal{I}^{2,\xi}_{\tau,\Pi} (\mathbf{t})\big| \rightarrow 0 \quad \mbox{as}\quad \tau \rightarrow \infty\quad \mbox{uniformly in} \,\, \xi \in \mathbf{B}_r\,,
\end{equation}
as in \eqref{Approximation_result_2_unbounded_I1-I2}. In order to do this, we use a telescoping argument as in \eqref{Approximation_result_2_unbounded_I1-I2_telescoping} and we define $\widetilde{\mathcal{W}}_{\tau}^a$ as in \eqref{cal W_tilde}. By Definition \ref{endpoint_admissible_w}, we have that $\widetilde{\mathcal{W}}_{\tau}^a$ satisfies  \eqref{W_bound_bold1} and \eqref{W_bound_bold4} (albeit with a different value of $C$).
Furthermore, \eqref{W_bound_bold2} gets replaced by
\begin{equation}
\label{W_bound_bold2'} 
\big|\big(\widetilde{\mathcal{W}}_{\tau}^a\big)\,\,\widehat{}\,\,\big|\;\leq\;1+ L\,,
\end{equation}
which implies 
\begin{equation}
\label{Induction_Base_Case_2B_bold_term1B'}
|(f*\widetilde{\mathcal{W}}_{\tau}^{a_2}){\,\,\widehat{}\,\,}|\;\leq\; (1+L)|\hat{f}|\,,
\end{equation}
as in \eqref{Induction_Base_Case_2B_bold_term1B}.
Finally, \eqref{2.125}
gets replaced by
\begin{equation}
\label{2.125'}
\|\widetilde{\mathcal{W}}_\tau^{a_0}\|_{H^{-1+\delta}(\Lambda)} \rightarrow 0 \quad \mbox{as}\quad \tau \rightarrow \infty\quad \mbox{whenever} \quad \widetilde{\mathcal{W}}_\tau^{a_0}=w_\tau-w\,.
\end{equation}
As in \eqref{2.125}, a \emph{unique} such $a_0 \in \mathcal{V}_2$ exists.

Given $\mathfrak{P} \in \conn (\mathcal{E})$, we define
\begin{equation}
\label{I'(P)_endpoint}
\mathfrak{I}'(\mathfrak {P}) \;\deq\; \int \dd \mathbf{y}_{\mathcal{V}_2(\mathfrak{P})} \, \prod_{e \in \mathfrak{P}} \mathcal{J}_{\tau, e}^{(1)}(\mathbf{y}_e, \mathbf{s})\, \prod_{a \in \mathcal{V}_2(\mathfrak{P})} \widetilde{\mathcal{W}}_\tau^{a}(y_{a}-y_{a^*})\,.
\end{equation}
Note that \eqref{I'(P)_endpoint} differs from \eqref{I'(P)} in the sense that we are not taking absolute values of the interactions. This affects only the \emph{unique} factor corresponding to $a=a_0$ as in \eqref{2.125'}, since by Definition \ref{endpoint_admissible_w}, Lemma \ref{w_positive_approximation} (iv), and \eqref{cal W_tilde} all the other factors $\widetilde{\mathcal{W}}_\tau^{a}, a \neq a_0$ are pointwise nonnegative.

Let us first show that if $\mathfrak{P} \in \conn(\mathcal{E})$ is such that $a_0 \in \mathcal{V}_2(\mathfrak{P})$, then we have
\begin{equation}
\label{a_0_vertex_factor}
\|\mathfrak{I}'(\mathfrak{P})\|_{L^{2}_{\mathbf{y_1}} L^{\infty}_{\mathcal{V}^{*}(\mathfrak{P})}} \rightarrow 0 \quad \mbox{as}\quad \tau \rightarrow \infty\,, 
\end{equation}
as in \eqref{a_0_vertex_factor1}.
In order to show \eqref{a_0_vertex_factor}, we argue as in the the proof of Proposition \ref{Product of subgraphs_endpoint}.

We need to consider two cases. The first case is when $n \deq |\mathcal{V}(\mathfrak{P})|=2$.
Here we need to consider two subcases\footnote{This is unlike the three possible subcases in the proof of Proposition \ref{Product of subgraphs_endpoint}, since now $\mathfrak{P}$ cannot be an open path.}.

\begin{itemize}
\item[(1)] $\mathfrak{P}$ is a closed path connecting $a_0$ and $a_0^*$. 

We argue as in \eqref{Base_n=2_(1)_endpoint}--\eqref{Base_n=2_(1)_endpoint_2} and use duality, Lemma \ref{Sobolev_product_estimate} with $s=1$ and $\alpha=\frac{\delta}{2}$, and Lemma \ref{Q_bounds_endpoint} (ii) to deduce that 
\begin{align}
\notag
&|\mathfrak{I}'(\mathfrak{P})| \;=\; \bigg|\int \dd x \,\big(w_{\tau}-w\big)(x)\,\mathbf{G}_{\tau}^2(x)\bigg| 
\\
&
\label{I'_endpoint_bound1}
\;\leq\; 
\|w_\tau-w\|_{H^{-1+\delta}(\Lambda)} \, \|G_\tau^2\|_{H^{1-\delta}(\Lambda)} \;\lesssim\; \|w_\tau-w\|_{H^{-1+\delta}(\Lambda)}\,.
\end{align}
\item[(2)] $\mathfrak{P}$ is a closed path connecting $a_0 \equiv a_1$ and $a_2 \in \mathcal{V}_2$ satisfying $a_2 \neq a_1^{*}$.
(Here we write $a_1 \equiv a_0$ to be consistent with the earlier notation).

We argue as in \eqref{Induction_Base_Case_2B_bold} and use the fact that $\mathbf{Q}^{(1)}_{\tau,\zeta}$ is even (by \eqref{positivity} and \eqref{Q^12_bold}) in order to write 
\begin{align}
\notag
\mathfrak{I}'(\mathfrak{P})  &\;=\;\int \dd y_{a_1}\,\dd y_{a_2}\, \widetilde{\mathcal{W}}_\tau^{a_1}(y_{a_1}-y_{a_1^*})\,\widetilde{\mathcal{W}}_\tau^{a_2}(y_{a_2}-y_{a_2^*})\,\mathbf{Q}_{\tau,\zeta_1}^{(1)}(y_{a_1}-y_{a_2})\,\mathbf{Q}_{\tau,-\zeta_1}^{(1)}(y_{a_1}-y_{a_2})
\\
\label{Induction_Base_Case_2B_bold'}
&\;=\;\int \dd y_{a_1}\, \widetilde{\mathcal{W}}_\tau^{a_1}(y_{a_1}-y_{a_1^*})\,
\bigg[\widetilde{\mathcal{W}}_\tau^{a_2}\,*\, \Big(\mathbf{Q}_{\tau,\zeta_1}^{(1)}\,\mathbf{Q}_{\tau,-\zeta_1}^{(1)}\Big)\bigg](y_{a_1}-y_{a_2^*})\,.
\end{align}
By applying duality in \eqref{Induction_Base_Case_2B_bold'}, we deduce that 
\begin{equation*}
|\mathfrak{I}'(\mathfrak{P})| \;\leq\;\big\|\widetilde{\mathcal{W}}_\tau^{a_1}(\cdot-y_{a_1^*})\big\|_{H^{-1+\delta}(\Lambda)} \,\Big\| \widetilde{\mathcal{W}}_\tau^{a_2}\,*\, \Big(\mathbf{Q}_{\tau,\zeta_1}^{(1)}\,\mathbf{Q}_{\tau,-\zeta_1}^{(1)}\Big)\Big\|_{H^{1-\delta}(\Lambda)}\,,
\end{equation*}
which by translation invariance of Sobolev norms, the choice of $a_1 \equiv a_0$, \eqref{Induction_Base_Case_2B_bold_term1B'}, Lemma \ref{Sobolev_product_estimate} with $s=1$ and $\alpha=\frac{\delta}{2}$ and Lemma \ref{Q_bounds_endpoint} (ii) is
\begin{equation}
\label{I'_endpoint_bound2}
\;\lesssim\; (1+L) \|w_\tau-w\|_{H^{-1+\delta}(\Lambda)}\,.
\end{equation}
Note that in the proof of \eqref{I'_endpoint_bound2}, we are using Lemma \ref{Sobolev_product_estimate} with with $f=\mathbf{Q}_{\tau,\zeta_1}^{(1)}$ and $g=\mathbf{Q}_{\tau,-\zeta_1}^{(1)}$ and the fact that $\widetilde{\mathcal{W}}_\tau^{a_1}=w_\tau-w$.
\end{itemize}

When $n \geq 3$, our goal is to integrate the vertex $a_0$ first. After integrating this vertex, by Definition \ref{endpoint_admissible_w}, Lemma \ref{w_endpoint_admissible_approximation} (iv), and \eqref{positivity}, we are left with a nonnegative integrand.
With analogous notation for the vertices of the path $\mathfrak{P}$ as in the proof of Proposition \ref{Product of subgraphs_endpoint}, we want to consider the case when $a_2 \equiv a_0$. Note that, by \eqref{cal W_tilde}, this is consistent with \eqref{integration_vertex_inductive_step}, so we can indeed arrange for this to be the first vertex which we integrate.

Let us observe that the only dependence on $y_{a_0} \equiv y_{a_2}$ in $\mathfrak{I}'(\mathfrak{P})$ is given by
\begin{equation}
\label{a_0_vertex_factor1}
\int \dd y_{a_2}\,\big(w_\tau-w\big)(y_{a_2}-y_{a_2^*})\, \mathbf{Q}_{\tau,\zeta_1}^{(1)}(y_{a_2}-y_{a_1})\,\mathbf{Q}_{\tau,\zeta_2}^{(1)}(y_{a_2}-y_{a_3})\,.
\end{equation}
By duality in $y_{a_2}$, the expression in \eqref{a_0_vertex_factor1} is in absolute value
\begin{equation*}
\;\leq\;\big\|\big(w_\tau-w\big)(\cdot-y_{a_2^*})\big\|_{H^{-1+\delta}(\Lambda)} \, \big\| \mathbf{Q}_{\tau,\zeta_1}^{(1)}(\cdot-y_{a_1})\,\mathbf{Q}_{\tau,\zeta_2}^{(1)}(\cdot-y_{a_3})\big\|_{H^{1-\delta}(\Lambda)}\,,
\end{equation*}
which by Lemma \ref{Sobolev_product_estimate}  with $s=1-\delta$ and $\alpha=\frac{\delta}{2}$ is
\begin{align*}
&\;\lesssim\; 
\big\|\big(w_\tau-w\big)(\cdot-y_{a_2^*})\big\|_{H^{-1+\delta}(\Lambda)} \, \big\| \mathbf{Q}_{\tau,\zeta_1}^{(1)}(\cdot-y_{a_1})\big\|_{H^{1-\delta/2}(\Lambda)}\,\big\|\mathbf{Q}_{\tau,\zeta_2}^{(1)}(\cdot-y_{a_3})\big\|_{H^{1-\delta/2}(\Lambda)}
\\
&\;=\;\|w_\tau-w\|_{H^{-1+\delta}(\Lambda)} \, \big\| \mathbf{Q}_{\tau,\zeta_1}^{(1)}\big\|_{H^{1-\delta/2}(\Lambda)}\,\big\|\mathbf{Q}_{\tau,\zeta_2}^{(1)}\big\|_{H^{1-\delta/2}(\Lambda)}\,,
\end{align*}
which by Lemma \ref{Q_bounds_endpoint} (ii), (v) is 
\begin{equation}
\label{I'_endpoint_bound3}
\;\lesssim\; \|w_\tau-w\|_{H^{-1+\delta}(\Lambda)}\,\mathbf{Q}^{(1)}_{\tau,\zeta_1+\zeta_2}(y_{a_1}-y_{a_3})\,.
\end{equation}
Using \eqref{I'_endpoint_bound3}, the observation that all the factors in the integrand of \eqref{I'(P)_endpoint} except $\big(w_\tau-w\big)(y_{a_2}-y_{a_2^*})$ are nonnegative, and the triangle inequality, we obtain that 
\begin{equation}
\label{I'_endpoint_bound4}
|\mathfrak{I}'(\mathfrak{P})|\;\leq\;C_0 \,\|w_\tau-w\|_{H^{-1+\delta}(\Lambda)} \,\mathfrak{I}'(\hat{\mathfrak{P}})\,,
\end{equation}
for $\hat{\mathfrak{P}}$ as defined in \eqref{I(P)_induction_unbounded}.
We estimate $\mathfrak{I}'(\hat{\mathfrak{P}})$ by arguing inductively as in the proof of Proposition 
\ref{Product of subgraphs_endpoint}. In doing so, we keep in mind Remark \ref{Product_of_subgraphs_Remark2}, the fact that $\widetilde{\mathcal{W}}_{\tau}^a$ satisfies \eqref{W_bound_bold1}, \eqref{W_bound_bold4}, and \eqref{W_bound_bold2'}, and the recall the base cases \eqref{Case_3_Induction_Base}, \eqref{I'_endpoint_bound1}, \eqref{I'_endpoint_bound2}, 
Substituting the obtained bound on $\mathfrak{I}'(\hat{\mathfrak{P}})$ into \eqref{I'_endpoint_bound4}, we deduce that 
\begin{equation}
\label{I'_endpoint_bound5}
\|\mathfrak{I}'(\mathfrak{P})\|_{L^{2}_{\mathbf{y_1}} L^{\infty}_{\mathcal{V}^{*}(\mathfrak{P})}}\;\leq\; C_0^{|\mathcal{V}(\mathfrak{P})|}\,\big(1+\|w\|_{L^1(\Lambda)}+L\big)^{|\mathcal{V}_2(\mathfrak{P)}|-1}\,\|w_\tau-w\|_{H^{-1+\delta}(\Lambda)}\,.
\end{equation}
We hence obtain \eqref{a_0_vertex_factor} from \eqref{I'_endpoint_bound5} and Lemma \ref{w_1D_approximation} (iv).

Suppose that now $\mathfrak{P} \in \conn(\mathcal{E})$ is such that $a_0 \notin \mathcal{V}_2(\mathfrak{P})$. Then, arguing as in the proof of Proposition \ref{Product of subgraphs_endpoint}, we get that
\begin{equation}
\label{not_a_0_vertex_factor}
\|\mathfrak{I}'(\mathfrak{P})\|_{L^{2}_{\mathbf{y_1}} L^{\infty}_{\mathcal{V}^{*}(\mathfrak{P})}}\;\leq\; C_0^{|\mathcal{V}(\mathfrak{P})|}\,\big(1+\|w\|_{L^1(\Lambda)}+L\big)^{|\mathcal{V}_2(\mathfrak{P)}|}\,. 
\end{equation}
Putting \eqref{a_0_vertex_factor} and \eqref{not_a_0_vertex_factor} together, using a cycle decomposition and arguing as in the proof of \eqref{Approximation_result_2_unbounded_I1-I2}, we deduce \eqref{Approximation_result_2_unbounded_I1-I2_endpoint}.

We now show that 
\begin{equation}
\label{Approximation_result_2_unbounded_I2-I_endpoint}
\mathcal{I}^{2,\xi}_{\tau,\Pi}(\mathbf{t}) \rightarrow \mathcal{I}^{\xi}_{\Pi} \quad \mbox{as}\quad \tau \rightarrow \infty\quad \mbox{uniformly in} \,\, \xi \in \mathbf{B}_r\,,
\end{equation}
as in \eqref{Approximation_result_2_unbounded_I2-I}. 
This is done by another telescoping argument. Namely, we have \eqref{Telescoping_2_application} (now without absolute values on $w$).

We fix $e_0 \in \mathcal{E}$ and recall the definition of $\tilde{\mathcal{J}}_{\tau,e}(\mathbf{y}_e,\mathbf{s})$ given by \eqref{J_tilde}. As in the proof of Proposition \ref{Approximation_result_2_unbounded}, given $a \in \mathcal{V}_2$, we define $\mathcal{W}^a$ by \eqref{cal W_unbounded} with $w_\tau$ replaced by $w$. Note that $\mathcal{W}a$ is then pointwise nonnegative and it satisfies the bounds in \eqref{W_bound_bold1}, \eqref{W_bound_bold2}, and \eqref{W_bound_bold4}.
With $\mathfrak{I}''(\mathfrak {P})$ as given by \eqref{I''_definition}, we have
\begin{equation*}
\mathfrak{I}''(\mathfrak {P}) \;=\; \int \dd \mathbf{y}_{\mathcal{V}_2(\mathfrak{P})} \, \prod_{e \in \mathfrak{P}} \tilde{\mathcal{J}}_{\tau, e}(\mathbf{y}_e, \mathbf{s})\, \prod_{a \in \mathcal{V}_2(\mathfrak{P})} \mathcal{W}^{a}(y_{a}-y_{a^*})\,.
\end{equation*}

We first consider $\mathfrak{P} \in \mathrm{conn} (\mathcal{E})$ which is the connected component with edge $e_0$. We want to show that 
\begin{equation}
\label{e_0_edge_factor}
\|\mathfrak{I}''(\mathfrak{P})\|_{L^{2}_{\mathbf{y_1}} L^{\infty}_{\mathcal{V}^{*}(\mathfrak{P})}} \rightarrow 0 \quad \mbox{as}\quad \tau \rightarrow \infty\,.
\end{equation}
In order to prove \eqref{e_0_edge_factor}, we use the fact that for all  $t \in (-1,1)$ and $s \in (0,1)$, we have
\begin{equation}
\label{J_tilde_e0_term}
\big\||\mathbf{Q}^{(1)}_{\tau,t}-\mathbf{G}|\big\|_{H^s(\Lambda)} \rightarrow 0 \quad \mbox{as}\quad \tau \rightarrow \infty\,.
\end{equation}
Note that  \eqref{J_tilde_e0_term} follows Lemma \ref{Q1_convergence_lemma_endpoint} and the observation that for $s \in (0,1)$ we have
\begin{equation}
\label{absolute_value_Sobolev_bounded}
\| |f| \|_{H^s(\Lambda)} \;\lesssim_s\; \|f\|_{H^s(\Lambda)}\,,
\end{equation}
In order to prove \eqref{absolute_value_Sobolev_bounded}, we use the \emph{Sobolev-Slobodeckij (physical space)} characterisation of homogeneous Sobolev spaces. More precisely, if $[x]$ denotes the unique element of the set $(x + \mathbb{Z}^2) \cap [-1/2,1/2)^2$, we have
\begin{equation}
\label{Sobolev_Slobodeckij_characterisation}
\|f\|_{\dot{H}^s(\Lambda)}^2 \;\sim_s\; \int \dd x \, \int \dd y\, \frac{|f(x)-f(y)|^2}{[x-y]^{n+2s}}
\end{equation}
For a self-contained proof of \eqref{Sobolev_Slobodeckij_characterisation} in the periodic setting, we refer the reader to \cite[Proposition 1.3]{Benyi_Oh}.
We deduce \eqref{J_tilde_e0_term} by using \eqref{Sobolev_Slobodeckij_characterisation} and the triangle inequality. 
We now deduce \eqref{e_0_edge_factor} by applying \eqref{J_tilde}, \eqref{J_tilde_e0_term}, Lemma \ref{Q_bounds_endpoint} (ii), the fact that $\|\mathbf{G}\|_{H^s(\Lambda)} \leq C(s)$ for all $s \in (0,1)$, as well as \eqref{Product_of_subgraphs_Remark_n=2a_endpoint}--\eqref{Product_of_subgraphs_Remark_n>=3_endpoint} and Remark \ref{Product_of_subgraphs_Remark2} with appropriate modifications to the context of $\mathfrak{I}''$.

Furthermore, if $\mathfrak{P} \in \mathrm{conn} (\mathcal{E})$ does not contain $e_0$ as an edge, the same arguments as above (as in the proof of Proposition \ref{Product of subgraphs_endpoint}) show that we have
\begin{equation}
\label{not_e_0_edge_factor}
\|\mathfrak{I}''(\mathfrak{P})\|_{L^{2}_{\mathbf{y_1}} L^{\infty}_{\mathcal{V}^{*}(\mathfrak{P})}}\;\leq\; C_0^{|\mathcal{V}(\mathfrak{P})|}\,\big(1+\|w\|_{L^1(\Lambda)}+L\big)^{|\mathcal{V}_2(\mathfrak{P)}|}\,. 
\end{equation}
Here we do not take absolute values on any of the factors $\tilde{\mathcal{J}}_{\tau,e}(\mathbf{y}_e,\mathbf{s})$ by \eqref{J_tilde}.
Using a cycle decomposition and applying \eqref{e_0_edge_factor} and \eqref{not_e_0_edge_factor}, we deduce \eqref{Approximation_result_2_unbounded_I2-I_endpoint}. The claim now follows from \eqref{Approximation_result_2_unbounded_I1-I2_endpoint} and \eqref{Approximation_result_2_unbounded_I2-I_endpoint}.
\end{proof}

\begin{remark}
\label{absolute_value_unbounded_negative_Sobolev}
We note that \eqref{absolute_value_Sobolev_bounded} does not hold for negative $s$. This is the reason why we assume that $w \geq 0$ and $w_\tau \geq 0$ pointwise.
See \eqref{I'_endpoint_bound4} and the part of the proof of Proposition \ref{Product of subgraphs_endpoint} that follows.
%See Definition \ref{endpoint_admissible_w} (iii), Lemma \ref{w_endpoint_admissible_approximation} (iv), and the proofs of Proposition \ref{Product of subgraphs_endpoint}, Lemma \ref{Approximation_result_Q1_endpoint}, and Proposition \ref{Approximation_result_2_unbounded_endpoint}, where we estimate the interactions in $H^{-1+\delta}(\Lambda)$.
In order to see that \eqref{absolute_value_Sobolev_bounded} does not hold for negative $s$, we give a counterexample\footnote{This counterexample was shown to the author by Sebastian Herr.}. Consider the sequence of functions $(f_n)$ given by $f_n(x)= \langle n \rangle^{-s} \,\ee^{2\pi \ii n \cdot x}$. Then $\|f_n\|_{H^s(\Lambda)}$ is bounded uniformly in $n$, but
$\| |f_n| \|_{H^s(\Lambda)} \sim \langle n \rangle^{-s}$ diverges as $n \rightarrow \infty$.
\end{remark}

Given $m,r \in \mathbb{N}$ and $\xi \in \mathbf{B}_r$, we recall $a^{\xi}_{\infty,m}$ given by \eqref{a_infty}. We have an analogue of Proposition \ref{a_convergence}.
\begin{proposition}
\label{a_convergence_endpoint}
Let $m,r \in \mathbb{N}$ be given. We have 
\begin{equation*}
a^{\xi}_{\tau,m} \rightarrow a^{\xi}_{\infty,m} \quad \mbox{as}\quad \tau \rightarrow \infty\quad \mbox{uniformly in} \,\, \xi \in \mathbf{B}_r\,.
\end{equation*}
\end{proposition}
\begin{proof}
The claim follows from  \eqref{Wick_application_identity_unbounded_all_D}, \eqref{a_infty}, Proposition \ref{Approximation_result_2_unbounded_endpoint}, Proposition \ref{Product of subgraphs_endpoint} (i), and the dominated convergence theorem.
\end{proof}

The remainder term \eqref{Remainder_term_R_unbounded} in this context is estimated as in Section \ref{Bounds on the remainder term}. As before, we are taking $\eta \in (0,1/4]$.
We have the following analogue of Proposition \ref{remainder_term}. 
\begin{proposition} 
\label{remainder_term_endpoint}
Let $r,M \in \mathbb{N}, \xi \in \mathbf{B}_r$, and $z \in \mathbb{C}$ with $\re z \geq 0$ be given. Then $R^{\xi}_{\tau,M}(z)$ given by \eqref{Remainder_term_R_unbounded} satisfies 
\begin{equation*}
|R_{\tau,M}^{\xi}(z)|\;\leq\; \bigg(\frac{Cr}{\eta^2}\bigg)^r \, \Bigg[\frac{C\big(1+\|w\|_{L^1(\Lambda)}+L\big)}{\eta^2}\Bigg]^M\,|z|^M\,M!\,.
\end{equation*}
\end{proposition}

\begin{proof}
The proof of Proposition \ref{remainder_term_endpoint} proceeds analogously as the proof of Proposition \ref{remainder_term}. We apply Proposition \ref{Product of subgraphs_endpoint} (iii) instead of Proposition \ref{Product of subgraphs} (iii). Note that by \eqref{W_tau_positivity} and Lemma \ref{w_endpoint_admissible_approximation} (ii), the interaction $W_\tau$ is indeed a positive operator.
\end{proof}
We adapt Lemma \ref{A^xi_{tau,z}_analytic} to this setting.
\begin{lemma} 
\label{A^xi_{tau,z}_analytic_endpoint}
Let $r \in \mathbb{N}$ and $\xi \in \mathbf{B}_r$ be given.
The function $A^{\xi}_{\tau}$ given in \eqref{Duhamel_expansion_unbounded_A}  is analytic in the right-half plane.
\end{lemma}
\begin{proof}
The proof is analogous to that of Lemma \ref{A^xi_{tau,z}_analytic}. The only difference is that the boundedness of $W_{\tau}$ on the range of $P^{(\leq n)}$ now follows from Lemma \ref{w_endpoint_admissible_approximation} (v).
\end{proof}

\subsection{The classical system} 
\label{Endpoint-admissible interaction potentials_3}

We now study the classical system with endpoint-admissible interaction potentials. The analysis is quite similar to that given in Section \ref{Analysis of the classical system}. We hence just emphasise the main differences. 

Given $K \in \mathbb{N}$ and recalling \eqref{G_K}, we define the function $\mathbf{G}_{[K]} : \Lambda \rightarrow \mathbb{C}$ 
by
\begin{equation}
\label{G_K_endpoint}
\mathbf{G}_{[K]}(x) \;\deq\; G_{[K]}(x;0)\,.
\end{equation}
Instead of Lemma \ref{GK_convergence_lemma}, we use a convergence result in Sobolev spaces.

\begin{lemma}
\label{GK_convergence_lemma_endpoint}
Recalling \eqref{G_endpoint}, we have
\begin{equation*}
\|\mathbf{G}_{[K]}-\mathbf{G}\|_{H^{-1+\delta}(\Lambda)} \rightarrow 0 \quad \mbox{as}\quad K \rightarrow \infty\,.
\end{equation*}
\end{lemma}

\begin{proof}
The claim follows by using \eqref{GK_convergence_lemma_proof} with $y=0$, \eqref{G_K_endpoint}, and by arguing as in the proof of Proposition \ref{G_{tau,0}_bound} (i).
\end{proof}

In this section, we use the same graphical setup as in Section \ref{Analysis of the classical system}.
We now give a proof of Lemma \ref{W_Wick-ordered} (ii) which justifies the construction of the Wick-ordered interaction.
\begin{proof}[Proof of Lemma \ref{W_Wick-ordered} (ii)]
The proof is very similar to that of Lemma \ref{W_Wick-ordered} (i). We adopt the same notation as in the proof of Lemma \ref{W_Wick-ordered} (i) and we just emphasise the required modifications in the arguments.
With $\mathbf{G}_{[K]}$ given by \eqref{G_K_endpoint}, \eqref{phi_K2} becomes
\begin{equation*}
\int \dd \mu \, \phi_{[L]}(x)\,\bar{\phi}_{[M]}(y)\;=\; \mathbf{G}_{[L \wedge M ]}(x-y)\,.
\end{equation*}
In particular, we can rewrite \eqref{I_{K,Pi}} as
\begin{equation}
\label{I_{K,Pi}_endpoint} 
\mathcal{I}_{\mathbf{K},\Pi}\;=\;\int_{\mathcal{V}} \dd \mathbf{y}\,\Bigg[\prod_{i=1}^{m}w(y_{i,1}-y_{i,2})\Bigg]\,
\prod_{\{a,b\} \in \mathcal{E}} \mathbf{G}_{[K_{i_a} \wedge K_{i_b}]}(y_a-y_b)\,.
\end{equation}
Likewise, we can rewrite \eqref{I_Pi} as
\begin{equation}
\label{I_Pi_endpoint}
\mathcal{I}_{\Pi}\;=\;\int_{\mathcal{V}} \dd \mathbf{y}\,\Bigg[\prod_{i=1}^{m}w(y_{i,1}-y_{i,2})\Bigg]\,
\prod_{\{a,b\} \in \mathcal{E}} \mathbf{G}(y_a-y_b)\,.
\end{equation}
We recall that $\mathcal{I}_{\Pi}$ is indeed well-defined and finite by Proposition \ref{Product of subgraphs_endpoint} (i) and Proposition \ref{Approximation_result_2_unbounded_endpoint}.
Finally, we note that \eqref{I_K_convergence} holds by applying the telescoping argument from the proof of \eqref{Approximation_result_2_unbounded_I2-I_endpoint} and by using Lemma \ref{GK_convergence_lemma_endpoint}. We now obtain the claim as in Lemma \ref{W_Wick-ordered} (i).
\end{proof}

We recall the perturbative expansion \eqref{A^xi} with explicit terms given by \eqref{a^xi_m} and the definition of $a^{\xi}_{\infty,m}$ given by \eqref{a_infty}. An analogue of Lemma \ref{a^xi_m_identity} holds in this setting.
\begin{lemma} 
\label{a^xi_m_identity_endpoint}
Given $m,r \in \mathbb{N}$ and $\xi \in \mathbf{B}_r$, we have $a^{\xi}_m=a^{\xi}_{\infty,m}$.
\end{lemma}
\begin{proof}
The proof is very similar to that of Lemma \ref{a^xi_m_identity}. By using Lemma \ref{W_Wick-ordered} (ii), we have that \eqref{a^xi_m2} holds with $\mathcal{I}^{\xi}_{[K],\Pi}$ given by \eqref{a^xi_m_identity_proof_1}, where, for $e \in \mathcal{E} \equiv \mathcal{E}_{\Pi}$, we rewrite \eqref{a^xi_m_identity_proof_2} as
\begin{equation*}
\mathcal{J}_{[K],e}(\mathbf{y}_e) \;=\;
\begin{cases}
\mathbf{G}(y_{a}-y_{b}) & \mbox{if } e=\{a,b\}\,\,\mbox{for}\,\,a,b \in \mathcal{V}_1
\\
\mathbf{G}_{[K]}(y_{a}-y_{b}) & \mbox{if } e=\{a,b\}\,\,\mbox{for}\,\,a \in \mathcal{V}_2\,\,\mbox{or}\,\,a \in \mathcal{V}_2\,.
\end{cases}
\end{equation*}
We now apply a telescoping argument as in the proof of \eqref{Approximation_result_2_unbounded_I2-I_endpoint} and use Lemma \ref{GK_convergence_lemma_endpoint} to deduce that \eqref{I^xi_K_convergence} holds. The claim now follows as in Lemma \ref{a^xi_m_identity}.
\end{proof}

We can now obtain upper bounds on the explicit terms in the expansion \eqref{A^xi} analogously as in Corollary \eqref{a^xi_m_corollary}.
\begin{corollary} 
\label{a^xi_m_corollary_endpoint}
Given $m,r \in \mathbb{N}$ and $\xi \in \mathbf{B}_r$, we have 
\begin{equation*}
|a^{\xi}_{m}|\;\leq\;(Cr)^{r} \,C^m  \,(1+\|w\|_{L^1(\Lambda)}+L)^{m}\,m!\,.
\end{equation*}
\end{corollary}

\begin{proof}
The claim follows from Proposition \ref{Product of subgraphs_endpoint} (ii), Proposition \ref{a_convergence_endpoint}, and Lemma \ref{a^xi_m_identity_endpoint}.
\end{proof}

Furthermore, we can estimate the remainder term \eqref{R^xi_m} analogously as in Proposition \ref{R^xi_M_bound_proposition}.
\begin{proposition} 
\label{R^xi_M_bound_proposition_endpoint}
Given $M,r \in \mathbb{N}$, $\xi \in \mathbf{B}_r$, and $z \in \mathbb{C}$ with $\re z \geq 0$, we have
\begin{equation*}
|R^{\xi}_{M}(z)|\;\leq\;(Cr)^r \, C^M\,(1+\|w\|_{L^1(\Lambda)}+L)^M\,|z|^M\,M!\,.
\end{equation*}
\end{proposition}
\begin{proof}
We argue as in the proof of Proposition \ref{R^xi_M_bound_proposition}.
Note that \eqref{R^xi_M_bound_proposition_1} still holds, where we now use Definition \ref{endpoint_admissible_w} (ii). We estimate the third term in \eqref{R^xi_M_bound} by using Corollary \ref{a^xi_m_corollary_endpoint} with $m=2M$ and $r=0$.
Hence the upper bound in \eqref{R^xi_M_bound_proposition_2} gets replaced by
\begin{equation*}
\;\leq\; (Cr)^r \, C^M\,(1+\|w\|_{L^1(\Lambda)}+L)^M\,|z|^M\,M!
\end{equation*}
which implies the claim.
\end{proof}
Finally, we note that the analyticity result given by Lemma \ref{A^xi_analytic} holds in this setting.

\begin{lemma} 
\label{A^xi_analytic_endpoint}
Let $r \in \mathbb{N}$ and $\xi \in \mathbf{B}_r$ be given.
The function $A^{\xi}$ given by \eqref{A^xi} is analytic in the right-half plane. 
\end{lemma}
\begin{proof}
We argue as in the proof of Lemma \ref{A^xi_analytic}. Taking derivatives in $z$ by differentiating under the integral sign in \eqref{rho_hat_unbounded} with $X=\Theta(\xi)$ is now justified by using the arguments from the proof of Proposition \ref{R^xi_M_bound_proposition_endpoint}.
\end{proof}

\subsection{Proof of Theorem \ref{endpoint_admissible_result}}
\label{Endpoint-admissible interaction potentials_4}
We now have all the necessary ingredients to prove Theorem \ref{endpoint_admissible_result}.

\begin{proof}[Proof of Theorem \ref{endpoint_admissible_result}]
We argue as in the proof of Theorem \ref{main_result} (ii) and reduce the claim to showing \eqref{main_result_proof_2}. In order to prove \eqref{main_result_proof_2}, we again work with the functions $A^{\xi}_{\tau}$ and $A^{\xi}$ defined in \eqref{Duhamel_expansion_unbounded_A} and \eqref{A^xi} respectively and apply Proposition \ref{Borel_summation_unbounded}. 
By Lemma \ref{A^xi_{tau,z}_analytic_endpoint} and Lemma \ref{A^xi_analytic_endpoint}, these functions are analytic in the right-half plane.

In $\mathcal{C}_2$, we perform the expansion \eqref{Borel_summation_unbounded1} according to \eqref{Duhamel_expansion_unbounded_A} and \eqref{A^xi}.
By Proposition \ref{Product of subgraphs_endpoint} (ii), Corollary \ref{a^xi_m_corollary_endpoint}, Proposition \ref{remainder_term_endpoint}, and Proposition \ref{R^xi_M_bound_proposition_endpoint}, we know that \eqref{Borel_summation_unbounded2}--\eqref{Borel_summation_unbounded3} hold with $\nu$ as in \eqref{nu_sigma_choice} and with 
\begin{equation*}
\sigma\;=\;\frac{C(1+\|w\|_{L^1(\Lambda)}+L)}{\eta^2}\,.
\end{equation*}
We now deduce the claim as in the proof of Theorem \ref{main_result} (ii).
\end{proof}

\appendix

\section{Proofs of auxiliary results for the interaction potential}
\label{The_interaction_potential}
In this appendix, we give the proofs of Lemma \ref{admissible_w_nonempty}, Lemma \ref{w_positive_approximation}, and  Lemma \ref{w_endpoint_admissible_approximation} concerning the interaction potential.
In order to prove all of these results, we refer to the \emph{transference principle}, which allows us to analyse the mapping properties of Fourier multiplier operators in the periodic setting. Let us first recall the precise statement of this principle.

\begin{proposition}(The transference principle \cite[Theorem VII.3.8]{Stein_Weiss})
\label{transference_principle}
\\
Let $1 \leq p \leq \infty$ be fixed.
Let $T: L^p(\mathbb{R}^d) \rightarrow  L^p(\mathbb{R}^d)$ be
a bounded operator satisfying the following properties.
\begin{itemize}
\item[(i)] $T$ is a Fourier multiplier operator,  i.e.\  there exists $u \in \mathcal{S}'(\mathbb{R}^d)$ such that 
$Tf=u * f$ for all $f \in \mathcal{S}(\mathbb{R}^d)$.
\item[(ii)] The convolution kernel $u$ has the property that its Fourier transform $\hat{u} \in \mathcal{S}'(\mathbb{R}^d)$ is continuous at each point of $\mathbb{Z}^d$.
\end{itemize}
Then the \textbf{periodisation} of $T$
\begin{equation}
\label{T_tilde}
\tilde{T} f (x) \;\deq\; \sum_{k \in \mathbb{Z}^d} \hat{u}(k) \hat{f}(k) \,\ee^{2 \pi \ii k \cdot x}
\end{equation}
defines a bounded operator $\tilde{T}:  L^p(\Lambda) \rightarrow  L^p(\Lambda)$ with 
\begin{equation}
\label{T_tilde_T}
\|\tilde{T}\|_{L^p(\Lambda) \rightarrow  L^p(\Lambda)} \;\leq\;\|T\|_{L^p(\mathbb{R}^d) \rightarrow L^p(\mathbb{R}^d)}\,.
\end{equation}
\end{proposition}
\begin{remark}
\label{transference_principle_remark}
For $T$ as in the assumptions of Proposition \ref{transference_principle}, we have 
\begin{equation*}
(T f)\,\,\widehat{}\,\,(\xi) \;=\; \hat{u}(\xi)\,\hat{f}(\xi) \in \mathcal{S}'(\mathbb{R}^d) \quad \mbox{for all} \quad f \in L^p(\mathbb{R}^d)\,,\,\xi\in \mathbb{R}^d\,,
\end{equation*}
and the periodisation $\tilde{T}$ satisfies
\begin{equation*}
(\tilde{T} f)\,\,\widehat{}\,\,(k) \;=\; \hat{u}(k)\,\hat{f} (k) \quad \mbox{for all} \quad f \in L^p(\Lambda)\,, \,k \in \mathbb{Z}^d\,.
\end{equation*}
\end{remark}

We now have the necessary tools to prove Lemma \ref{admissible_w_nonempty}.

\begin{proof}[Proof of Lemma \ref{admissible_w_nonempty}]
We first prove (i). Note that the claim is immediate when $d=1$, so we need to consider the case when $d=2,3$.
Let us recall that given $p \in \mathcal{P}_d$ as in \eqref{P_d}, we need to find $w \in L^p(\Lambda) \setminus L^{\infty}(\Lambda)$ which is even and of positive type. We need to consider two cases depending on the size of $p$.
\begin{itemize}
\item[(a)] $2 \leq p<\infty$.

We take $q \in (1,p')$ and let $w : \Lambda \rightarrow \mathbb{C}$ be such that 
\begin{equation}
\label{admissible_w_nonempty_1}
\hat{w}(k) \;\deq\; \frac{1}{\langle k \rangle^{\frac{d}{q}}}\,\,\, \mbox{for} \,\, k \in \mathbb{Z}^d\,.
\end{equation}
From \eqref{admissible_w_nonempty_1}, we obtain that $\hat{w} \in \ell^{p'}(\mathbb{Z}^d)$. By the Hausdorff-Young inequality, it follows that $w \in L^p(\Lambda)$. Furthermore, $w$ is even and of positive type. 

We just need to show that $w \notin L^{\infty}(\Lambda)$.  In order to do this, we argue by contradiction.
Given $\delta>0$, we consider the Fourier multiplier $T_\delta$ on $\mathcal{S}(\mathbb{R}^d)$ given by 
\begin{equation}
\label{T_delta}
(T_\delta f)\,\,\widehat{}\,\,(\xi) \;=\;\ee^{- \delta \pi |\xi|^2} \, \hat{f}(\xi)\quad \mbox{for all} \quad f \in \mathcal{S}(\mathbb{R}^d), \,\xi \in \mathbb{R}^d\,.
\end{equation}
Note that \eqref{T_delta} can be rewritten as
\begin{equation*}
T_\delta f \;=\; \frac{1}{\delta^{d/2}} \, \ee^{-\pi |x|^2 /\delta} *f
\end{equation*}
and hence by Young's inequality we have
\begin{equation}
\label{T_delta2}
\|T_\delta f\|_{L^{\infty}(\mathbb{R}^d)} \;\leq\; \|f\|_{L^{\infty}(\mathbb{R}^d)}\,.
\end{equation}
In particular, $T_\delta$ extends to an operator on $L^{\infty}(\mathbb{R}^d)$.
We let $S_\delta$ be the periodisation of $T_\delta$, in other words,  
\begin{equation}
\label{S_delta}
S_\delta f(x)\;=\;\sum_{k \in \mathbb{Z}^d} \ee^{-\delta \pi |k|^2}\,\hat{f}(k)\,\ee^{2\pi \ii k \cdot x}\,,
\end{equation}
for $f \in L^{\infty}(\Lambda)$.
Using \eqref{T_delta2} and Proposition \ref{transference_principle}, we deduce that
\begin{equation}
\label{S_delta_bound}
\|S_\delta f\|_{L^{\infty}(\Lambda)} \;\leq\; \|f\|_{L^{\infty}(\Lambda)} \quad \mbox{for all} \quad f \in L^{\infty}(\Lambda)\,.
\end{equation}
In particular, \eqref{S_delta_bound} implies that $\|S_\delta w\|_{L^{\infty}(\Lambda)}$ is bounded uniformly in $\delta$.

We now substitute $f=w$ in \eqref{S_delta}. Using \eqref{admissible_w_nonempty_1}, \eqref{S_delta}, and the dominated convergence theorem it follows that $S_\delta w$ is a continuous function on $\Lambda$. Furthermore 
\begin{equation}
\label{S_delta_2}
S_\delta w(0) \;=\; \sum_{k \in \mathbb{Z}^d} \frac{\ee^{-\delta \pi |k|^2}}{\langle k \rangle^{\frac{d}{q}}} \rightarrow \infty \quad \mbox{as } \quad \delta \rightarrow 0\,,
\end{equation}
by using the monotone convergence theorem.
Hence, \eqref{S_delta_2} implies that $\|S_\delta w\|_{L^{\infty}(\Lambda)} \rightarrow \infty$ as $\delta \rightarrow 0$, thus giving a contradiction. We deduce that $w \notin L^{\infty}(\Lambda)$.
\item[(b)] $1 \leq p <2$.

Here, we recall that functions on $\mathbb{T}^d$ can be identified with $\mathbb{Z}^d$-periodic functions on $\mathbb{R}^d$. We use this identification throughout. Let us consider a radial function $\chi \in C_c^\infty(\mathbb{R}^d)$ which satisfies the following properties.
\begin{itemize}
\item[(1)] $0 \leq \chi \leq 1$.
\item[(2)] $\chi(\xi)=1$ for $|\xi| \leq \frac{1}{3}$.
\item[(3)] $\chi(\xi)=0$ for $|\xi|>\frac{1}{2}$.
\end{itemize}
We take $q \in (p,2)$ and let 
\begin{equation}
\label{admissible_w_nonempty_2}
f(x) \;\deq\; \frac{1}{|x|^{\frac{d}{q}}}\,\chi(x) \quad \mbox{for} \quad x \in [-1/2,1/2]^d\,.
\end{equation}
We extend $f$ by periodicity to a function on $\mathbb{R}^d$ (and we interpret the result as being a function on $\mathbb{T}^d$). Note that the resulting function $f$ then belongs to $L^p(\Lambda)$ and is even. Therefore, its Fourier coefficients are real-valued.
We let $w \deq f *_\Lambda f = \int_\Lambda \dd y\, f(\cdot-y)\, f(y)$. Then $w \in L^p(\Lambda)$ by Young's inequality. Furthermore, $w$ is even and $\hat{w} = (\hat{f})^2 \geq 0$, so $w$ is of positive type.

We need to show that $w \notin L^{\infty}(\Lambda)$. In order to show this, we use \eqref{admissible_w_nonempty_2} and the support properties of $\chi$ to deduce that for $x \in \Lambda$ with $|x| \leq 1/8$, we have
\begin{equation}
\label{admissible_w_nonempty_3}
w(x)\;=\; f*_\Lambda f (x) \;\geq\; \mathop{\int}_{|y| \leq |x|/2} \frac{1}{|x-y|^{\frac{d}{q}}}\,\frac{1}{|y|^{\frac{d}{q}}}\,\chi(x-y)\,\chi(y)\,\dd y\,.
\end{equation}
Note that in the integrand in \eqref{admissible_w_nonempty_3}, we have $\chi(x-y)=\chi(y)=1$. Furthermore, $|x-y|\sim|x|$. Hence the expression in \eqref{admissible_w_nonempty_3} is 
\begin{equation}
\label{admissible_w_nonempty_4}
\gtrsim \frac{1}{|x|^{\frac{d}{q}}}\,\mathop{\int} _{|y|\leq|x|/2} \frac{1}{|y|^{\frac{d}{q}}}\,\dd y \sim |x|^{\,d(1-\frac{2}{q})}\,.
\end{equation}
Since $q<2$, \eqref{admissible_w_nonempty_4} implies that $w \notin L^{\infty}(\Lambda)$. This finishes the proof of (i).
\end{itemize}
We now prove (ii). For $\epsilon>0$, we consider $f$
\begin{equation}
\label{admissible_w_nonempty_5}
\hat{f}(k) \;\deq\; \frac{1}{\langle k \rangle^{1+\epsilon}} \quad \mbox{for} \quad k \in \mathbb{Z}^2\,.
\end{equation}
Note that then $f \in L^2(\Lambda)$ is even and real-valued.
We let $w \deq f^2$. Then $w \in L^1(\Lambda)$ is even, and $w \geq 0$ pointwise. Furthermore, by \eqref{admissible_w_nonempty_5}, we have that for all $k \in \mathbb{Z}^2$
\begin{equation}
\label{admissible_w_nonempty_6}
\hat{w}(k)\;=\;\sum_{k' \in \mathbb{Z}^2} \hat{f}(k')\,\hat{f}(k-k')\;\geq\;0\,.
\end{equation}
It remains to check that $w$ satisfies conditions (iv) and (v) of Definition \ref{endpoint_admissible_w}. Let $k \in \mathbb{Z}^2$ be given. We consider the contributions when $|k'| < \frac{|k|}{2}$, $\frac{|k|}{2} \leq |k'| \leq 2|k|$ and $|k'| >2|k|$ in \eqref{admissible_w_nonempty_6} respectively.

By \eqref{admissible_w_nonempty_5}, we have 
\begin{equation*}
\sum_{|k'|<\frac{|k|}{2}}\hat{f}(k')\,\hat{f}(k-k') \sim \Bigg(\sum_{|k'|<\frac{|k|}{2}} \hat{f}(k')\Bigg)\,\hat{f}(k)\,,
\end{equation*}
which, by applying the Cauchy-Schwarz inequality and \eqref{admissible_w_nonempty_5} once more is
\begin{equation}
\label{admissible_w_nonempty_7A}
\;\lesssim\; \langle k \rangle \,\Bigg(\sum_{|k'|<\frac{|k|}{2}} |\hat{f}(k')|^2\Bigg)^{1/2}\,\hat{f}(k)\;\leq\;\langle k \rangle \,\|f\|_{L^2(\Lambda)}\,\hat{f}(k)\;\lesssim\; \langle k \rangle^{-\epsilon}\,.
\end{equation}
Furthermore, by \eqref{admissible_w_nonempty_5}, we have
\begin{equation*}
\sum_{\frac{|k|}{2} \leq |k'| \leq 2|k|}\hat{f}(k')\,\hat{f}(k-k') \sim \Bigg(\sum_{\frac{|k|}{2} \leq |k'| \leq 2|k|} \hat{f}(k-k')\Bigg)\,\hat{f}(k)\,,
\end{equation*}
which by the Cauchy-Schwarz inequality and \eqref{admissible_w_nonempty_5} is
\begin{equation}
\label{admissible_w_nonempty_7B}
\;\lesssim\; \langle k \rangle \,\|f\|_{L^2(\Lambda)}\,\hat{f}(k)\;\lesssim\; \langle k \rangle^{-\epsilon}\,.
\end{equation}
Finally, by  \eqref{admissible_w_nonempty_5}, we have
\begin{equation}
\label{admissible_w_nonempty_7C}
\sum_{|k'| > 2|k|}\hat{f}(k')\,\hat{f}(k-k') \sim \sum_{|k'| > 2|k|}\big(\hat{f}(k')\big)^2 \sim \sum_{|k'| > 2|k|} \langle k' \rangle^{-2-2\epsilon} \sim \langle k \rangle^{-2\epsilon}\,.
\end{equation}
From \eqref{admissible_w_nonempty_7A}--\eqref{admissible_w_nonempty_7C}, we deduce that 
\begin{equation}
\label{admissible_w_nonempty_7}
\langle k \rangle^{-2\epsilon} \lesssim \hat{w}(k) \lesssim \langle k \rangle^{-\epsilon}\,.
\end{equation}
The upper bound in \eqref{admissible_w_nonempty_7} implies that $w$ satisfies condition (iv) of Definition \ref{endpoint_admissible_w}. The lower bound in \eqref{admissible_w_nonempty_7} implies that $w$ satisfies condition (v) of Definition \ref{endpoint_admissible_w}, provided that we choose $\epsilon \leq 1$. Namely, then $w \notin L^2(\Lambda)$ by Plancherel's theorem and hence $w \notin L^{\infty}(\Lambda)$. This finishes the proof of (ii).
\end{proof}

We now prove Lemma \ref{w_positive_approximation}.

\begin{proof}[Proof of Lemma \ref{w_positive_approximation}]
Let us consider $\chi \in C_c^\infty(\mathbb{R}^d)$ satisfying the following properties.
\begin{itemize}
\item[(1)] $0 \leq \chi \leq 1$.
\item[(2)] $\chi(\xi)=1$ for $|\xi| \leq \frac{1}{2}$.
\item[(3)] $\chi(\xi)=0$ for $|\xi|>1$.
\end{itemize}
Given $M>0$, we consider the Fourier multiplier $T_M$ on $\mathcal{S}(\mathbb{R}^d)$ such that
\begin{equation}
\label{T_M1}
(T_M f)\,\,\widehat{}\,\,(\xi) \;=\; \chi(\xi/M) \,\hat{f}(\xi)\quad \mbox{for all} \quad f \in \mathcal{S}(\mathbb{R}^d), \,\xi \in \mathbb{R}^d\,.
\end{equation}
In particular, we can rewrite \eqref{T_M1} as
\begin{equation}
\label{T_M2}
T_M f \;=\; M^d \,\check{\chi}(M \,\cdot) * f\,,
\end{equation}
where $\check{\chi}$ denotes the inverse Fourier transform of $\chi$. By Young's inequality, \eqref{T_M2} implies that $T$ extends to $L^p(\mathbb{R}^d)$ and satisfies
\begin{equation}
\label{T_M3}
\|T_M f\|_{L^p(\mathbb{R}^d)} \;\leq\; C(\chi) \|f\|_{L^p(\mathbb{R}^d)} \quad \mbox{for all} \quad f \in L^p(\mathbb{R}^d)\,.
\end{equation}
Let $S_M$ denote the periodisation of $T_M$, i.e.\ 
\begin{equation}
\label{S_M1}
S_M f(x)\;=\;\sum_{k \in \mathbb{Z}^d} \chi(k/M)\,\hat{f}(k)\,\ee^{2\pi \ii k \cdot x}
\end{equation}
\label{S_M_bound}
Since $\chi \in C_c^{\infty}(\mathbb{R}^d)$, by Proposition \ref{transference_principle} and \eqref{T_M3}, we have
\begin{equation}
\label{S_M_bound}
\|S_M f\|_{L^p(\Lambda)} \;\leq\; C(p,\chi) \|f\|_{L^p(\Lambda)} \quad \mbox{for all} \quad f \in L^p(\Lambda)\,.
\end{equation}
Given $\tau \geq 1$, we define 
\begin{equation}
\label{w_tau_definition}
w_\tau \;\deq\; S_{M(\tau)}w\,.
\end{equation}
We now determine $M(\tau)>0$ so that $w_\tau$ defined in \eqref{w_tau_definition} satisfies the wanted properties (i)--(iii) stated in Lemma \ref{w_positive_approximation}.

Since $w$ is of positive type, we obtain that $w_\tau$ is also of positive type by using \eqref{S_M1}, \eqref{w_tau_definition}, and property (1) of $\chi$. Therefore $w_\tau$ satisfies property (i). By applying \eqref{S_M1} and properties (1) and (3) of $\chi$, it follows that $w_\tau \in L^{\infty}(\Lambda)$ and
\begin{equation*}
\|w_\tau\|_{L^{\infty}(\Lambda)} \;\leq\; \mathop{\sum_{k \in \mathbb{Z}^d}}_{|k| \leq M(\tau)} |\hat{w}(k)| \;\leq\; C M(\tau)^d \|w\|_{L^1(\Lambda)} \;\leq\; C M(\tau)^d \|w\|_{L^p(\Lambda)}\,. 
\end{equation*}
Hence, $w_\tau$ satisfies condition (ii) if we take 
\begin{equation}
\label{M_tau_choice}
M(\tau) \sim \bigg(\frac{\tau^{\beta}}{\|w\|_{L^p(\Lambda)}}\bigg)^{1/d}\,.
\end{equation}
By \eqref{S_M_bound} and \eqref{w_tau_definition}, it follows that $w_\tau$ satisfies condition (iii).
In order to prove that $w_\tau$ satisfies condition (iv), we recall that trigonometric polynomials are dense in $L^p(\Lambda)$ \cite[Corollary 1.1]{Duoandikoetxea}.  %(note that this is the only place where we are using the assumption that $p<\infty$). 
We now deduce the claim by arguing analogously as in the proof of \cite[Lemma 1.8]{Duoandikoetxea}. Namely, given $\epsilon>0$, we can find a trigonometric polynomial $g$ such that 
\begin{equation}
\label{trigonometric_polynomial1}
\|w-g\|_{L^p(\Lambda)}<\epsilon\,.
\end{equation}
By \eqref{S_M1}, \eqref{M_tau_choice}, and property (2) of $\chi$, it follows that there exists $\tau_0 \equiv \tau_0(g) \geq 1$ such that 
\begin{equation}
\label{trigonometric_polynomial2}
S_{M(\tau)}g\;=\;g \,\,\, \mbox{for} \,\, \tau\;\geq\;\tau_0\,.
\end{equation}
In particular, for $\tau \geq \tau_0$, we have
\begin{equation*}
\|w_\tau-w\|_{L^p(\Lambda)} \;\leq\; \|S_{M(\tau)}w-S_{M(\tau)}g\|_{L^p(\Lambda)}+ \|S_{M(\tau)}g-g\|_{L^p(\Lambda)}+ \|g-w\|_{L^p(\Lambda)}\;<\;(C(p,\chi)+1) \epsilon\,.
\end{equation*}
Here, we used  \eqref{S_M_bound}, \eqref{w_tau_definition}, \eqref{trigonometric_polynomial1}, and \eqref{trigonometric_polynomial2}. Therefore, $w_\tau$ satisfies condition (iii).
\end{proof}

We now prove Lemma \ref{w_endpoint_admissible_approximation}.

\begin{proof}[Proof of Lemma \ref{w_endpoint_admissible_approximation}]
We note that, by Definition \ref{endpoint_admissible_w} (iv) and by the construction of $\delta$, we indeed have that 
\begin{equation}
\label{w_H^{-1+delta}}
\|w\|_{H^{-1+\delta}(\Lambda)} \leq C L\,,
\end{equation}
for some $C \equiv C(\epsilon)>0$.
Let $g: \mathbb{R}^2 \rightarrow \mathbb{R}$ be given by $g(x)=\ee^{-\pi |x|^2}$. We define $g_{\tau}: \mathbb{R}^2 \rightarrow \mathbb{R}$ by
\begin{equation}
\label{endpoint_admissible_g_tau}
g_\tau (x) \;\deq\; g\bigg(\frac{x}{\tau^{\beta/2}}\bigg)\,.
\end{equation}
With $g_\tau$ as in \eqref{endpoint_admissible_g_tau}, we define $w_\tau$ in terms of its Fourier coefficients by
\begin{equation}
\label{endpoint_admissible_w_tau1}
\hat{w}_\tau (k) \;\deq\; g_{\tau}(k) \, \hat{w}(k)\,,
\end{equation}
for $k \in \mathbb{Z}^2$. Since by \eqref{endpoint_admissible_g_tau} we have $0 \leq g_\tau \leq 1$, 
we deduce claims (i), (ii), (iii) by applying \eqref{endpoint_admissible_w_tau1} and recalling \eqref{w_H^{-1+delta}}, Definition \ref{endpoint_admissible_w} (ii), and Definition \ref{endpoint_admissible_w} (iv) respectively.

In order to show (iv), we use \eqref{endpoint_admissible_w_tau1} to deduce that we can write 
\begin{equation}
\label{endpoint_admissible_w_tau2}
w_\tau \;=\;f_\tau *_{\Lambda} w \;=\; \int_{\Lambda} \,\dd y \, f_\tau(x-y)\,w(y) \,,
\end{equation}
where for $x \in \mathbb{R}^2$, we let
\begin{equation}
\label{f_tau}
f_\tau(x)\;\deq\; \sum_{k \in \mathbb{Z}^2} g_\tau(k) \,\ee^{2\pi \ii k \cdot x}\;=\; \sum_{k \in \mathbb{Z}^2} \check{g}_\tau (x+k) \;\geq\;0\,.
\end{equation}
In \eqref{f_tau}, we used the Poisson summation formula and let $\check{g}_\tau (y) \deq \tau^{\beta} g(\tau^{\beta/2}y) \geq 0$.
We hence deduce (iv) from \eqref{endpoint_admissible_w_tau2}, \eqref{f_tau} and Definition \ref{endpoint_admissible_w} (iii).

In order to show (v), we use \eqref{endpoint_admissible_g_tau}--\eqref{endpoint_admissible_w_tau1} and Definition \ref{endpoint_admissible_w} (ii) to deduce that 
\begin{equation*}
\|w_\tau\|_{L^{\infty}(\Lambda)} \;\leq\; \sum_{k \in \mathbb{Z}^2} g\bigg(\frac{k}{\tau^{\beta/2}}\bigg) \,\hat{w}(k)\,,
\end{equation*}
which by using $g(x) \leq \frac{C}{\langle x \rangle^{2}}$ and Definition \ref{endpoint_admissible_w} (iv) is 
\begin{equation*}
\;\leq\; C \sum_{k \in \mathbb{Z}^2} \frac{\tau^{\beta}}{\langle k \rangle^{2}}\,\frac{1}{\langle k \rangle^{\epsilon}} \;\leq\; C \tau^{\beta}\,.
\end{equation*}
The claim (v) now follows. In order to prove (vi), we use Proposition \ref{transference_principle} and argue as in the proof of \eqref{S_delta_bound} with $p=1$ instead of $p=\infty$ and the claim follows.

Finally, by \eqref{endpoint_admissible_w_tau1} we note that 
\begin{equation*}
\|w_\tau-w\|_{H^{-1+\delta}(\Lambda)} \sim \Bigg(\sum_{k \in \mathbb{Z}^2} |1-g_\tau(k)|^2 \, |\hat{w}(k)|^2 \, \langle k \rangle^{-2+2\delta}\Bigg)^{1/2} \rightarrow 0 \quad \mbox{as}\quad \tau \rightarrow \infty\,,
\end{equation*}
by recalling that $\delta=\epsilon/2$ and using Definition \ref{endpoint_admissible_w} (iv) as well as the dominated convergence theorem. This proves claim (vii).
\end{proof}


\begin{thebibliography}{References}
\bibitem{AFP} Z.~{Ammari}, M.~{Falconi}, B.~{Pawilowski},  \emph{On the rate of convergence for the mean-field approximation of many-body quantum dynamics},
Comm. Math. Sci. \textbf{14} (2016) 1417--1442. 
\bibitem{AN} Z.~{Ammari}, F.~{Nier}, \emph{Mean-field limit for bosons and propagation of Wigner measures}, J. Math. Phys. \textbf{50} (2009), 042107.
\bibitem{Bach} V.~{Bach}, \emph{Ionization energies of bosonic Coulomb systems}, Lett. Math. Phys. {\bf 21} (1991), 139--149.
\bibitem{BKS}
G.~{Ben Arous}, K.~{Kirkpatrick}, B.~{Schlein}, \emph{A central limit
  theorem in many-body quantum dynamics}, Comm. Math. Phys. \textbf{321} (2013), 371--417.
\bibitem{BenPorSch_Review} N.~{Benedikter}, M.~{Porta}, B.~{Schlein}, \emph{Effective Evolution Equations from Quantum Dynamics}, Springer Briefs in Mathematical Physics, Springer, 2016.
\bibitem{BL} R.~{Benguria}, E. H.~{Lieb}, \emph{Proof of the Stability of Highly Negative Ions in the
Absence of the Pauli Principle}, Phys. Rev. Lett. {\bf 50} (1983), 1771--1774.
\bibitem{Benyi_Oh} A. ~{B\'{e}nyi}, T. ~{Oh}, \emph{The Sobolev inequality on the torus revisited}, Publ. Math. Debrecen \textbf{83} (2013), no. 3, 359--374.
\bibitem{BBCS1}C.~{Boccato}, C.~{Brennecke}, S.~{Cenatiempo}, B.~{Schlein}, \emph{Complete Bose-Einstein condensation in the Gross-Pitaevskii regime}, Comm. Math. Phys. \textbf{359} (2018), no. 3, 975--1026.
\bibitem{BBCS2}C.~{Boccato}, C.~{Brennecke}, S.~{Cenatiempo}, B.~{Schlein}, \emph{Optimal Rate for Bose-Einstein Condensation in the Gross-Pitaevskii Regime}, Preprint arXiv: 1812.03086.
\bibitem{BCS} C.~{Boccato}, S.~{Cenatiempo}, B.~{Schlein}, \emph{Quantum many-body fluctuations around nonlinear Schr\"odinger dynamics}, Ann. Inst. H. Poincar\'{e} Anal. Non Lin\'{e}aire \textbf{18} (2017), 113--191. 
\bibitem{B} J. ~{Bourgain}, \emph{Periodic nonlinear Schr\"{o}dinger equation and invariant measures}, Comm. Math. Phys. {\bf 166} (1994), 1--26.
\bibitem{Bourgain_ZS} J. ~{Bourgain}, \emph{On the Cauchy problem and invariant measure problem for the periodic Zakharov system}, Duke Math. J. {\bf 76} (1994), 175--202.
\bibitem{B1} J. ~{Bourgain}, \emph{Invariant measures for the 2D-defocusing nonlinear Schr\"odinger equation}, Comm. Math. Phys. {\bf 176} (1996), 421-445. 
\bibitem{Bourgain_1997} J. ~{Bourgain}, \emph{Invariant measures for the Gross-Pitaevskii equation}, J. Math. Pures 
Appl. {\bf 76} (1997), 649--702.
\bibitem{B3} J. ~{Bourgain}, \emph{Invariant measures for {NLS} in infinite volume}, Comm. Math. Phys.  {\bf 210} (2000), 605--620. 
\bibitem{BourgainBulut} J.~{Bourgain}, A.~{Bulut}, \emph{Gibbs measure evolution in radial nonlinear wave and Schr\"{o}dinger equations on the ball}, C. R. Math. Acad. Sci. Paris {\bf 350} (2012), 571--575.
\bibitem{BourgainBulut2} J. ~{Bourgain}, A.~{Bulut}. 
\emph{Almost sure global well posedness for the radial nonlinear Schr\"{o}dinger equation on the unit ball II: the 3D case}, J. Eur. Math. Soc. (JEMS).  {\bf 16} (2014), 1289--1325.
\bibitem{BourgainBulut4} J.~{Bourgain}, A. ~{Bulut}, \emph{Almost sure global well posedness for the radial nonlinear Schr\"{o}dinger equation on the unit ball I: the 2D case},  Ann. Inst. H. Poincar\'{e} Anal. Non Lin\'{e}aire {\bf31} (2014), 1267--1288. 
\bibitem{BNNS} C. ~{Brennecke}, P.T. ~{Nam}, M. ~{Napiorkowski}, B.~{Schlein}, \emph{Fluctuations of $N$-particle quantum dynamics around the nonlinear Schr\"{o}dinger equation}, Preprint arXiv: 1710.09743.
\bibitem{Brennecke_Schlein} C. ~{Brennecke}, B.~{Schlein}, \emph{Gross-Pitaevskii dynamics for Bose-Einstein condensates}, 
Anal. PDE \textbf{12} (2019), no. 6, 1513--1596. 
\bibitem{BrydgesSlade} D. ~{Brydges}, G. ~{Slade}, \emph{Statistical mechanics of the $2D$ focusing nonlinear Schr\"{o}dinger equation}, Comm. Math. Phys. {\bf 182} (1996), 485--504.
\bibitem{BSS} S.~{Buchholz}, C.~{Saffirio}, B.~{Schlein}, \emph{Multivariate central limit theorem in quantum dynamics}, J. Stat. Phys. \textbf{154} (2014), 113--152.
\bibitem{BurqThomannTzvetkov} N.~{Burq}, L.~{Thomann}, N.~{Tzvetkov},  \emph{Long time dynamics for the one dimensional non linear Schr\"{o}dinger equation}, Ann. Inst. Fourier (Grenoble) {\bf 63} (2013), 2137--2198. 
\bibitem{Cacciafesta_deSuzzoni1}
F. ~{Cacciafesta}, A.-S. ~{de Suzzoni},
\emph{Invariant measure for the Schr\"{o}dinger equation on the real line}, J. Funct. Anal. {\bf 269} (2015), 271--324.
\bibitem{CLS}
L.~{Chen}, J.~{Oon Lee}, B.~{Schlein}, \emph{Rate of convergence towards
  {H}artree dynamics}, J. Stat. Phys. \textbf{144} (2011), no.4, 872--903.
\bibitem{CP} T.~{Chen}, N.~{Pavlovi\'{c}}, \emph{The quintic NLS as the mean-field limit of a boson gas with three-body interactions}. \emph{J. Funct. Anal.} \textbf{260} (2011), 959--997.
\bibitem{XC} X.~{Chen}, \emph{Second order corrections to mean-field evolution for weakly interacting
bosons in the case of three-body interactions}, Arch. Rational Mech. Anal. \textbf{203} (2012), 455--497.
\bibitem{CH} X.~{Chen}, J.~{Holmer}, \emph{Focusing quantum many-body dynamics: the rigorous derivation of the $1D$ focusing cubic nonlinear Schr\"odinger equation},
Arch. Rational Mech. Anal. \textbf{221} (2016), 631--676.
\bibitem{CH2} X.~{Chen}, J.~{Holmer}, \emph{Focusing quantum many-body dynamics II: the rigorous derivation of the $1D$ focusing cubic nonlinear Schr\"odinger equation from 3D},
\emph{Anal.  PDE} \textbf{10--3} (2017), 589--633.
\bibitem{CH_2018} X.~{Chen}, J.~{Holmer}, \emph{The derivation of the $\mathbb{T}^3$ Energy-critical NLS from Quantum Many-body Dynamics}, Preprint arXiv: 1803.08082.
\bibitem{DengTzvetkovVisciglia} Y.~{Deng}, N.~{Tzvetkov}, N.~{Visciglia}, \emph{Invariant Measures and Long Time Behaviour for the Benjamin-Ono Equation III}, Comm. Math. Phys. {\bf 339} (2015), 815--857.
\bibitem{Duoandikoetxea} J. ~{Duoandikoetxea}, \emph{Fourier Analysis}, AMS Graduate Studies in Mathematics, Volume \textbf{29} (2001). Translated and revised by D. ~{Cruz-Uribe}, SFO.
\bibitem{EESY} A.~{Elgart}, L.~{Erd\H{o}s}, B.~{Schlein}, H.-T.~{Yau}, \emph{Gross-Pitaevskii Equation as the Mean Field Limit of Weakly Coupled Bosons}, Arch. Rational Mech. Anal. \textbf{179} (2006), 265--283. 
\bibitem{ES} A.~{Elgart}, B.~{Schlein},  \emph{Mean-field dynamics for boson
stars}, Comm. Pure Applied Math. \textbf{60} (2007), no. 4, 500--545.
\bibitem{ESY1} L.~{Erd\H{o}s}, B.~{Schlein}, H.-T.~{Yau}, \emph{Derivation of the Gross-Pitaevskii hierarchy for the dynamics of Bose-Einstein condensate}, Comm. Pure Appl. Math. \textbf{59} (2006), no. 12, 1659--1741.
\bibitem{ESY2} L.~{Erd\H{o}s}, B.~{Schlein}, H.-T.~{Yau}, \emph{Derivation of the cubic non-linear Schr\"{o}dinger equation from quantum dynamics of many-body systems}, Invent. Math. \textbf{167} (2007), no. 3, 515--614.
\bibitem{ESY3} L.~{Erd\H{o}s}, B.~{Schlein}, H.-T.~{Yau}, \emph{Rigorous derivation of the Gross-Pitaevskii equation}, Phys. Rev. Lett. \textbf{98} (2007), no.4, 040404.
\bibitem{ESY4} L.~{Erd\H{o}s}, B.~{Schlein}, H.-T.~{Yau}, \emph{Rigorous derivation of the Gross-Pitaevskii equation with a large interaction potential}, J. Amer. Math. Soc. \textbf{22} (2009), no. 4, 1099--1156.
\bibitem{ESY5} L.~{Erd\H{o}s}, B.~{Schlein}, H.-T.~{Yau}, \emph{Derivation of the Gross-Pitaevskii equation for the dynamics of Bose-Einstein condensate}, Ann. of Math. (2). \textbf{172} (2010), no. 1, 291--370.
\bibitem{EY}
L.~{Erd{\H{o}}s} and H.-T.~{Yau}, \emph{Derivation of the nonlinear {S}chr\"{o}dinger equation from a many-body {C}oulomb system}, Adv. Theor. Math. Phys. \textbf{5} (2001), no.~6, 1169--1205.
\bibitem{FSV} M. ~{Fannes}, H.~{Spohn}, A. ~{Verbeure}, \emph{Equilibrium states for mean field models}, J.
Math. Phys. \textbf{21} (1980), 355--358.
\bibitem{FKP} J.~{Fr\"ohlich}, A.~{Knowles}, A.~{Pizzo}, \emph{Atomism and quantization}, J.\ Phys.\ A: Math.\ Theor. \textbf{40} (2007), 3033--3045.
\bibitem{FrKnScSo1} J.~{Fr\"{o}hlich}, A.~{Knowles}, B.~{Schlein}, V.~{Sohinger},  \emph{Gibbs measures of nonlinear Schr\"{o}dinger equations as limits of many-body quantum states in dimensions $d \leq 3$}, Comm. Math. Phys. \textbf{356} (2017), no. 3, 883--980.
\bibitem{FrKnScSo2} J.~{Fr\"{o}hlich}, A.~{Knowles}, B.~{Schlein}, V.~{Sohinger}, \emph{A microscopic derivation of time-dependent correlation functions of the $1D$ cubic nonlinear Schr\"{o}dinger equation}, Preprint arXiv: 1703.04465.
\bibitem{FKS}
J.~{Fr\"{o}hlich}, A.~{Knowles}, S.~{Schwarz}, \emph{On the  mean-field limit of bosons with {C}oulomb two-body interaction}, Comm. Math. Phys. \textbf{288} (2009), no. 3, 1023--1059.
\bibitem{GLV1} G. ~{Genovese}, R. ~{Luc{\`{a}}}, D. ~{Valeri}, \emph{Gibbs measures associated to the integrals of motion of the periodic {dNLS}}, Sel. Math. New Ser. \textbf{22} (2016), no. 3, 1663--1702.
\bibitem{GLV2} G. ~{Genovese}, R. ~{Luc{\`{a}}}, D. ~{Valeri}, \emph{Invariant measures for the periodic derivative nonlinear Schr\"{o}dinger equation}, Preprint arXiv: 1801.03152.
\bibitem{GV}
J. ~{Ginibre}, G. ~{Velo}, \emph{The classical field limit of scattering theory
  for nonrelativistic many-boson systems. {I} and {II}}, Comm. Math. Phys. \textbf{66}
  (1979), no. 1, 37--76, and \textbf{68} (1979), no.~1, 45--68.
\bibitem{Glimm_Jaffe} J.~{Glimm}, A.~{Jaffe}, \emph{Quantum Physics. A Functional Integral Point of View}, Springer-Verlag, Second edition, 1987.
\bibitem{Golse_Review} F.~{Golse}, \emph{On the Dynamics of Large Particle Systems in the Mean Field Limit}, in 
Macroscopic and Large Scale Phenomena: Coarse Graining, Mean Field Limits and Ergodicity, 1-144, Springer , 2016.
\bibitem{GS}  P.~{Grech}, R.~{Seiringer}, \emph{The excitation spectrum for weakly interacting bosons in
a trap}, Comm. Math. Phys. \textbf{322} (2013),  559--591.
\bibitem{GM} M. ~{Grillakis}, M. ~{Machedon}, \emph{Beyond mean field: on the role of pair excitations in the evolution of condensates}, J. Fixed Point Theory Appl. \textbf{14} (2013), no.1, 91--111.
\bibitem{GMM1}
M.~{Grillakis}, M.~{Machedon}, D.~{Margetis}, \emph{Second-order corrections to 
mean-field evolution of weakly interacting bosons {I}}, Comm. Math. Phys. \textbf{294} 
(2010), no.1, 273--301.
\bibitem{GMM2}
M.~{Grillakis}, M.~{Machedon}, D.~{Margetis}, \emph{Second-order corrections to mean-field evolution of weakly interacting bosons {II}}, Adv. Math. \textbf{228} (2011), no. 3, 1788--1815.
\bibitem{Hardy} G. H. ~{Hardy}, \emph{Divergent Series}, Oxford at the Clarendon Press, 1949. 
\bibitem{H}
K.~{Hepp}, \emph{The classical limit for quantum mechanical correlation
  functions}, Comm. Math. Phys. \textbf{35} (1974), 265--277.
\bibitem{HS} S.~{Herr}, V. ~{Sohinger}, \emph{The Gross-Pitaevskii hierarchy on general rectangular tori}, Arch. Rational Mech Anal. \textbf{220} (2016), no. 3, 1119--1158. 
\bibitem{K} M. K.-H. ~{Kiessling}, \emph{The Hartree limit of Born's ensemble for the ground state of a bosonic atom or ion}, J. 
Math. Phys. {\bf 53} (2012), 095223.
\bibitem{KSS}
K.~{Kirkpatrick}, B.~{Schlein} and G.~{Staffilani}, \emph{Derivation of the two dimensional nonlinear
{S}chr\"odinger equation from many-body quantum dynamics}, Amer. J. Math. \textbf{133} (2011), no. 1, 91--130.
\bibitem{Knowles_Thesis} A.~{Knowles}, \emph{Limiting dynamics in large quantum systems}, ETH Z\"{u}rich Doctoral Thesis (2009). ETHZ e-collection 18517. 
\bibitem{KP}
A.~{Knowles}, P.~{Pickl}, \emph{Mean-field dynamics: singular potentials and rate of convergence}, Comm. Math. Phys. \textbf{298} (2010), no.1, 101--138.
\bibitem{LRS} J. ~{Lebowitz}, H. ~{Rose}, E. ~{Speer}, \emph{Statistical mechanics of the nonlinear Schr\"{o}dinger equation}, J. Stat. Phys. {\bf 50} (1988), 657--687.
\bibitem{LNR0} M. ~{Lewin}, P.T. ~{Nam}, N. ~{Rougerie}, \emph{Derivation of Hartree's theory for generic mean-field Bose systems}, Adv. Math. {\bf 254} (2014), 570-621.
\bibitem{Lewin_Nam_Rougerie} M. ~{Lewin}, P. T. ~{Nam}, N. ~{Rougerie}, \emph{Derivation of nonlinear Gibbs measures from many-body quantum mechanics}, J. \'{E}c. Polytech. Math. \textbf{2} (2015), 65--115.
\bibitem{Lewin_Nam_Rougerie_Survey1} M. ~{Lewin}, P. T. ~{Nam}, N. ~{Rougerie},
\emph{Bose Gases at Positive Temperature and Non-Linear Gibbs Measures}, Proceedings of the 18th ICMP, Santiago de Chile, July 2015, Preprint arXiv: 1602.05166.
\bibitem{Lewin_Nam_Rougerie2} M. ~{Lewin}, P. T. ~{Nam}, N. ~{Rougerie}, \emph{Gibbs measures based on 1D (an)harmonic oscillators as mean-field limits}, J. Math. Phys. \textbf{59} (2018), no.4, 041901.
\bibitem{Lewin_Nam_Rougerie3}  M. ~{Lewin}, P.T. ~{Nam}, N. ~{Rougerie}, \emph{Classical field theory limit of 2D many-body quantum Gibbs states}, Preprint arXiv: 1805.08370.
\bibitem{Lewin_Nam_Rougerie_Survey2} M. ~{Lewin}, P. T. ~{Nam}, N. ~{Rougerie}, \emph{The interacting 2D Bose gas and nonlinear Gibbs measures}, Oberwolfach Abstract, Preprint arXiv:1805.03506.
\bibitem{Lewin_Nam_Rougerie_Survey3} M. ~{Lewin}, P. T. ~{Nam}, N. ~{Rougerie}, \emph{Derivation of renormalized Gibbs measures from equilibrium many-body quantum Bose gases}, Contribution to Proceedings of the International Congress of Mathematical Physics, Montreal, Canada, July 23--28, 2018, Preprint arXiv: 1903.01271.
\bibitem{LNS} M.~{Lewin}, P.~T.~{Nam} and B.~{Schlein}, \emph{Fluctuations around Hartree states in the mean-field regime}, Amer. J. Math \textbf{137} (2015), 1613--1650. 
\bibitem{LNSS}  M.~{Lewin}, P.~T.~{Nam}, S. ~{Serfaty}, J.P. ~{Solovej}, \emph{Bogoliubov spectrum of interacting Bose gases}, Comm. Pure Appl. Math. \textbf{68} (2012), 413-471.
\bibitem{LS} E.H.~{Lieb}, R.~{Seiringer}.
\emph{Proof of {B}ose-{E}instein condensation for dilute trapped gases},
Phys. Rev. Lett. \textbf{88} (2002), 170409-1-4.
\bibitem{LSY} E.H.~{Lieb}, R. ~{Seiringer}, J. ~{Yngvason}, \emph{Bosons in a trap:
a rigorous derivation of the {G}ross-{P}itaevskii energy functional},
Phys. Rev A \textbf{61} (2000), 043602.
\bibitem{LY} E. H.~{Lieb}, H.-T.~{Yau}, \emph{The Chandrasekhar theory of stellar collapse as the limit of quantum mechanics}, Comm. Math. Phys. \textbf{112} (1987), 147--174.
\bibitem{McKean_Vaninsky1} H. P.~{McKean}, K. L.~{Vaninsky}, \emph{Action-angle variables for the cubic Schr\"{o}dinger equation}, Comm. Pure Appl. Math. \textbf{50} (1997), no. 6, 489--562.
\bibitem{McKean_Vaninsky2} H. P.~{McKean}, K. L.~{Vaninsky}, \emph{Cubic Schr\"{o}dinger: the petit canonical ensemble in action-angle variables}, Comm. Pure Appl. Math \textbf{50} (1997), no. 7, 593--622.
\bibitem{NORBS} A.~{Nahmod}, T.~{Oh}, L.~{Rey-Bellet}, G.~{Staffilani}. 
\emph{Invariant weighted Wiener measures and almost sure global well-posedness for the periodic derivative NLS}, J. Eur. Math. Soc. {\bf 14} (2012), 1275--1330.
\bibitem{NRBSS} A.~{Nahmod}, L.~{Rey-Bellet}, S.~{Sheffield}, G.~{Staffilani}, \emph{Absolute continuity of Brownian bridges under certain gauge transformations}, Math. Res. Lett. {\bf 18} (2011), 875--887.
\bibitem{NN} P.T. ~{Nam}, M. ~{Napiorkowski}, \emph{Bogoliubov correction to the mean-field dynamics of interacting bosons}, 
Adv. Theor. Math. Phys. \textbf{21} (2017), no. 3, 683--738. 
\bibitem{Nelson1} E. ~{Nelson}, \emph{The free Markoff field}, J. Funct. Anal. \textbf{12} (1973), 211--227.
\bibitem{Nevanlinna} F. ~{Nevanlinna}, \emph{Zur Theorie der asymptotischen Potenzreihen}, Ann. Acad. Sci. Fen. Ser. A \textbf{12}, no. 3, 1918--1919.
\bibitem{OQ}
T.~{Oh}, J.~{Quastel}, \emph{On invariant Gibbs measures conditioned on mass and
momentum}, J. Math. Soc. Japan \textbf{65} (2013), 13--35.
\bibitem{Rademacher_Schlein} S. ~{Rademacher}, B ~{Schlein}, \emph{Central Limit Theorem for Bose-Einstein Condensates}, Preprint arXiv: 1903.00365.
\bibitem{RW} G. A.~{Raggio}, R. F.~{Werner}, \emph{Quantum statistical mechanics of general mean field
systems}, Helv. Phys. Acta \textbf{62} (1989), 980--1003.
\bibitem{RS}
I.~{Rodnianski}, B.~{Schlein}, \emph{Quantum fluctuations and rate of convergence towards mean-field dynamics}, Comm. Math. Phys. \textbf{291} (2009), no.1, 31--61.
\bibitem{Sch_Review} B.~{Schlein}, \emph{Derivation of Effective Evolution Equations from Microscopic Quantum Dynamics}, Lecture notes, CMI 2008 Summer School on Evolution Equations. 
\bibitem{Simon74} B.~{Simon}, \emph{The $P(\Phi)_2$ Euclidean (Quantum) Field Theory}, Princeton Univ. Press, 1974.
\bibitem{Simon05} B. ~{Simon}, \emph{Trace ideals and their applications}, Amer. Math. Soc., Second edition, 2005.
\bibitem{S2} V.~{Sohinger}, \emph{A rigorous derivation of the defocusing nonlinear Schr\"{o}dinger equation on $\mathbb{T}^3$ from the dynamics of many-body quantum systems}, Ann. Inst. H. Poincar\'{e} Anal. Non Lin\'{e}aire. \textbf{32} (2015), no. 6, 1337--1365.
\bibitem{Sokal} A.D.~{Sokal}, \emph{An improvement of Watson's theorem on Borel summability}, J. Math. Phys. \textbf{21} (1980), no. 2, 261--263.
\bibitem{S} J. P. ~{Solovej}, \emph{Asymptotics for bosonic atoms}, Lett. Math. Phys. {\bf 20} (1990), 165--172.
\bibitem{2S} H. Spohn, \emph{Kinetic equations from Hamiltonian Dynamics}, Rev. Mod. Phys. \textbf{52} (1980), no. 3, 569--615.
\bibitem{Stein_Weiss} E. ~{Stein}, G. ~{Weiss}, \emph{Introduction to Fourier Analysis on Euclidean Spaces}, Princeton University Press, Second edition, 1975.
\bibitem{TTz}
L. ~{Thomann}, N. ~{Tzvetkov}, \emph{Gibbs measure for the periodic derivative nonlinear {S}chr\"odinger equation}, Nonlinearity \textbf{23} (2010), 2771--2791.
\bibitem{Tz1}
N.~{Tzvetkov},
\emph{Invariant measures for the Nonlinear Schr\"{o}dinger equation on the disc},
Dynamics of PDE, {\bf 2} (2006) 111--160.
\bibitem{Tz}
N.~{Tzvetkov}, \emph{Invariant measures for the defocusing nonlinear Schr\"odinger equation}, Ann. Inst. Fourier (Grenoble) \textbf{58} (2008), 2543--2604.
\bibitem{Watson} G.N. ~{Watson}, \emph{A Theory of Asymptotic Series}, Philos. Trans. Soc. London, Ser. A \textbf{211} (1912).
\bibitem{Zhidkov} P.E. ~{Zhidkov}, \emph{An invariant measure for the nonlinear Schr\"{o}dinger equation} (Russian) Dokl. Akad. Nauk SSSR {\bf 317} (1991) 543--546; translation in Soviet Math. Dokl. {\bf 43}, 431--434.
\end{thebibliography}
\end{document}